%% file: Main.tex
\pdfoutput=1\relax
\documentclass[reqno]{amsart}

\input{Preamble.tex}

\usepackage{etoolbox}
\newtoggle{draft}
\togglefalse{draft}

\iftoggle{draft} {
\usepackage[margin=1.5in]{geometry}
\newcommand{\NB}[1]{\todo[color=gray!40]{#1}}
\newcommand{\TODO}[1]{\todo[color=red]{#1}}
}{ 
\usepackage[margin=1.5in]{geometry}
\newcommand{\NB}[1]{}
\newcommand{\TODO}[1]{}
\renewcommand{\todo}[1]{}
\renewcommand{\todo}[1]{}
}
\title{On the boundaries of highly connected, almost closed manifolds}

\author{Robert Burklund}
\address{Department of Mathematics, Massachusetts Institute of Technology, Cambridge, MA, USA}
\email{burklund@mit.edu}

\author{Jeremy Hahn} 
\address{Department of Mathematics, Massachusetts Institute of Technology, Cambridge, MA, USA}
\email{jhahn01@mit.edu}

\author{Andrew Senger}
\address{Department of Mathematics, Harvard University, Cambridge, MA, USA}
\email{senger@math.harvard.edu}

\begin{document}
\begin{abstract}
Building on work of Stolz, we prove for integers $0 \le d \le 3$ and $k>232$ that the boundaries of $(k-1)$-connected, almost closed $(2k+d)$-manifolds also bound parallelizable manifolds.
Away from finitely many dimensions, this settles longstanding questions of C.T.C. Wall, determines all Stein fillable homotopy spheres, and proves a conjecture of Galatius and Randal-Williams.
Implications are drawn for both the classification of highly connected manifolds and, via work of Kreck and Krannich, the calculation of their mapping class groups.

Our technique is to recast the Galatius and Randal-Williams conjecture in terms of the vanishing of a certain Toda bracket, and then to analyze this Toda bracket by bounding its $\mathrm{H}\mathbb{F}_p$-Adams filtrations for all primes $p$.
We additionally prove new vanishing lines in the $\mathrm{H}\mathbb{F}_p$-Adams spectral sequences of spheres and Moore spectra, which are likely to be of independent interest. Several of these vanishing lines rely on an Appendix by Robert Burklund, which answers a question of Mathew about vanishing curves in $\mathrm{BP} \langle n \rangle$-based Adams spectral sequences.
\end{abstract}
\maketitle

\tableofcontents
\vbadness 5000


\addtocontents{toc}{\protect\setcounter{tocdepth}{1}}
\section{Introduction}
\input{Introduction2.tex}
\addtocontents{toc}{\protect\setcounter{tocdepth}{2}}

\section{The classification of \texorpdfstring{$(n-1)$}{(n-1)}-connected \texorpdfstring{$(2n)$}{(2n)}-manifolds} \label{sec:classification}
\input{ManifoldClassification2.tex}
\section{Additional applications} \label{sec:Applications}
\input{Applications.tex}
\addtocontents{toc}{\protect\setcounter{tocdepth}{1}}

\section{Calculations with the Goodwillie TAQ tower} \label{sec:Goodwillie}
\input{Goodwillie.tex}

\section{\texorpdfstring{$\MO \langle 4n \rangle$}{MO<4n>} as a homotopy cofiber} \label{sec:ThomPushout}
\input{ThomPushout.tex}

\section{The remaining problem as a Toda bracket} \label{sec:Toda}
\input{Toda2.tex}




\section{The Galatius \& Randal-Williams conjecture} \label{sec:Finale}
\input{Finale.tex}
\section{The proof of Theorem \ref{thm:mainboundary}} \label{sec:FinishingStolz}
\input{StolzProof.tex}

\section{Synthetic spectra} \label{sec:SyntheticReview}
\input{SynRevIntro.tex}
\input{SynRevBigraded.tex}

\section{A synthetic Toda bracket} \label{sec:SyntheticToda}
\input{SyntheticToda.tex}

\section{Vanishing lines in synthetic spectra} \label{sec:AppendixVL}
\input{VanishingLines.tex}

\section{An Adams--Novikov vanishing line} \label{sec:ANvl}
\input{ANSSLine2.tex}

\section{Banded vanishing lines} \label{sec:BandVan}
\input{BandedGenericity2.tex}

\section{A banded vanishing line for \texorpdfstring{$Y$}{Y}} \label{sec:Y}
\input{Y-band.tex}

\section{The mod \texorpdfstring{$8$}{8} Moore spectrum} \label{sec:mod8}
\input{Mod8.tex}

\appendix

\section{Synthetic homotopy groups} 
\label{sec:synthetic-appendix}

\input{SynRevHom.tex}

\subsection{The proof of \Cref{thm:synthetic-Adams}}
\label{sec:proof-synth-Adams}\ 

\input{SynRevAdams.tex}

\subsection{Bigraded homotopy groups in the Toda range}
\label{subsec:syn-toda-range}\ 

\input{SyntheticTodaRange.tex}

\section{Vanishing curves in Adams spectral sequences, \texorpdfstring{\\}{} by Robert Burklund} \label{sec:Appendix}

\input{Appendix-s0.tex}

\subsection{Preliminaries and statements}
\label{app:prelim}\ 

\input{Appendix-s1.tex}

\subsection{Comparing vanishing lines}
\label{app:ABvl}\ 

\input{Appendix-s4.tex}

\subsection{Comparing vanishing lines (continued)}
\label{app:blocks}\ 

\input{Appendix-s5-E1.tex}

\subsection{The proof of \Cref{thm:app-main}}
\label{app:reductions}\ 

\input{Appendix-s2.tex}

\bibliographystyle{alpha}
\bibliography{bibliography}

\end{document}

%% file: Preamble.tex
\usepackage{verbatim}
\usepackage[textsize=scriptsize]{todonotes}
\usepackage{tikz-cd}
\usepackage{etoolbox}
\usepackage{etex}
\usepackage{hyperref}
\usepackage{cleveref}
\usepackage[T1]{fontenc}

\usepackage{chemarr}
\usepackage{amssymb}
\usepackage{amsmath}
\usepackage{comment}
\usepackage{cleveref}
\usepackage{mathtools}
\usepackage{rotating}
\usepackage{wrapfig}
\usepackage{outlines}
\usepackage{graphicx}

\theoremstyle{definition}
\newtheorem{nul}{}[section]
\newtheorem{dfn}[nul]{Definition}

\newtheorem{rmk}[nul]{Remark}

\newtheorem{cnstr}[nul]{Construction}
\newtheorem{cnv}[nul]{Convention}
\newtheorem{ntn}[nul]{Notation}
\newtheorem{exm}[nul]{Example}

\newtheorem{rec}[nul]{Recollection}

\newtheorem{wrn}[nul]{Warning}
\newtheorem{qst}[nul]{Question}
\newtheorem*{dfn*}{Definition}
\newtheorem*{axm*}{Axiom}
\newtheorem*{ntn*}{Notation}
\newtheorem*{exm*}{Example}
\newtheorem*{exr*}{Exercise}
\newtheorem*{int*}{Intuition}
\newtheorem*{qst*}{Question}
\newtheorem*{rmk*}{Remark}

\theoremstyle{plain}

\newtheorem{thm}[nul]{Theorem}
\newtheorem{prop}[nul]{Proposition}
\newtheorem{lem}[nul]{Lemma}

\newtheorem{cnj}[nul]{Conjecture}
\newtheorem{cor}[nul]{Corollary}
\newtheorem{fact}[nul]{Fact}
\newtheorem*{thm*}{Theorem}
\newtheorem*{prop*}{Proposition}
\newtheorem*{cor*}{Corollary}
\newtheorem*{lem*}{Lemma}
\newtheorem*{cnj*}{Conjecture}

\let\oldwidetilde\widetilde
\protected\def\widetilde{\oldwidetilde}

\DeclareMathOperator{\im}{\mathrm{im}}

\DeclareMathOperator{\fib}{\mathrm{fib}}
\DeclareMathOperator{\sBar}{\mathrm{Bar}}

\DeclareMathOperator{\Hom}{\mathrm{Hom}} 
\DeclareMathOperator{\End}{\mathrm{End}}

\DeclareMathOperator{\coker}{\mathrm{coker}}

\DeclareMathOperator{\Ss}{\mathbb{S}}
\DeclareMathOperator{\FF}{\mathbb{F}}
\DeclareMathOperator{\F}{\mathbb{F}}
\DeclareMathOperator{\NN}{\mathbb{N}}
\DeclareMathOperator{\A}{\mathcal{A}}

\DeclareMathOperator{\E}{\mathbb{E}}

\DeclareMathOperator{\MO}{\mathrm{MO}}

\DeclareMathOperator{\bS}{\mathbb{S}}

\DeclareMathOperator{\Ext}{\mathrm{Ext}}

\DeclareMathOperator{\Alg}{\mathrm{Alg}}
\DeclareMathOperator{\Sp}{\mathrm{Sp}}

\DeclareMathOperator{\J}{\mathcal{J}}

\DeclareMathOperator{\GL}{\mathrm{GL}}

\DeclareMathOperator{\cof}{\mathrm{cof}}
\DeclareMathOperator{\DD}{\mathbb{cofib}}
\DeclareMathOperator{\Sq}{\mathrm{Sq}}

\newcommand{\spOf}{\Sigma^{\infty}_+ \mathrm{O \langle 4n-1\rangle}}
\newcommand{\sOf}{\Sigma^{\infty} \mathrm{O \langle 4n-1\rangle}}
\newcommand{\Of}{\mathrm{O \langle 4n-1\rangle}}
\newcommand{\wt}{\widetilde}

\newcommand{\BOm}{\mathrm{BO \langle m \rangle}}
\newcommand{\BO}{\mathrm{BO}}
\newcommand{\Om}{\mathrm{O \langle m -1 \rangle}}
\newcommand{\Oo}{\mathrm{O}}
\newcommand{\U}{\mathrm{U}}
\newcommand{\BP}{\mathrm{BP}}
\newcommand{\Syn}{\mathrm{Syn}}
\newcommand{\HZ}{\mathrm{H}\mathbb{Z}}
\newcommand{\HFp}{\mathrm{H}\mathbb{F}_p}
\newcommand{\HFt}{\mathrm{H}\mathbb{F}_2}

\newcommand{\qoppa}{\kappa}
\newcommand{\sampi}{\lambda}

\newcommand{\ZZ}{\mathbb{Z}}
\newcommand{\Z}{\mathbb{Z}}

\DeclarePairedDelimiter\abs{\lvert}{\rvert}%
\makeatletter
\let\oldabs\abs
\def\abs{\@ifstar{\oldabs}{\oldabs*}}

\usepackage{tikz}
\usetikzlibrary{matrix,arrows,decorations}
\usepackage{tikz-cd}

\usepackage{adjustbox}

\let\oldtocsection=\tocsection
 
\let\oldtocsubsection=\tocsubsection
 
\let\oldtocsubsubsection=\tocsubsubsection
 
\renewcommand{\tocsection}[2]{\hspace{0em}\oldtocsection{#1}{#2}}
\renewcommand{\tocsubsection}[2]{\hspace{1em}\oldtocsubsection{#1}{#2}}
\renewcommand{\tocsubsubsection}[2]{\hspace{2em}\oldtocsubsubsection{#1}{#2}}


%% file: Introduction2.tex
For each integer $m \ge 5$, the Kervaire--Milnor \cite{KervaireMilnor} group of homotopy spheres $\Theta_{m}$ is the group under connected sum of $h$-cobordism classes of closed, smooth, oriented manifolds $\Sigma$ that are homotopy equivalent to the $m$-sphere $S^{m}$.  The Kervaire--Milnor exact sequence
$$0 \to \mathrm{bP}_{m+1} \to \Theta_{m} \to \mathrm{coker}(J)_{m},$$
expresses $\Theta_{m}$ in terms of the finite cyclic group $\mathrm{bP}_{m+1}$ and the mysterious, but amenable to methods of homotopy theory, finite group $\mathrm{coker}(J)_{m}$.  The subgroup $\mathrm{bP}_{m+1} \subset \Theta_{m}$ consists of all homotopy spheres that are the boundaries of parallelizable $(m+1)$-manifolds.  When $m$ is even, $\mathrm{bP}_{m+1}$ is trivial \cite[Theorem 5.1]{KervaireMilnor}.

A not-necessarily parallelizable, compact, oriented, smooth manifold $M$ is said to be \emph{almost closed} if its boundary $\partial M$ is a homotopy sphere.  The main theorem of our work is as follows:

\begin{thm} \label{thm:mainboundary}
Let $k>232$ and $0 \le d \le 3$ be integers.  Suppose that $M$ is a $(k-1)$-connected, almost closed $(2k+d)$-manifold.  Then the boundary 
$\partial M \in \Theta_{2k+d-1}$ has trivial image $$0=[\partial M] \in \mathrm{coker}(J)_{2k+d-1}.$$  In particular, $\partial M$ bounds a parallelizable manifold.
\end{thm}

\begin{rmk}
The bounds $k>232$ and $d \le 3$ can likely be improved (cf. \Cref{sec:open-problems} and \Cref{rmk:open-questions}).  However, there are examples (due to Frank \cite[Example 1]{Frank} and Stolz \cite[Satz 12.1]{StolzBook}, respectively) of:
\begin{itemize}
 \item A $3$-connected, almost closed $9$-manifold with boundary non-trivial in $\mathrm{coker}(J)_8$.
 \item A $7$-connected, almost closed $17$-manifold with boundary non-trivial in $\mathrm{coker}(J)_{16}$.
\end{itemize}
Theorem \ref{thm:mainboundary} demonstrates that these examples exhibit fundamentally low-dimensional phenomena.
\end{rmk}

\begin{rmk}
Many special cases of Theorem \ref{thm:mainboundary} were known antecedent to this work. 
Theorem B of \cite{StolzBook} summarizes the prior state of the art, and our work can be viewed as the completion of a program by Stolz to answer questions raised by Wall in \cite{Wall62,Wall67}.  Our theorem is new when $d=0$ and $k \equiv 0$ mod $4$, when $d=1$ and $k \equiv 1$ mod $8$, when $d=2$ and $k \equiv 3$ mod $4$, and when $d=3$ and $k \equiv 0$ mod $4$.
\end{rmk}

Theorem \ref{thm:mainboundary} is most interesting in the case $d=0$, where it was previously unknown for $k \equiv 0$ modulo $4$.
Work of Stolz \cite[Lemma 12.5]{StolzBook} reduces this case of our main theorem to the following result:

\begin{thm}[Conjecture of Galatius and Randal-Williams] \label{thm:intro-main}
Let $\MO \langle 4n \rangle$ denote the Thom spectrum of the canonical map
$$\tau_{\ge 4n} \mathrm{BO} \to \mathrm{BO},$$
where $\tau_{\ge 4n} \mathrm{BO}$ denotes the $(4n-1)$-connected cover of $\mathrm{BO}$.
For all $n>31$, the unit map
$$\pi_{8n-1} \bS \to \pi_{8n-1} \MO\langle 4n \rangle$$
is surjective, with kernel exactly the image of the $J$-homomorphism
$$\pi_{8n-1} \mathrm{O} \to \pi_{8n-1} \bS.$$
\end{thm}

We label Theorem \ref{thm:intro-main} a conjecture of Galatius and Randal-Williams since it is, when $n>31$, equivalent to Conjectures A and B of their work \cite{GRAbelianQuotients}.
Theorem \ref{thm:intro-main} allows us to improve the bound $k>232$ in the $d=0$ case of Theorem \ref{thm:mainboundary}.  For details, see \Cref{thm:strongmainboundary}.

\begin{rmk} \label{rmk:Qdef}
Much of Theorem \ref{thm:intro-main} is classical: the surjectivity statement follows from surgery as in \cite[Theorem 6.6]{KervaireMilnor}, while the Pontryagin-Thom correspondence guarantees that the image of the $J$-homomorphism is contained in the kernel of the unit map $\pi_{8n-1} \bS \to \pi_{8n-1} \MO\langle 4n\rangle$.  The difficult point is to prove that the kernel of this unit map contains \emph{only} the image of $J$.

A priori, there could be additional elements in this kernel, and the concern has a geometric interpretation.  Let $\Sigma_Q \in \Theta_{8n-1}$ denote the boundary of the manifold obtained by plumbing together two copies of the $4n$-dimensional linear disk bundle over $S^{4n}$ that generates the image of $\pi_{4n} \mathrm{BSO}(4n-1)$ in $\pi_{4n}\mathrm{BSO}(4n)$. Theorem \ref{thm:intro-main} is equivalent to the claim that, for $n>31$, the class $[\Sigma_Q] \in \mathrm{coker}(J)_{8n-1}$ is trivial \cite[Lemma 10.3]{StolzBook}.
\end{rmk}

Our proof of Theorem \ref{thm:intro-main} follows a general strategy due to Stolz \cite{StolzBook}, which he applied to prove some cases of Theorem \ref{thm:mainboundary}.  For each prime number $p$ we compute a lower bound on the $\mathrm{H}\mathbb{F}_p$-Adams filtrations of classes in the kernel of the unit map
$$\pi_{8n-1} \bS \to \pi_{8n-1} \MO \langle 4n \rangle.$$
Our lower bound is given in \Cref{thm:AdamsBound}, and it is one of the main technical achievements of this paper.  It is approximately double the bound obtained by Stolz in \cite[Satz 12.7]{StolzBook}, and we devote Sections \ref{sec:Goodwillie}-\ref{sec:Toda} and Sections \ref{sec:SyntheticReview}-\ref{sec:SyntheticToda} to its proof.

\begin{rmk}
A key portion of the argument for \Cref{thm:AdamsBound} takes place in Pstr\k{a}gowski's category of \emph{synthetic spectra} \cite{Pstragowski} (c.f. \cite{GIKR} for an alternative construction of $\mathrm{BP}$-synthetic spectra).  Other users of this category may be interested in our omnibus \Cref{thm:synthetic-Adams}, which relates Adams spectral sequences to synthetic homotopy groups.

Let us comment briefly on how synthetic technology allows us to prove stronger results than we could otherwise.  Many of our results, such as \Cref{thm:AdamsBound} and \Cref{prop:band-in-cofiber-seqs}, can be stated without reference to synthetic language.  While it should in principle be possible to prove such statements without synthetic technology, in practice we suspect such proofs would be technically demanding, difficult to verify, and vastly expand the length of the paper.


The most delicate point in the paper is \Cref{cnstr:synth-toda-diagram}, which bounds the Adams filtration of the class $w$ constructed in \Cref{lem:toda-main}. To make this construction requires a \emph{simultaneous} solution to two difficulties. The first is that we must choose a certain nullhomotopy to be of sufficiently high Adams filtration. In the synthetic category it is both simple and natural to maintain control over the Adams filtration of a homotopy. The second is that we must reason about how $\mathbb{E}_\infty$ ring structures interact with Adams filtration, which we cleanly accomplish using the symmetric monoidal structure on Pstr\k{a}gowski's category. The authors found it challenging to rigorously address both of these points, and their interaction, without the synthetic category. 

\end{rmk}

To make effective use of \Cref{thm:AdamsBound}, and also to prove the remaining cases of Theorem \ref{thm:mainboundary}, we need to explicitly understand all elements of $\pi_*(\mathbb{S}^{\wedge}_{p})$ of large $\mathrm{H}\mathbb{F}_p$-Adams filtration.  This is a problem of significant independent interest in pure homotopy theory, so we summarize our new results as \Cref{thm:Burklund} and \Cref{thm:intro-mod-8} below. For the definition of the $\mu$-family, see \cite{AdamsJIV}, and note that we write $\mathbb{S}_p^{\wedge}$ to denote the $p$-completion of the sphere spectrum.

\begin{thm}[Burklund, proved as Theorem \ref{thm:app-main}] \label{thm:Burklund}
  For each prime number $p>2$ and each integer $k>0$, let $\Gamma_p(k)$ denote the largest Adams filtration attained by a class in $\pi_{k} \mathbb{S}^{\wedge}_p$ that is not in the image of $J$.
  Similarly, let $\Gamma_2(k)$ denote the largest Adams filtration attained by a class in $\pi_k \mathbb{S}^{\wedge}_2$ that is not in the subgroup generated by the image of $J$ and the $\mu$-family.
  \begin{enumerate}
  \item For any prime $p$, 
    $$ \Gamma_p(k) \leq \frac{(2p-1)k}{(2p-2)(2p^2-2)} + o(k),$$
    where $o(k)$ denotes a sublinear error term. 
  \item If $k>0$ is any integer, then
    $$ \Gamma_3(k) \leq \frac{25}{184}k + 20 + \ell(k),$$
    where $\ell(k)$ is $0$ unless $k+2 \equiv 0$ modulo $4$, in which case $\ell(k)$ is the $3$-adic valuation of $k+2$.
  \end{enumerate}
\end{thm}

This theorem is due solely to the first author, and is proved in Appendix \ref{sec:Appendix} at the end of the work.  Part (2) of Burklund's theorem, at the prime $p=3$, is essential to our proof of Theorem \ref{thm:intro-main}.  Sections \ref{sec:AppendixVL} and \ref{sec:ANvl} of the main paper develop the tools necessary to deduce part (2) of the theorem from a more precise version of part (1).
Experts in Adams spectral sequences will want to examine the introduction to Appendix \ref{sec:Appendix} for additional and more precise results, including a solution to a question of Mathew about vanishing curves in $\mathrm{BP} \langle n \rangle$-based Adams spectral sequences.

\begin{rmk}
Previous upper bounds for $\Gamma_p(k)$ were proved by Davis and Mahowald when $p=2$ \cite{DM3}, and by Gonz\'alez \cite{Gonzalez} for $p>3$.
We make much use of their bounds in this paper, which complement our own.
In particular, while Burklund proves better asymptotic behavior of $\Gamma_p(k)$ than implied by any previous work, the explicit constants of Davis, Mahowald and Gonz\'alez are more useful for our geometric applications.
At $p=3$, the best prior known bound for $\Gamma_3(k)$ is due to Andrews \cite{Andrews}, who in his thesis computed the entire $3$-primary Adams spectral sequence above a line of slope $1/5$.
Part $(2)$ of Burklund's theorem contains stronger information about the $3$-primary $\mathrm{E}_{\infty}$-page, at the cost of having nothing to say about earlier pages.
\end{rmk}

Our other major result, Theorem \ref{thm:intro-mod-8} below, applies only to $8$-torsion classes in $\pi_*(\bS)$.  When it applies, it is stronger than \Cref{thm:Burklund}.
\begin{thm}[Proved as Theorem \ref{thm:mod8-main-thm} in the main text] \label{thm:intro-mod-8}
    Let $C(8)$ denote the mod $8$ Moore spectrum, and let $F^s \pi_k (C(8)) \subseteq \pi_k (C(8))$ denote the subgroup of elements of $\HFt$-Adams filtration at least $s$.
Then, for $k \geq 126$, the image of the Bockstein map
    \[F^{\frac{1}{5} k + 15} \pi_k (C(8)) \to \pi_{k-1} (\Ss)\]
is contained in the subgroup of $\pi_{k-1} (\Ss)$ generated by the image of $J$ and the $\mu$-family.
\end{thm}
We devote Sections \ref{sec:BandVan}-\ref{sec:mod8} to the proof of Theorem \ref{thm:intro-mod-8}.

\begin{rmk}
At key points in the arguments for \cite[Theorems B \& D]{StolzBook}, Stolz applies an analog, for the mod $2$ Moore spectrum, of our Theorem \ref{thm:intro-mod-8}.
This analog is due to Mahowald.
While Mahowald announced the result in \cite{MahBull}, and it is also claimed in \cite{MahTokyo} and \cite[p.41]{DM3}, to the best of our knowledge no proof has appeared in print.  In \Cref{sec:mod8} we prove a version of Mahowald's result in order to close this gap in the literature.  We then study in turn the mod $4$ and mod $8$ Moore spectra in order to prove Theorem \ref{thm:intro-mod-8}, the full strength of which is necessary to conclude Theorem \ref{thm:mainboundary}.

These Moore spectra results are closely related to Mahowald and Miller's proofs \cite{Miller, MahJEHP} of the height $1$ telescope conjecture, and we record a quick proof of the height $1$ telescope conjecture at $p=2$ as Corollary \ref{cor:telescope}.
\end{rmk}

Before launching into our arguments, we use Sections \ref{sec:classification}-\ref{sec:Applications} to give four applications of the above theorems.  In brief, these applications consist of:
\begin{enumerate}
\item For $n>31$, a classification of smooth, $(4n-1)$-connected, closed $(8n)$-manifolds up to diffeomorphism.  This completes, away from finitely many exceptional dimensions, the classification of $(n-1)$-connected $(2n)$-manifolds sought after in Wall's 1962 paper \cite{Wall62}.

\item In dimensions larger than $247$, a classification of all Stein fillable homotopy spheres.  Away from finitely many exceptional dimensions, this answers a question raised by Eliashberg \cite[3.8]{ContactWorkshop} and proves a conjecture of Bowden, Crowley, and Stipsicz \cite[Conjecture 5.9]{SteinFillable}.

\item For $\ell>31$ and $g \ge 1$, a computation of the mapping class group of the manifold
$$\sharp^{g} \left(S^{4\ell-1} \times S^{4\ell-1} \right).$$
The computation follows from inputting our result into theorems of Kreck and Krannich \cite{Kreck,KrannichMappingClass}.  With additional input, due to Galatius--Randal-Williams and Krannich--Reinhold \cite{GRAbelianQuotients,KrannichCharacteristic}, we make further comments about the classifying space
$$\mathrm{BDiff}^+\left(\sharp^{g} \left(S^{4\ell-1} \times S^{4\ell-1} \right) \right).$$

\item The best known upper bounds for the exponents of the stable stems $\pi_{*}(\Ss_{(p)})$.
\end{enumerate}

\subsection{An outline of the paper}
\label{subsec:sketch}\ 

The proofs of our main theorems begin in Section \ref{sec:Goodwillie}.  We outline our strategy below:

\vspace{1ex}
\noindent \textbf{Sections \ref{sec:Goodwillie}-\ref{sec:Toda}:}
For $n \ge 3$ an integer, we begin our analysis of $\pi_{8n-1}(\MO \langle 4n \rangle)$.  Our main tool in these sections is the relative bar construction
$$\MO \langle 4n \rangle \simeq \bS \otimes_{\spOf} \bS.$$
The bar construction allows us to reduce our study of the Thom spectrum $\MO \langle 4n \rangle$ to a study of the suspension spectrum $\spOf$.  In Section \ref{sec:Goodwillie}, we study $\spOf$ by means of the Goodwillie tower of the identity in augmented $\mathbb{E}_\infty$-algebras.  The idea of applying the Goodwillie calculus is due to Tyler Lawson, and it neatly resolves the `Problem' that Stolz identifies in \cite[p. XIII]{StolzBook}.  In Section \ref{sec:ThomPushout} we describe a variant of the bar construction that is equivalent in the metastable range.  Finally, in Section \ref{sec:Toda}, we reduce the calculation of the unit map $\pi_{8n-1} \bS \to \pi_{8n-1}(\MO \langle 4n \rangle)$ to the calculation of a certain Toda bracket $w$.  We postpone further analysis of this Toda bracket to Section \ref{sec:SyntheticToda}.

\vspace{1ex}
\noindent \textbf{Sections \ref{sec:Finale}-\ref{sec:FinishingStolz}:} We prove Theorems \ref{thm:mainboundary} and \ref{thm:intro-main} in these two sections, using three results from later in the paper as black boxes.  In Section \ref{sec:Finale}, we prove Theorem \ref{thm:intro-main} using Theorems \ref{thm:Burklund} and \ref{thm:AdamsBound}.  Theorem \ref{thm:AdamsBound} is the main result of Section \ref{sec:SyntheticToda}, and it consists of a lower bound on the $\mathrm{H}\mathbb{F}_p$-Adams filtrations of the Toda bracket $w$.  In Section \ref{sec:FinishingStolz}, we give the proof of Theorem \ref{thm:mainboundary} assuming Theorem \ref{thm:intro-mod-8} as well as some results from Stolz's book \cite{StolzBook}.  The arguments in both Sections \ref{sec:Finale} and \ref{sec:FinishingStolz} are straightforward analogs of arguments from Stolz's book, and we are able to go farther than Stolz only because our three black box theorems are stronger than the results he references.

\vspace{1ex}
\noindent \textbf{Sections \ref{sec:SyntheticReview}-\ref{sec:SyntheticToda}:} In these sections we undertake an analysis of the Toda bracket $w$.
The definition of $w$ critically hinges on the following fact: given an element $x \in \pi_{4n-1} \bS$, there is a canonical nullhomotopy of $2x^2$ (since $\bS$ is $\mathbb{E}_\infty$ there is a canonical witness to the Koszul sign rule, or a homotopy between $x^2$ and $-x^2$, or a nullhomotopy of $2x^2$).
The interaction of this nullhomotopy with Adams spectral sequences has some history, going back to work of Kahn, Milgram, M\"{a}kinen, and Bruner \cite[VI]{BMMS} on $\cup_1$ operations in the Adams spectral sequence.
We do not know how to apply Bruner's work directly to our (somewhat more complicated) situation, but it is morally related.
Instead of relying on results of Bruner, we analyze the situation from scratch: here enters for the first time a major tool in our work, the recently developed category of \emph{synthetic spectra}.

Synthetic spectra were developed by Piotr Pstr\k{a}gowski in \cite{Pstragowski}.
They constitute a homotopy theory, or symmetric monoidal stable $\infty$-category, of formal Adams spectral sequences.
Lax symmetric monoidal functors $\nu$ and $\tau^{-1}$ to and from the $\infty$-category $\Sp$ of spectra allow for a particularly clear analysis of the interaction between Adams spectral sequences and $\mathbb{E}_\infty$-ring structures.

In Section \ref{sec:SyntheticReview} we recall Pstr\k{a}gowski's work and develop a few additional properties of synthetic spectra that we require.  In Section \ref{sec:SyntheticToda} we apply all of the theory thus far to bound the $\mathrm{H}\mathbb{F}_p$-Adams filtrations of $w$ for all primes $p$. 

\vspace{1ex}
\noindent \textbf{Sections \ref{sec:AppendixVL}-\ref{sec:ANvl}:} We begin the latter half of the paper, which aims to prove Theorems \ref{thm:Burklund} and \ref{thm:intro-mod-8}.
In Section \ref{sec:AppendixVL}, we give a general discussion of vanishing lines in $E$-based Adams spectral sequences.  We study the behavior of vanishing lines under extensions, and recover results of Hopkins--Palmieri--Smith \cite{HPS} in the language of synthetic spectra.
In Section \ref{sec:ANvl}, we combine the general theory of Section \ref{sec:AppendixVL} with concrete computations of Belmont \cite{Eva} and  Ravenel \cite{GreenBook} to deduce vanishing lines in Adams--Novikov spectral sequences.

\vspace{1ex}
\noindent \textbf{Section \ref{sec:BandVan}-\ref{sec:mod8}:} In Section \ref{sec:BandVan}, we introduce the notion of a \emph{$v_1$-banded vanishing line}.  While Adams spectral sequences are not zero above $v_1$-banded vanishing lines, elements above such lines are essentially $K(1)$-local and hence related to the image of $J$.  We show variants of the results of Section \ref{sec:AppendixVL}, in particular demonstrating that $v_1$-banded vanishing lines are preserved under extensions and cofibers of synthetic spectra.  In Section \ref{sec:Y}, we apply machinery of Haynes Miller \cite{Miller} to prove a $v_1$-banded vanishing line in the $\mathrm{H}\mathbb{F}_2$-based Adams spectral sequence for the spectrum $Y = C(2) \otimes C(\eta)$.  In more classical language this result is known to experts, and follows from combining Miller's tools with computational results of Davis and Mahowald \cite{v1ExtDavisMahowald}.  In Section \ref{sec:mod8}, we establish a $v_1$-banded vanishing line in the modified $\mathrm{H}\mathbb{F}_2$-Adams spectral sequence for the Moore spectrum $C(8)$ and conclude, in particular, Theorem \ref{thm:intro-mod-8}. 
%

\vspace{1ex}
\noindent \textbf{Appendix \ref{sec:synthetic-appendix}:} The first part of this appendix is devoted to a technical proof of \Cref{thm:synthetic-Adams}. The theorem provides the means to translate statements about $E$-based Adams spectral sequences into statements about $E$-based synthetic spectra, and vice-versa. The proofs in this section are mostly a matter of careful bookkeeping. 

The second part of the appendix contains a computation of the $\mathrm{H}\mathbb{F}_2$-synthetic homotopy groups of the $2$-complete sphere through the Toda range.  We find that this computation illustrates many of the subtleties of \Cref{thm:synthetic-Adams} and effectively demonstrates the process of moving between Adams spectral sequence information and synthetic information.

\vspace{1ex}
\noindent \textbf{Appendix \ref{sec:Appendix}:} This appendix, due solely to the first author, proves Theorem \ref{thm:Burklund} and settles Question 3.33 of \cite{Akhil}.
Classically, results similar to Theorem \ref{thm:Burklund} are proved in two independent steps via the study of $bo$-resolutions \cite{DM3, Gonzalez}. The first step establishes vanishing curves in $bo$-based Adams spectral sequences. The second (and more technically difficult) step relates the canonical $bo$- and $\HFp$-resolutions of the sphere. This appendix provides an improvement on the vanishing curve of the first step.

The main idea is a new, and surprisingly elementary, method of analyzing vanishing curves in $\BP \langle 1 \rangle$-based Adams spectral sequences.  More generally, using only the fact that $\tau_{<|v_{n+1}|} \BP \langle n \rangle \simeq \tau_{<|v_{n+1}|} \BP$, Burklund relates $\BP \langle n \rangle$-based Adams spectral sequences to $\BP$-based Adams spectral sequences.
Vanishing curves in $\BP$-based Adams spectral sequences are understood through a strong form of the Nilpotence Theorem of Devinatz, Hopkins, and Smith \cite{DHS}, which provides the key input necessary to prove \Cref{thm:Burklund}(1). At the prime $3$, the main result of Section \ref{sec:ANvl} provides the precise numerical control needed to deduce Theorem \ref{thm:Burklund}(2).

\subsection{Conventions}\

Beginning in Section \ref{sec:ThomPushout}, we fix an integer $n \ge 3$.  We use $\mathbb{S}$ to denote the sphere spectrum, $\mathbb{S}^n$ to denote the stable $n$-sphere, and $S^{n}$ to denote the unstable $n$-sphere.  For integers $k>0$, we use $\mathcal{J}_{k}$ to denote the image of $J$ subgroup of $\pi_{k} \mathbb{S}$.  All manifolds are smooth and oriented, and all diffeomorphisms are orientation-preserving.  Throughout the work, we freely use the language of $\infty$-categories as set out in \cite{HA,HTT}. In particular, all limits and colimits are taken in the homotopy invariant sense of \cite{HTT}.

\subsection{Acknowledgments}\

We thank Manuel Krannich for giving an inspiring lecture in the MIT Topology Seminar in March $2019$, as well as for comments on a draft of the paper.  It is because of Manuel's encouragement that we became aware of, and began to work on, the problems we solve here.
We thank Mike Hopkins for conversations regarding the Toda bracket manipulations of Section \ref{sec:Toda}. 
Mike is the PhD advisor of the first author, was the advisor of the middle author, and is always an invaluable resource.
We thank Haynes Miller for his help and encouragement as the PhD advisor of the third author and postdoctoral mentor of the middle author.
We thank Mark Behrens for discussions regarding Section \ref{sec:Y}, Dexter Chua for identifying the sign in Theorem \ref{thm:synthetic-Adams}(1c), and Stephan Stolz for a useful conversation about bordism of almost closed manifolds.
We thank Araminta Amabel, Shaul Barkan, Diarmuid Crowley, Sander Kupers, Oscar Randal--Williams, John Rognes, Zhouli Xu, and Adela YiYu Zhang for helpful comments.

We thank the anonymous referees for their detailed comments.

The authors offer their greatest thanks to Tyler Lawson, who suggested the use of Goodwillie calculus in Section \ref{sec:Goodwillie} and offered detailed comments on a draft of the work.
While Tyler declined to be a coauthor, his influence on this paper is substantial. Tyler has been a personal mentor to the latter two authors for many years now---we would not be algebraic topologists without him.
We hope that his guidance, kindness, and keen mathematical vision raise up generations of homotopy theorists to come.

During the course of the work, the second author was supported by NSF Grant DMS-1803273, and the third author was supported by an NSF GRFP fellowship under Grant No. 1745302.


%% file: ManifoldClassification2.tex
Recall our convention that all manifolds are smooth and oriented, and all diffeomorphisms are orientation-preserving.
Interest in the boundaries of highly connected manifolds may be traced back to the late 1950s and early 1960s, due to relations with the following question:

\begin{qst} \label{qst:Classification}
Let $n \ge 3$ be an integer.
Is it possible to classify, or enumerate, all $(n-1)$-connected, closed $(2n)$-manifolds up to diffeomorphism?
\end{qst}

In \cite{MilnorTalk}, Milnor explains how his study of Question \ref{qst:Classification} led to the discovery of exotic spheres.  Major strides toward the classification were provided by C.T.C. Wall in \cite{Wall62}, who used Smale's $h$-cobordism theorem to classify $(n-1)$-connected, almost closed $(2n)$-manifolds.  We recall some of that work below.

\begin{rec} \label{rec:manifoldinvariants}
Suppose that $M$ is an $(n-1)$-connected, closed $(2n)$-manifold.
By Poincar\'e duality, the middle homology group 
$$H=H_n(M;\mathbb{Z})$$
must be free abelian of finite rank.
Associated to this middle homology group is a canonical bilinear, unimodular form, the intersection pairing
$$H \otimes H \to \mathbb{Z}.$$
The pairing is symmetric if $n$ is even and skew-symmetric if $n$ is odd---in general, one says that the pairing is \emph{$n$-symmetric}.

A slightly more delicate invariant, which depends on the smooth structure of $M$, is the \emph{normal bundle data}
$$\alpha:H \to \pi_{n-1}SO(n).$$
Following Wall \cite{Wall62}, we define this function $\alpha$ via a theorem of Haefliger \cite{Haefliger}.
If $n=3$, then $\pi_2 SO(3)$ is trivial, so there is nothing to define. In general, the Hurewicz theorem gives a canonical isomorphism $H \cong \pi_n(M)$.
For $n \ge 4$,  Haefliger's theorem implies that an element $x \in \pi_n(M)$ may be represented, uniquely up to isotopy, by an embedded sphere $x:S^n \to M$.
The normal bundle of this embedding is then $n$-dimensional, and so classified by an element $\alpha(x):S^n \to BSO(n)$.
\end{rec}

\begin{rec} \label{rec:Wallrelations}
Wall proved universal relationships between the intersection pairing $H \otimes H \to \mathbb{Z}$ and the function $\alpha$.
To describe them, let 
$$HJ:\pi_{n-1} SO(n) \to \mathbb{Z}$$ 
denote the composite of the unstable $J$-homomorphism $\pi_{n-1} SO(n) \to \pi_{2n-1} S^{n}$ and the Hopf invariant $\pi_{2n-1} S^{n} \to \mathbb{Z}$ (this composite may alternatively be described as $\pi_{n-1}$ applied to the projection $SO(n) \to S^{n-1}$).
Furthermore, let $\tau_{S^n} \in \pi_{n-1} SO(n) \cong \pi_n BSO(n)$ denote the map classifying the tangent bundle to the $n$-sphere.
Finally, for $x,y \in H$, let $xy$ denote the intersection pairing of $x$ with $y$, and let $x^2$ denote the intersection pairing of $x$ with itself.  

For all $x,y \in H$, Wall proved \cite[Lemma 2]{Wall62} the following relations:
\begin{equation} \label{eqn:Wallalpha2}
x^2 = HJ(\alpha(x)), \text{ and}
\end{equation}
\begin{equation} \label{eqn:Wallalpha}
\alpha(x+y) = \alpha(x)+\alpha(y) + (xy)(\tau_{S^n}).
\end{equation}
\end{rec}

\begin{dfn} \label{dfn:nspace}
Following \cite[p. 169]{Wall62}, we call a triple
$$I=(H,H \otimes H \to \mathbb{Z}, \alpha)$$
an \emph{$n$-space} whenever $H$ is a free, finite rank abelian group, $H \otimes H \to \mathbb{Z}$ is a unimodular, $n$-symmetric bilinear form, $\alpha$ is a map of pointed sets, and the triple $I$ satisfies the relations $(\ref{eqn:Wallalpha2})$ and $(\ref{eqn:Wallalpha})$ of \Cref{rec:Wallrelations}.
\end{dfn}

\begin{dfn}
Recollections \ref{rec:manifoldinvariants} and \ref{rec:Wallrelations} allow us to define a map of sets
$$
\begin{tikzcd}
\left\{ \begin{matrix} (n-1)\text{-connected,} \\ \text{closed }2n\text{-manifolds} \end{matrix} \right\}\hspace{-1.2ex}\raisebox{-1.5ex}{\bigg/}
\hspace{-.8ex}\raisebox{-2ex}{\text{diffeomorphism}}
\arrow{r}{\Psi} &  
\left\{\begin{matrix} n\text{-spaces} \\ \end{matrix} \right\}\hspace{-1.2ex}\raisebox{-1.5ex}{\bigg/}
\hspace{-.8ex}\raisebox{-2ex}{\text{isomorphism.}}
\end{tikzcd}
$$
%
An isomorphism of $n$-spaces is just an isomorphism of the underlying abelian group $H$ that respects both the bilinear form $H \otimes H \to \mathbb{Z}$ and the function $\alpha$.
\end{dfn}

The theorem below was first proved in \cite[p.170]{Wall62}.

\begin{thm}[Wall]
Suppose that $M$ and $N$ are two $(n-1)$-connected, closed $(2n)$-manifolds such that $\Psi(M)=\Psi(N)$.  Then there exists a homotopy sphere $\Sigma \in \Theta_{2n}$ such that $M \sharp \Sigma$ is diffeomorphic to $N$.
\end{thm}

\begin{rmk} \label{rmk:inertia}
Suppose that $M$ is an $(n-1)$-connected $2n$-manifold and that $\Sigma \in \Theta_{2n}$ is a homotopy sphere not diffeomorphic to $S^{2n}$.  One may ask whether the diffeomorphism types of $M$ and $M \sharp \Sigma$ differ.  Wall proved this to be the case whenever $n \not \equiv 0,1,4$ mod $8$ \cite[p.289]{Wall67} (in other words, when $n \not \equiv 0,1,4$ mod $8$, Wall proved that the inertia group of $M$ is trivial).
Building on work of Kosi\'{n}ski \cite{Kosinski}, Stolz expanded this to show  that, if $n \ge 106$ and $n \equiv 0$ mod $4$, then the diffeomorphism types of $M$ and $M \sharp \Sigma$ differ \cite[Theorem D]{StolzBook}.

Our theorems settle, at least in large dimensions, the remaining case $n \equiv 1$ mod $8$.
Indeed, work of Wall \cite[Theorem 10]{Wall67} proves that, if $M$ and $M \sharp \Sigma$ have the same diffeomorphism type, then $\Sigma$ is the boundary of an $(n-1)$-connected $(2n+1)$-manifold.
Using Theorem \ref{thm:mainboundary}, we may conclude in dimensions $2n+1 > 465$ that $\Sigma$ bounds a parallelizable manifold.
However, any $\Sigma \in \Theta_{2n}$ that bounds a parallelizable manifold must be standard, by \cite[Theorem 5.1]{KervaireMilnor}. In sufficiently large dimensions, it follows that the preimage under $\Psi$ of any triple is either empty, or consists of exactly $|\Theta_{2n}|$ different diffeomorphism types.
\end{rmk}

In light of the above theorem and remark, the complete enumeration of $(n-1)$-connected $(2n)$-manifolds is reduced, in sufficiently large dimensions, to the following question:

\begin{qst} \label{qst:image-of-psi}
What is the image of the map $\Psi$?  In other words, what combinations of middle homology group, intersection pairing, and normal bundle data arise from $(n-1)$-connected, closed $(2n)$-manifolds?
\end{qst}

\begin{rmk}
Wall solved Question \ref{qst:image-of-psi} for all $n \equiv 6$ mod $8$ \cite[Case 4]{Wall62}. More specifically, the solution may be read off from \cite[Proposition 5]{Wall62} together with the fact that $\pi_{n-1} SO =0$ for $n \equiv 6$ mod $8$ (so $\chi=0$, always).  Schultz \cite[Corollary 3.2]{Schultz} solved Question \ref{qst:image-of-psi} in the case $n \equiv 2$ mod $8$ (c.f. the discussion immediately subsequent to  \cite[Corollary 3.2]{Schultz}, about the work of Frank). Stolz \cite[Theorems B \& C]{StolzBook} settled the case $n \equiv 1$ mod $8$ when $n \ge 113$.
\end{rmk}

\begin{exm}
Suppose that $n>7$ is $\equiv 3,5,$ or $7$ modulo $8$.  In these cases, the function
$$\alpha:H \to \pi_{n-1} SO(n) \cong \mathbb{Z}/2\mathbb{Z}$$
satisfies the formula
$$\alpha(x+y) = \alpha(x) + \alpha(y) + (xy \text{ mod }2).$$
In other words, $\alpha$ is a quadratic refinement of the intersection pairing, and so one can associate an Arf--Kervaire invariant
$$\Phi(\alpha) \in \mathbb{Z}/2\mathbb{Z}.$$
Thus, if one is able to settle Question \ref{qst:image-of-psi}, then one is in particular able to answer the following question:
\end{exm}

\begin{qst} \label{qst:kervaire}
Suppose $n \equiv 3,5,$ or $7$ mod $8$.  Does there exist an $(n-1)$-connected, closed $(2n)$-manifold of Kervaire invariant $1$?
\end{qst}

\begin{rmk}
Barratt, Jones, and Mahowald constructed a $62$-dimensional manifold of Kervaire Invariant $1$ \cite{BJMKervaire} (cf. \cite{XuKervaire}).  Such manifolds are also known to exist in dimensions $2,6,14,$ and $30$.  On the other hand, deep work of Hill-- Hopkins--Ravenel \cite{HHR} proves that there is no manifold of Kervaire Invariant $1$ of dimension larger than $126$.
\end{rmk}

\begin{rmk}
When $n \equiv 3,5,$ or $7$ mod $8$, Wall completely reduced \cite[Lemma 5]{Wall62} Question \ref{qst:image-of-psi} to Question \ref{qst:kervaire}.  Question \ref{qst:kervaire} was settled by Brown and Peterson for $n \equiv 5$ mod $8$ \cite{BPKervaire}, by Browder for $n \equiv 3$ mod $8$ \cite{BrowderKervaire}, and by Hill--Hopkins--Ravenel for $n \equiv 7$ mod $8$ and $n>63$ \cite{HHR}.
\end{rmk}

For $n \ge 113$, the above work leaves Question \ref{qst:image-of-psi} open only when $n \equiv 0$ modulo $4$, and so we focus on this case now.

\begin{rec} \label{rec:chi}
Suppose $n \ge 12$, $n \equiv 0$ mod $4$, and $M$ is an $(n-1)$-connected, closed $(2n)$-manifold with middle homology group $H$ (the case $n=8$ has slightly different features, because of the existence of Hopf invariant $1$ elements and the octonionic projective plane).  Since the intersection pairing
$$H \otimes H \to \mathbb{Z}$$
is unimodular, it provides a canonical isomorphism between $H$ and its dual $\mathrm{Hom}(H,\mathbb{Z})$.
Wall notes \cite[Case 1]{Wall62} that the composite
$$H \stackrel{\alpha}{\longrightarrow} \pi_{n-1} SO(n) \longrightarrow \pi_{n-1} SO \cong \mathbb{Z}$$
is a homomorphism of abelian groups (by relation (2) in \Cref{rec:Wallrelations} together with the fact that spheres are stably parallelizable), and so via the intersection pairing determines a class $\chi(\alpha) \in H$.
In fact, the function $\alpha$ is entirely determined by the relations $(\ref{eqn:Wallalpha2})$ and $(\ref{eqn:Wallalpha})$ and the class $\chi(\alpha)$ \cite[p. 174]{Wall62}.
\end{rec}

\begin{cnstr}(\cite{Wall62}) \label{cnstr:Wall}
Let $n \ge 3$ denote any integer, and let 
$$I = (H,H \otimes H \to \mathbb{Z},\alpha)$$
denote an $n$-space.  From this data Wall constructs an $(n-1)$-connected, almost closed $(2n)$-manifold $N_I$ \cite[p.170]{Wall62}.
Wall further notes that there is an $(n-1)$-connected, closed $(2n)$-manifold $M$, with $\Psi(M)=I$, if and only if $\partial N_I$ is diffeomorphic to $S^{2n-1}$ \cite[p.177]{Wall62}.
\end{cnstr}

Suppose now that $n \ge 12$ is a multiple of $4$.
Given an $n$-space $I$, it remains to understand the boundary $\partial N_I \in \Theta_{2n-1}$.

By work of Brumfiel \cite{Brumfiel}, when $n \equiv 0$ mod $4$ the Kervaire--Milnor exact sequence splits to give a direct sum decomposition
$$\Theta_{2n-1} \cong \mathrm{bP}_{2n} \oplus \mathrm{coker}(J)_{2n-1}.$$
It thus suffices to analyze separately the images of $\partial N_I$ within $\mathrm{bP}_{2n}$ and $\mathrm{coker}(J)_{2n-1}$.
By applying a formula of Stolz \cite{StolzbP} and elaborating on work of Lampe \cite{Lampe}, Krannich and Reinhold \cite[Lemma 2.7]{KrannichCharacteristic} determined when the image of $\partial N_I$ vanishes in $\mathrm{bP}_{2n}$:

\begin{dfn} \label{dfn:manynumbers}
Let $m>2$ denote a positive integer.  Following \cite{KrannichCharacteristic}, we let 
\begin{itemize}
\item $B_{2m}$ denote the $(2m)^{\mathrm{th}}$ Bernoulli number.
\item $j_{m}$ denote $$j_{m}=\mathrm{denom}\left(\frac{\abs{B_{2m}}}{4m}\right),$$ the denominator of the absolute value of $\frac{B_{2m}}{4m}$ when written in lowest terms.
\item $a_m$ denote $1$ if $m$ is even and $2$ if $m$ is odd.
\item $\sigma_m$ denote the integer $$\sigma_m=a_m 2^{2m+1}(2^{2m-1}-1)\mathrm{num}\left(\frac{\abs{B_{2m}}}{4m}\right).$$
\item $c_m$ and $d_m$ denote integers such that
$$c_m \mathrm{num}\left(\frac{\abs{B_{2m}}}{4m}\right) + d_m \mathrm{denom}\left(\frac{\abs{B_{2m}}}{4m}\right) = 1.$$
\end{itemize}
If $m=2k>4$ is an even integer, we additionally follow \cite[Lemma 2.7]{KrannichCharacteristic} and let $s(Q)_{2k}$ denote the integer
$$s(Q)_{2k} = \frac{-1}{8 j^2_{k}}
\left( \sigma_k^2 + a_k^2 \sigma_{2k}\mathrm{num}\left(\frac{\abs{B_{2k}}}{4k}\right)\right)
\left( c_{2k} \mathrm{num}\left(\frac{\abs{B_{2k}}}{4k}\right) + 2(-1)^{k} d_{2k} j_k \right).
$$
\end{dfn}

\begin{thm}[Lampe, Krannich--Reinhold] Suppose $n \ge 12$ is a multiple of $4$, and let $I$ denote an $n$-space.  Then the boundary $\partial N_I$ has trivial image in $\mathrm{bP}_{2n}$ if and only if
$$\frac{\mathrm{sig}}{8}+\frac{\chi(\alpha)^2}{2}s(Q)_{n/2} \equiv 0 \text{ modulo } \frac{\sigma_{n/2}}{8}.$$
Here, $\mathrm{sig}$ denotes the \emph{signature} of the intersection form, and $\chi(\alpha)^2$ refers to the product of $\chi(\alpha)$ with itself via the intersection form.
\end{thm}

\begin{proof}
See \cite[Section 2]{KrannichCharacteristic} and \cite[Section 3.2.2]{KrannichMappingClass}.
\end{proof}

We thus obtain, as a consequence of our work in this paper, the following result:

\begin{thm} \label{thm:classification}
Suppose $n>124$ is divisible by $4$.  Then there exists an $(n-1)$-connected, closed $(2n)$-manifold with middle homology group $H$, intersection pairing $H \otimes H \to \mathbb{Z}$, and normal bundle data $\alpha:H \to \pi_{n-1}SO(n)$ if and only if the following conditions both hold:
\begin{enumerate}
\item The collection $(H, H \otimes H \to \mathbb{Z},\alpha)$ forms an $n$-space in the sense of \Cref{dfn:nspace}.
\item The relation 
$$\frac{\mathrm{sig}}{8}+\frac{\chi(\alpha)^2}{2}s(Q)_{n/2} \equiv 0 \text{ modulo } \frac{\sigma_{n/2}}{8}$$
 is satisfied, where $\mathrm{sig}$ denotes the signature of the intersection pairing and $\chi(\alpha)$ is defined as in \Cref{rec:chi}.
\end{enumerate}
If the conditions hold, so that a manifold exists, then the number of choices of such up to diffeomorphism is exactly $|\Theta_{2n}| = |\mathrm{coker}(J)_{2n}|,$ and they form a free orbit under the $\Theta_{2n}$ action by connected sum.
\end{thm}

\begin{proof}
The last sentence of the theorem follows, as in Remark \ref{rmk:inertia}, from Stolz's theorem \cite[Theorem D]{StolzBook}.  The remainder of the result follows by combining the above discussion with Theorem \ref{thm:strongmainboundary}.
\end{proof}

\begin{rmk}
  Our results also have implications for the classification of $(n-1)$-connected, closed $(2n+1)$-manifolds.  For $n \ge 8$, Wall classified all $(n-1)$-connected, almost closed $(2n+1)$-manifolds \cite{Wall67}.  Since $bP_{2n+1}$ is trivial, our Theorem \ref{thm:mainboundary} proves that the boundaries of Wall's almost closed manifolds are diffeomorphic to $S^{2n}$ whenever $n>232$.  This was previously unknown for $n \equiv 1$ modulo $8$ \cite[Theorem B]{StolzBook}.  There follows a classification of $(n-1)$-connected, closed $(2n+1)$-manifolds up to connected sum with a homotopy sphere.

  The problem of determining the inertia groups is somewhat subtle,
  but tractable \cite[Theorem D]{StolzBook}.
  In later work which builds on the methods developed in this paper the authors have analyzed boundary spheres and inertia groups substantially deeper into the metastable range \cite{burklund2020inertia}.
\end{rmk}


%% file: Applications.tex
\subsection{The classification of Stein fillable homotopy spheres}
\

\input{SteinFillable.tex}

\subsection{Calculations of mapping class groups}
\

\input{MappingClassGroups.tex}

\subsection{Bounds on the exponent of \texorpdfstring{$\mathrm{coker}(J)$}{coker(J)}}
\label{ssec:exp}\

\input{Exponents.tex}


%% file: SteinFillable.tex
Recall that a \emph{Stein domain} is a compact, complex manifold with boundary, such that the boundary is a regular level set of a strictly plurisubharmonic function (see, e.g., \cite[p.1-3]{SteinFillable} or \cite{SteinBook}).  The boundaries of Stein domains are naturally equipped with contact structures.  A contact $(2q+1)$-manifold $M$ is \emph{Stein fillable} if it may be realized as the boundary of a Stein domain.

Eliashberg has raised the question \cite[3.8]{ContactWorkshop} of which homotopy spheres $\Sigma \in \Theta_{2q+1}$ admit Stein fillable contact structures.  Eliashberg explicitly noted in \cite[3.8]{ContactWorkshop} that such $\Sigma$ necessarily bound $q$-connected, almost closed $(2q+2)$-manifolds, and that this might already be restrictive. For the proof that such $\Sigma$ must bound $q$-connected, almost closed $(2q+2)$-manifolds, see \cite[Proof of Theorem 5.4]{SteinFillable} and \cite[1.2.2]{EliTopChar}.

Bowden, Crowley, and Stipsicz took up Eliashberg's question, and applied Wall and Schultz's work \cite{Wall62, Wall67, Schultz} to settle it when $q \ne 9$ and $q+1 \not \equiv 0$ modulo $4$ \cite[Theorem 5.4]{SteinFillable}.  We offer the following additional theorem, which answers all but finitely many cases of Conjecture 5.9 from \cite{SteinFillable}:

\begin{thm} \label{app:SteinFillable}
Suppose that $q>123$.  A homotopy sphere $\Sigma \in \Theta_{2q+1}$ admits a Stein fillable contact structure if and only if $\Sigma \in \mathrm{bP}_{2q+2}$.  
\end{thm}

\begin{proof}
By the theorem of Bowden, Crowley, and Stipsicz, this is true whenever $q+1 \not \equiv 0$ modulo $4$ \cite[Theorem 5.4]{SteinFillable}.  We therefore suppose that $q+1 \equiv 0$ modulo $4$.  As we have discussed above, any Stein fillable $\Sigma \in \Theta_{2q+1}$ must bound a $q$-connected, almost closed $(2q+2)$-manifold.  It follows from Theorem \ref{thm:strongmainboundary} that, if $\Sigma \in \Theta_{2q+1}$ is Stein fillable, then $\Sigma \in \mathrm{bP}_{2q+2}$.  The converse is another result of Bowden, Crowley, and Stipsicz \cite[Proposition 5.3]{SteinFillable}.
\end{proof}


%% file: MappingClassGroups.tex
In Section \ref{sec:classification}, our theorems were used to classify $(n-1)$-connected $(2n)$-manifolds up to diffeomorphism.
We explain here how work of Galatius, Krannich, Kreck, Randal-Williams, and Reinhold connects our results to the study of diffeomorphisms of the manifold
$$W_g = \sharp^{g} \left( S^n \times S^n \right),$$
with $g \ge 1$.
This $(n-1)$-connected $(2n)$-manifold is a higher dimensional analog of a genus $g$ surface. 
As $g$ varies, the $W_g$ play a fundamental role in some approaches to the moduli spaces of manifolds, as outlined in the survey article \cite{GRHandbook}.
A discussion of how diffeomorphisms of $W_g$ relate to diffeomorphisms of other $(n-1)$-connected $(2n)$-maifolds appears below Theorem G in \cite{KrannichMappingClass}.

We consider in particular the classifying space
$$\mathcal{M}_{g} = \mathrm{BDiff}^{+}(W_g)$$
of orientation-preserving diffeomorphisms of $W_{g}$.
The first homotopy group $\pi_1(\mathcal{M}_{g})$ is an example of a higher dimensional mapping class group; it is the group of isotopy classes of orientation-preserving diffeomorphisms of $W_g$ (so, for example, $\pi_1 \mathcal{M}_0 \cong \Theta_{2n+1}$).  The first homology group $H_1(\mathcal{M}_{g};\mathbb{Z})$ is the abelianization of this mapping class group.  Higher cohomology groups, such as $H^2(\mathcal{M}_g;\mathbb{Z})$, include Miller--Morita--Mumford characteristic classes of bundles with fiber $W_{g}$.  At least for some values of $n$ and $g$, our theorems have something to say about each of these groups.

\begin{rec}
Suppose $n \ge 3$, and consider the mapping class group $\pi_1(\mathcal{M}_g)$.  This group was determined up to two extension problems by Kreck \cite{Kreck}.  Following Krannich \cite{KrannichMappingClass}, we write these extensions as
\begin{equation} \label{eq:extension1}
0 \to \Theta_{2n+1} \to \pi_1(\mathcal{M}_g) \to \pi_1(\mathcal{M}_g)/\Theta_{2n+1} \to 0
\end{equation} 
and
\begin{equation} \label{eq:extension2}
0 \to H_n(W_g) \otimes S\pi_n SO(n) \to \pi_1(\mathcal{M}_g)/\Theta_{2n+1} \to G_g \to 0.
\end{equation}
Here, $\Theta_{2n+1}$ is the Kervaire--Milnor group of homotopy $(2n+1)$-spheres, and $S\pi_n(SO(n))$ is the image of the stabilization map $S:\pi_n SO(n) \to \pi_n SO(n+1)$.  The group $G_g \subset \mathrm{GL}_{2g}(\mathbb{Z})$ is the subgroup of automorphisms of $H_n(W_g) \cong \mathbb{Z}^{2g}$ that are realized by diffeomorphisms.  It is explicitly described in \cite[Section 1.2]{KrannichMappingClass}.

The extension problems (\ref{eq:extension1}) and (\ref{eq:extension2}) have proven difficult to resolve, with special cases studied in \cite{Sato, Fried, Kry02, Kry03, Crowley, GRAbelianQuotients,KrannichMappingClass}.
\end{rec}

Recent work of Krannich resolves these extensions geometrically, in the case of $n$ odd, with the answers phrased in terms of certain elements $\Sigma_P, \Sigma_Q \in \Theta_{2n+1}$.  To be precise, Krannich proves for $n>7$ odd that the extension (\ref{eq:extension2}) splits \cite[Theorem A]{KrannichMappingClass}, and the extension $(\ref{eq:extension1})$ is classified \cite[Theorem B]{KrannichMappingClass} by a certain element
$$\frac{\mathrm{sgn}}{8} \Sigma_P + \frac{\chi^2}{2} \Sigma_Q \in \mathrm{H}^2\left(\pi_1(\mathcal{M}_g)/\Theta_{2n+1}; \Theta_{2n+1}\right).$$
The element $\Sigma_P \in \Theta_{2n+1}$ is a generator of $\mathrm{bP}_{2n+2}$.  The element $\Sigma_Q$ is $0$ whenever $n \equiv 1$ modulo $4$, and when $n \equiv 3$ modulo $4$ it is the boundary of the plumbing discussed in Remark \ref{rmk:Qdef}.  A consequence of our work here is a more explicit description of $\Sigma_Q$:

\begin{thm}
Suppose that $n>123$ is congruent to $3$ modulo $4$, and let $s(Q)_{(n+1)/2}$ denote the integer defined in Definition \ref{dfn:manynumbers}.  Then the element $\Sigma_Q \in \Theta_{2n+1}$ of \cite[Theorem B]{KrannichMappingClass} is equal to $s(Q)_{(n+1)/2} \Sigma_P$.  In particular, $\Sigma_Q$ is an element of the subgroup $\mathrm{bP}_{2n+2}$.
\end{thm}

\begin{proof}
  The last sentence of the theorem follows immediately from the definition of $\Sigma_Q$ (see e.g. the paragraph above \cite[Theorem B]{KrannichMappingClass}) and our Theorem \ref{thm:strongmainboundary}.  The exact formula $\Sigma_Q = s(Q)_{(n+1)/2} \Sigma_P$ is a consequence of \cite[Lemma 2.7]{KrannichCharacteristic}.
\end{proof}

The original motivation of Galatius and Randal-Williams in conjecturing \Cref{thm:intro-main} was to study the homology group $H_1(\mathcal{M}_g;\mathbb{Z})$.  It was understood in \cite[Theorem 1.3]{GRAbelianQuotients} and \cite[Corollary E]{KrannichMappingClass} that Theorem \ref{thm:intro-main} would lead to an explicit calculation of $H_1(\mathcal{M}_g;\mathbb{Z})$. By combining these results with our work, we conclude the following corollary:

\begin{cor}
Suppose that $n>123$ is congruent to $3$ modulo $4$ and $g \ge 3$.  Then
$$
H_1(\mathcal{M}_g;\mathbb{Z}) \cong  \left(\mathbb{Z}/4\mathbb{Z} \right) \oplus \mathrm{coker}(J)_{2n+1}.
$$
If $n>123$ is congruent to $3$ modulo $4$ and $g=2$, then
$$
H_1(\mathcal{M}_g;\mathbb{Z}) \cong  \left(\mathbb{Z}/4\mathbb{Z} \oplus \mathbb{Z}/2\mathbb{Z} \right) \oplus \mathrm{coker}(J)_{2n+1}.
$$
\end{cor}

\begin{rmk}
For the implications of our result when $g=1$, see \cite[Corollary E]{KrannichMappingClass}.
\end{rmk}

\begin{rmk} \label{rmk:MillerMoritaMumford}
As pointed out in \cite[Remark 6.2]{GRHandbook}, the universal coefficients formula expresses the finite group $H_1(\mathcal{M}_g;\mathbb{Z})$ as the torsion subgroup of $H^2(\mathcal{M}_g;\mathbb{Z})$. 
In \cite{KrannichCharacteristic}, Krannich and Reinhold calculated the torsion-free quotient of $H^2(\mathcal{M}_g;\mathbb{Z})$, for $g \ge 7$, in terms of $\Sigma_Q$.
\end{rmk}

It remains an interesting open question to determine the higher homology and cohomology groups of $\mathcal{M}_g$.  We expect that the methods of this paper have more to say about these groups, especially when $g \gg 0$ so that the work of Galatius and Randal-Williams \cite{GRAbelianQuotients, GRHandbook} applies (when $g \gg 0$, Galatius and Randal-Williams prove these groups isomorphic to the (co)homology groups of $\Omega^{\infty}$ of a Thom spectrum, placing the problem firmly in the realm of stable homotopy theory).

%% file: Exponents.tex
We end by giving an application, to stable homotopy theory, of Burklund's Theorem \ref{thm:Burklund}. Fix a prime number $p$.

\begin{dfn}
For each integer $n \ge 1$, Serre proved \cite{SerreFinite} that the $p$-local $n$th stable stem 
$\pi_n(\mathbb{S}_{(p)})$
is a finite $p$-group.  The \emph{exponent} of the $n$th stable stem, denoted here by
$$\mathrm{exp}(\pi_n(\mathbb{S}_{(p)})),$$
is the smallest integer $a$ such that all elements of $\pi_n(\mathbb{S}_{(p)})$ are $p^a$-torsion. 
\end{dfn}

Upper bounds for the exponent have been considered in several papers \cite{Ad:SHT, liul, Arlettaz, Gonz00, MathewExponent}.  For $p>3$, the best prior bounds are due to Gonz\'alez \cite[Corollary 4.1.4]{Gonz00}. As in Gonz\'alez's work, our bounds on the exponent are deduced from an upper bound on $\Gamma_p$.

\begin{thm}[Burklund] \label{thm:torsion-bound}
There is an inequality
$$\mathrm{exp}(\pi_n(\mathbb{S}_{(p)})) \le \frac{(2p-1)n}{(2p-2)(2p^2-2)} + o(n),$$
where $o(n)$ is the sublinear error term appearing in the statement of Theorem \ref{thm:Burklund}(1).
\end{thm}

\begin{proof}
At odd primes, Adams showed that the image of $J$ is a direct summand of $\pi_n(\mathbb{S}_{(p)})$ \cite{Ad:SHT, AdamsJIV}.  At the prime $2$, Adams and Quillen proved that the subgroup generated by the image of $J$ and the $\mu$-family is a direct summand \cite{QuillenAdamsConj}.  These papers also calculate the order of the image of $J$, from which it follows that the exponents of these summands grow logarithmically in $n$.

Suppose now that $x$ is an element of the complementary summand of $\pi_n(\mathbb{S}_{(p)})$, and let $\Gamma_p(n)$ denote the function from the statement of Theorem \ref{thm:Burklund}.  Since multiplication by $p$ raises $\mathrm{H}\mathbb{F}_p$-Adams filtration by at least $1$, Theorem \ref{thm:Burklund} implies that $p^{\Gamma_p(n)}x$ is either in the image of $J$, or, if $p=2$, in the subgroup generated by the image of $J$ and the $\mu$-family.  Since we assumed that $x$ is in the complementary summand, it follows that $p^{\Gamma_p(n)}x=0$.
\end{proof}

\begin{rmk}
  Like previous bounds on the torsion exponent of the stable stems,
  \Cref{thm:torsion-bound} is a linear bound.
  By contrast it is expected that $\mathrm{exp}(\pi_n(\mathbb{S}_{(p)}))$ grows sublinearly in $n$,
  though it remains unclear what the specific asymptotics of this function should be.
  The best known lower bound is logarithmic and comes from the image of $J$.
\end{rmk}



%% file: Goodwillie.tex
{\bf An overview of Sections \ref{sec:Goodwillie}-\ref{sec:Toda}}
Fix an integer $n \ge 3$.  Recall that $\MO \langle 4n \rangle$ is, by definition, the Thom spectrum \cite{BMMS,ABGHR} of the composite spectrum map
\begin{equation} \label{eq:ThomMap}
\tau_{\ge 4n} ko \to \tau_{\ge 1}ko \to \Sigma \mathrm{gl}_1(\mathbb{S}).
\end{equation}
Here, the spectrum map $\tau_{\ge 1}ko \to \Sigma \mathrm{gl}_1(\mathbb{S})$ is the infinite delooping of the real $J$ homomorphism $\mathrm{BO} \to \mathrm{BGL}_1(\mathbb{S})$, which exists because the one-point compactification functor $\mathrm{Vect}_{\mathbb{R}} \to \mathrm{Top}_*$ is symmetric monoidal.

Our first aim in this paper is to prove \Cref{thm:intro-main}, or equivalently to understand the unit map
$$\pi_{8n-1} \bS \to \pi_{8n-1} \MO \langle 4n \rangle.$$
To begin to do so, we fix some notation and recall more precisely how a Thom spectrum such as $\MO \langle 4n \rangle$ is defined.

\begin{dfn}
Taking $\Omega^{\infty+1}$ of the sequence (\ref{eq:ThomMap}) gives maps of infinite loop spaces
$$\Omega \Omega^{\infty} \tau_{\ge 4n} ko \to \mathrm{O} \to \GL_1(\mathbb{S}).$$
We use the notation $\Of$ to denote the infinite loop space $\Omega \Omega^{\infty} \tau_{\ge 4n} ko$.
The infinite loop map
$$\Of \to \GL_1(\mathbb{S})$$
gives rise, by the universal property of $\GL_1(\mathbb{S})$ \cite[Theorem 5.1]{ABGHR}, to a map of $\E_{\infty}$-rings
$$J_+:\spOf \to \bS.$$
Here, $\spOf$ is a spherical group ring, with underlying spectrum the suspension spectrum of the pointed space $\Of_+$.
Contracting $\Of$ to a point gives rise to a second $\mathbb{E}_\infty$-ring map
$$\epsilon:\spOf \to \bS,$$
the \emph{augmentation map.}  We use $J:\sOf \to \bS$ to refer to the composite 
$$\sOf = \fib(\epsilon) \longrightarrow \spOf \stackrel{J_+}{\longrightarrow} \bS,$$
which is a map of non-unital $\mathbb{E}_\infty$-rings.
\end{dfn}

\begin{cnstr} \label{cnstr:barconstruction}
By Definition 4.1 of \cite{ABGHR}, the spectrum $\MO \langle 4n \rangle$ can be presented as the geometric realization of the two-sided bar construction
$$\MO \langle 4n \rangle \simeq \left|\sBar(\bS,\spOf, \bS)_{\bullet} \right|,$$
which computes the relative tensor product $\bS \otimes_{\spOf} \bS$.
Here, the action of $\spOf$ on the left copy of $\bS$ is via $\epsilon$, and the action on the right copy of $\bS$ is via $J_+$.
\end{cnstr}

One may view our work in the first half of the paper as a computation of the map $\pi_{8n-1} \Ss \to \pi_{8n-1} \MO \langle 4n \rangle$ via the associated bar spectral sequence. We will see that the only possible differentials affecting $\pi_{8n-1} \MO \langle 4n \rangle$ are $d_1$-differentials and a single $d_2$-differential, so what we need to do may be summarized as follows:
\begin{enumerate}
    \item Compute the $\mathrm{E}_1$-page in the relevant range.
    \item Compute the relevant $d_1$-differentials.
    \item Compute the single relevant $d_2$-differential.
\end{enumerate}
Later in this section, we study the homotopy of $\sOf$ using the Goodwillie tower in augmented $\mathbb{E}_\infty$-ring spectra.  This is enough to resolve (1) and most of (2) above.

The key to proving \Cref{thm:intro-main} is to show that the single relevant $d_2$-differential vanishes.  Rather than using the language of spectral sequences, we will cast the computation of this $d_2$ as a computation of a certain Toda bracket $w$.  One of the main theorems of this paper, \Cref{thm:AdamsBound}, is a lower bound on the $\HFp$-Adams filtration of $w$ for each prime $p$.  Our goals in Sections \ref{sec:ThomPushout} and \ref{sec:Toda} will be to define $w$, to reduce Theorem \ref{thm:intro-main} to the computation of $w$, and to express $w$ in as convenient a form as possible.  In particular, \Cref{lem:toda-main} will express $w$ in a form amenable to Adams filtration arguments, though we postpone any serious discussion of Adams filtration to Sections \ref{sec:Finale} and later.

\subsection{The homotopy of $\sOf$}\

The main body of this section is a computation of the homotopy of the reduced suspension spectrum $\sOf$, for $n\ge 3$.  To explain our results, it is helpful to assign names to a few elements in $\pi_*(\sOf)$ and $\pi_*(\mathbb{S})$:

\begin{dfn} \label{dfn:x}
Let $$x \in \pi_{4n-1}(\sOf)$$
denote a generator of the bottom non-zero homotopy group of $\sOf$.  Since $\sOf$ is a non-unital $\E_\infty$-ring, we may speak of the class
$$x^2 \in \pi_{8n-2}(\sOf).$$
Finally, there is a class
$$J(x) \in \pi_{4n-1}(\bS),$$
defined as the composite
$S^{4n-1} \stackrel{x}{\longrightarrow} \sOf \stackrel{J}{\longrightarrow} \bS$
\end{dfn}
The remainder of this section will consist of proofs of the following facts:
\begin{enumerate}
\item The group $\pi_{8n-2}\sOf$ is isomorphic to $\mathbb{Z}/2\mathbb{Z}$, generated by the element $x^2$ of Definition \ref{dfn:x}.
    Furthermore, the group $\pi_{8n-1} \sOf$ is isomorphic to $\pi_{8n} ko \cong \mathbb{Z}$.
\item The element $xJ(x) \in \pi_{8n-2} \sOf$, defined using the right $\pi_*(\bS)$-module structure on $\pi_*(\sOf)$, is zero.
\item
Suppose $4n-1 \le \ell \le 8n-1$.  Then the image of the map
$$\pi_{\ell}(\sOf) \stackrel{J}{\longrightarrow} \pi_{\ell}(\mathbb{S})$$
is exactly $\mathcal{J}_{\ell}$.
\end{enumerate}

The first of these facts will be proved as \Cref{cor:gen-by-x2}, the second as \Cref{lem:xJ(x)null}, and the last as \Cref{thm:imJwelldefined}.
Our key tool will be the Goodwillie tower of the identity in augmented $\mathbb{E}_\infty$-ring spectra, the basic structure of which was worked out by Nick Kuhn \cite{KuhnTAQ}.
We thank Tyler Lawson for suggesting the relevance of this tower.

\begin{dfn}
Let $X$ be a spectrum and $m \ge 1$ a natural number.  We denote by $D_m(X)$ the extended power spectrum
$$D_m(X) \coloneqq (X^{\otimes m})_{h\Sigma_m}.$$
\end{dfn}

\begin{lem} \label{lem:Kuhn}
Suppose that $R$ is an $\E_\infty$-ring spectrum, equipped with an augmentation
$$\epsilon:R \to \bS.$$
Suppose further that the fiber of $\epsilon$ is $0$-connected.  Then there is a convergent tower of $\E_{\infty}$-ring spectra:
\begin{center}
  \begin{tikzcd} [column sep=tiny]
    & & D_m(\mathrm{TAQ}(R; \bS)) \ar[d] &
    & D_2(\mathrm{TAQ}(R; \bS)) \ar[d] &
    \mathrm{TAQ}(R; \bS) \ar[d] &
    \bS \ar[d, "\simeq"]  \\
    R \ar[r] &
    \cdots \ar[r] &
    P_m(R) \ar[r] &
    \cdots \ar[r] &
    P_2(R) \ar[r] &
    P_1(R) \ar[r] &
    P_0(R)
  \end{tikzcd}
\end{center}
such that the composite map $R \to P_0(R)$ is the augmentation map $\epsilon$.
\end{lem}

\begin{proof}
See \cite[Theorem 3.10]{KuhnTAQ}.
\end{proof}

\begin{cor} \label{cor:goodwillie}
There is a convergent tower of non-unital $\mathbb{E}_\infty$ ring spectra
\begin{center}
  \begin{tikzcd}[column sep = small]
    & & D_3(\Sigma^{-1} \tau_{\ge 4n} ko) \ar[d] &
    D_2(\Sigma^{-1} \tau_{\ge 4n} ko) \ar[d] &
    \Sigma^{-1} \tau_{\ge 4n} ko \ar[d, "\simeq"] \\
    \sOf \ar[r] &
    \cdots \ar[r] &
    Q_3 \ar[r] &
    Q_2 \ar[r] &
    Q_1 
  \end{tikzcd}
\end{center}
\end{cor}

\begin{proof}
We apply the previous lemma to $R=\spOf$ with its augmentation map $\epsilon$.  Note that, since
$$R \simeq \Sigma^{\infty}_+ \Omega^{\infty} \Sigma^{-1} \tau_{\ge 4n} ko,$$
we learn from \cite[Example 3.9]{KuhnTAQ} that $\mathrm{TAQ}(R;\bS) \simeq \Sigma^{-1} \tau_{\ge 4n} ko$.  The corollary follows by setting
$Q_i = \mathrm{fib}(P_i \to P_0).$
\end{proof}

\begin{lem}\label{lemm:d2-ko-comp}
    For $n \ge 3$, the bottom two homotopy groups of $D_2 (\Sigma^{-1} \tau_{\ge 4n} ko)$ are \[\pi_{8n-2} D_2 (\Sigma^{-1} \tau_{\ge 4n} ko) \cong \Z/2\Z\] and \[\pi_{8n-1} D_2 (\Sigma^{-1} \tau_{\ge 4n} ko) \cong 0.\]

    Moreover, the generator of $\Z/2\Z \cong \pi_{8n-2} D_2 (\Sigma^{-1} \tau_{\ge 4n} ko)$ survives in the spectral sequence associated to the tower of \Cref{cor:goodwillie} to detect $x^2 \in \pi_{8n-2} \sOf$.
\end{lem}

\begin{proof}
   There is a $4n$-connected map $\Ss^{4n-1} \to \Sigma^{-1} \tau_{\ge 4n} ko$ which induces an $(8n-1)$-connected map
	\[D_2 (\Ss^{4n-1}) \to D_2 (\Sigma^{-1} \tau_{\ge 4n} ko).\]
	Thus, there is an isomorphism 
	\[\pi_{8n-2} D_2 (\Ss^{4n-1}) \cong \pi_{8n-2} D_2 (\Sigma^{-1} \tau_{\ge 4n} ko)\]
	and a surjective map 
	\[\pi_{8n-1} D_2 (\Ss^{4n-1}) \twoheadrightarrow \pi_{8n-1} D_2 (\Sigma^{-1} \tau_{\ge 4n} ko).\]
	It therefore suffices to make the desired homotopy group computations for $D_2 (\Ss^{4n-1})$.

    There is an equivalence $D_2 (\Ss^{4n-1}) \simeq \Sigma^{4n-1}\mathbb{RP}^\infty _{4n-1}$, arising from the fact that 
		\[D_2(\Ss^{4n-1}) \simeq \Ss^{(4n-1)\rho}_{hC_2}\]
		is the Thom spectrum of the bundle $(4n-1)\mathbf{1} + (4n-1)\gamma$ over $BC_2 \simeq \mathbb{RP}^{\infty}$, and \cite[Proposition V.3.1]{BMMS} computes $\mathbb{RP}_{4n-1} ^{4n+1} \simeq \Ss^{4n-1} \cup_{2} \mathrm{e}^{4n} \cup_{\eta} \mathrm{e}^{4n+1}$. We therefore determine $\pi_{8n-2} D_2 (\Ss^{4n-1}) \cong \Z/2\Z$ and $\pi_{8n-1} D_2 (\Ss^{4n-1}) \cong 0$, as desired.

    Comparison with the Goodwillie tower of $\Sigma^{\infty} _+ \Omega^{\infty} \Ss^{4n-1}$, which recovers the Snaith splitting \cite{KuhnTAQ}, shows that the generator of $\pi_{8n-2} D_2 (\Sigma^{-1} \tau_{\ge 4n} ko)$ detects
    \[ x^2 \in \pi_{8n-2} \sOf. \qedhere\]
\end{proof}

\begin{cor} \label{cor:gen-by-x2}
For $n \ge 3$, the group $\pi_{8n-2}\sOf$ is a copy of $\mathbb{Z}/2\mathbb{Z}$, generated by the element $x^2$ of Definition \ref{dfn:x}.
    Furthermore, the group $\pi_{8n-1} \sOf$ is isomorphic to $\pi_{8n} ko$.
\end{cor}

\begin{proof}
    Note that $\tau_{\le 8n-1} D_k(\Sigma^{-1} \tau_{\ge 4n} ko)$ is trivial for $k >2$.  Thus, Corollary \ref{cor:goodwillie} and \Cref{lemm:d2-ko-comp} imply the existence of a long exact sequence
$$
\begin{tikzcd}
0 \arrow{r} & \pi_{8n-1}(\sOf) \arrow{r} \arrow[draw=none]{d}[name=Z, shape=coordinate]{} & \pi_{8n-1}(\Sigma^{-1} \tau_{\ge 4n} ko) \arrow[rounded corners,to path={ -- ([xshift=2ex]\tikztostart.east)|- (Z) [near end]\tikztonodes-| ([xshift=-2ex]\tikztotarget.west)-- (\tikztotarget)}]{dll} \\
\Z/2\Z \arrow{r}{x^2} & \pi_{8n-2}(\sOf) \arrow{r} & \pi_{8n-2}(\Sigma^{-1} \tau_{\ge 4n} ko).
\end{tikzcd}
$$

We now claim that the maps
$\pi_{k}(\sOf) \to \pi_{k}(\Sigma^{-1} \tau_{\ge 4n} ko)$ are surjective.  To see this, we note that these maps are $\pi_{k}$ of the map of spectra
$$\sOf \longrightarrow \Sigma^{-1} \tau_{\ge 4n} ko$$
that is adjoint to the identity homomorphism
$$\Of \stackrel{\simeq}{\longrightarrow} \Omega^{\infty} \Sigma^{-1} \tau_{\ge 4n} ko.$$
Indeed, it follows from the combination of \cite[Example 3.9]{KuhnTAQ} and \cite[Theorem 3.10(2)]{KuhnTAQ} that the map
$$\mathrm{hofib}(\epsilon:R \to \mathbb{S}) \to \mathrm{TAQ}(R;\mathbb{S})$$
induced by the tower of \Cref{lem:Kuhn} agrees for $R=\Sigma^{\infty}_+ \Omega^{\infty} X$ with the counit $\Sigma^{\infty}_+ \Omega^{\infty}X \to X$.
In particular, the map $\sOf \to \Sigma^{-1} \tau_{\ge 4n} ko$ admits a section after applying $\Omega^{\infty}.$ 

Identifying $\pi_{8n-2} (\Sigma^{-1} \tau_{\ge 4n} ko)$ with zero, we obtain isomorphisms
    \[\pi_{8n-1} (\sOf) \cong \pi_{8n-1} (\Sigma^{-1} \tau_{\ge 4n} ko) \cong \pi_{8n} ko\]
    and
    \[\Z/2\Z \cong \pi_{8n-2} (\sOf),\]
    with the latter sending the generator to $x^2$, as desired.
\end{proof}

\begin{cnstr}
Recall that the element $J(x) \in \pi_{4n-1} \bS$ was defined, in Definition \ref{dfn:x}, as the composite
$$S^{4n-1} \stackrel{x}{\longrightarrow} \sOf \stackrel{J}{\longrightarrow} \bS.$$
The right $\pi_*(\bS)$-module structure on $\pi_*(\sOf)$ allows us to define an element
$$xJ(x) \in \pi_{8n-2} \sOf.$$
\end{cnstr}

\begin{lem} \label{lem:xJ(x)null}
For $n \ge 3$, the element $xJ(x) \in \pi_{8n-2} \sOf$, defined using the right $\pi_*(\bS)$-module structure on $\pi_*(\sOf)$, is zero.
\end{lem}

\begin{proof}
By Corollary \ref{cor:gen-by-x2}, we know that $\pi_{8n-2} \sOf$ is isomorphic to $\Z/2\Z$, generated by the element $x^2$.
It follows that, if $xJ(x) \ne 0$, then $xJ(x) =x^2$.
In \Cref{rmk:x2-filt} we determine that the element $x^2$ has $\mathrm{H}\mathbb{F}_2$-Adams filtration $1$.
However, $xJ(x)$ has $\mathrm{H}\mathbb{F}_2$-Adams filtration at least that of $J(x)$.
Note now that $$x \in \pi_{4n-1} \sOf$$ is the suspension of an unstable class, and thus $J(x) \in \pi_{4n-1} \mathbb{S}$ is in $\J_{4n-1}$.
In particular, since $n\ge 3$, $J(x)$ has $\mathrm{H}\mathbb{F}_2$-Adams filtration larger than $1$.
\end{proof}

\subsection{The image of $J$ in $\pi_*(\mathbb{S})$}\ 

Classically, the phrase ``image of $J$'' in $\pi_{\ell}(\mathbb{S})$ refers to the image of the map
$$\pi_\ell(\mathrm{O}) \to \pi_{\ell}\GL_1(\mathbb{S}) \cong \pi_{\ell}(\mathbb{S}) \text{ for } \ell>0.$$
Recall that we use $\mathcal{J}_{\ell} \subseteq \pi_{\ell}(\mathbb{S})$ to denote this subset.

Unfortunately, we have introduced a second possible meaning of the phrase ``image of $J$,'' namely the image of the map
$$\pi_{\ell}(\sOf) \stackrel{J}{\longrightarrow} \pi_{\ell}(\mathbb{S}).$$

For general $\ell$, these two images may be different.  We prove here, however, that they are the same in our range of interest, and so no ambiguity has been introduced.
\begin{thm} \label{thm:imJwelldefined}
Suppose $n \ge 3$ and $4n-1 \le \ell \le 8n-1$.  Then the image of the map
$$\pi_{\ell}(\sOf) \stackrel{J}{\longrightarrow} \pi_{\ell}(\mathbb{S})$$
is exactly $\mathcal{J}_{\ell}$.
\end{thm}

\begin{proof}
This will automatically be true so long as every class in 
$\pi_{\ell} \sOf$ is the suspension of an unstable class in $\pi_{\ell} \Of.$
According to \Cref{cor:goodwillie}, there can be no difficulty unless $\ell=8n-1$ or $\ell=8n-2$.
The case $\ell=8n-1$ follows from \Cref{cor:gen-by-x2}.

To handle the case $\ell=8n-2$, we must check that $J(x^2)$ is an element of $\mathcal{J}_{8n-2}=0$.
Since $J$ is a non-unital ring map, $J(x^2)=J(x)^2$.  Now, it is known (by, e.g. \cite[Lemma 3]{Novikov}) that $\mathcal{J}_i \cdot \mathcal{J}_j \subseteq \mathcal{J}_{i+j}$ if $i,j>7$.  Using the hypothesis that $n \ge 3$, we have in short that $J(x)^2$ is an element of $\mathcal{J}_{8n-2} = 0$.
\end{proof}


%% file: ThomPushout.tex
In \Cref{cnstr:barconstruction}, we recalled that $\MO \langle 4n \rangle$ can be computed via a two-sided bar construction.  In this section we give a description of the bar construction, valid through a range of homotopy groups, which is particularly well-suited to the explicit identification of the $d_2$ in the bar spectral sequence as a Toda bracket.  Our main result is \Cref{thm:cofiber}.

\begin{cnstr} \label{cnstr:square}
Since $J$ is a map of non-unital rings in the homotopy category of spectra, the following diagram commutes up to homotopy:
$$
\begin{tikzcd}
\sOf^{\otimes 2} \arrow{r}{1 \otimes J} \arrow{d}{m} & \sOf \arrow{d}{J} \\
\sOf \arrow{r}{J} & \mathbb{S},
\end{tikzcd}
$$
where $m$ is the product map.
In the $\infty$-category of spectra, the fact that $J$ is a ring map is not a property, but actually additional structure.
In particular, there is a \emph{canonical homotopy} $a$ filling the above square:
$$
\begin{tikzcd}[column sep = huge]
\sOf^{\otimes 2} \arrow{dd}{m} \arrow{r}{1 \otimes J} & \sOf \arrow{dd}{J} \\ \\
\sOf \arrow{r}{J} \arrow[Leftrightarrow]{ruu}{a} & \bS.
\end{tikzcd}
$$
This homotopy $a$ may alternatively be viewed as a specific \emph{nullhomotopy} of $J \circ (1 \otimes J-m)$.
Let $P$ denote the homotopy cofiber of the map
$$
\begin{tikzcd}[column sep=huge]
\sOf^{\otimes 2} \arrow{r}{1 \otimes J-m} & \sOf
\end{tikzcd}
$$
Then the homotopy $a$ provides a canonical factorization
$$
\begin{tikzcd}[column sep=huge]
\sOf^{\otimes 2} \arrow{r}{1 \otimes J-m} & \sOf \arrow{d}{J} \arrow{r} & P \arrow[dashed]{dl} \\
& \mathbb{S}.
\end{tikzcd}
$$
\end{cnstr}

The main theorem of this section is the identification of the Thom spectrum $\MO \langle 4n \rangle$ with the cofiber of the above map $P \to \Ss$ in a range:

\begin{thm} \label{thm:cofiber}
Let $C$ denote the  cofiber of the map $P \to \mathbb{S}$ constructed above.
Then there is an equivalence of spectra $\tau_{\le 12n-2} C \simeq \tau_{\le 12n-2} \MO \langle 4n \rangle$.  Furthermore, the unit map
$$\tau_{\le 12n-2} \mathbb{S} \to \tau_{\le 12n-2} \MO \langle 4n \rangle$$ agrees with the natural map $\tau_{\le 12n-2} \mathbb{S} \to \tau_{\le 12n-2} C$.
\end{thm}

Before proving \Cref{thm:cofiber}, let us recall \Cref{cnstr:barconstruction}.  \Cref{cnstr:barconstruction} says that the spectrum $\MO \langle 4n \rangle$ can be calculated as the geometric realization of a two-sided bar construction
$$\MO \langle 4n \rangle = \left|\sBar(\bS,\spOf, \bS)_{\bullet} \right|.$$
Here, the action of $\spOf$ on the leftmost $\bS$ is via $\epsilon$, and the action on the rightmost $\bS$ is via $J_+$.

We may display this bar construction as a simplicial object,
$$
\begin{tikzcd}
\cdots \arrow[r,yshift=-3ex] \arrow[r,yshift=-1ex] \arrow[r,yshift=1ex] \arrow[r,yshift=3ex]
& \spOf^{\otimes 2} \arrow[r,yshift=-2ex,"\epsilon \otimes 1"] \arrow[r,"m"] \arrow[r,yshift=2ex,"1 \otimes J_+"]
& \spOf \arrow[r,yshift=-1ex,"\epsilon"] \arrow[r,yshift=1ex, "J_+"] 
& \bS,
\end{tikzcd}
$$
where the leftward degeneracy maps are omitted.  The key point is that, if we only wish to study $\tau_{\le 12n-2} \MO \langle 4n \rangle$, we need only study the partial simplicial diagram
$$
\begin{tikzcd}
\spOf^{\otimes 2} \arrow[r,yshift=-2ex,"\epsilon \otimes 1"] \arrow[r,"m"] \arrow[r,yshift=2ex,"1 \otimes J_+"]
& \spOf \arrow[r,yshift=-1ex,"\epsilon"] \arrow[r,yshift=1ex, "J_+"] 
& \bS.
\end{tikzcd}
$$
In the language of \cite[Lemma 1.2.4.17]{HA}, this is a diagram $\Delta^{\text{op}}_{\le 2} \to \Sp$.

\begin{lem} \label{lem:partialbar}
Let $X$ denote the  colimit of the partial simplicial diagram $\Delta^{\mathrm{op}}_{\le 2} \to \Sp$ given by $\mathrm{Bar}(\bS,\spOf,\bS)_{\le 2}$.  Then there is an equivalence of spectra $$\tau_{\le 12n-2} \MO \langle 4n \rangle \simeq \tau_{\le 12n-2} X.$$
\end{lem}

\begin{proof}
    Set $\mathrm{Bar}_{\leq k} = \abs{\sBar(\Ss, \spOf, \Ss)_{\leq k}}$. Then $\Ss = \sBar_{\leq 0}$ and $\MO \langle 4n \rangle$ is the  colimit of the $\sBar_{\leq k}$. We have a diagram
    \begin{center}
        \begin{tikzcd}[column sep=tiny]
            \Ss \arrow[r] & \sBar_{\leq 1} \arrow[r] \arrow[d] & \sBar_{\leq 2} \arrow[r] \arrow[d] & \sBar_{\leq 3} \arrow[r] \arrow[d] & \cdots \arrow[r] & \MO \langle 4n \rangle, \\
            & \Sigma \sOf & \Sigma^{2} \sOf^{\otimes 2} & \Sigma^{3} \sOf^{\otimes 3} & & 
        \end{tikzcd}
    \end{center}
    where the vertical maps are the  cofibers of the horizontal maps. The result now follows from the fact that, when $k \geq 3$, $\Sigma^{k} \sOf ^{\otimes k}$ is $(12n)$-connective.

\end{proof}

\begin{proof}[Proof of \Cref{thm:cofiber}]
In \Cref{lem:partialbar}, we established that $\tau_{\le 12n-2} \MO \langle 4n \rangle$ may be calculated as $\tau_{\le 12n-2}$ of the  colimit $X$ of a simplicial diagram
$$
\begin{tikzcd}
\spOf^{\otimes 2} \arrow[r,yshift=-2ex,"\epsilon \otimes 1"] \arrow[r,"m"] \arrow[r,yshift=2ex,"1 \otimes J_+"]
& \spOf \arrow[r,yshift=-1ex,"\epsilon"] \arrow[r,yshift=1ex, "J_+"] 
& \bS,
\end{tikzcd}
$$
where we have omitted the degeneracies from the notation.
According to the cofinality statement of \cite[Lemma 1.2.4.17]{HA}, $X$ may be characterized as the lower right corner of the following cocartesian cube:
$$
\begin{tikzcd}[row sep={35,between origins}, column sep={45,between origins}]
      & \spOf^{\otimes 2} \ar[rr,"1 \otimes J_+"] \ar[dd,"\epsilon \otimes 1" near end] \ar[dl,"m"] & & \spOf \ar[dd,"\epsilon"] \ar[dl,"J_+"] \\
    \spOf \ar[rr,crossing over, "J_+" near end] \ar[dd,"\epsilon"] & & \mathbb{S}  \\
      & \spOf  \ar[rr, "J_+" near end] \ar[dl,"\epsilon"] & &  \mathbb{S}  \arrow{dl}  \\
    \mathbb{S} \arrow[rr,""] & & X. \ar[from=uu,crossing over]
\end{tikzcd}
$$
We finish the proof by showing that the $X$ appearing in the cube is equivalent to the spectrum $C$ from the theorem statement.  Indeed, taking fibers in the vertical direction, we learn that $X$ is the total cofiber of the square
$$
\begin{tikzcd}[column sep= huge]
\mathrm{fiber}(\spOf^{\otimes 2} \stackrel{\epsilon \otimes 1}{\longrightarrow} \spOf) \arrow{r} \arrow{d} &  \mathrm{fiber}(\spOf \stackrel{\epsilon}{\longrightarrow} \mathbb{S}) \arrow{d} \\
\mathrm{fiber}(\spOf \stackrel{\epsilon}{\longrightarrow} \mathbb{S}) \arrow{r} & \mathbb{S},
\end{tikzcd}
$$
which simplifies to the square
$$
\begin{tikzcd}[column sep= huge]
\sOf \oplus \sOf^{\otimes 2} \arrow{r}{(1,1 \otimes J)} \arrow{d}{(1,m)} &  \sOf \arrow{d}{J} \\
\sOf \arrow{r}{J} \arrow[Leftrightarrow]{ru}{1\oplus a} & \mathbb{S}.
\end{tikzcd}
$$
The pushout of the arrows $(1,1 \otimes J)$ and $(1,m)$ is calculated as the cofiber of the map
$$
\begin{tikzcd} [column sep = huge, ampersand replacement=\&]
\sOf \oplus \sOf^{\otimes 2} \arrow{r}{\begin{pmatrix} 1 & 1 \otimes J \\ -1 & -m \end{pmatrix}} \& \sOf \oplus \sOf,
\end{tikzcd}
$$
or equivalently the cofiber of the map
$$
\begin{tikzcd} [column sep = huge]
\sOf^{\otimes 2} \arrow{r}{1 \otimes J-m} & \sOf.
\end{tikzcd}  
$$
This cofiber is the spectrum $P$ of the theorem statement.  To obtain the final sentence of the theorem, note that the unit map from $\bS$ to $\MO \langle 4n \rangle$ is the map from $\mathrm{Bar}(\bS,\spOf,\bS)_0$ into the geometric realization of the full bar construction.  This factors through the partial bar construction $|\mathrm{Bar}(\bS,\spOf,\bS)_{\le 2}| \simeq X$, via the map $\bS \to X$ that appears three times in the above cocartesian cube.
\qedhere
\end{proof}


%% file: Toda2.tex
In this section we will use the theory built up in Sections \ref{sec:Goodwillie} and \ref{sec:ThomPushout} to reduce Theorem \ref{thm:intro-main} to a concrete Toda bracket computation.
The final result of this section, \Cref{lem:toda-main}, is the only statement from Sections \ref{sec:Goodwillie}-\ref{sec:Toda} that is used later in the paper. The lemma expresses the Toda bracket in an explicit enough form that we will be able to bound its $\HFp$-Adams filtrations in \Cref{sec:SyntheticToda}.

Recall once more the fundamental square
$$
\begin{tikzcd}[column sep = huge]
\sOf^{\otimes 2} \arrow{dd}{m} \arrow{r}{1 \otimes J} & \sOf \arrow{dd}{J} \\ \\
\sOf \arrow{r}{J} \arrow[Leftrightarrow]{ruu}{a} & \bS.
\end{tikzcd}
$$
Let $P$ denote the  cofiber
$$
\begin{tikzcd} [column sep = huge]
P=\mathrm{cofiber}(\sOf^{\otimes 2} \arrow{r}{1 \otimes J - m} & \sOf).
\end{tikzcd}
$$
Then, as explained in Remark \ref{cnstr:square}, the homotopy $a$ gives rise to a canonical map $P \to \mathbb{S}$.  According to Theorem \ref{thm:cofiber}, there is a long exact sequence
\begin{equation} \label{eq:les}
  \pi_{8n-1} (P) \to \pi_{8n-1} (\mathbb{S}) \to \pi_{8n-1} \left(\MO \langle 4n \rangle\right) \to \pi_{8n-2} (P) \to \pi_{8n-2} (\mathbb{S}),
\end{equation}
which we will use to compute $\pi_{8n-1} \left( \MO \langle 4n \rangle \right)$.

\begin{lem} \label{lem:unit-surjectivity}
    The group $\pi_{8n-2} (P)$ is trivial.
\end{lem}

\begin{proof}
Consider the long exact sequence
\begin{center}
\begin{tikzcd}[column sep=tiny]
		\pi_{8n-2} \left( \sOf^{\otimes 2} \right) \arrow[rr] & \arrow[d, phantom, ""{coordinate, name=Z}] & \pi_{8n-2} \left( \sOf \right) \arrow[dll, rounded corners, to path={ -- ([xshift=2ex]\tikztostart.east) |- (Z) [near end]\tikztonodes -| ([xshift=-2ex]\tikztotarget.west) --(\tikztotarget)}] \\
		\pi_{8n-2} (P) \arrow[rr] & \vphantom{a} & \pi_{8n-3} \left( \sOf^{\otimes 2} \right) \cong 0.
	\end{tikzcd}
\end{center}
According to Corollary \ref{cor:gen-by-x2}, $\pi_{8n-2} \left( \sOf \right) \cong \mathbb{Z}/2\mathbb{Z}$, generated by $x^2$.
We will thus be done upon showing that $x^2$ is in the image of the map
$$\pi_{8n-2} \left( \sOf^{\otimes 2} \right) \to \pi_{8n-2} \left( \sOf \right).$$
The class $x \otimes x$ in the domain is sent to $xJ(x)-x^2$.  By Lemma \ref{lem:xJ(x)null}, $xJ(x)=0$.
\end{proof}

To complete the proof of Theorem \ref{thm:intro-main}, it remains to compute the image of the map
$$\pi_{8n-1} (P) \to \pi_{8n-1} (\mathbb{S}).$$
Note that the definition of $P$ as the cofiber of a map $\sOf^{\otimes 2} \to \sOf$ means that there is a canonical map
$\pi_{8n-1} (P) \to \pi_{8n-2} \left( \sOf^{\otimes 2} \right).$

\begin{lem} \label{lem:Reduction1}
Suppose $\ell$ is any class in $\pi_{8n-1} P$ which maps to 
$$2(x \otimes x) \in \pi_{8n-2} \left( \sOf^{\otimes 2} \right) \cong \Z\{x \otimes x\}.$$
Then $\J_{8n-1}$ and the image of $\ell$ in $\pi_{8n-1} (\mathbb{S})$ generate the kernel of the map
$$\pi_{8n-1} (\mathbb{S}) \to \pi_{8n-1} (\MO \langle 4n \rangle).$$
\end{lem}

\begin{proof}
  By equation (\ref{eq:les}), it suffices to show that the image of $\pi_{8n-1} (P) \to \pi_{8n-1} (\mathbb{S})$ is generated by $\J_{8n-1}$ and $\ell$.
  We will argue using the cofiber sequence
  $$ \sOf \to P \to \Sigma \sOf^{\otimes 2}. $$
  The composite map $\sOf \to P \to \mathbb{S}$ is by definition $J$. Therefore, by Theorem \ref{thm:imJwelldefined}, it has image exactly $\mathcal{J}_{8n-1}$ in degree $8n-1$. What remains is to show that $2 (x \otimes x)$ generates the subgroup of elements of $\pi_{8n-2} \left(\sOf^{\otimes 2} \right)$ that lift to $\pi_{8n-1} (P)$. Consider the map
  $$ \Z\{x \otimes x\} \cong \pi_{8n-2} \left(\sOf^{\otimes 2}\right) \to \pi_{8n-2} \left(\sOf\right) \cong (\Z/2)\{x^2\} $$
  The class $x \otimes x$ maps to 
    $$x^2 -x J(x) = x^2 \in \pi_{8n-2} \left( \sOf \right),$$
    since $x J(x) = 0$ by \Cref{lem:xJ(x)null}. Therefore, $2(x \otimes x)$ is a generator of the subgroup of elements which lift.
\end{proof}

Unwinding the definition of $P$, it is helpful to restate Lemma \ref{lem:Reduction1} in the following equivalent form:

\begin{cnstr} \label{cnstr:toda-basic}
Recall from \Cref{cor:gen-by-x2} and \Cref{lem:xJ(x)null} that $2xJ(x)=2x^2=0$ in $\pi_{8n-2} \left(\spOf\right)$.  We may therefore choose (completely arbitrary) nullhomotopies $f$ and $b$ completing the following diagram: 
$$
\begin{tikzcd}[column sep = huge]
\bS^{8n-2} \arrow{r}{2(x \otimes x)}
\arrow[rdd, bend right=20,""{name=D}] \arrow[rdd, bend right=20,swap,"0"] \arrow[rr,bend left = 30,""{name=U},"0"]
& \sOf^{\otimes 2} \arrow[Leftrightarrow, from=D, "b"] \arrow[Leftrightarrow, from=U, "f"] \arrow{dd}{m} \arrow{r}{1 \otimes J}
& \sOf \arrow{dd}{J}  \\ \\
& \sOf \arrow[Leftrightarrow]{ruu}{a} \arrow{r}{J}  & \bS. 
\end{tikzcd}
$$
Composing all three of these homotopies yields a homotopy between the map
$$0:\bS^{8n-2} \to \bS$$
  and itself, or equivalently a loop in the pointed mapping space $\mathrm{Hom}_{*}(\bS^{8n-2},\bS)$, or an element $z \in \pi_{8n-1} \bS$.  This Toda bracket $z$ is well-defined up to changing the nullhomotopies $f$ and $b$. The sets of homotopy classes of nullhomotopies $f$ and $b$ are torsors for $\pi_{8n-1} \sOf$, and changing either $f$ or $b$ by an element $y \in \pi_{8n-1} \sOf$ has the effect of changing the element $z$ by $J(y)$. The class $z$ therefore has indeterminancy equal to $\pi_{8n-1}J$, which is equal to $\J_{8n-1}$ by \Cref{thm:imJwelldefined}.
\end{cnstr}

\begin{lem} \label{lem:toda-basic}
The kernel of the map
$$\pi_{8n-1} (\mathbb{S}) \to \pi_{8n-1} (\MO \langle 4n \rangle)$$
is generated by $\J_{8n-1}$ and the Toda bracket $z$ of \Cref{cnstr:toda-basic}.
\end{lem}

\begin{proof}
The nullhomotopies $f$ and $b$ combine to give a nullhomotopy of the composite
$$
\begin{tikzcd}[column sep = huge]
\bS^{8n-2} \arrow{r}{2(x \otimes x)} & \sOf^{\otimes 2} \arrow{r}{1 \otimes J -m} & \sOf,
\end{tikzcd}
$$
which is exactly the data of a lift of $2(x \otimes x)$ to a class in
\[
\begin{tikzcd} [column sep = huge]
    \pi_{8n-2}\mathrm{fib}(\sOf^{\otimes 2} \arrow{r}{1 \otimes J - m} & \sOf) = \pi_{8n-2} \left( \Sigma^{-1} P \right) \cong \pi_{8n-1} (P).
\end{tikzcd}
\]
The conclusion therefore follows from \Cref{lem:Reduction1}.
\end{proof}

Our strategy will be to choose the nullhomotopies $f$ and $b$, or equivalently the lift $\ell$ in Lemma \ref{lem:Reduction1}, as judiciously as possible. It will be because of these choices that we will be able to establish our $\HFp$-Adams filtration bounds in \Cref{sec:SyntheticToda}. Let us begin by making a careful choice of the nullhomotopy $b$:

\begin{rec}
Suppose that $R$ is a homotopy commutative ring spectrum, and $r$ an element of $\pi_{2*+1} R$.  Then the graded commutativity of $\pi_*(R)$ ensures that $2r^2 =0$ in $\pi_{4*+2}(R)$.
  A small part of the data of an $\E_\infty$-structure on $R$ is a \emph{canonical} nullhomotopy of $2r^2$.
  Indeed, given $r : \bS^{2n+1} \to R$, there is a canonical factorization of $r^2$ as
  \[\bS^{4n+2} \to D_2 (\bS^{2n+1}) \to D_2 (R) \to R.\]
  The canonical nullhomotopy of $2r^2$ arises from fixing an identification of the $(4n+3)$-skeleton of $D_2 (\bS^{2n+1})$ with $\Sigma^{4n+2} C(2)$.
\end{rec}

\begin{cnstr} \label{cnstr:toda-main-preskeleton}
Let $h$ denote the canonical nullhomotopy of $2J(x)^2$ that arises from the fact that $J(x) \in \pi_{4n-1} \bS$ is an element in the odd degree homotopy of the $\E_{\infty}$-ring $\bS$.  Let $g$ denote the canonical homotopy $J(xJ(x))\simeq J(x)^2$ that arises from $J$ being a map of right $\bS$-modules, and let $f$ denote a completely arbitrary nullhomotopy of $2xJ(x)$.  Then we may form the following diagram,
$$
\begin{tikzcd}[column sep = huge]
\mathbb{S}^{8n-2} \arrow{r}{2}
\arrow[rrdd, bend right=20,""{name=D}] \arrow[rrdd, bend right=20,swap,"0"] \arrow[rr,bend left = 30,""{name=U},"0"]
\arrow[rrdd, bend right=20,""{name=D}] \arrow[rrdd, bend right=20,swap,"0"] \arrow[rr,bend left = 30,""{name=U},"0"]
& \mathbb{S}^{8n-2} \arrow[Leftrightarrow, from=D, "h"] \arrow[Leftrightarrow, from=U, "f"] \arrow{r}{xJ(x)} \arrow[rdd,""{name=MU},"J(x)^2" description] \arrow[rdd,swap, ""{name=MD},"J(x)^2" description]
& \sOf \arrow{dd}{J}  \arrow[bend left=20,Leftrightarrow, from=MU, swap, "g"] \\ \\
& & \mathbb{S}, 
\end{tikzcd}
$$
which composes to give a Toda bracket $w \in \pi_{8n-1} \bS$.
\end{cnstr}

\begin{lem} \label{lem:toda-main-preskeleton}
Let $w$ denote the Toda bracket of Construction \ref{cnstr:toda-main-preskeleton}.
Then the kernel of the map
$$\pi_{8n-1} (\mathbb{S}) \to \pi_{8n-1} (\MO \langle 4n \rangle)$$
is generated by $\J_{8n-1}$ and $w$. 
\end{lem}

\begin{proof}
Composing (whiskering) the homotopy $a$
$$
\begin{tikzcd}[column sep = huge]
\sOf^{\otimes 2} \arrow{dd}{m} \arrow{r}{1 \otimes J} & \sOf \arrow{dd}{J} \\ \\
\sOf \arrow{r}{J} \arrow[Leftrightarrow]{ruu}{a} & \bS.
\end{tikzcd}
$$
along the map $\mathbb{S}^{8n-2} \stackrel{x \otimes x}{\longrightarrow} \sOf^{\otimes 2}$ yields the diagram
$$
\begin{tikzcd}[column sep = huge]
\bS^{8n-2} \arrow{dd}{x^2} \arrow{r}{xJ(x)} \arrow[rdd,""{name=MU},"J(x)^2" description] \arrow[rdd,swap, ""{name=MD},"J(x)^2" description] &
\sOf \arrow{dd}{J} \arrow[bend left=10,Leftrightarrow, from=MU, swap, "g"] \\ \\
\sOf \arrow{r}{J} \arrow[bend left=10,Leftrightarrow, from=MD, swap, "c"]  & \bS. 
\end{tikzcd}
$$
Here, $g$ is the homotopy from Construction \ref{cnstr:toda-main-preskeleton}, and $c$ is the natural homotopy arising from the structure of $J$ as a ring homomorphism.  Consider now the slightly extended diagram 
$$
\begin{tikzcd}[column sep = huge]
\bS^{8n-2} \arrow{r}{2} &\bS^{8n-2} \arrow{dd}{x^2} \arrow{r}{xJ(x)} \arrow[rdd,""{name=MU},"J(x)^2" description] \arrow[rdd,swap, ""{name=MD},"J(x)^2" description] &
\sOf \arrow{dd}{J} \arrow[bend left=10,Leftrightarrow, from=MU, swap, "g"] \\ \\
&\sOf \arrow{r}{J} \arrow[bend left=10,Leftrightarrow, from=MD, swap, "c"]  & \bS. 
\end{tikzcd}
$$
To put ourselves in the situation of Lemma \ref{lem:toda-basic}, we must choose a nullhomotopy $f$ of $2xJ(x)$ as well as a nullhomotopy $b$ of $2x^2$.  The result follows from choosing $b$ to the canonical nullhomotopy of $2x^2$ arising from the fact that $\sOf$ is a (non-unital) $\mathbb{E}_\infty$-ring spectrum.  Since $J$ is naturally a map of $\mathbb{E}_\infty$-rings, and not just of $\mathbb{A}_\infty$-rings, this canonical nullhomotopy of $2x^2$ will compose with $c$ to be the canonical nullhomotopy $h$ of $2J(x)^2$.
\end{proof}

We record one final technical reduction to end this section.

\begin{dfn} \label{dfn:skeleton}
Let
	$$M \longrightarrow \Sigma^{\infty} \Oo \mathrm{\langle 4n-1 \rangle}$$
	denote the inclusion of an $(8n-1)$-skeleton of $\Sigma^{\infty} \Oo \mathrm{\langle 4n-1 \rangle}$. By the inclusion of an $(8n-1)$-skeleton, we mean in particular that the induced map
	$$\mathrm{H}_*(M;\F_p) \longrightarrow \mathrm{H}_*(\Sigma^{\infty} \Oo \mathrm{\langle 4n-1 \rangle};\F_p)$$
is an isomorphism for $* < 8n-1$ and a surjection for $* = 8n-1$, and that $\mathrm{H}_*(M;\F_p) \cong 0$ for $* > 8n-1$.
The generator $x \in \pi_{4n-1} \left(\sOf \right)$ is the image of some class in $\pi_{4n-1} M$, which by abuse of notation we also denote by $x$.  We additionally abuse notation by using $J$ to denote the composite map $$M \longrightarrow \Sigma^{\infty} \Oo \mathrm{\langle 4n-1 \rangle} \stackrel{J}{\longrightarrow} \mathbb{S}.$$
\end{dfn}

\begin{lem} \label{lem:toda-main}
Let $h$ denote the canonical nullhomotopy of $2J(x)^2$ that arises from the fact that $J(x) \in \pi_{4n-1} \bS$ is an element in the odd degree homotopy of the $\E_{\infty}$-ring $\bS$.  Let $g$ denote the canonical homotopy $J(xJ(x))\simeq J(x)^2$ that arises from $J$ being a map of right $\bS$-modules, and let $f$ denote a completely arbitrary nullhomotopy of $2xJ(x)$.  Then we may form the following diagram,
$$
\begin{tikzcd}[column sep = huge]
\mathbb{S}^{8n-2} \arrow{r}{2}
\arrow[rrdd, bend right=20,""{name=D}] \arrow[rrdd, bend right=20,swap,"0"] \arrow[rr,bend left = 30,""{name=U},"0"]
\arrow[rrdd, bend right=20,""{name=D}] \arrow[rrdd, bend right=20,swap,"0"] \arrow[rr,bend left = 30,""{name=U},"0"]
& \mathbb{S}^{8n-2} \arrow[Leftrightarrow, from=D, "h"] \arrow[Leftrightarrow, from=U, "f"] \arrow{r}{xJ(x)} \arrow[rdd,""{name=MU},"J(x)^2" description] \arrow[rdd,swap, ""{name=MD},"J(x)^2" description]
& M \arrow{dd}{J}  \arrow[bend left=20,Leftrightarrow, from=MU, swap, "g"] \\ \\
& & \mathbb{S}, 
\end{tikzcd}
$$
which composes to give a Toda bracket $w \in \pi_{8n-1} (\bS)$.
The kernel of the map
$$\pi_{8n-1} (\mathbb{S}) \to \pi_{8n-1} (\MO \langle 4n \rangle)$$
is generated by $\J_{8n-1}$ and $w$. 
\end{lem}

\begin{proof}
    Since $M$ is an $(8n-1)$-skeleton, $2xJ(x)$ is trivial not just in $\pi_{8n-2} \left(\sOf \right)$, but also in $\pi_{8n-2} (M)$.  A nullhomotopy $f$ of $2x J(x)$ inside of $\pi_{8n-2} (M)$ in particular induces such a nullhomotopy in $\pi_{8n-2} \left( \sOf \right)$.  Also, the map $M \to \sOf$ is a map of right $\mathbb{S}$-modules (since it is a map of spectra), and so the homotopy $g$ from this lemma composes with the inclusion of $M$ to give the homotopy $g$ of Construction \ref{cnstr:toda-main-preskeleton}.
\end{proof}


%% file: Finale.tex
In this section, we will prove Theorem \ref{thm:intro-main} assuming two results from later in the paper.
Recall our standing assumption that $n \ge 3$ is a positive integer.  In Sections \ref{sec:ThomPushout}-\ref{sec:Toda}, we studied the unit map
$$\pi_{8n-1} \bS \to \pi_{8n-1} \MO \langle 4n \rangle.$$
In Lemma \ref{lem:unit-surjectivity}, we showed that this map is surjective.  In Lemma \ref{lem:Reduction1}, we showed that the subgroup $\mathcal{J}_{8n-1}$ is in the kernel of this map.\footnote{We remark that this and the preceeding statement are not original to us and may be proven using classical tools of geometric topology---see \Cref{rmk:Qdef}} Furthermore, modulo $\mathcal{J}_{8n-1}$, every element in the kernel is an integer multiple of a single class, which by Lemma \ref{lem:toda-main} is given by the Toda bracket $w$.

Our task here is to show that, for $n \geq 32$, this element $w$ is trivial modulo $\mathcal{J}_{8n-1}$. Our strategy will be to prove, separately for each prime number $p$, that $w$ is trivial after $p$-localization. 

\begin{thm} \label{thm:w=0}
  Fix a prime number $p$.  The element
  $$w \in (\pi_{8n-1} \bS) / \J_{8n-1}$$
  is $p$-locally trivial if any of the following conditions are met:
  \begin{itemize}
  \item $p > 3$.
  \item $n \geq 32$ and $p=3$.
  \item $n \geq 17$ and $p=2$.
  \end{itemize}
\end{thm}

\begin{cnv}
For the rest of this section we will work $p$-locally for a fixed prime number $p$.  For example, we use $\pi_*\Ss$ to denote $\pi_*\Ss_{(p)}$.
\end{cnv}

The proof that $w \in \mathcal{J}_{8n-1}$ will proceed by using two results from later in the paper.   These results respectively
\begin{enumerate}
    \item establish a lower bound on the $\HFp$-Adams filtration of $w$, and
    \item exhibit an upper bound on the $\HFp$-Adams filtrations of elements of $\coker (J)$.
\end{enumerate}
To explain further, we recall the following definition, which appeared in the statement of \Cref{thm:Burklund}:

\begin{dfn} \label{dfn:h}
  For each prime number $p>2$ and each integer $k>0$, let $\Gamma_p(k)$ denote the minimal $m$ such that every $\alpha \in \pi_{k}\Ss_{(p)}$ with $\HFp$-Adams filtration strictly greater than $m$ is in the image of $J$.
  Similarly, let $\Gamma_2(k)$ denote the minimal $m$ such that every $\alpha \in \pi_{k}\Ss_{(2)}$ with $\HFt$-Adams filtration strictly greater than $m$ is in the subgroup generated by the image of $J$ and the $\mu$-family.
\end{dfn}

\begin{rmk}
At $p=2$, the elements of $\pi_*\Ss$ in the $\mu$ family are of degrees $1$ and $2$ modulo $8$, and in particular do not occur in degree $8n-1$.  Thus, an element in $\pi_{8n-1}\Ss_{(p)}$ with $\HFt$-Adams filtration greater than $\Gamma_p(8n-1)$ is automatically in the image of $J$.
\end{rmk}

If we let $AF(w)$ denote the $\HFp$-Adams filtration of (some choice of) $w$, then it will suffice to show that
\begin{align}\label{eqn:BIG} \Gamma_p (8n-1) < AF(w). \end{align}
\Cref{sec:SyntheticToda} will be devoted to establishing a lower bound on $AF(w)$. To state this bound, we establish additional notation.

\begin{dfn}
We define the integer $N_2$ by the formula
\[N_2 = h(4n-1) - \lfloor \log_2 (8n) \rfloor + 1,\]
where $h(k)$ is the number of integers $0 < s \leq k$ which are congruent to $0,1,2$ or $4$ mod $8$.
For an odd prime $p$, we define 
\[N_p = \left\lfloor \frac{4n}{2p-2} \right\rfloor - \left\lfloor \log_p (4n) \right\rfloor.\]
\end{dfn}

\begin{thm}[Proven as \Cref{thm:AdamsBound}]\label{thm:filt-bound-finale}
There exists a choice of $f$ in the statement of Lemma \ref{lem:toda-main} such that the $\mathrm{H}\F_p$-Adams filtration of the Toda bracket $w$ is at least $2N_p-1$.
\end{thm}


\begin{rmk}
  Our use of Adams filtrations in the proof of \Cref{thm:w=0} is inspired by arguments of Stolz in \cite{StolzBook}. 
  In partiular, it follows from a theorem of Stolz that there is a lower bound of size $N_2$ on the $\HFt$-Adams filtration of $w$ \cite[Satz 12.7]{StolzBook}.
  The doubling of Stolz's lower bound is one of the main contributions of this paper.
\end{rmk}

Additionally, upper bounds on $\Gamma_p (8n-1)$ have been previously studied by Davis--Mahowald \cite{DM3} (at the prime $2$) and Gonz\'alez \cite{Gonzalez} (at odd primes). At the prime $3$, we will require a novel bound established by the first named author in \Cref{sec:Appendix}.

\begin{thm} \label{thm:gamma-upper}
  We have the following upper bounds on the function $\Gamma_{p} (8n-1)$:
  \begin{enumerate}
    \item (Davis--Mahowald, \cite[Corollary 1.3]{DM3})
      \[\Gamma_2 (8n-1) \leq \frac{3(8n-1)}{10}+7+v_2 (n). \]
    \item (Gonz\'alez, \cite[Theorem 5.1]{Gonzalez}) Assume $p \neq 2$. Then:
      \[\Gamma_p (8n-1) \leq 3+\frac{(2p-1)(8n-1)}{(2p-2)(p^2-p-1)}. \]
    \item (Burklund, \Cref{thm:app-main} (4))\footnote{The statement of \Cref{thm:app-main}(4) in \Cref{sec:Appendix} contains a term depending on a function $\ell(k)$, but this function vanishes by definition when $k=8n-1$.}
      \[\Gamma_3 (8n-1) \leq \frac{25(8n-1)}{184}+19+\frac{1133}{1472}.\]
  \end{enumerate}
\end{thm}

\begin{rmk}
\Cref{thm:AdamsBound} and \Cref{thm:app-main} are the only results from the latter half of this paper that are required to settle the Galatius \& Randal--Williams conjecture.  The paper is structured so that the reader willing to assume \Cref{thm:app-main} need not read past \Cref{sec:SyntheticToda} to understand the proof of \Cref{thm:intro-main}.
\end{rmk}

%
At this point, the main work ahead of us in this section is to understand when the bound of \Cref{thm:filt-bound-finale} exceeds the bounds of \Cref{thm:gamma-upper}. To this end, we introduce some compact notation.


\begin{ntn}
  Let
  $$ A_p := 2N_p - 1\ \ \text{ and }\ \ B_p := \begin{cases} \frac{3(8n-1)}{10} + 7 + v_2(n), & p=2 \\ \frac{25(8n-1)}{184} + 19 + \frac{1133}{1472}, & p=3 \\ 3 + \frac{(2p-1)(8n-1)}{(2p-2)(p^2 - p - 1)}, & p \geq 5 \end{cases}. $$
\end{ntn}

\begin{lem} \label{lem:logs}
  The element $w \in (\pi_{8n-1} \mathbb{S}) / \J_{8n-1}$ is $p$-locally trivial if $A_p > B_p$,
  which occurs for 
  \begin{itemize}
  \item $p=2$ and $n \geq 17$,
  \item $p=3$ and $n \geq 32$,
  \item $p=5$ and $n \geq 16$,
  \item $p=7$ and $n \geq 21$,
  \item $p \ge 11$ and $n \geq 2(2p-2)$.
  \end{itemize}
\end{lem}

\begin{proof}
  The first claim follows from the preceding discussion.
  Verifying the remaining claims is just a matter of checking inequalities between elementary functions.
  For each of these we follow the same basic strategy:
  \begin{enumerate}
  \item Find a smooth function that acts as a lower bound for $A_p - B_p$.
  \item Argue that the derivative of this function is positive.
  \item Find a point where this lower bound is positive.
  \item Go back and fill in small values of $n$ by directly computing $A_p - B_p$.
  \end{enumerate}
  
  We begin with the $p=2$ case.
  \begin{align*}
    A_2 - B_2
    &= 2 \bigg(h(4n-1) - \lfloor \log_2(8n) \rfloor + 1 \bigg) - 1 - \bigg( \frac{3}{10}(8n-1) + 7 + v_2(n) \bigg) \\
    &\geq (4n - 2) - 2\log_2(8n)  + 1 - \frac{3}{10}(8n-1) - 7 - \log_2(n) \\
    &= \frac{8}{5}n - 3\log_2(n)  - 13.7
  \end{align*}
  Since the quantity on the final line is positive for $n=17$ (it is approximately $1.24$) and its derivative with respect to $n$ is positive for $n \geq 3$ we may conclude that $A_2>B_2$ for $n \geq 17$.

  Next we handle the $p=3$ case, which proceeds in the same fashion as $p=2$.
  \begin{align*}
    A_3 - B_3
    &= 2 \left( \left\lfloor \frac{4n}{4} \right\rfloor - \lfloor \log_3(4n) \rfloor \right) - 1 - \left( \frac{25}{184}(8n-1) + 19 + \frac{1133}{1472} \right) \\
    &\geq 2n - 2\log_3(4n) - 1 - \frac{200}{184}n - 20 \\
    &= \frac{21}{23}n - 2\log_3(4n) - 21      
  \end{align*}
  The quantity on the final line is positive for $n=33$ (it is approximately $0.24$) and its derivative with respect to $n$ is positive for $n \geq 2$, thus we may conclude that $A_3>B_3$ for $n \geq 33$.
  The inequality for the remaining value of $n$ can now be read off from \Cref{tbl:big}.
  \begin{table}
    \begin{center}
      \renewcommand{\arraystretch}{1.1}
      \begin{tabular}{|c||c|c|}\hline
        $n$ & $A_3$ & $B_3$ \\\hline\hline
        31 & 53 & 53.33 \\\hline
        32 & 55 & 54.42 \\\hline
        33 & 57 & 55.50 \\\hline
      \end{tabular}
      \hspace{0.5cm}
      \begin{tabular}{|c||c|c|}\hline
        $n$ & $A_5$ & $B_5$  \\\hline\hline
          16 &   11 & 10.52  \\\hline
          17 &   11 & 10.99  \\\hline
          18 &   13 & 11.47  \\\hline
          19 &   13 & 11.94  \\\hline
          20 &   15 & 12.41  \\\hline
          21 &   15 & 12.89  \\\hline
          22 &   17 & 13.36  \\\hline
          23 &   17 & 13.84  \\\hline
          24 &   19 & 14.31  \\\hline
      \end{tabular}
      \hspace{0.5cm}
      \begin{tabular}{|c||c|c|}\hline
        $n$ & $A_7$ & $B_7$ \\\hline\hline
        20 &  7 & 7.20 \\\hline
        21 &  9 & 7.41 \\\hline
        22 &  9 & 7.62 \\\hline
        23 &  9 & 7.84 \\\hline
        24 &  11 & 8.05 \\\hline
      \end{tabular}
      \hspace{0.5cm}
      \renewcommand{\arraystretch}{1.0}
        \vspace{0.5cm}
        \caption{}
        \label{tbl:big}
    \end{center}
  \end{table}

  We handle the remaining values of $p$ uniformly.
  Let $(p-1)k = n$, then
  \begin{align*}
    A_p - B_p &= 2 \left( \left\lfloor \frac{4n}{2p-2} \right\rfloor - \lfloor \log_{p}(4n) \rfloor \right) - 1 - \left( 3 + \frac{(2p-1)(8n-1)}{(2p-2)(p^2-p-1)} \right) \\
    &\geq 2 \frac{4n}{2p-2} - 2\log_{p}(4n) - 6 - \frac{(2p-1)8n}{(2p-2)(p^2-p-1)} \\
    &\geq 4k - 2\log_{p}(4(p-1)k) - 6 - \frac{(2p-1)4k}{p^2 - p - 1} \\
    &\geq 4k - \frac{(2p-1)4k}{p^2 - p - 1} - 8 - 2\log_{p}(4k). 
  \end{align*}
  Let $C_p$ denote the quantity on the final line.
  For $p \geq 5$ and $k \geq 2$ we have that
  \begin{align*}
    \frac{\partial}{\partial k} C_p
    &= 4 - \frac{(2p-1)4}{p^2 - p - 1} - \frac{2}{\log(p) k} \geq 4 - \frac{8p}{p^2 - 2p} - \frac{2}{\log(p) k} \\
    &\geq 4 - \frac{8}{3} - \frac{2}{2\log(5)} > 0
  \end{align*}  
  and
  \[ \frac{\partial}{\partial p} C_p = 4k\frac{2p^2 - 2p + 3}{(p^2 - p - 1)^2} + \frac{2\log(4k)}{(\log(p))^2 p} = \frac{2k((2p - 1)^2 + 5)}{(p^2 - p - 1)^2} + \frac{2\log(4k)}{(\log(p))^2 p}> 0. \]
  When $p=5$ and $k=6$ the value of $C_p$ is approximately $0.68$ and
  when $p=7$ and $k=4$ the value of $C_p$ is approximately $0.08$.
  Translating back into terms of $p$ and $n$ instead of $k$ this completes the proof of the lemma for $p \geq 11$ and all but finitely many cases for $p=5,7$. 
  A second consultation with \Cref{tbl:big} now completes the proof.
\end{proof}

\begin{proof}[Proof of Theorem \ref{thm:w=0}]
  
  In order to finish the proof of \Cref{thm:w=0} at $p\geq 5$ it will suffice to show that $\pi_{8n-1}\Ss^0_{(p)}$ is generated by the image of $J$ for each $n$ below the bound from \Cref{lem:logs}.
  Through this range the $\mathrm{E}_2$-page of the Adams--Novikov spectral sequence is calculated in \cite[Theorem 4.4.20]{GreenBook} and the spectral sequence degenerates at the $\mathrm{E}_2$-page for degree reasons.

  In \Cref{tbl:small-stems} we list the generators of the cokernel of $J$ in degrees below the bound from \Cref{lem:logs} at primes $5,7,11$ and $13$.
  None of these generators are in a degree congruent to $-1$ modulo 8.
  At primes $p \geq 17$ we argue as follows:
  again using \cite[Theorem 4.4.20]{GreenBook} we know that the first element of $\coker(J)$ at odd primes is $\beta_1$ in the $(2p^2 - 2p -2)$ stem.
  On the other hand, since
  \[8n-1 \leq 16(2p-2) - 1 < 2p^2 - 2p - 2 \]
  the bound from \Cref{lem:logs} is below this point.
\end{proof}
  

  \begin{table}
    \begin{center}
      \renewcommand{\arraystretch}{1.1}
      \begin{tabular}{|c|c|c|c|}\hline
        $p$ & Max $n$ & Max degree & Classes in $\coker(J)$ \\\hline\hline
        5 & 15 & 119 & $ \beta_1,\ \alpha_1\beta_1,\ \beta_1^2,\ \alpha_1\beta_1^2,\ \beta_2,\ \alpha_1\beta_2,\ \beta_1^3 $\hfill\vadjust{} \\\hline
        7 & 20 & 159 & $\beta_1,\ \alpha_1\beta_1$\hfill\vadjust{} \\\hline
        11 & 39 & 311 & $\beta_1,\ \alpha_1\beta_1$\hfill\vadjust{} \\\hline
        13 & 47 & 375 & $\beta_1,\ \alpha_1\beta_1$\hfill\vadjust{} \\\hline             
      \end{tabular}
      \renewcommand{\arraystretch}{1.0}
      \vspace{0.5cm}
      \caption{}
      \label{tbl:small-stems}
    \end{center}
  \end{table}
  





\subsection{A note on the remaining dimensions}
\label{sec:open-problems}\ 

Galatius and Randal--Williams conjecture \cite[Conjecture A]{GRAbelianQuotients} that the map
$$\pi_{8n-1} \bS \to \pi_{8n-1} \MO \langle 4n \rangle$$
has kernel equal to $\mathcal{J}_{8n-1}$ for all $n \ge 1$, and not just for $n >31$.

It is known that the conjecture is true when $n=1$ and when $n=2$ \cite[p. 13]{GRAbelianQuotients}.  Indeed, the case $n=1$ follows from the fact that there is nothing in $\pi_7 \bS$ not in the image of $J$.  When $n=2$, it follows from direct calculation of $\pi_{15} \MO \langle 8 \rangle=\pi_{15} \mathrm{MString}$ as in \cite{GiambalvoO8,GiambalvoO8corrected}.  The methods of this paper first apply when $n \ge 3$.

The first and third named authors returned to this question in \cite{ManSmall} and proved the following theorem, completely resolving the remaining cases of the Galatius and Randal--Williams conjecture:
\begin{thm}[{\cite[Theorem 1.4]{ManSmall}}]
  The kernel of the map
  \[\pi_{2k-1} \bS \to \pi_{2k-1} \MO \langle k \rangle\]
  is equal to $\mathcal{J}_{2k-1}$ if $k \neq 9,12$.
  When $k=9$, it is generated by $\mathcal{J}_{17}$ and $\eta \eta_4 \in \pi_{17} (\bS)$.
  When $k=12$, it is generated by $\mathcal{J}_{23}$ and $\eta^3 \overline{\kappa} \in \pi_{23} (\bS)$.

  In particular, the Galatius and Randal--Williams conjecture is true precisely when $n \neq 3$.
\end{thm}

%% file: StolzProof.tex
In this section, we will prove Theorem \ref{thm:mainboundary}.  Our methods are due to Stephan Stolz \cite{StolzBook}, and we rely heavily on his work.  We are able to improve on Stolz's results by a combination of Theorem \ref{thm:intro-main} and the following:

\begin{thm} \label{thm:mod8}
Suppose that a class $\alpha$ in the $(2k+d)$th homotopy group of the mod $8$ Moore spectrum
    $$\alpha \in \pi_{2k+d}(C(8))$$
has $\mathrm{H}\mathbb{F}_2$-Adams filtration at least $\frac{2k+d}{5}+15$.  Then, if $2k+d \ge 126$, the image of $\alpha$ under the Bockstein map
$$\pi_{2k+d}(C(8)) \to \pi_{2k+d-1}(\mathbb{S})$$
is contained in the subgroup of $\pi_{2k+d-1}(\mathbb{S})$ generated by $\mathcal{J}_{2k+d-1}$ and Adams's $\mu$-family.
\end{thm}

Theorem \ref{thm:mod8} will be proved as Theorem \ref{thm:mod8-main-thm} in the subsequent half of the paper.

\begin{rmk}
    Stolz relied on a similar result for the mod $2$ Moore spectrum, which he attributes to Mahowald \cite[Satz 12.9]{StolzBook}.  Though Mahowald announced such a result in \cite{MahBull}, and again in \cite{MahTokyo} (the reference that Stolz cites), to the best of our knowledge no proof has appeared in print.  Part of our motivation in proving Theorem \ref{thm:mod8} is to fill this gap in the literature. The other motivation is that we obtain stronger geometric consequences from the mod $8$ Moore spectrum.
\end{rmk}

Recall that $\mathrm{MO} \langle k \rangle$ denotes the Thom spectrum of the map
$$\tau_{\ge k} \mathrm{BO} \to \mathrm{BO}.$$
There is a unit map $\mathbb{S} \to \mathrm{MO} \langle k \rangle$, which may be extended to a cofiber sequence
$$\mathbb{S} \to \mathrm{MO} \langle k \rangle \to \mathrm{MO} \langle k \rangle / \mathbb{S} \stackrel{\partial}{\to} \mathbb{S}^1.$$
Stolz constructed \cite[Satz 3.1]{StolzBook} a spectrum $A[k]$ together with a map $$b:A[k] \to \mathrm{MO} \langle k \rangle / \mathbb{S}$$ such that the following is true:

\begin{thm}[{\cite[Lemma 12.5]{StolzBook}}]
Let $k>2$ and $d \ge 0$ be integers.  Suppose that, for every element $\alpha \in \pi_{2k+d}(A[k])$, the image of $\alpha$ under the composite
$$\pi_{2k+d}(A[k]) \stackrel{b_*}{\longrightarrow} \pi_{2k+d} \left( \mathrm{MO} \langle k \rangle / \mathbb{S} \right) \stackrel{\partial_*}{\longrightarrow} \pi_{2k+d}(\mathbb{S}^1) \cong \pi_{2k+d-1}(\mathbb{S})$$
is in $\mathcal{J}_{2k+d-1}$.  Then the boundary of any $(k-1)$-connected, almost closed $(2k+d)$-manifold also bounds a parallelizable manifold.
\end{thm}

Stolz then proved the following two theorems. The first is contained in the proof of \cite[Satz 12.7]{StolzBook}:

\begin{thm}\cite[Proof of Satz 12.7]{StolzBook}  \label{thm:Stolz-Adams}
    Suppose that $k>2$ and $d \ge 0$, and let $M_{2k+d+1} \hookrightarrow A[k]$ denote a $(2k+d+1)$-skeleton of $A[k]$. Then the composite \[M_{2k+d+1} \hookrightarrow A[k] \to \mathrm{MO} \langle k \rangle / \Ss \to \Ss^1\] has $\mathrm{H}\mathbb{F}_2$-Adams filtration at least $$h(k-1)-\lfloor \mathrm{log}_2(2k+d+1) \rfloor +1,$$ where $h(k-1)$ is the number of integers of $s$ with $0<s \le k-1$ and $s \equiv 0,1,2,$ or $4$ modulo $8$.
\end{thm}

The second follows from \cite[Theorem A]{StolzBook} along with the equivalences $A[k] \simeq A[k+1]$ for $k \equiv 3,5,6,7 \mod 8$ \cite[Satz 3.1(ii)]{StolzBook}:

\begin{thm}[{\cite[Theorem A]{StolzBook}}] \label{thm:Stolz-8torsion}
Suppose that $k \ge 9$ and that $0 \le d \le 3$.  Then, unless one of the following conditions are met, every element of $\pi_{2k+d} A[k]$ is $8$-torsion:
\begin{itemize}
\item $k \equiv 0$ modulo $4$ and $d=0$.
\item $k \equiv 3$ modulo $4$ and $d=2$.
\end{itemize}
\end{thm}

We now proceed with our proof of Theorem \ref{thm:mainboundary}.

\begin{thm} \label{thm:strongmainboundary}
Let $k>124$ and $0 \le d \le 3$ be integers.  Suppose that $k \equiv 0$ modulo $4$ and $d=0$ or $k \equiv 3$ modulo $3$ and $d=2$.  Then the boundary of any $(k-1)$-connected, almost closed $(2k+d)$-manifold also bounds a parallelizable manifold.
\end{thm}

\begin{proof}
We need to show that the image of the composite
$$\pi_{2k+d} A[k] \to \pi_{2k+d}(\MO \langle k \rangle/\mathbb{S}) \to \pi_{2k+d}(\mathbb{S}^1)$$
contains only classes in $\mathcal{J}_{2k+d-1}$.  In fact, we will show that this is already true of the image of
$$\pi_{2k+d}(\MO \langle k \rangle/\mathbb{S}) \to \pi_{2k+d}(\mathbb{S}^1),$$
or equivalently that the unit map
$$\pi_{2k+d-1} \mathbb{S} \to \pi_{2k+d-1} \MO \langle k \rangle$$
has kernel consisting only of classes in $\mathcal{J}_{2k+d-1}$.  If $d=0$ and $k \equiv 0$ modulo $4$, this follows from Theorem \ref{thm:intro-main}.  If $d=2$ and $k \equiv 3$ modulo $4$, then $\MO \langle k \rangle \simeq \MO \langle k+1 \rangle$, and so $\pi_{2k+2} \MO \langle k \rangle \cong \pi_{2k+2} \MO \langle k+1 \rangle$ and this again follows from Theorem \ref{thm:intro-main}.
\end{proof}

\begin{thm} \label{thm:weakmainboundary}
Let $k>232$ and $0 \le d \le 3$ be integers.  Suppose that $k$ and $d$ satisfy neither of the exceptional conditions under which Theorem \ref{thm:Stolz-8torsion} fails and Theorem \ref{thm:strongmainboundary} applies.  Then the boundary of any $(k-1)$-connected, almost closed $(2k+d)$-manifold also bounds a parallelizable manifold. 
\end{thm}

\begin{proof}
We construct an argument very similar to that found on \cite[p.107]{StolzBook}.  Namely, consider the diagram
    \[
\begin{tikzcd}
    \Sigma^{-1} C(8) \otimes M_{2k+d+1} \arrow{r}{1 \otimes \iota} \arrow{d} & \Sigma^{-1}C(8) \otimes A[k] \arrow{r}{1 \otimes (\partial \circ b)} \arrow{d} & \Sigma^{-1}C(8) \otimes \mathbb{S}^1 \arrow{d} \\   
    M_{2k+d+1} \arrow{r}{\iota} & A[k] \arrow{r}{\partial \circ b} &  \mathbb{S}^1
\end{tikzcd}
    \]
    where $M_{2k+d+1} \to A[k]$ is a $(2k+d+1)$-skeleton.

    Let $\alpha$ denote a map $\Ss^{2k+d} \to A[k]$. Then we may factor $\alpha$ through an $8$-torsion map $\Ss^{2k+d} \to M_{2k+d+1}$ by \Cref{thm:Stolz-8torsion}, and thus we may choose a lift $\bar{\alpha}:\Ss^{2k+d} \to \Sigma^{-1}C(8) \otimes M_{2k+d+1}$. Since \[M_{2k+d+1} \xrightarrow{\iota} A[k] \xrightarrow{\partial\circ b} \Ss^{1}\] is of $\HFt$-Adams filtration at least $h(k-1)-\lfloor \mathrm{log}_2(2k+d+1) \rfloor +1$ by \Cref{thm:Stolz-Adams}, so is \[\Sigma^{-1}C(8) \otimes M_{2k+d+1} \xrightarrow{1 \otimes \iota} \Sigma^{-1} C(8) \otimes A[k] \xrightarrow{1 \otimes (\partial \circ b)} \Sigma^{-1} C(8) \otimes \mathbb{S}^1.\] It follows that $(1 \otimes \partial) \circ (1 \otimes b) \circ (1 \otimes \iota) \circ \bar{\alpha} \in \pi_{2k+d} \left(C(8)\right)$ is too. Thus, by Theorem \ref{thm:mod8}, so long as
$$2k+d \ge 126, \text{ and}$$
$$h(k-1)-\lfloor \mathrm{log}_2(2k+d+1) \rfloor +1 \ge \frac{2k+d}{5}+15,$$
the image of $\alpha$ in $\pi_{2k+d-1} \mathbb{S}$ must be in the subgroup generated by $\mathcal{J}_{2k+d-1}$ and Adams's $\mu$-family.  In \Cref{lemm:inequality}, we show that both of these conditions are satisfied under our assumptions $k > 232$ and $0 \le d \le 3$.
Now, we claim that the image of $\alpha$ in $\pi_{2k+d-1} \mathbb{S}$ must actually be in the subgroup generated by $\mathcal{J}_{2k+d-1}$, without Adams's $\mu$-family.  This follows simply from the fact that this class is in the image of the map
$$\pi_{2k+d} \MO \langle k \rangle / \mathbb{S} \to \pi_{2k+d} \mathbb{S}^1,$$
and therefore in the kernel of the map
$$\pi_{2k+d-1} \mathbb{S} \to \MO \langle k \rangle.$$
Recall that the Atiyah--Bott--Shapiro orientation \cite{ABS} determines a unital map 
$$\mathrm{MO} \langle 3 \rangle = \mathrm{MSpin} \to \mathrm{KO}.$$
The composite map
$$\pi_{2k+d-1} \mathbb{S} \to \pi_{2k+d-1} \MO \langle k \rangle \to \pi_{2k+d-1} \MO \langle 3 \rangle \to \pi_{2k+d-1} \mathrm{KO}$$
has the effect of killing $\mathcal{J}_{2k+d-1}$ without killing any of Adams's $\mu$-family.  Thus, any class in the kernel of the map
$$\pi_{2k+d-1} \mathbb{S} \to \pi_{2k+d-1} \MO \langle k \rangle$$
cannot be a sum of a class in $\mathcal{J}_{2k+d-1}$ and a non-trivial element of the $\mu$-family.
\end{proof}

\begin{proof} [Proof of Theorem \ref{thm:mainboundary}]
This follows by combining Theorem \ref{thm:strongmainboundary} and Theorem \ref{thm:weakmainboundary}.
\end{proof}

\begin{lem}\label{lemm:inequality}
    The inequality \[h(k-1) - \lfloor \log_2(2k+d+1) \rfloor + 1 \geq \frac{2k+d}{5} + 15\] holds for $k > 232$ and $0 \leq d \leq 3$.
\end{lem}

\begin{proof}
  Without loss of generality we may assume $d=3$.
  Then, using the inequality $h(k-1) \geq \frac{k}{2}-1$, it will suffice to show that
  \begin{align}
    \frac{k}{10} \geq \frac{3}{5} + 15 + \log_2 (2k+4). \label{eq:stolz-inequality}
  \end{align}
  Taking derivatives we see that the left hand side increases faster than the right hand side as soon as $k \geq 13$. Using a computer we find that \Cref{eq:stolz-inequality} holds for $k=246$, so the lemma holds for $k \geq 246$. For the remaining values of $k$, we compute each side of the desired inequality
  \[h(k-1) - \lfloor \log_2(2k+d+1) \rfloor + 1 \geq \frac{2k+d}{5} + 15\]
  for $d=3$, and display their difference, $\Delta$, in the following table:

  \begin{table}[ht!]
    \begin{center}
      \begin{tabular}{|c||c|c|c|c|c|c|c|c|c|c|c|c|c|}\hline
        $k$ & 233 & 234 & 235 & 236 & 237 & 238 & 239 & 240 & 241 & 242 & 243 & 244 & 245 \\\hline
        $\Delta$ & 0.2 & 0.8 & 1.4 & 1.0 & 1.6 & 1.2 & 0.8 & 0.4 & 1.0 & 1.6 & 2.2 & 1.8 & 2.4 \\\hline
      \end{tabular}
    \end{center}
  \end{table}
  

\end{proof}

\begin{rmk} \label{rmk:open-questions}
In Section \ref{sec:open-problems} we discussed possible improvements to Theorem \ref{thm:strongmainboundary}.
We have spent comparatively little effort optimizing Theorem \ref{thm:weakmainboundary}, and it would be interesting to see an improvement of the bounds $k>232$ and $d \le 3$.

For a fixed dimension $m$, it would be interesting to know the largest integer $\ell$ such that a smooth, $\ell$-connected, almost closed $m$-manifold bounds an element non-trivial in $\mathrm{coker}(J)_{m-1}$.
Conjecture \ref{cnj:vanishing-line} suggests that this integer $\ell$ should be closer to $\frac{m}{3}$ than $\frac{m}{2}$.
\end{rmk}


%% file: SynRevIntro.tex
At this point, we have reduced our main theorems to three technical results, which will appear as Theorems \ref{thm:AdamsBound}, \ref{thm:mod8-main-thm} and \ref{thm:app-main}.
Additionally, in \Cref{sec:Goodwillie} we referred to \Cref{rmk:x2-filt}.
Each of these results relies on an analysis of Adams filtration.

First, we focus on \Cref{thm:AdamsBound}, which bounds the Adams filtration of the Toda bracket $w \in \pi_{8n-1} \bS$ defined in \Cref{lem:toda-main}.
To understand why Toda brackets have controllable Adams filtration, it is helpful to consider the following facts:
\begin{enumerate}
\item Adams filtration is super-additive under function composition,

  \noindent i.e. $AF(fg) \geq AF(f) + AF(g)$.
\item Toda brackets are a kind of \emph{secondary} composition operation.
\end{enumerate}
These facts suggest that we should be able to compute lower bounds for the Adams filtration of a Toda bracket. In practice, this can be subtle, since such bounds require us to keep track not only of the Adams filtrations of maps but also of the Adams filtrations of homotopies.

We believe that questions involving the Adams filtrations of homotopies are greatly clarified by recent work of Piotr Pstr\k{a}gowski, and in particular his development of the category of \emph{synthetic spectra} \cite{Pstragowski}. For $E$ an Adams-type homology theory, $E$-based synthetic spectra form an $\infty$-category $\mathrm{Syn}_{E}$ of formal $E$-based Adams spectral sequences.  We devote this section to a review of the basic properties of $\mathrm{Syn}_E$, some of which have not appeared in the literature.

\begin{dfn}
  Suppose that $E$ is a homotopy associative ring spectrum such that $E_*$ and $E_*E$ are graded commutative rings. Following \cite[Definition 3.12]{Pstragowski}, we say that a finite spectrum $X$ is \emph{finite $E_*$-projective} (or simply \emph{finite projective} if $E$ is clear from context) if $E_*X$ is a projective $E_*$-module.
  We say that $E$ is of \emph{Adams type} if $E$ can be written as a filtered colimit of finite projective spectra $E_\alpha$ such that the natural maps $$E^*E_\alpha \to \Hom_{E_*}(E_*E_\alpha, E_*)$$ are isomorphisms.
\end{dfn}

\begin{exm}
  In this paper, we will make use only of the examples $E=\BP$ and $E=\mathrm{H}\mathbb{F}_p$ for some prime $p$, both of which are of Adams type.
\end{exm}

\begin{cnstr}[Pstr\k{a}gowski]
Let $E$ denote an Adams type homology theory.  Then there is a stable, presentably symmetric monoidal $\infty$-category $\mathrm{Syn}_E$ together with a functor
$$\nu_E:\Sp \to \mathrm{Syn}_E,$$ 
which is lax symmetric monoidal and preserves filtered colimits \cite[Lemma 4.4]{Pstragowski}.  However, $\nu_E$ does \emph{not} preserve cofiber sequences in general. When $E$ is clear from context, we will often denote $\nu_E$ by $\nu$.
\end{cnstr}

\begin{rmk}
The tensor product in synthetic spectra preserves colimits in each variable separately.
\end{rmk}

\begin{rmk} \label{rmk:strong-monoidal}
If $X$ and $Y$ are any two spectra, then the lax symmetric monoidal structure on $\nu$ provides us with a natural comparison map
$$ \nu(X) \otimes \nu(Y) \to \nu(X \otimes Y). $$
In some cases this comparison map is actually an equivalence.
For example, $\nu$ is symmetric monoidal when restricted to the full subcategory of finite projectives.
More generally, \cite[Lemma 4.24]{Pstragowski} proves that the comparison map is an equivalence whenever $X$ is a filtered colimit of finite projectives. Note that this condition is only on $X$, and $Y$ may be arbitrary.

If $E=\mathrm{H}\mathbb{F}_p$, then every finite spectrum is finite projective, and so every spectrum $X$ satisfies the condition above. Thus, $\nu_{\mathrm{H}\mathbb{F}_p}$ is symmetric monoidal, rather than merely lax symmetric monoidal.
\end{rmk}

\begin{rmk} \label{rmk:dualizable-generators}
  As proved in \cite[Lemma 3.18]{Pstragowski}, the full subcategory of spectra spanned by the finite projective spectra is rigid symmetric monoidal.
  Furthermore, \cite[Remark 4.14]{Pstragowski} proves that the set of $\Sigma^k \nu P$ with $k \in \Z$ and $P$ finite projective is a family of compact generators of $\mathrm{Syn}_E$. The fact that $\nu$ is symmetric monoidal when restricted to finite projective spectra implies that this is a family of dualizable compact generators.
\end{rmk}

If $X$ is a spectrum, then $\nu X$ records detailed information about the $E$-based Adams tower for $X$.  A first hint of this is found in the following proposition:

\begin{lem}[{\cite[Lemma 4.23]{Pstragowski}}]   \label{lemm:syn-cof}
  Suppose that
  $$ A \to B \to C $$
	is a cofiber sequence of spectra. Then
  $$ \nu A \to \nu B \to \nu C $$
	is a cofiber sequence of synthetic spectra if and only if
  $$ 0 \to E_*A \to E_*B \to E_*C \to 0 $$
  is a short exact sequence of $E_*E$-comodules.\footnote{The condition that a cofiber sequence become short exact on $E$-homology is exactly the condition under which there is a long exact sequence on the level of Adams $\mathrm{E}_2$-pages.}
\end{lem}


%% file: SynRevBigraded.tex
To precisely relate $\nu X$ to the $E$-based Adams spectral sequence for $X$, we must introduce \emph{bigraded spheres} and the canonical bigraded homotopy element $\tau$. 

\begin{dfn}[{\cite[Definitions 4.6 and 4.9]{Pstragowski}}]
The \emph{bigraded sphere} $\bS^{n,n}$ is defined to be $\nu(\bS^{n})$.  Since $\mathrm{Syn}_E$ is stable, we more generally define
$\bS^{a,b}$
to be $\Sigma^{a-b} \bS^{b,b}$, which makes sense even if $a-b<0$. For any synthetic spectrum $X$, the \emph{bigraded homotopy groups} $\pi_{a,b}(X)$ are defined to be the abelian groups
$$ \pi_{a,b}(X) = \pi_0 \Hom(\bS^{a,b},X). $$
\end{dfn}


\begin{rmk}
  The fact that $\nu$ is symmetric monoidal when restricted to finite projectives (such as $\bS^b$) implies that each of the bigraded spheres $\bS^{a,b}$ is $\otimes$-invertible. Thus, bigraded homotopy groups are particular instances of Picard-graded homotopy groups.
\end{rmk}

\begin{rmk} \label{rmk:colimit-compare}
Recall that, if $F:\mathcal{C} \to \mathcal{D}$ is any functor of pointed, cocomplete $\infty$-categories, the definition of $\Sigma$ as a pushout gives natural comparison morphisms
$$\Sigma F(c) \to F(\Sigma c).$$
\end{rmk}

\begin{dfn}[{\cite[Definition 4.27]{Pstragowski}}]
The natural comparison map 
$$\bS^{0,-1}=\Sigma \nu(\bS^{-1}) \longrightarrow \nu(\Sigma \bS^{-1})=\bS^{0,0}$$
is denoted by $\tau$.  In short, $\tau$ is a canonical element of $\pi_{0,-1} \bS^{0,0}$.   The symbol $C\tau$ denotes the cofiber of $\tau$.  A synthetic spectrum $X$ is said to be \emph{$\tau$-invertible} if the map $$\tau:\Sigma^{0,-1} X \to X$$ is an equivalence.  
\end{dfn}

Using $\tau$ we can now give a global description of the category of synthetic spectra.
Although the description given in \Cref{thm:tau-inv} is concise and powerful,
we will ultimately trade it in for the more computationally precise \Cref{thm:synthetic-Adams}.

\begin{thm}[Pstr\k{a}gowski] \label{thm:tau-inv}\ 

  \begin{enumerate}
  \item The localization functor given by inverting $\tau$ is symmetric monoidal. \todo{should we say it's smashing?}
  \item The full subcategory of $\tau$-invertible synthetic spectra is equivalent to the category of spectra.
  \item The composite $\tau^{-1} \circ \nu$ is equivalent to the identity functor on $\Sp$.
  \item The object $C\tau$ admits the structure of an $\mathbb{E}_\infty$-ring in $\Syn_E$.
  \item Suppose that $E$ is homotopy commutative.
    Then there is a natural fully faithful, monoidal functor
    $$ \mathrm{Mod}_{C\tau} \to \mathrm{Stable}_{E_*E},$$
    where the target is Hovey's stable $\infty$-category of comodules and
    the composition of $\nu(-) \otimes C\tau$ with this functor is equivalent to $E_*(-)$.
  \end{enumerate}
  We can construct the following diagram, where every arrow except $\nu$ and $E_*(-)$ is a left adjoint.
  \begin{center}
    \begin{tikzcd}
      & \Sp \ar[dl, swap, "1"] \ar[d, "\nu"] \ar[drr, bend left, "E_*(-)"] & & \\
      \Sp & \mathrm{Syn}_E \ar[l, "\tau^{-1}"] \ar[r, "- \otimes C\tau"] &
      \mathrm{Mod}_{C\tau} \ar[r] & \mathrm{Stable}_{E_*E}
    \end{tikzcd}
  \end{center}
\end{thm}

Before proving \Cref{thm:tau-inv} we record the following useful corollary.

\begin{cor}[{\cite[Lemma 4.56]{Pstragowski}}] \label{lemm:ctauE2}
  For any spectrum $X$, there is a natural isomorphism of bigraded abelian groups
  $$ \pi_{t-s,t}(C\tau \otimes \nu X) \cong \Ext_{E_*E}^{s,t}(E_*,E_*X).$$
  Note that the latter object is the $\mathrm{E}_2$-page of the $E$-based Adams spectrum sequence for $X$.
\end{cor}

\begin{proof}[Proof of \Cref{thm:tau-inv}]
  Except for the claim that the functor in (5) is symmetric monoidal, this theorem is just a combination of citations to \cite{Pstragowski}:
  (1) is \cite[Theorem 4.36 and Proposition 4.39]{Pstragowski}, (2) is \cite[Theorem 4.36]{Pstragowski}, (3) is \cite[Proposition 4.39]{Pstragowski}, (4) is \cite[Corollary 4.45]{Pstragowski} and most of (5) is \cite[Theorem 4.46 and Remark 4.55]{Pstragowski}.

  We finish by proving the remaining claim.
  By \cite[Lemma 4.43]{Pstragowski} the left adjoint
  $$ \epsilon_* : \mathrm{Syn}_E \to \mathrm{Stable}_{E_*E} $$
  is symmetric monoidal because $E$ is homotopy commutative.\footnote{While \cite[Lemma 4.43]{Pstragowski} as written does not state that $E$ needs to be homotopy commutative, this hypothesis is necessary: we refer the reader to Lemma 4.44 in the second arXiv version of \cite{Pstragowski}.} Then, by \cite[Corollary I.2.5.5.3]{SAG} and \cite[Lemma 4.44]{Pstragowski}, there is a factorization of lax symmetric monoidal right adjoints
  $$ \mathrm{Stable}_{E_*E} \to \mathrm{Mod}_{C\tau} \to \mathrm{Syn}_E.$$
  In particular, this means that the left adjoint $\mathrm{Mod}_{C\tau} \to \mathrm{Stable}_{E_* E}$ canonically acquires the structure of an oplax symmetric monoidal functor \cite[Proposition A]{OpLaxAdjoint}.
  It remains to check that the comparison maps provided by the oplax structure are equivalences.
  Because the tensor products on $\mathrm{Syn}_E$ and $\mathrm{Stable}_{E_*E}$ are cocontinuous in each variable, it suffices to check this on compact generators. This follows from \cite[Lemma 4.43]{Pstragowski} and the fact that $\mathrm{Mod}_{C\tau}$ is compactly generated by objects of the form $C\tau \otimes M$. 
\end{proof}

\begin{rmk}
Altogether, \Cref{thm:tau-inv} suggests the following geometric picture of synthetic spectra:
Synthetic spectra form a $\mathbb{G}_m$-equivariant family over $\mathbb{A}^1$, where $\tau$ is the coordinate on $\mathbb{A}^1$.
The special fiber of this family is a category of comodules while the generic fiber is the category of spectra.
We will not pursue this perspective further in the present paper.
\end{rmk}

\begin{lem} \label{lemm:adams-fil1}
  If a map $f:X \longrightarrow Y$ of spectra has $E$-Adams filtration $k$, then there exists a factorization:
  \begin{center}
    \begin{tikzcd}
      & \Sigma^{0,-k} \nu(Y) \arrow{d}{\tau^{k}} \\
        \nu(X) \arrow[dashed]{ur}{\wt{f}} \arrow{r}{\nu(f)} & \nu(Y)
    \end{tikzcd}
  \end{center}
\end{lem}

\begin{proof}
  Any map which is of Adams filtration $k$ can be factored into a composite of $k$ maps each of Adams filtration 1. Then, by pasting diagrams as shown below it will suffice to prove the lemma for $k=1$.
  \begin{center}
    \begin{tikzcd}
      & & \Sigma^{0,-a-b} \nu(C) \ar[d, "\tau^b"] \\
      & \Sigma^{0,-a} \nu(B) \ar[ur, dashed] \ar[r, "\nu(h)"] \ar[d, "\tau^a"] & \Sigma^{0,-a}\nu(C) \ar[d, "\tau^a"] \\
      \nu(A) \ar[r, "\nu(g)"] \ar[ur, dashed] & \nu(B) \ar[r, "\nu(h)"] & \nu(C)
    \end{tikzcd}
  \end{center}  

  Using the associated cofiber sequence
  $$ \Sigma^{-1}Y \xrightarrow{g} Z \xrightarrow{h} X \xrightarrow{f} Y,  $$
  we can build the diagram below.
  \begin{center}
    \begin{tikzcd}
      \nu (\Sigma^{-1} Y) \ar[r] \ar[d, "\nu(g)"] & 0 \ar[r] \ar[d] & \Sigma \nu (\Sigma^{-1}Y) \ar[r, "\tau"] \ar[d] & \nu(Y) \ar[d, equal] \\
      \nu (Z) \ar[r, "\nu(h)"] & \nu(X) \ar[r] \ar[rr, bend right, "\nu(f)"] & \text{cof}(\nu(h)) \ar[r] & \nu(\mathrm{cof}(h)).
    \end{tikzcd}
  \end{center}

  In this diagram, the first pair of maps in each row form cofiber sequences and the right-most map in each row is an assembly map.
  Now, since $f$ has positive Adams filtration it is zero on $E$-homology and therefore $g$ and $h$ satisfy the conditions of \Cref{lemm:syn-cof}.
  This implies that the left-most square is cocartesian and the third vertical map is an equivalence, which provides the desired factorization of $\nu(f)$.
\end{proof}

Before we can relate the bigraded homotopy groups of $\nu X$ to the $E$-based Adams spectral sequence of $X$, we must engage in a brief discussion of completion and convergence.

\begin{dfn} \label{rmk:E-complete} \label{dfn:E-complete}
  A spectrum $X$ is said to be \emph{$E$-nilpotent complete} if the $E$-based Adams resolution for $X$ converges to $X$. In this paper we will make use of two instances of this:
  \begin{itemize}
      \item Any bounded below, $p$-local spectrum $X$ is $\BP$-nilpotent complete \cite[Theorem 6.5]{BousfieldLocalization}.
      \item Any bounded below, $p$-complete spectrum $X$ is $\mathrm{H}\mathbb{F}_p$-nilpotent complete \cite[Theorem 6.6]{BousfieldLocalization}.
  \end{itemize}
\end{dfn}

\begin{dfn} \label{dfn:strong-conv}
    Following Boardman \cite[Defintion 5.2]{Boardman}, we will say that the $E$-based Adams spectral sequence for a spectrum $X$ is \emph{strongly convergent} if:
    \begin{itemize}
        \item The $E$-Adams filtration $F^{\bullet} \pi_{*} (X)$ of the homotopy groups of $X$ is complete and Hausdorff.
        \item There are isomorphisms $F^{s} \pi_{t-s} (X) / F^{s+1} \pi_{t-s} (X) \cong \mathrm{E}^{s,t}_\infty (X)$, where $\mathrm{E}_{\infty} ^{s,t} (X)$ is the $\mathrm{E}_\infty$-page of the $E$-Adams spectral sequence for $X$.
    \end{itemize}
\end{dfn}


\begin{rmk}
    \Cref{dfn:E-complete} and \Cref{dfn:strong-conv} have obvious analogs for synthetic spectra, and we will make use of these analogs without further mention.
\end{rmk}

Our strongest result concerning the relationship between synthetic spectra and Adams spectral sequences is \Cref{thm:synthetic-Adams} (stated below). This theorem provides a dictionary between the structure of the $E$-Adams spectral sequence for $X$ and the structure of the bigraded homotopy groups of $\nu X$. The proof of \Cref{thm:synthetic-Adams} is quite technical, and we defer it to \Cref{sec:proof-synth-Adams}. We have structured the paper so that the reader willing to assume \Cref{thm:synthetic-Adams} need not read \Cref{sec:proof-synth-Adams}.

In order to highlight many of the subtleties which can arise in applying \Cref{thm:synthetic-Adams}, we give example calculations of $\pi_{*,*}(\nu_{\HFt} \bS_2^\wedge)$ through the Toda range in \Cref{subsec:syn-toda-range}.  We strongly recommend that any reader seeking to understand Theorem \ref{thm:synthetic-Adams} examine \Cref{subsec:syn-toda-range}.

\begin{thm} \label{thm:synthetic-Adams}
  Let $X$ denote an $E$-nilpotent complete spectrum with strongly convergent $E$-based Adams spectral sequence.
  Then we have the following description of the bigraded homotopy groups of $\nu X$.

  Let $x$ denote a class in topological degree $k$ and filtration $s$ of the $\mathrm{E}_2$-page of the $E$-based Adams spectral sequence for $X$. The following are equivalent:
  \begin{itemize}
  \item[(1a)] Each of the differentials $d_2$,...,$d_r$ vanish on $x$.
  \item[(1b)] $x$, viewed as an element of $\pi_{k,k+s}(C\tau \otimes \nu X)$, lifts to $\pi_{k,k+s}(C\tau^r \otimes \nu X)$.
  \item[(1c)] $x$ admits a lift to $\pi_{k,k+s}(C\tau^r \otimes \nu X)$ whose image under the $\tau$-Bockstein
  $$ C\tau^r \otimes \nu X \to \Sigma^{1,-r} C\tau \otimes \nu X $$
  is equal to $-d_{r+1}(x)$.
  \end{itemize}

  If we moreover assume that $x$ is a permanent cycle, then there exists a (not necessarily unique) lift of $x$ along the map
  $\pi_{k,k+s}(\nu X) \to \pi_{k,k+s}(C\tau \otimes \nu X)$.  For any such lift, $\wt{x}$, the following statements are true:
  \begin{itemize}
  \item[(2a)] If $x$ survives to the $\mathrm{E}_{r+1}$-page, then $\tau^{r-1} \wt{x} \neq 0$.
  \item[(2b)] If $x$ survives to the $\mathrm{E}_\infty$-page, then the image of $\wt{x}$ in $\pi_{k} (X)$ is of $E$-Adams filtration $s$ and detected by $x$ in the $E$-based Adams spectral sequence.\footnote{The image of $\wt{x}$ in $\pi_{k} (X)$ refers to the image of $\wt{x}$ under the map $\pi_{k,k+s} (\nu X) \to \pi_{k} (\tau^{-1} \nu X) \cong \pi_k (X)$ induced by the functor $\tau^{-1}$ of \Cref{thm:tau-inv}.}
  \end{itemize}
  Furthermore, there always exists a choice of lift $\wt{x}$ satisfying additional properties:
  \begin{itemize}
  \item[(3a)] If $x$ is the target of a $d_{r+1}$-differential, then we may choose $\wt{x}$ so that $\tau^r \wt{x} = 0$.
  \item[(3b)] If $x$ survives to the $\mathrm{E}_\infty$-page, and $\alpha \in \pi_k X$ is detected by $x$, then we may choose $\wt{x}$ so that
    $\tau^{-1} \wt{x} = \alpha$. In this case we will often write $\wt{\alpha}$ for $\wt{x}$.
  \end{itemize}
  Finally, the following generation statement holds:
  \begin{itemize}
  \item[(4)] Fix any collection of $\wt{x}$ (not necessarily chosen according to (3)) such that the $x$ span the permanent cycles in topological degree $k$.  Then the $\tau$-adic completion of the $\Z[\tau]$-submodule of $\pi_{k,*}(\nu X)$ generated by those $\wt{x}$ is equal to $\pi_{k,*} (\nu X)$.\footnote{We consider $\pi_{k,*}(\nu X)$ as a graded abelian group with an operation $\tau$ which decreases the grading by 1.}
  \end{itemize}
\end{thm}

The proof is somewhat involved, so we postpone it to \Cref{sec:proof-synth-Adams}.
We extract below some more digestible corollaries of the above omnibus theorem.

\begin{cor}\label{lemm:ctaugoodenough}
  Let $X$ denote an $E$-nilpotent complete spectrum with strongly convergent $E$-based Adams spectral sequence.
  Suppose for fixed integers $a$ and $b$ that 
  $$\pi_{a,b+s} (C\tau \otimes \nu X) = 0$$
	for all integers $s \geq 0$.   Then it is also true that $\pi_{a,b+s} (\nu X) = 0$ for all $s \geq 0$.
\end{cor}

\begin{proof}
    This follows by combining the vanishing assumption and \Cref{thm:synthetic-Adams}(4).
\end{proof}

We next note that the filtration by ``divisibility by $\tau$'' coincides with the Adams filtration:

\begin{cor}\label{cor:synth-filt}
  Let $X$ denote an $E$-nilpotent complete spectrum with strongly convergent $E$-based Adams spectral sequence.
  Then the filtration of $\pi_k(X)$ given by
  $$ F^s \pi_k(X) := \mathrm{im}(\pi_{k,k+s}(\nu X) \to \pi_k(X)) $$
  coincides with the $E$-Adams filtration on $\pi_k(X)$.
\end{cor}

\begin{proof}
  We show that each filtration contains the other.
  \Cref{lemm:adams-fil1} provides an inclusion in one direction: if $x \in \pi_k(X)$ has $E$-Adams filtration $\geq s$, then $x \in F^s\pi_k(X)$.

    Suppose now that $x \in F^s \pi_{k} (X)$, so that we may find some $\wt{x} \in \pi_{k,k+s} (\nu X)$ that maps to $x$. We may assume without loss of generality that $s$ was chosen maximally. Let $y$ be the image of $\wt{x}$ in $\pi_{k,k+s} (C\tau \otimes \nu X)$.

    Suppose that $y$ is a boundary in the $E$-Adams spectral sequence. Then by \Cref{thm:synthetic-Adams}(3a) there exists a $\tau$-power torsion element $\wt{y}$ lifting $y$. It follows that $\wt{x}-\wt{y} = \tau \wt{z}$ for some $\wt{z} \in \pi_{k,k+s+1} (\nu X)$. But then $\wt{z}$ maps to $x \in \pi_k (X)$ under $\tau^{-1}$, which implies that $x \in F^{s+1} \pi_{k} (X)$. This contradicts the maximality assumption on $s$.

    We conclude that $y$ cannot be a boundary. Then \Cref{thm:synthetic-Adams}(2b) finishes the proof.
  %
%
%
\end{proof}

Let $\pi_{k,k+s}(\nu X)^{\mathrm{tor}}$ denote the subgroup of $\tau$-power torsion elements.
We obtain the following $\tau$-power torsion order bound:

\begin{cor}\label{cor:tau-torsion-bound}
  Let $X$ denote an $E$-nilpotent complete spectrum with strongly convergent $E$-based Adams spectral sequence.
  Then the $\tau$-torsion order of $\pi_{k,k+s}(\nu X)^{\mathrm{tor}}$ is equal to the maximum of
  \begin{enumerate}
  \item one less than the $\tau$-torsion order of $\pi_{k,k+s+1}(\nu X)^{\mathrm{tor}}$,
  \item one less than the length of the longest Adams differential entering $\mathrm{E}_*^{s,k+s}$.
  \end{enumerate}
\end{cor}

\begin{proof}
  Suppose $x \in \pi_{k,k+s}(\nu X)^{\mathrm{tor}}$, and let $y$ denote the image of $x$ in $\pi_{k,k+s}(C\tau \otimes \nu X)$.
  Choose a lift $\wt{y}$ of $y$ as in \Cref{thm:synthetic-Adams}(3).
  
  Suppose that $y$ is not a boundary in the $E$-Adams spectral sequence.
  Then, $\wt{y} - x$ is divisible by $\tau$ while
  $ \tau^{-1}(\wt{y} - x) = \tau^{-1} \wt{y} $
  is detected by $y$, which contradicts \Cref{cor:synth-filt}.
  
  We conclude that $y$ must be a boundary in the $E$-Adams spectral sequence. Let $r$ denote the length of the differential that hits $y$.
    Then $\wt{y} - x$ is divisible by $\tau$ and $\wt{y}$ is $\tau^{r-1}$-torsion, so the desired bound on the $\tau$-torsion order of $x$ follows.
\end{proof}

%

%

%% file: SyntheticToda.tex
In this section, we will prove \Cref{thm:filt-bound-finale}, which we used in \Cref{sec:Finale} to prove \Cref{thm:intro-main}. That is to say, we provide for each prime $p$ a bound on the $\HFp$-Adams filtration of the Toda bracket $w$ of \Cref{lem:toda-main}.

To accomplish this, we will lift the Toda bracket along the functor \[ \tau^{-1} : \mathrm{Syn}_{\HFp} \to \Sp\] in such a way that \Cref{cor:synth-filt} implies the existence of the desired bound on the $\HFp$-Adams filtration.

The first ingredient that we will need is a bound on the $\HFp$-Adams filtration of the map \[J : \Sigma^{\infty} \Oo\mathrm{\langle4n-1 \rangle} \to \Ss,\] at least when restricted to a skeleton of $\Sigma^{\infty} \Oo\mathrm{\langle 4n-1 \rangle}$. We must restrict to a skeleton because the map $J$ does not otherwise have high $\HFp$-Adams filtration.

\begin{cnv}
In the remainder of this section, we fix a prime $p$ and implicitly $p$-complete all spectra. Furthermore, all synthetic spectra will be taken with respect to $\HFp$.
\end{cnv}
\todo{Jeremy: Check numbers in this section.}

\begin{rmk}
 Recall from Definition \ref{dfn:skeleton} that
	$$M \longrightarrow \Sigma^{\infty} \Oo \mathrm{\langle 4n-1 \rangle}$$
	denotes the inclusion of an $(8n-1)$-skeleton of $\Sigma^{\infty} \Oo \mathrm{\langle 4n-1 \rangle}$. In particular, the induced map
	$$\mathrm{H}_*(M;\F_p) \longrightarrow \mathrm{H}_*(\Sigma^{\infty} \Oo \mathrm{\langle 4n-1 \rangle};\F_p)$$
is an isomorphism for $* < 8n-1$ and a surjection for $* = 8n-1$, and $\mathrm{H}_*(M;\F_p) \cong 0$ for $* > 8n-1$.
\end{rmk}

\begin{ntn}
    Let $h(k)$ denote the number of integers $0 < s \leq k$ which are congruent to one of $0,1,2$ or $4$ mod $8$. Then we set \[N_2 = h(4n-1) - \lfloor \log_2 (8n) \rfloor + 1\] and, for $p$ odd, \[N_p = \left\lfloor \frac{4n}{2p-2} \right\rfloor - \left\lfloor \log_p (4n) \right\rfloor.\] Note that this notation suppresses the dependence of $N_2$ and $N_p$ on $n$.
\end{ntn}

\begin{lem}\label{lemm:filtration}
The $\HFp$-Adams filtration of the composite map of spectra
	$$M \longrightarrow \Sigma^{\infty} \Oo \mathrm{\langle 4n-1 \rangle} \stackrel{J}{\longrightarrow} \mathbb{S}$$
	is at least $N_p$.
\end{lem}

We will prove \Cref{lemm:filtration} in \Cref{subsec:filtration}.
Using \Cref{lemm:filtration}, we proceed to construct a lift of the diagram defining the Toda bracket $w$ to $\mathrm{Syn}_{\HFp}$.
We take the first step of this construction below:

\begin{cnstr}
By \Cref{lemm:adams-fil1}, \Cref{lemm:filtration} implies the existence of a factorization in $\mathrm{Syn}_{\HFp}$
\begin{center}
  \begin{tikzcd}
    & \arrow{d}{\tau^{N_p}} \Ss^{0,-N_p} \\
    \nu M \arrow[dashed]{ur} \arrow{r}{J} & \Ss^{0,0},
  \end{tikzcd}
\end{center}
which we will prefer to view as a morphism
$$\wt{J}:\Sigma^{0,N_p}\nu M \longrightarrow \Ss^{0,0}.$$
As in \Cref{dfn:skeleton}, we view $x$ as an element of $\pi_{4n-1} M$.
We may then obtain a class $y$ as the composition
$$ y:\Ss^{4n-1,4n+N_p-1} \xrightarrow{\nu(x)} \Sigma^{0,N_p} \nu M \xrightarrow{\wt{J}} \Ss^{0,0}. $$
This element $y$ is a member of the bigraded homotopy group $\pi_{4n-1,4n+N_p-1}\Ss^{0,0}$.
\end{cnstr}

Before constructing our lift of the Toda bracket $w$ we reproduce the relevant diagram, which appeared in Lemma \ref{lem:toda-main}, for the convenience of the reader.
\begin{center}
  \begin{tikzcd}[column sep = huge]
    \mathbb{S}^{8n-2} \arrow{r}{2}
    \arrow[rrdd, bend right=20,""{name=D}] \arrow[rrdd, bend right=20,swap,"0"] \arrow[rr,bend left = 30,""{name=U},"0"]
    \arrow[rrdd, bend right=20,""{name=D}] \arrow[rrdd, bend right=20,swap,"0"] \arrow[rr,bend left = 30,""{name=U},"0"]
    & \mathbb{S}^{8n-2} \arrow[Leftrightarrow, from=D, "h"] \arrow[Leftrightarrow, from=U, "f"] \arrow{r}{xJ(x)} \arrow[rdd,""{name=MU},"J(x)^2" description] \arrow[rdd,swap, ""{name=MD},"J(x)^2" description]
    & M \arrow{dd}{J}  \arrow[bend left=20,Leftrightarrow, from=MU, swap, "g"] \\ \\
    & & \mathbb{S}.
  \end{tikzcd}
\end{center}
The homotopies $f,g$ and $h$ are chosen as follows:
\begin{itemize}
\item $f$ is an arbitrary nullhomotopy.
\item $g$ is the canonical homotopy associated to the fact that $J$ is a map of $\Ss$-modules.
\item $h$ is the canonical nullhomotopy given by the $\mathbb{E}_\infty$-ring structure on $\Ss$.
\end{itemize}

\begin{cnstr}\label{cnstr:synth-toda-diagram}
  We may form the following diagram of morphisms and homotopies in $\Syn_{\HFp}$
\begin{center}
  \begin{tikzcd}[column sep = huge]
    \mathbb{S}^{8n-2,8n+2N_p-2} \arrow{r}{2}
    \arrow[rrdd, bend right=20,""{name=D}] \arrow[rrdd, bend right=20,swap,"0"] \arrow[rr,bend left = 30,""{name=U},"0"]
    \arrow[rrdd, bend right=20,""{name=D}] \arrow[rrdd, bend right=20,swap,"0"] \arrow[rr,bend left = 30,""{name=U},"0"]
    & \mathbb{S}^{8n-2,8n+2N_p-2} \arrow[Leftrightarrow, from=D, "\tilde{h}"] \arrow[Leftrightarrow, from=U, "\tilde{f}"] \arrow{r}{\nu(x)y} \arrow[rdd,""{name=MU},"y^2" description]
    & \Sigma^{0,N_p}\nu(M) \arrow{dd}{\wt{J}}  \arrow[swap,Leftrightarrow, from=MU,"\tilde{g}"] \\ \\
    & & \Ss^{0,0}, 
  \end{tikzcd}
\end{center}
where the homotopies $\tilde{f},\tilde{g}$ and $\tilde{h}$ are chosen as follows:
\begin{itemize}
\item $\tilde{f}$ is an arbitrary nullhomotopy, which exists as a consequence of \Cref{thm:vanishinggroup}.
\item $\tilde{g}$ is the canonical homotopy that expresses the fact that $\tilde{J}$ is a map of right $\Ss^{0,0}$-modules.
\item $\tilde{h}$ is the canonical nullhomotopy that comes from the fact that $\Ss^{0,0}$ is an $\mathbb{E}_\infty$-ring in the symmetric moniodal $\infty$-category $\mathrm{Syn}_{\HFp}$.
\end{itemize}

\end{cnstr}

\begin{prop}\label{thm:vanishinggroup}
    The bigraded homotopy group $\pi_{8n-2,8n+N_p-2}(\nu M )$ is trivial for $n \geq 3$.
\end{prop}

We will prove \Cref{thm:vanishinggroup} in \Cref{subsec:vanishinggroup}. By construction, the diagram of \Cref{cnstr:synth-toda-diagram} maps under the symmetric monoidal functor $\tau^{-1}$ to the diagram of \Cref{lem:toda-main}. We are therefore able to read off the following $\HFp$-Adams filtration bound on the resulting Toda bracket.

\begin{thm} \label{thm:AdamsBound}
There exists a choice of $f$ in the statement of Lemma \ref{lem:toda-main} such that the $\HFp$-Adams filtration of the Toda bracket $w$ is at least $2N_p-1$.
\end{thm}

\begin{proof}
	On the one hand, applying $\tau^{-1}$ to the diagram of \Cref{cnstr:synth-toda-diagram} yields the diagram of \Cref{lem:toda-main}. On the other hand, the Toda bracket presented by \Cref{cnstr:synth-toda-diagram} is given by an element of $\pi_{8n-1,8n+2N_p-2} \Ss^{0,0}.$ Therefore \Cref{cor:synth-filt} implies that it realizes to an element of Adams filtration at least \[\left(8n+2N_p-2\right)-(8n-1)=2N_p-1. \qedhere\]
\end{proof}

\begin{rmk}
    As mentioned in the introduction, this is an improvement on a bound of Stolz (cf. \cite[Satz 12.7]{StolzBook}), who works at $p=2$ and bounds the $\HFt$-Adams filtration of $w$ by approximately $N_2$.
\end{rmk}

In the rest of this section, we will prove \Cref{lemm:filtration} and \Cref{thm:vanishinggroup}.

\subsection{Proof of \Cref{lemm:filtration}}
\label{subsec:filtration}\ 

Our proof of \Cref{lemm:filtration} is similar to Stolz's proof of \cite[Satz 12.7]{StolzBook}.

At the prime $2$, our argument will be based on Stong's computation of the cohomology of $\BO\mathrm{\langle m \rangle}$ in \cite{Stong}. At an odd prime, we base our argument on Singer's computation of the cohomology of $\U \mathrm{\langle 2m-1 \rangle}$ in \cite{Singer}. We begin with some notation.

\begin{ntn}
    Given a prime $p$ and an integer $n$ with $p$-adic expansion $n = \Sigma_i a_i p^i$, we let $\sigma_p (n) = \Sigma_i a_i$.
\end{ntn}

\begin{ntn}
    Let $\theta_i \in \mathrm{H}^i (\BO; \F_2)$ for $i \geq 1$ denote the polynomial generators fixed by Stong in \cite{Stong}, so that $\mathrm{H}^i (\BO; \F_2) \cong \F_2 [\theta_i \vert i \geq 1]$.

    Moreover, let $G_m$ denote the image of the canonical map \[\mathrm{H}^*(K(\pi_m \BOm, m); \F_2) \to \mathrm{H}^* (\BOm; \F_2).\]
\end{ntn}

\begin{thm}[{\cite[Theorem A and Corollary on p. 542]{Stong}}] \label{thm:BOmcoh2}
	There is an isomorphism
    \[\mathrm{H}^* (\BOm; \F_2) \cong \FF_2 [\theta_i \vert \sigma_2 (i-1) \geq h(m)] \otimes G_m.\] Moreover, $G_m$ is a polynomial algebra.
\end{thm}

\begin{ntn}
    Fix an odd prime $p$. We let $\mu_{2i +1} \in \mathrm{H}^{2i+1} (\U; \F_p)$ for $i \geq 0$ denote the exterior generators fixed by Singer in \cite{Singer}, so that $\mathrm{H}^* (\U; \F_p) \cong \Lambda_{\F_p} (\mu_{2i+1} \vert i \geq 0)$.
\end{ntn}

\begin{thm}[{\cite[Equations (4.14n) and (4.15n)]{Singer}}]\label{thm:Umcohp}
    Let $p$ be an odd prime. Then there is an isomorphism \[\mathrm{H}^* (\U\mathrm{\langle 2m-1 \rangle}; \F_p) \cong \frac{\mathrm{H}^{*} (\U;\F_p)}{(\mu_{2i+1} \vert \sigma_p (i) < m-1)} \otimes H_m,\] where $H_m \subseteq \mathrm{H}^* (\U\mathrm{\langle 2m-1 \rangle}; \F_p)$ is a sub-Hopf algebra.

    Moreover, the image of the map \[\mathrm{H}^* (\U\mathrm{\langle 2m-2p+1\rangle}; \F_p) \to \mathrm{H}^* (\U\mathrm{\langle 2m-1\rangle}; \F_p)\] is \[\frac{\mathrm{H}^{*} (\U;\F_p)}{(\mu_{2i+1} \vert \sigma_p (i) < m-1)} \otimes 1.\]
\end{thm}

From the above results we can read off the behavior of mod $p$ cohomology under the maps in the Whitehead tower of $\Oo$.

\begin{cor}\label{cor:zerocohmap}
	Assume that $m \equiv 0,1,2,4 \text{ mod } 8$. Then the map $\Oo\mathrm{\langle m \rangle} \to \Om$ induces zero on $\mathrm{H}^* (-; \F_2)$ for $0 < * < 2^{h(m)} - 1$.
\end{cor}

\begin{proof}
    It follows from \Cref{thm:BOmcoh2} that the mod $2$ cohomology of $\BOm$ is polynomial. It therefore follows from \cite[Part II, Corollary 3.2]{SmithEM} that the Eilenberg-Moore spectral sequence for $\mathrm{H}^*(\Oo\mathrm{\langle m-1 \rangle}; \F_2)$ collapses at the $\mathrm{E}_2$-page with $\mathrm{E}_2$-term an exterior algebra on the transgressions of polynomial generators for $\mathrm{H}^*(\BOm; \F_2)$.

    Since $\mathrm{H}^*(K(\pi_m \BO, m); \F_2)$ is also polynomial by \cite[Th\'{e}or\`{e}mes 2 and 3]{SerreCohEM}, the Eilenberg-Moore spectral sequence for $\mathrm{H}^* (K(\pi_m \BO, m-1), \F_2)$ similarly degenerates at the $\mathrm{E}_2$-page with $\mathrm{E}_2$-term an exterior algebra on the transgressions of polynomial generators for $\mathrm{H}^*(K(\pi_m \BO, m); \F_2)$.

    As \[\BOm \to K(\pi_m \BO, m) \] induces a surjective map on $\mathrm{H}^* (-; \F_2)$ for $* < 2^{h(m)}$, we find by the above that it induces a surjective map on $\mathrm{E}_2$ and therefore $E_{\infty}$ page of the Eilenberg-Moore spectral sequence through degree $2^{h(m)} - 2$. We conclude that the bottom Postnikov map \[\Oo\mathrm{\langle m-1 \rangle} \to K(\pi_m \BO, m-1)\] induces a surjection on cohomology through degree $2^{h(m)} - 2$, and therefore that \[\Oo \mathrm{\langle m \rangle} \to \Om\] induces zero on $\mathrm{H}^* (-; \F_2)$ in the desired range.
\end{proof}

As we will need it later, we state the following corollary to the proof of \Cref{cor:zerocohmap}.

\begin{cor}\label{cor:coh-surj-range}
    Assume that $m \equiv 0,1,2,4 \text{ mod } 8$. Then the map \[\Oo\mathrm{\langle m - 1\rangle} \to K(\pi_m \BO, m-1)\] induces a surjective map on $\mathrm{H}^*(-;\F_2)$ for $* \leq 2^{h(m)} - 2$.
\end{cor}

\begin{cor}\label{cor:zerocohmapp}
	Let $p$ be an odd prime. Then the map $\Oo\mathrm{\langle 4m+2p-3 \rangle} \to \Oo \mathrm{\langle 4m-1 \rangle}$ induces zero on $\mathrm{H}^* (-; \F_p)$ for \[0 < * < 2p^{\frac{2m}{p-1}} - 1.\]
\end{cor}

\begin{proof}
  We begin by noting that, since $\Oo\mathrm{\langle n \rangle}$ is a summand of $\U \mathrm{\langle n \rangle}$ compatibly in $n$ (recall that we have implicitly completed at an odd prime), it suffices to prove that the map \[\U \mathrm{\langle 4m+2p-3 \rangle} \to \U \mathrm{\langle 4m-1 \rangle}\] induces zero on $\mathrm{H}^* (-; \F_p)$ for \[0 < * < 2p^{\frac{2m}{p-1}} - 1.\]

	By \Cref{thm:Umcohp}, the image of the map \[\mathrm{H}^* (\U \mathrm{\langle 4m-1 \rangle}; \F_p) \to \mathrm{H}^* (\U \mathrm{\langle 4m+2p-3\rangle}; \F_p)\] is of the form \[\frac{\mathrm{H}^{*} (\U; \F_p)}{(\mu_{2i+1} \vert \sigma_p (i) < 2m+p-2)}.\] It follows that the lowest positive degree element of the image is $\mu_{2j+1}$, where $j$ is the smallest integer such that $\sigma_p (j) = 2m+p-2$.

	This implies that
	\begin{align*}
		j &\geq \sum_{i=1} ^{\lfloor \frac{2m-1}{p-1} \rfloor + 1} (p-1) p^{i-1} 
		= p^{\lfloor \frac{2m-1}{p-1} + 1 \rfloor} - 1 
		\geq p^{\frac{2m}{p-1}} -1,
	\end{align*}
	so that \[2j+1 \geq 2p^{\frac{2m}{p-1}} - 1,\] from which the result follows.
%
%
\end{proof}

We are now able to prove the desired Adams filtration bounds.

\begin{lem}\label{prop:genHF2filt}
    Let $M_{k} \to \Sigma^{\infty} \Oo\mathrm{\langle m-1 \rangle}$ denote the inclusion of a $k$-skeleton for $k \geq 1$. Then the composite map $M_{k} \to \Sigma^{\infty} \Oo\mathrm{\langle m-1 \rangle} \xrightarrow{J} \Ss$ has $\HFt$-Adams filtration at least \[h(m-1) - \lfloor \log_2 (k+1) \rfloor + 1.\]
\end{lem}

\begin{proof}
    Factoring the map $\Sigma^{\infty} \Oo\mathrm{\langle m-1 \rangle} \to \Sigma^{\infty} \mathrm{O \langle 1 \rangle} = \Sigma^{\infty} \mathrm{SO}$ through the Whitehead tower and taking $k$-skeleta, we find that the resulting map has $\HFt$-Adams filtration at least \[\abs{\{s \in \NN \vert s \equiv 0,1,2,4 \text{ mod } 8 \text{ and } 2 \leq s \leq m-1 \text{ and } k < 2^{h(s)} -1\}}\] by \Cref{cor:zerocohmap}.
    Since $k \geq 1$, this is bounded below by
    \[h(m-1) - \abs{\{s \in \NN \vert s \equiv 0,1,2,4 \text{ mod } 8 \text{ and } \log_2 (k+1) \geq h(s)\}},\]
    which is equal to
    \[h(m-1) - \lfloor \log_2 (k+1) \rfloor.\]

    Since $J : \Sigma^{\infty} \mathrm{SO} \to \Ss$ is also zero on $\mathrm{H}^* (-; \F_2)$, we conclude that the $\HFt$-Adams filtration of
    \[M_k \to \Sigma^{\infty} \Oo\mathrm{\langle m-1 \rangle} \xrightarrow{J} \Ss\]
    is at least
    \[h(m-1) - \lfloor \log_2 (k+1) \rfloor + 1. \qedhere\]
\end{proof}

\begin{lem}\label{prop:genHFpfilt}
    Let $p$ be an odd prime. Then if $M_{k} \to \Sigma^{\infty} \Oo\mathrm{\langle 4m-1 \rangle}$ denotes the inclusion of a $k$-skeleton, the composite map $M_{k} \to \Sigma^{\infty} \Oo\mathrm{\langle 4m-1 \rangle} \xrightarrow{J} \Ss$ has $\HFp$-Adams filtration at least \[\left\lfloor \frac{4m}{2p-2} \right\rfloor - \left\lfloor \log_p \left( \frac{k+1}{2} \right) \right\rfloor.\]
\end{lem}

\begin{proof}
	Again, factoring the map $\Sigma^{\infty} \Oo\mathrm{\langle 4m -1 \rangle} \to \Sigma^\infty \mathrm{SO}$ through the Whitehead tower and taking $k$-skeleta, we find that \Cref{cor:zerocohmapp} implies that the resulting map has $\HFp$-Adams filtration at least
	\begin{align*}
		\abs{\left\{s \in \NN \vert s \equiv 0 \text{ mod } 2p-2 \text{ and } 2 \leq s \leq 4m-2p-2 \text{ and } k < 2p^{\frac{s}{2p-2}}-1\right\}}.
	\end{align*}
	This is at least as large as \[\left\lfloor \frac{4m-2p-2}{2p-2} \right\rfloor - \abs{\left\{ s \in \NN \vert s \equiv 0 \text{ mod } 2p-2 \text{ and } \log_p \left( \frac{k+1}{2} \right) \geq \frac{s}{2p-2}\right\}},\]
	which is equal to \[\left\lfloor \frac{4m}{2p-2} \right\rfloor - 1 - \left\lfloor \log_p \left( \frac{k+1}{2} \right) \right\rfloor.\]
    Since \[J : \Sigma^{\infty} \mathrm{SO} \to \Ss\] is also zero on $\mathrm{H}^* (-; \F_p)$, we conclude that the $\HFp$-Adams filtration of \[M_k \to \Sigma^{\infty} \Oo\mathrm{\langle 4m-1 \rangle} \xrightarrow{J} \Ss\] is at least \[\left\lfloor \frac{4m}{2p-2} \right\rfloor - \left\lfloor \log_p \left( \frac{k+1}{2} \right) \right\rfloor. \qedhere\]
\end{proof}

\begin{proof}[Proof of \Cref{lemm:filtration}]
	Set $m = 4n$ and $k = 8n-1$ in \Cref{prop:genHF2filt} and $m=n$ and $k=8n-1$ in \Cref{prop:genHFpfilt}.
\end{proof}

\subsection{Proof of \Cref{thm:vanishinggroup}}
\label{subsec:vanishinggroup}\ 

We will prove \Cref{thm:vanishinggroup} by using \Cref{lemm:ctauE2} and \Cref{lemm:ctaugoodenough} to reduce it to a statement about the vanishing of certain bidegrees in the $\mathrm{E}_2$-page of the Adams spectral sequence for $\Sigma^{\infty} \Oo\mathrm{\langle 4n-1 \rangle}$. Thus our task is to compute this $\mathrm{E}_2$-page in a range. We begin with the proof in the odd-primary case.

\begin{proof}[Proof of \Cref{thm:vanishinggroup} for odd $p$]
  We will show that
  \[ \pi_{8n-2, 8n-2 + k} (\nu M ) = 0 \]
    for all $k \geq 0$. Since $M$ is finite and implicitly $p$-completed, its $\HFp$-based Adams spectral sequence converges strongly \cite[Theorem 6.6]{BousfieldLocalization}, and we may apply \Cref{lemm:ctaugoodenough}. It therefore suffices to show that
  \[\pi_{8n-2, 8n-2+k} (C\tau \otimes \nu M) = 0\]
  for all $k \geq 0$.
  By \Cref{lemm:ctauE2}, 
  \begin{align*}
    \pi_{8n-2, 8n-2+k} (C\tau \otimes \nu M ) &\cong \mathrm{E}_2^{k,8n-2+k}(M) 
  \end{align*}

  We prove this last group is zero by comparison with the $\mathrm{E}_2$-page of the $\HFp$-Adams spectral sequence for $ko$. First, we note that it follows from the definition of $M$, \Cref{cor:goodwillie} and \Cref{lemm:d2-ko-comp} that $M \to \sOf \to \Sigma^{-1} \tau_{\geq 4n} ko$ is $(8n-1)$-connected at an odd prime. It therefore suffices to show that \[\mathrm{E}_2 ^{k, 8n-1+k} (\tau_{\geq 4n} ko) = 0.\] This follows from the structure of the $\HFp$-Adams spectral sequence for $\tau_{\geq 4n} ko$, which is equivalent to $\Sigma^{4n} ko$ since we are working at an odd prime. The structure of this spectral sequence may be deduced from \cite[Theorem 3.1.16]{GreenBook} and the fact that $ko$ is a summand of $ku$ at odd primes.
\end{proof}

We begin the proof for $p=2$ with the following lemma.

\begin{lem}\label{lemm:cohsurj2}
	The canonical map \[\Sigma^{\infty} \Oo\mathrm{\langle 4n-1 \rangle} \to \Sigma^{-1} \tau_{\geq 4n} ko\] is surjective on $\mathrm{H}^* (-; \F_2)$ for $* \leq 8n-1$ whenever $n \geq 3$.
\end{lem}

\begin{proof}
  Consider the following diagram:
  \begin{center}
    \begin{tikzcd}
      \Sigma^{\infty} \Oo \mathrm{\langle 4n-1 \rangle} \arrow[r] \arrow[d] & \Sigma^{-1} \tau_{\geq 4n} ko \arrow[d] \\
      \Sigma^{\infty} K(\ZZ, 4n-1) \arrow[r] & \Sigma^{4n-1} \HZ,
    \end{tikzcd}
  \end{center}
  where the vertical maps come from taking the first nonzero Postnikov sections of $\Oo\mathrm{\langle 4n-1 \rangle}$ and $\Sigma^{-1} \tau_{\geq 4n} ko$.

  The left vertical map is surjective on $\mathrm{H}^* (-; \F_2)$ for $* \leq 2^{h(4n)} - 2$ by \Cref{cor:coh-surj-range}. Therefore under our assumption that $n \geq 3$, it suffices to show that the bottom horizontal map is surjective on $\mathrm{H}^* (-; \F_2)$ for $* \leq 8n-1$. The algebra $\mathrm{H}^*(K(\ZZ, 4n-1); \F_2)$ is generated as an algebra by the image of $\mathrm{H}^* (\Sigma^{4n-1} \HZ)$ by \cite[Th\'{e}or\`{e}me 3]{SerreCohEM}. Letting $i_{4n -1} \in \mathrm{H}^{4n-1} (K(\ZZ,4n-1);\F_2)$ denote the fundamental class, it follows that the only classes in the relevant range that might not be in the image are $i_{4n-1}^2$ and $(\Sq^1 i_{4n-1})(i_{4n-1})$. But $i_{4n-1}^2 = \Sq^{4n-1} i_{4n-1}$ and $\Sq^1 i_{4n-1} = 0$, so the result follows.
\end{proof}

\begin{prop}\label{prop:E2compute}
  Assume that $n \geq 3$.
  In the range $t-s \leq 8n-3$, there is an isomorphism of $\mathrm{E}_2$-pages of $\HFt$-Adams spectral sequences
  \[ \mathrm{E}_2 ^{s,t} (\Sigma^\infty \Oo\mathrm{\langle 4n-1 \rangle}) \cong \mathrm{E}_2 ^{s,t} (\Sigma^{-1} \tau_{\geq 4n} ko). \]
  Moreover, for $t-s = 8n-2$, we have an isomorphism:
  \begin{align*}
    \mathrm{E}_2 ^{s,t} (\Sigma^\infty \Oo\mathrm{\langle 4n-1 \rangle}) &\cong \mathrm{E}_2 ^{s-1,t} (\Sigma D_2 (\Sigma^{-1} \tau_{\geq 4n} ko)) \\
    &\cong \begin{cases} 0 & \text{if }(s,t) \neq (1,8n-1) \\ \ZZ/2\ZZ & \text{if } (s,t) = (1,8n-1). \end{cases}
  \end{align*}
\end{prop}

\begin{proof}
  By \Cref{lemm:cohsurj2} and the tower of \Cref{cor:goodwillie}, there is a short exact sequence
  \[ 0 \to \mathrm{H}^* (\Sigma D_2 (\Sigma^{-1} \tau_{\geq 4n} ko)) \to \mathrm{H}^* (\Sigma^{-1} \tau_{\geq 4n} ko) \to \mathrm{H}^* (\Oo\mathrm{\langle 4n-1 \rangle)} \to 0 \]
  for $* \leq 8n-1.$

  By \Cref{lemm:d2-ko-comp}, the bottom homotopy group of $\Sigma D_2 (\Sigma^{-1} \tau_{\geq 4n} ko)$ is \[\pi_{8n-1} \Sigma D_2 (\Sigma^{-1} \tau_{\geq 4n} ko) \cong \ZZ/2\ZZ.\] It follows that \[\mathrm{E}_2 ^{s,t} (\Sigma D_2 (\Sigma^{-1} \tau_{\geq 4n} ko)) \cong \begin{cases} 0 & \text{if } t-s \leq 8n-1, (s,t) \neq (0,8n-1) \\ \ZZ/2\ZZ & \text{if } (t,s) = (0,8n-1) \end{cases}.\]

  The desired result now follows from the long exact sequence on $\mathrm{E}_2$-terms induced by the short exact sequence on cohomology, since nontrivial connecting maps are ruled out for bidegree reasons.
\end{proof}

\begin{rmk}\label{rmk:x2-filt}
  Since we know by \Cref{cor:gen-by-x2} that $\pi_{8n-2} \Sigma^{\infty} \Oo\mathrm{\langle 4n-1 \rangle} \cong \ZZ/2\ZZ$ is generated by the class $x^2$, the nonzero class in $\mathrm{E}_2 ^{1,8n-1} (\Sigma^{\infty} \Oo\mathrm{\langle 4n-1 \rangle}) \cong \ZZ/2\ZZ$ must represent $x^2$ on the $E_\infty$ page.
\end{rmk}

To illustrate the result of \Cref{prop:E2compute}, we include a picture of the Adams spectral sequence for $\Sigma^{\infty} \Oo \mathrm{\langle 15 \rangle}$ in the range determined by \Cref{prop:E2compute}. Note that the spectral sequence must collapse in this range for sparsity reasons.

\begin{figure}
  \centering
   \includegraphics[trim={3.9cm 19cm 3.9cm 4.2cm}, clip, scale=1]{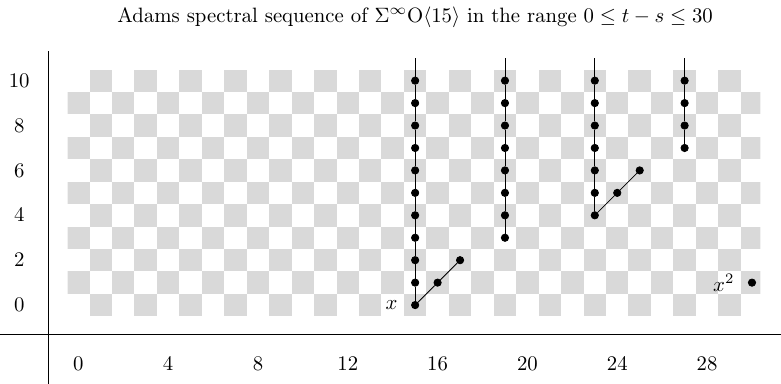}
\end{figure}

\begin{proof}[Proof of \Cref{thm:vanishinggroup} when $p=2$]
  We will show that
  \[\pi_{8n-2, 8n-2+k} (\nu M) = 0\]
    for all $k \geq 2$, which is sufficient since $N_2 \geq 3$ for $n \geq 3$. Since $M$ is finite and implicitly $2$-completed, its $\HFt$-based Adams spectral sequence converges strongly \cite[Theorem 6.6]{BousfieldLocalization}, and we may apply \Cref{lemm:ctaugoodenough}. It therefore suffices to show that
  \[\pi_{8n-2, 8n-2+k} (C\tau \otimes \nu M ) = 0\]
  for all $k \geq 2.$ By \Cref{lemm:ctauE2}, 
  \begin{align*}
    \pi_{8n-2, 8n-2+k} (\nu M \otimes C\tau) &\cong \mathrm{E}_2 ^{k,8n-2+k} (M) \\
    &\cong \mathrm{E}_2 ^{k, 8n-2+k} (\Sigma^{\infty} \Oo \mathrm{\langle 4n-1 \rangle}),
  \end{align*}
  which is zero for $k \geq 2$ by \Cref{prop:E2compute}.
\end{proof}


%% file: VanishingLines.tex
%
This section begins our study of vanishing lines in Adams spectral sequences, which is subject of Sections \ref{sec:AppendixVL}-\ref{sec:mod8} and \Cref{sec:Appendix}. In this section, our main concern will be the genericity properties of various notions of vanishing lines in synthetic spectra. A key feature of our methods is that they make clear how the intercepts of such vanishing lines change in cofiber sequences. Our results are used in \Cref{sec:ANvl} to obtain an explicit vanishing line in the Adams-Novikov spectral sequence for the $p$-local sphere, for each $p \ge 3$. In \Cref{sec:Appendix} the results of Sections \ref{sec:AppendixVL} and \ref{sec:ANvl} are used to deduce \Cref{thm:Burklund}(2).

Our genericity results recover versions of the genericity results of Hopkins, Palmieri and Smith \cite{HPS} for finite-page vanishing lines in $E$-Adams spectral sequences.  One side effect of our use of synthetic spectra is that we only prove results for $E$ of Adams type.





\begin{dfn}
  A \emph{thick subcategory} $\mathcal{C}$ of $\Sp$ (resp. $\Syn_E$) is a full subcategory which satisfies the following properties:
  \begin{itemize}
  \item it is closed under suspensions $\Sigma^n$ (resp. $\Sigma^{p,q}$) for $n,p,q \in \Z$,
  \item it is closed under retracts,
  \item if $X \to Y \to Z$ is a cofiber sequence and any two of $X,Y,Z$ are in $\mathcal{C}$, then so is the third.
  \end{itemize}

  Following \cite[Definition 1.1]{HPS}, we say that a property of (synthetic) spectra is \emph{generic} if the full subcategory of (synthetic) spectra satisfying that property is thick.
\end{dfn}

We now define four notions of vanishing line.

\begin{dfn}
  Given a synthetic spectrum $X$, we will say that
  \begin{enumerate}
  \item $X$ has a vanishing line of slope $m$ and intercept $c$
    if $\pi_{k,k+s}(X) = 0$ whenever $s > mk+c$.
  \item $X$ has a strong vanishing line of slope $m$ and intercept $c$
    if $X \otimes \nu_E (Y)$ has a vanishing line of slope $m$ and intercept $c$ for every connective spectrum $Y \in \Sp_{\geq 0}$.
  \item $X$ has a finite-page vanishing line of slope $m$, intercept $c$ and torsion level $r$ 
    if every class in $\pi_{k,k+s}(X)$ is $\tau^r$-torsion when $s > mk+c$.
  \item $X$ has a strong finite-page vanishing line of slope $m$, intercept $c$ and torsion level $r$
    if $X \otimes \nu_E (Y)$ has a finite-page vanishing line of slope $m$, intercept $c$ and torsion level $r$ for every $Y \in \Sp_{\geq 0}$.
  \end{enumerate}
\end{dfn}

\begin{rmk}
    The compatibility of $\nu_E$ with filtered colimits implies that the presence of a strong (finite-page) vanishing line need only be checked on \emph{finite} $Y \in \Sp_{\geq 0}$.

\end{rmk}

\begin{rmk}
  \label{rmk:strict-as-pc}
  A (strong) vanishing line of slope $m$ and intercept $c$ is equivalent to a (strong) finite-page vanishing line of slope $m$, intercept $c$ and torsion level $0$.
\end{rmk}

\begin{rmk}
    Given an $E$-nilpotent complete spectrum $X$, we will say that the $E$-based Adams spectral sequence for $X$ admits a (strong) (finite-page) vanishing line if $\nu_E (X)$ does. This is justified by the following proposition.
\end{rmk}

\begin{prop}\label{prop:vanishing-line-Adams}
    Given an $E$-nilpotent complete spectrum $Y$, $\nu_E (Y)$ admits a finite-page vanishing line of slope $m$, intercept $c$ and torsion level $r$ if and only if $\mathrm{E}^{s,k+s} _{r+2} = 0$ for $s > mk+c$.
\end{prop}

\begin{figure}
  \centering
    \includegraphics[scale=0.6]{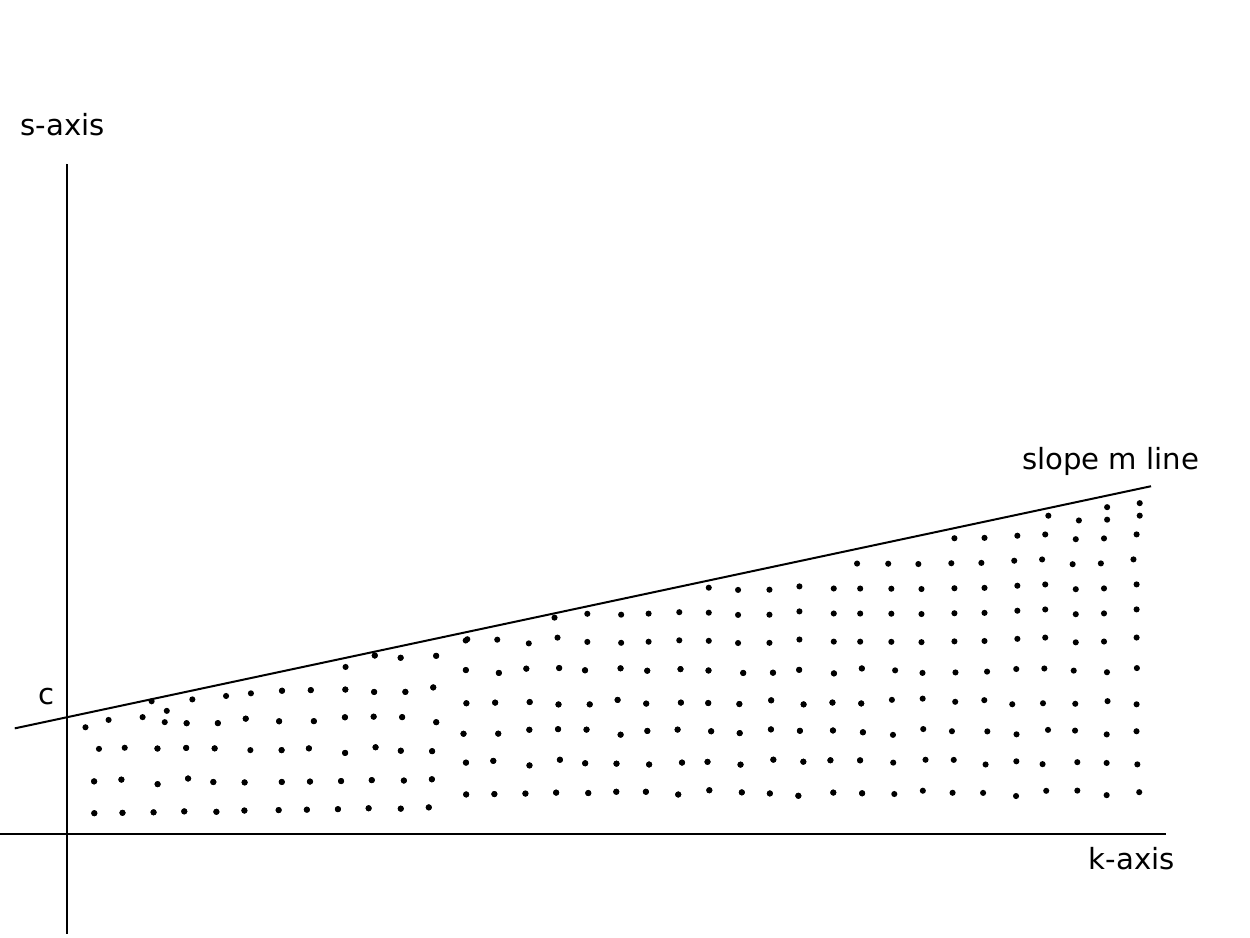}
    \caption{In this figure we display a picture of what the $\mathrm{E}_{r+2}$-page of the $E$-Adams spectral sequence of an $E$-nilpotent complete spectrum that admits a finite-page vanishing line of slope $m$, intercept $c$ and torsion level $r$ looks like. Namely, the region strictly above the vanishing line must consist of zero groups on the $\mathrm{E}_{r+2}$-page.}
    \label{fig:banded}
\end{figure}

We will need the following technical lemmas in the proof of \Cref{prop:vanishing-line-Adams}.

\begin{lem}\label{lemm:vanishing-line-converge}
    Given an $E$-nilpotent complete spectrum $Y$, the $E$-based Adams spectral sequence for $Y$ converges strongly if $\nu_E (Y)$ admits a finite-page vanishing line of positive slope.
\end{lem}

\begin{proof}
    By \Cref{cor:strong-conv}, it suffices to show that the $\tau$-Bockstein spectral sequence for $\nu_E (Y)$ converges strongly. By \Cref{thm:strong-conv}, to show strong convergence it will suffice to show that there are only finitely many differentials exiting each tridegree. But the finite-page vanishing line for $\nu_E Y$ implies that every $d_{s} ^{\tau-\mathrm{BSS}}$ with $s > r+1$ and target above the vanishing line must be zero. This implies that each group in the $\tau$-Bockstein spectral sequence may only be the source of only finitely many differentials, as required.
\end{proof}

\begin{lem}\label{lemm:vanishing-converge}
    Let $Y$ denote an $E$-nilpotent complete spectrum and suppose that there exist numbers $m > 0$, $c$ and $r$ for which the $E$-Adams spectral sequence of $Y$ satisfies $\mathrm{E}^{s,k+s} _{r} = 0$ for $s > mk+c$. Then the $E$-Adams spectral sequence for $Y$ converges strongly.
\end{lem}

\begin{proof}
    It follows from the assumption that each group in the spectral sequence can only have finitely many differentials originating from it, so the result follows from \Cref{thm:strong-conv}.
\end{proof}

\begin{proof}[Proof of \Cref{prop:vanishing-line-Adams}]
    Let $Y$ denote an $E$-nilpotent complete spectrum satisfying one of the conditions in the statement of the proposition. By either \Cref{lemm:vanishing-line-converge} or \Cref{lemm:vanishing-converge}, the $E$-Adams spectral sequence for $Y$ converges strongly. We are therefore free to invoke \Cref{thm:synthetic-Adams} in the following.

    Assume that $\nu_E (Y)$ admits a finite-page vanishing line of slope $m$, intercept $c$ and torsion level $r$, and suppose that there exists $0 \neq x \in \mathrm{E}^{s,k+s} _{r+2}$ with $s > mk +c$. If $x$ is the source of a differential, we may replace it by its target and therefore assume without loss of generality that $x$ is a permanent cycle. Let $y \in \mathrm{E}^{s,k+s} _2$ be a representative of $x$. Invoking \Cref{thm:synthetic-Adams}, we conclude that there exists $\wt{y} \in \pi_{s,k+s} (\nu_E Y)$ which is not $\tau^r$-torsion, a contradiction.

    Now suppose that $\mathrm{E}^{s,k+s} _{r+2} = 0$ when $s > mk+c$.  Applying \Cref{thm:synthetic-Adams}, we see that every element of the form $\wt{x}$ above the vanishing line is $\tau^r$-torsion. \Cref{thm:synthetic-Adams} also implies that the $\tau$-adic completion of the $\Z[\tau]$-submodule of the bigraded homotopy generated by such $\wt{x}$ is exactly $\pi_{*,*} (\nu_E Y)$. From the uniform bound on the $\tau$-torsion order, we learn that the completion was unnecessary. It follows that every class in $\pi_{k,k+s} (\nu_E Y)$ is $\tau^r$-torsion when $s > mk + c$, i.e. that $\nu_E Y$ admits a finite-page vanishing line of slope $m$, intercept $c$ and torsion level $r$.
\end{proof}

We now state the main result of this section.

\begin{thm}
  \label{thm:vl-are-generic}
  Given a slope $m>0$, the following four conditions on a synthetic spectrum $X$ are generic:
  \begin{enumerate}
  \item $X$ has a vanishing line of slope $m$.
  \item $X$ has a strong vanishing line of slope $m$.
  \item $X$ has a finite-page vanishing line of slope $m$.
  \item $X$ has a strong finite-page vanishing line of slope $m$.
  \end{enumerate}
\end{thm}

The proof of this theorem will be given over the course of two lemmas.

\begin{lem}
  \label{lemm:suspension-line}
  Suppose that a synthetic spectrum $X$ has a (strong) (finite-page) vanishing line of slope $m$, intercept $c$ and torsion level $r$. Then:
  \begin{enumerate}
  \item Any retract of $X$ has a (strong) (finite-page) vanishing line of slope $m$, intercept $c$ and torsion level $r$.
  \item $\Sigma^{k,k} X$ has a (strong) (finite-page) vanishing line of slope $m$, intercept $c-mk$ and torsion level $r$.
  \item $\Sigma^{0,s} X$ has a (strong) (finite-page) vanishing line of slope $m$, intercept $c+s$ and torsion level $r$.
  \end{enumerate}
\end{lem}

\begin{proof}
  The first claim follows from the fact that the bigraded homotopy groups of a retract of $X$ are a retract of the bigraded homotopy groups of $X$.
  The second and third claim follow from keeping track of the change in indexing of the bigraded homotopy groups under bigraded suspensions.
\end{proof}

\begin{lem}
  \label{lemm:cof-line}
  Given a cofiber sequence of synthetic spectra
  $$ A \xrightarrow{f} B \xrightarrow{g} C $$
  such that
  \begin{enumerate}
  \item $A$ has a (strong) (finite-page) vanishing line of slope $m$, intercept $c_1$ and torsion level $r_1$,
  \item $C$ has a (strong) (finite-page) vanishing line of slope $m$, intercept $c_2$ and torsion level $r_2$.
  \end{enumerate}
  Then, $B$ has a (strong) (finite-page) vanishing line of slope $m$, intercept $\max(c_1+r_2,c_2)$ and torsion level $r_1+r_2$.
\end{lem}

\begin{proof}
  Remark \ref{rmk:strict-as-pc} implies that it will suffice to prove the finite-page versions of this lemma. One can also easily see that the strong versions follow from the weak versions applied to all cofiber sequences of the form
  $$ A \otimes \nu_E (Y) \to B \otimes \nu_E (Y) \to C \otimes \nu_E (Y) $$
  where $Y \in \text{Sp}_{\geq 0}$.
  

  We finish the proof by proving the statement for finite-page vanishing lines. Suppose that $\alpha \in \pi_{k,k+s}(B)$ with $s > mk + \max(c_1+r_2,c_2)$. From the finite-page vanishing line for $C$ the class $g(\alpha)$ is $\tau^{r_2}$-torsion. Thus, there is a class $\alpha' \in \pi_{k,k+s-r_2}(A)$ such that $f(\alpha') = \tau^{r_2}\alpha$. By assumption, $s-r_2 > mk + c_1$, so the finite-page vanishing line for $A$ tells us that $\alpha'$ is $\tau^{r_1}$-torsion. In particular, $\tau^{r_1+r_2}\alpha = 0$, as desired.
\end{proof}

Combining \Cref{lemm:suspension-line} and \Cref{lemm:cof-line} gives the proof of \Cref{thm:vl-are-generic}.
We record the following corollary, which will find use in \Cref{sec:ANvl}.

\begin{cor} \label{cor:ctau-vl}
  The synthetic spectrum $C\tau^M$ has a strong finite-page vanishing line of slope $m$, intercept $c$ and torsion level $M$ for every slope $m$ and intercept $c$.
\end{cor}

\begin{proof}
  We proceed by induction on $M$.
  The base case follows from the fact that $C\tau$ is a ring \cite[Corollary 4.45]{Pstragowski} and therefore every $C\tau$-module has homotopy groups which are simple $\tau$-torsion. For $M > 1$, we apply \Cref{lemm:cof-line} to the cofiber sequences
  \[ \Sigma^{0,-1} C\tau^{M-1} \to C\tau^M \to C\tau. \qedhere \]
\end{proof}

Next we prove a version of \cite[Theorem 1.3]{HPS}.

\begin{thm}\label{thm:HPS}
  Given a slope $m > 0$, the following conditions on an $E$-local spectrum $X$ are generic:
  \begin{enumerate}
  \item $\nu_E (X)$ admits a finite-page vanishing line of slope $m$.
  \item $\nu_E (X)$ admits a strong finite-page vanishing line of slope $m$.
  \end{enumerate}
\end{thm}

\begin{rmk}
  Specializing to the case where $X$ is $E$-nilpotent complete, \Cref{prop:vanishing-line-Adams} and \Cref{thm:HPS}(1) together recovers a version of \cite[Theorem 1.3 (i)]{HPS}.

  Our assumptions on $E$ in \Cref{thm:HPS} differ from those given in \cite[Condition 1.2]{HPS}. We assume that $E$ is of Adams type, whereas \cite{HPS} assumes, among other things, that $E$ is connective.
\end{rmk}


To deduce \Cref{thm:HPS} from \Cref{thm:vl-are-generic}, we need to bound the extent to which $\nu_E$ fails to preserve cofiber sequences.

\begin{lem}\label{lemm:comparison-tau-tor}
  Let $X \to Y \to Z$ be a cofiber sequence of $E$-local spectra. Moreover, let $C$ denote the cofiber of $\nu_E (X) \to \nu_E (Y)$. Then the cofiber $D$ of the induced map $C \to \nu_E (Z)$ is a $C\tau$-module.
\end{lem}

\begin{proof}
	We may build a commutative diagram
	\begin{center}
	\begin{tikzcd}
		\nu_E (X) \arrow[r] \arrow[d, "\mathrm{id}_{\nu_E (X)}"] & \nu_E (Y) \arrow[r] \arrow[d, "\mathrm{id}_{\nu_E (X)}"] & C \arrow[r] \arrow[d] & \Sigma^{1,0} \nu_E (X) \arrow[r] \arrow[d, "\tau"] & \Sigma^{1,0} \nu_E (Y) \arrow[d, "\tau"] \\
		\nu_E (X) \arrow[r] & \nu_E (Y) \arrow[r] & \nu_E (Z) \arrow[r] & \nu_E (\Sigma X) \arrow[r] & \nu_E (\Sigma Y) 
	\end{tikzcd}
	\end{center}
    out of the comparison maps between colimits before applying $\nu_E $ and after. By \cite[Remark 4.61]{Pstragowski}, there is a natural isomorphism \[(\nu E)_{k,k+*} (\nu_E(W)) \cong E_{k} (W)[\tau]\] for any spectrum $W$. This is an isomorphism of bigraded groups if $E_k (W)$ is considered to have bidegree $(k,k)$ and $\tau$ is given bidegree $(0,1)$. Applying $\nu E_{*,*} (-)$, we obtain a diagram
	\begin{center}
    \begin{tikzcd}[column sep = small]
		E_{k} (X)[\tau] \arrow[r] \arrow[d,"\text{id}"] & E_{k} (Y)[\tau] \arrow[r] \arrow[d,"\text{id}"] & \nu E_{k,k+*} (C) \arrow[r] \arrow[d] & \Sigma^{0,-1} E_{k-1} (X) [\tau] \arrow[r] \arrow[d, "\cdot \tau"] & \Sigma^{0,-1} E_{k-1} (Y) [\tau] \arrow[d, "\cdot \tau"] \\
		E_{k} (X)[\tau] \arrow[r] & E_{k} (Y)[\tau] \arrow[r] & E_{k} (Z)[\tau] \arrow[r] & E_{k-1} (X) [\tau] \arrow[r] & E_{k-1} (Y) [\tau],
	\end{tikzcd}
	\end{center}
	where both the top and bottom rows are exact: the top is exact because it arose from applying $\nu E_{*,*}$ to a cofiber sequence, and the bottom is exact because it is obtained by adjoining $\tau$ to an exact sequence.
    Letting $f : E_{k} (X) \to E_k (Y)$ denote the map induced by $X \to Y$, we find that \[ 0 \to \nu E_{k,k+*} (C) \to E_k (Z) [\tau] \to \ker(f)_{k-1} \to 0 \] is exact. Recalling that we defined $D$ to be the cofiber of $C \to \nu_E (Z)$, we conclude that \[\nu E_{k,k+\ell} (D) = \begin{cases} \ker(f)_{k-1},& \text{if } \ell = 0\\ 0,& \text{otherwise.} \end{cases}\]
        This is sufficient to conclude that $D$ is a $C\tau$-module, from our assumptions that $X$, $Y$ and $Z$ are $E$-local. Indeed, this is a combination of citations to \cite{Pstragowski}. In the language of that paper $D$ is hypercomplete \cite[Propositions 5.4 and 5.6]{Pstragowski}. Therefore \cite[Theorem 4.18]{Pstragowski} implies that $D$ lies in the heart of the natural $t$-structure on $\Syn_E$, which is discussed in \cite[Section 4.2]{Pstragowski}.
        By \cite[Lemmas 4.42 and 4.43]{Pstragowski}, there is an adjunction $\epsilon_* : \mathrm{Syn}_E \leftrightarrows \mathrm{Stable}_{E_* E} : \epsilon^*$ with $\epsilon^*$ lax symmetric monoidal and which induces an equivalence on the hearts. It follows that $D \simeq \epsilon^* (\epsilon_* (D))$. Since $C\tau \simeq \epsilon^* (E_*)$ as $\mathbb{E}_{\infty}$-rings \cite[Corollary 4.45]{Pstragowski}, $D \simeq \epsilon^* (\epsilon_* (D))$ is a $C\tau$-module by lax symmetric monoidality of $\epsilon^*$.
%
\end{proof}

\begin{cor}\label{cor:HPS}
	Let $A \to B \to C$ be a cofiber sequence of $E$-local spectra and suppose that
	\begin{enumerate}
		\item $\nu_E (A)$ has a (strong) finite-page vanishing line of slope $m$, intercept $c_1$ and torsion level $r_1$,
		\item $\nu_E (C)$ has a (strong) finite-page vanishing line of slope $m$, intercept $c_2$ and torsion level $r_2$.
	\end{enumerate}
	Then $\nu_E (B)$ has a (strong) finite-page vanishing line of slope $m$, intercept $\max(c_1+r_2, c_2)+1$ and torsion level $r_1 + r_2 + 1$.
\end{cor}

\begin{proof}
	By \Cref{lemm:cof-line}, the cofiber $X$ of $\nu_E (\Sigma C) \to \nu_E (A)$ has a (strong) finite-page vanishing line of slope $m$, intercept $\max(c_1+r_2, c_2)$ and torsion level $r_1 + r_2$. By \Cref{lemm:comparison-tau-tor}, the cofiber $Y$ of $X \to \nu_E (B)$ is a $C\tau$-module. It follows that $Y$ has a strong finite-page vanishing line of slope $m$, arbitrary negative intercept and torsion level $1$.

	Applying \Cref{lemm:cof-line} to $X \to \nu_E (B) \to Y$, we obtain the desired result.
\end{proof}

\begin{proof}[Proof of \Cref{thm:HPS}]
	This follows from \Cref{cor:HPS}, \Cref{lemm:suspension-line} and the fact that $\nu_E$ sends retracts to retracts and suspensions to bigraded suspensions.
\end{proof}

Finally, we record a lemma which is useful in establishing (strong) vanishing lines. We say that a synthetic spectrum is $\tau$-complete if the natural map $X \to \varprojlim X \otimes C\tau^{n}$ is an equivalence.

\begin{lem} \label{lemm:vl-ctau}
	A $\tau$-complete synthetic spectrum $X$ has a vanishing line (resp. strong vanishing line) of slope $m \geq 0$ and intercept $c$ if and only if $X \otimes C\tau$ does.
\end{lem}

\begin{proof}
  The ``only if'' direction is easy and follows from considering the exact sequence
  $$ \pi_{a,b}(X) \to \pi_{a,b}(X \otimes C\tau) \to \pi_{1,-1}(X). $$

  For the ``if'' direction we first note by induction that $X \otimes C\tau^n$ admits a (strong) vanishing line of slope $m$ and intercept $c$. For this it suffices to apply Lemmas \ref{lemm:suspension-line} and \ref{lemm:cof-line} to the cofiber sequences
  $$ \Sigma^{0,-n} C\tau \to C\tau^{n+1} \to C\tau^n. $$

  In the non-strong case, the result now follows from the $\tau$-completeness of $X$. The potential $\lim^1$ vanishes because of the assumed vanishing line.
  


    In the strong case we must also prove that $ X \otimes \nu_E(Y) $ is $\tau$-complete for finite $Y$. We will show that the collection of $Y$ for which $X \otimes \nu_E (Y)$ is $\tau$-complete is thick. Since it contains $\Ss^0$, it will then contain all finite spectra. It is clearly closed under suspensions and retracts. Suppose that $Z_1 \to Z_2 \to Z_3$ is a cofiber sequence with the property that $X \otimes \nu_E (Z_1)$ and $X \otimes \nu_E (Z_2)$ are $\tau$-complete. We will show that $X \otimes \nu_E (Z_3)$ is $\tau$-complete.

    Write $C$ for the cofiber of $\nu_E (Z_1) \to \nu_E (Z_2)$. Then $X \otimes C$ is $\tau$-complete, and there is a cofiber sequence $X \otimes C \to X \otimes \nu_E (Z_3) \to X \otimes D$ where $D$ is a $C\tau$-module by \Cref{lemm:comparison-tau-tor}. It follows that $X \otimes D$ is $\tau$-complete and hence that $X \otimes \nu_E (Z_3)$ is $\tau$-complete, as desired.
\end{proof}

%% file: ANSSLine2.tex
Using ideas from \Cref{sec:AppendixVL}, we prove a strong finite-page vanishing line on $\nu_{\BP}(\Ss^0)$. This line is not visible on the $\mathrm{E}_2$-page of the spectral sequence. The vanishing line will be used in \Cref{sec:Appendix} to provide the explicit numerical control over the function $\Gamma (k)$ required in \Cref{sec:Finale}.

\begin{cnv}
    In this section, we will fix a prime $p$ and implicitly $p$-localize all spectra. Furthermore, all synthetic spectra will be taken with respect to $\BP$.
\end{cnv}

\begin{thm}
  \label{thm:AN-vl}
  For $p\geq 3$, the $\BP$-synthetic sphere $\nu_{\BP} (\Ss^{0})$ has a strong finite-page vanishing line of slope $m$, intercept $c$ and torsion level $r$ where
  $$ m = \frac{1}{p^3 - p -1} \text{, }\ \ c = 2p^2 - 4p + 9 - \frac{2p^2+2p-10}{p^3 - p -1}\ \ \text{   and   }\ \  r = 2p^2-4p+2. $$
\end{thm}

\begin{rmk}
  The key content of \Cref{thm:AN-vl} is not the slope of the vanishing line,
  but rather the explicit values for the intercepts and torsion levels.\footnote{Indeed, an argument of Hopkins and Smith shows that the Devinatz--Hopkins--Smith nilpotence theorem is equivalent to the following statement:
  given any positive slope $\epsilon > 0$,
  the Adams-Novikov spectral sequence for $\Ss^0$ admits a vanishing line of slope $\epsilon$
  at some finite page (see \cite[Theorem 3.30]{Akhil} for a published account of this argument).
  By \Cref{prop:vanishing-line-Adams} this is equivalent to saying that $\nu_{\BP} \mathbb{S}^0$ has a finite page vanishing line of slope $\epsilon$ for every positive $\epsilon$.}
\end{rmk}


\begin{ntn}
    We let $\wt{\beta}_1 \in \pi_{2p^2-2p-2,2p^2-2p} (\nu_{\BP} (\Ss^0))$ denote a synthetic lift of $\beta_1 \in \pi_{2p^2-2p-2} (\Ss^0)$ as in \Cref{thm:synthetic-Adams}(3).
\end{ntn}
The proof of \Cref{thm:AN-vl} consists of two main steps:
\begin{enumerate}
    \item We show that $C(\wt{\beta}_1)$ admits a strong vanishing line of slope $\frac{1}{p^3-p-1}$ and explicit intercept.
    \item Using the fact that $\beta_1$ is nilpotent topologically, we apply step (1) and the results of \Cref{sec:AppendixVL} to show that $\nu_{\BP} (\Ss^0)$ admits the desired strong finite-page vanishing line.
\end{enumerate}

Our proof of (1) will be based on the homological algebra of $P_*$-comodules, where $P_*$ is the polynomial part of the dual Steenrod algebra.

\begin{rec}
  Let $r$ denote the map of Hopf algebroids
  \[ (\BP_*, \BP_*\BP) \xrightarrow{r} (\F_p, P_*). \]
  Thinking in terms of the associated stacks we have a pullback/pushforward adjunction between the asscociated categories of sheaves
  \[ r^* : \mathrm{Stable}_{\BP_* \BP} \rightleftharpoons \mathrm{Stable}_{P_*} : r_*. \]
  Concretely, at the level of comodules $r^*$ sends a $\BP_*\BP$-comodule $M$ to the tensor product $\F_p \otimes_{\BP_*} M$ viewed as $P_*$-comodule.
\end{rec}

From \cite[Proposition 4.53]{Pstragowski} we know that the symmetric monoidal embedding \[\mathrm{Mod}_{C\tau} \to \mathrm{Stable}_{\BP_* \BP}\]
of \Cref{thm:tau-inv} is an equivalence.
As such, we will abuse notation by applying $r^*$ directly to $C\tau$-modules.

The reduction to $P_*$-comodules is carried out by the following lemma whose main input is an algebraic lemma of Krause \cite[Proposition 4.22]{AchimThesis}. For the statement of this lemma we use the notion of vanishing lines in Ext groups (which are a sort of bigraded homotopy groups) from Section 4 of Krause's work.

\begin{lem} \label{lemm:reduce-to-P}
  If $X$ is a compact and $\tau$-complete object of $\Syn_{\BP}$ such that
  $r^*(C\tau \otimes X)$ admits a vanishing line of slope $m$ and intercept $c$,
  then $X$ admits a strong vanishing line of the same slope and intercept.
\end{lem}

\begin{proof}
    By \Cref{lemm:vl-ctau} it suffices to show that $C\tau \otimes X \otimes \nu_{\BP} (A)$ admits a vanishing line of slope $m$ and intercept $c$ for all $A \in \Sp_{\geq 0}$. The vanishing statements we wish to prove are compatible with filtered colimits (as is $\nu_{\BP}$), therefore it suffices to restrict to the case where $A$ is finite.

  The symmetric monoidal equivalence of $\mathrm{Mod}_{C\tau}$ and $\mathrm{Stable}_{\BP_*\BP}$ provides us with a derived $\BP_*\BP$-comodule $\overline{X}$ associated to $C\tau \otimes X$ and equivalences 
    $$ \pi_{t-s,t}(C\tau \otimes X \otimes \nu_{\BP} (A)) \cong \Ext_{\BP_*\BP} ^{s,t} (\BP_*, \overline{X} \otimes_{\BP_*} \BP_*A). $$
    Let us now show that it suffices to address the case when $A = \Ss^0.$ The category of connective comodules over the Hopf algebroid $(\BP_*, \BP_*\BP)$ has enough projectives by the main result of \cite{SalchEnoughProj}, so we may fix a resolution $A_\bullet$ of $\BP_* A$ whose associated graded consists of positive shifts of $\BP_*$.
    We'll prove that the desired vanishing line already exists on the first page of the spectral sequence associated to the filtered object $\overline{X} \otimes_{\BP_*} A_\bullet$. By our choice of filtration this reduces to showing that $\Ext_{\BP_*\BP}^{s,t}(\BP_*, \overline{X})$ has the desired vanishing line, as we wanted.
  
    By our assumption that $X$ is compact, $\overline{X}$ is compact as an object of $\mathrm{Stable}_{\BP_* \BP}$. It therefore follows from the proof of \cite[Proposition 4.22]{AchimThesis} that $\overline{X}$ and $r^*\overline{X}$ admit the same vanishing lines, as desired.
\end{proof}

As a corollary to the above, we obtain the following well-known vanishing line. We will make use of it in our proof of \Cref{thm:AN-vl}.

\begin{prop}\label{prop:easy-ANSS-vl}
    Let $p$ be an odd prime. Then
  \label{prop:AN-vl}
  $\nu_{\BP}(\Ss^0)$ is $\tau$-complete and has a vanishing line of slope $m$ and intercept $c$ where
  $$ m = \frac{1}{p^2 - p -1}\ \ \text{   and   }\ \ c = 1 - \frac{2p-3}{p^2 - p -1}. $$
\end{prop}

\begin{proof}
    Since $\Ss^0$ is $\BP$-nilpotent complete, \Cref{lemm:Elocal-taucomplete} implies that $\nu_{\BP}(\Ss^0)$ is $\tau$-complete. Hence, by \Cref{lemm:reduce-to-P}, it suffices to show that the desired vanishing line is present in
  $$ \Ext_{P_*}^{s,t}(\F_p, \F_p). $$
    This vanishing line is already visible in the $\mathrm{E}_1$-page of the May spectral sequence for Ext over $P_*$.
\end{proof}

We now establish the desired vanishing line for $C(\wt{\beta}_1)$.

\begin{lem} \todo{It would probably be good from an expositional point of view to pull out the statement of \cite[Theorem 2.3.1]{PalmieriBook} from the proof of Lemma. I decided not to because of Palmieri's unfortunate notion of $CL$-spectrum. The other possibility would be trying to write our own theorem with explicit intercept in Achim's framework. I don't have time for this at the moment, though.}
  \label{lemm:cbeta-vl}
    The synthetic spectrum $C(\wt{\beta}_1)$ has a strong vanishing line of slope $m$ and intercept $c$, where
  $$ m = \frac{1}{p^3 - p - 1} \text{  and  } c = 8 - \frac{4p^2 - 11}{p^3 - p -1}. $$
\end{lem}

\begin{proof}
    We begin by noting that $C(\wt{\beta}_1)$ is $\tau$-complete by \Cref{prop:easy-ANSS-vl} and the fact that $\tau$-completeness is closed under finite colimits.
  Thus, by \Cref{lemm:reduce-to-P}, it suffices to show that 
    $ \Ext_{P_*}^{s,t}(\F_p, C(\beta_1)) $
    has the desired vanishing line, where $C(\beta_1)$ is the cofiber of the element $\beta_1 \in \Ext_{P_*} ^{2,2p^2-2p} (\F_p, \F_p)$ in $\mathrm{Stable}_{P_*}$.

    In \cite[Section 3]{Eva} Belmont shows that $C(\beta_1)$ satisfies the conditions of \cite[Theorem 2.3.1]{PalmieriBook} with Palmieri's paramter $d$ equal to $p^3 - p$. While this theorem is stated for $\A_*$ in \cite{PalmieriBook}, the proof carries over for $P_*$. Thus, we learn that
    $$ \Ext_{P_*}^{s,t}(\F_p, C(\beta_1)) = 0 \ \ \text{ for all } s > \frac{1}{d-1}(t-s + \alpha(d)) + 1, $$
  where $n$ is the minimal integer such that $2(p-1)p^n > d$ (in our case $n=2$) and
    $$ \alpha(d) := \left( \sum_{s+t \leq n \\ |\xi_t^{p^s}| \leq d} d + (p-1)|\xi_t^{p^s}| \right). $$
    The above calculation of the intercept is (the $P_*$ version of) \cite[Remark 2.3.4]{PalmieriBook}. We note that the $+1$ term at the end of the above inequality for $s$ comes from the $i_1$ term in \cite[Remark 2.3.4]{PalmieriBook}.
  We calculate that
  $\alpha(p^3 - p) = 7p^3 - 4p^2 - 7p + 4$ and thereby see that $C\wt{\beta}_1$ has the desired vanishing line.
\end{proof}


We now move on to step (2) of our proof of \Cref{thm:AN-vl}. Since $\beta_1$ is not nilpotent on the $\mathrm{E}_2$-page of the Adams-Novikov spectral sequence, the synthetic class $\wt{\beta}_1$ is not nilpotent. It follows that we cannot complete step (2) through a direct application of the genericity results of \Cref{sec:AppendixVL}. Instead, \Cref{thm:synthetic-Adams} will show that $\wt{\beta}_1 ^N \tau^M = 0$ for some large $N$ and $M$. In this situation, we have the following lemma:

\begin{lem}
  \label{lemm:syn-lines}
  Suppose that $X$ is a synthetic spectrum with a self map
  $b: \Sigma^{u,u+v}X \to X$ such that,
  \begin{itemize}
  \item $b^N\tau^M = 0$,
  \item $\Sigma^{-|b|} C(b)$ has a (strong) vanishing line of slope $m$ and intercept $c$, and
  \item $ \frac{v}{u} \geq m$.
  \end{itemize}
  Then $X$ has a (strong) finite-page vanishing line of slope $m$, intercept $c'$ and torsion level $M$, where
  $$ c' = c + \min(N(v-mu), M+m+1). $$
\end{lem}

\begin{proof}

  Consider the family of cofiber sequences
  $$ \Sigma^{-|b|} C(b) \to \Sigma^{-n|b|} C(b^n) \to \Sigma^{-n|b|} C(b^{n-1}) $$
  as $n$ varies. We will prove by induction that $\Sigma^{-n|b|} C(b^n)$ has a (strong) vanishing line of slope $m$ and intercept $c$. The base case is one of our hypotheses. In order to handle the induction step, we apply Lemma \ref{lemm:cof-line} to the cofiber sequence above. By assumption (and \Cref{lemm:suspension-line}), $\Sigma^{-n|b|} C(b^{n-1})$ has a (strong) vanishing line of slope $m$ and intercept $c + mu - v$. Thus, $\Sigma^{-n|b|} C(b^n)$ has a (strong) vanishing line of slope $m$ and intercept $\max(c, c + mu-v) = c$.
  
  Next, we apply Lemmas \ref{lemm:cof-line} and \ref{cor:ctau-vl} to the cofiber sequence 
  $$ C\tau^M \xrightarrow{f} \Sigma^{-N|b|} C(b^N\tau^M) \xrightarrow{g} \Sigma^{-N|b|} C(b^N) $$
  in order to conclude that $\Sigma^{-N|b|} C(b^N\tau^M)$ has a (strong) finite-page vanishing line of slope $m$, intercept $c$ and torsion level $M$. Finally, using the splitting
  $$ C(b^N\tau^M) \simeq X \oplus \Sigma^{(1,-M) + N|b|}X, $$
  we obtain the desired (strong) finite-page vanishing line. \qedhere
  

\end{proof}

To apply this lemma to prove \Cref{thm:AN-vl}, we need to determine the constants that we have called $N$ and $M$ for $X = \nu_{\BP} (\Ss^0)$ and $b = \wt{\beta}_1$. By \Cref{thm:synthetic-Adams}, this comes down to the following lemma. 

\begin{lem}[Ravenel]\label{lemm:beta1-values}
    There are differentials
    \begin{align*}
        d_9 (\alpha_1 \beta_4) &= \beta_1 ^6 \text{ at } p=3 \text{ and} \\
        d_{33} (\gamma_3) &= \beta_1 ^{18} \text{ at } p=5
    \end{align*}
    in the Adams-Novikov spectral sequence. Moreover, we have 
    $\beta_1 ^{p^2-p+1} = 0$
    at any odd prime $p$. The Adams-Novikov differential with target $\beta_1 ^{p^2-p+1}$ has length at most
    $$2p^2 - 4p + 3. $$
\end{lem}

\begin{proof}
    The $3$-primary differential is part of \cite[Theorem 7.5.3]{GreenBook} and the $5$-primary differential is \cite[Theorem 7.6.1]{GreenBook}. 
    The general bound on the order of nilpotence of $\beta_1$ is proven shortly after the statement of Theorem 7.6.1 in \cite{GreenBook}, where Ravenel recounts a classical argument of Toda for this relation. Finally, the bound on the length of the differential follows from sparsity and the fact that there are no differentials off the $1$-line of the Adams-Novikov spectral sequence at odd primes.
\end{proof}


\begin{proof}[Proof of Theorem \ref{thm:AN-vl}]
    In order to prove the theorem we apply \Cref{lemm:syn-lines} to $X = \nu_{\BP} (\Ss^{0})$ and $b = \wt{\beta}_1$. The remainder of the proof is just a matter of computing $m,c,u,v,N$ and $M$.

  The element $\wt{\beta}_1$ has bidegree $(2p^2-2p-2,2p^2-2p)$ so $u = 2p^2 -2p - 2$ and $v = 2$. By Lemmas \ref{lemm:cbeta-vl} and \ref{lemm:suspension-line} we know $\Sigma^{-|\wt{\beta}_1|}C\wt{\beta}_1$ has a strong vanishing line of slope $m = (p^3 - p - 1)^{-1}$ and intercept
  \begin{align*}
    c &= \left(8 - \frac{4p^2 - 11}{p^3 - p - 1} \right) - \left( 2 - \frac{2p^2 - 2p - 2}{p^3 - p - 1} \right) =  6 - \frac{2p^2+2p-9}{p^3 - p - 1}
  \end{align*}

    Suppose that there exists an $a$ in the $\mathrm{E}_{r+1}$-term of the Adams-Novikov spectral sequence such that $d_{r+1}(a) = \beta_1^N$. Then, by \Cref{thm:synthetic-Adams} there exists a $\wt{\beta_1^N}$ such that $\wt{\beta_1^N}\tau^{r} = 0$. A priori it may not be true that $\wt{\beta}_1^N = \wt{\beta_1^N}$, though we do know their difference maps to zero in $C\tau$ and is therefore divisible by $\tau$. In this case we can then use \Cref{prop:easy-ANSS-vl} to conclude that this ``difference divided by $\tau$'' is zero---seeing as it lives in a bigrading which is zero. To summarize, we learn that if $\beta_1^N$ is hit by a $d_{r+1}$-differential in the Adams-Novikov spectral sequence, then $\wt{\beta}_1^N \tau^r = 0$.
  

    We may therefore cite \Cref{lemm:beta1-values} to obtain the values of $N$ and $M$. We summarize the values we have computed in the following table:
  \renewcommand{\arraystretch}{1.6}
  \begin{center}
    \begin{tabular}{|c||c|c|c|c|c|}\hline
      prime  & $m$ & $u$ & $v$ & $N$ & $M$ \\\hline\hline
      $3$ & $\frac{1}{23}$ & $10$ & $2$ & $6$ & $8$ \\\hline
      $5$   & $\frac{1}{119}$ & $38$ & $2$ & $18$ & $32$ \\\hline
      $\geq 7$ & $\frac{1}{p^3 - p^2 - 1}$ & $2p^2-2p-2$ & $2$ & $p^2-p+1$ & $2p^2-4p+2$ \\\hline
    \end{tabular}
  \end{center}
  \renewcommand{\arraystretch}{1.0}
  At the prime $3$ the intercept is
  $$ 6 - \frac{15}{23} + \min \left(6\left( 2 - \frac{10}{23} \right), 9 + \frac{1}{23} \right) = 14 + \frac{9}{23}. $$
  At the prime $5$ the intercept is
  $$ 6 - \frac{51}{119} + \min \left(18 \left( 2 - \frac{38}{119} \right), 33 + \frac{1}{119} \right) < 38 + \frac{69}{119}. $$
  At primes $\geq 7$ the intercept is
  \begin{align*}
    6 - &\frac{2p^2+2p-9}{p^3 - p - 1} + \min \left((p^2 - p -1) \left( 2 - \frac{2p^2 - 2p - 2}{p^3 - p -1} \right), 2p^2 - 4p + 3 + \frac{1}{p^3 - p - 1} \right) \\
    &= 2p^2 - 4p + 9 - \frac{2p^2+2p-10}{p^3 - p^2 -1}.
  \end{align*}
  Note that the bound we write down for all primes is in fact equal to
  \[2p^2 - 4p + 9 + \frac{2p^2+2p-10}{p^3 - p^2 -1}. \qedhere\]
\end{proof}



%% file: BandedGenericity2.tex
{\bf An overview of Sections \ref{sec:BandVan}-\ref{sec:mod8}}
This and the following two sections are devoted to the proof of \Cref{thm:mod8}, which will be proven as \Cref{thm:mod8-main-thm}. To prove \Cref{thm:mod8-main-thm}, we will show that there exists a line of slope $\frac{1}{5}$ on some finite page of the modified $\HFt$-Adams spectral sequence of the mod $8$ Moore spectrum $C(8)$ above which the only classes are those detecting the $K(1)$-local homotopy of $C(8)$.\footnote{For the notion of a modified Adams spectral sequence, see \cite[Section 3]{BHHM}.}

In this section, we will axiomatize this property into the defintion of a $v_1$-banded vanishing line on a synthetic spectrum. We will then show that the property of having a $v_1$-banded vanishing line is generic, i.e. is closed under retractions, bigraded suspensions and cofiber sequences of synthetic spectra. In \Cref{sec:Y}, we will show that $\nu_{\HFt} (Y) = \nu_{\HFt} (C(2) \otimes C(\eta))$ admits a $v_1$-banded vanishing line.
In \Cref{sec:mod8}, we will establish a $v_1$-banded vanishing line on $C(\wt{8})$ and use this line to prove \Cref{thm:mod8-main-thm}.\footnote{The synthetic spectrum $C(\wt{8})$ is defined in \Cref{sec:mod8}. It encodes the modified $\HFt$-Adams spectral sequence of $C(8)$.}  The proof of the $v_1$-banded vanishing line of $C(\wt{8})$ is a genericity argument, building from the case of $\nu_{\HFt} (Y)$. As in Sections \ref{sec:AppendixVL} and \ref{sec:ANvl}, we will sedulously keep track of intercepts and torsion levels throughout.
%
%

\begin{dfn}
    Given a $\Z[\tau]$-module $M$, we let $M_{\mathrm{tor}} \subset M$ denote the subgroup of $\tau$-power torsion elements and $M_{\mathrm{tf}}$ the torsion free quotient $M / M_{\mathrm{tor}}$. When there are other subscripts present, we will sometimes find it convenient to write $M^{\mathrm{tor}}$ and $M^{\mathrm{tf}}$ in place of $M_{\mathrm{tor}}$ and $M_{\mathrm{tf}}$, respectively.
\end{dfn}

\begin{dfn}\label{dfn:synth-filt}
    Given a synthetic spectrum $X$, we let $F^s \pi_{k} (\tau^{-1} X) \subset \pi_{k} (\tau^{-1} X)$ denote the image of $\pi_{k,k+s} X \to \pi_{k} (\tau^{-1} X)$. This defines a descending filtration on $\pi_k (\tau^{-1} X)$, which is natural in $X$.
\end{dfn}

\begin{rmk}
  The natural map $\pi_{k,k+s} (X)_{\mathrm{tf}} \to F^s \pi_k (\tau^{-1} X)$ is an isomorphism.
\end{rmk}

\begin{rmk}
    Let $Y$ be a $E$-nilpotent complete spectrum whose $E$-Adams spectral sequence converges strongly. By \Cref{cor:synth-filt}, the filtration $F^s \pi_k (\tau^{-1} \nu_E (Y))$ coincides with the $E$-Adams filtration on $\pi_k (Y) \cong \pi_k (\tau^{-1} \nu_E (Y))$.
\end{rmk}

\begin{cnv}
    In the remainder of this section, we will fix a prime $p$ and work exclusively with the category $\mathrm{Syn}_{\HFp}$ of synthetic spectra with respect to $\HFp$.
\end{cnv}

\begin{dfn}\label{dfn:banded-vanishing-line}
  We say that a synthetic spectrum $X$ has a $v_1$-banded vanishing line with
  \begin{itemize}
  \item band intercepts $b \leq d$
  \item range of validity $v$
  \item line of slope $m < \frac{1}{2p-2}$ and intercept $c$
  \item torsion bound $r$
  \end{itemize}
  if the following conditions hold:
  \begin{enumerate}
  \item every class in $\pi_{k,k+s} (X)_{\mathrm{tor}}$ is $\tau^{r}$-torsion for $s \geq mk+c$ and $k \geq v$,
  \item the natural map
    $$ F^{\frac{1}{2p-2} k + b} \pi_k (\tau^{-1} X) \to F^{mk + c} \pi_k (\tau^{-1} X) $$
    is an isomorphism for $k \geq v$,
  \item the composite
    $$ F^{\frac{1}{2p-2} k + b} \pi_{k} (\tau^{-1} X) \to \pi_k(\tau^{-1}X) \to \pi_k (L_{K(1)} \tau^{-1} X) $$
    is an equivalence for $k \geq v$,
  \item $\pi_{k,k+s} (X) = 0$ for $s > \frac{1}{2p-2} k + d$.
  \end{enumerate}
  
  More concisely, we will say that that $X$ has a $v_1$-banded vanishing line with parameters $(b \leq d, v, m, c, r)$.
\end{dfn}


\begin{rmk}
  Given an $\HFp$-nilpotent complete spectrum $X$, we will say that the $\HFp$-Adams spectral sequence of $X$ admits a $v_1$-banded vanishing line with parameters $(b \leq d, v, m, c, r)$ if $\nu_{\HFp} (X)$ admits one. This is justified by the following proposition:
\end{rmk}

\begin{prop} \label{prop:banded-adams}
    Given an $\HFp$-nilpotent complete spectrum $X$, $\nu_{\HFp} (X)$ admits a $v_1$-banded vanishing line with parameters $(b \leq d, v, m, c, r)$ if and only if the $\HFp$-based Adams spectral sequence for $X$ satisfies the following conditions:
    \begin{enumerate}
        \item[($1'$)] $\mathrm{E}_{r+2} ^{s,k+s} = \mathrm{E}_{\infty} ^{s,k+s}$ for $s \geq mk+c$ and $k \geq v$.
        \item[($2'$)] $\mathrm{E}_{r+2} ^{s,k+s} = 0$ for $mk+c \leq s < \frac{1}{2p-2}k + b$ and $k \geq v$.
        \item[($3'$)]$F^{\frac{1}{2p-2} k + b} \pi_k (X) \to \pi_k (L_{K(1)} X)$ is an isomorphism for $k \geq v$, where $F$ is the $\HFp$-Adams filtration.
        \item[($4'$)] $\mathrm{E}_2 ^{s,k+s} = 0$ for all $s > \frac{1}{2p-2}k+d$.
    \end{enumerate}
\end{prop}

\begin{proof}
    It follows from \Cref{lemm:vanishing-line-converge} and \Cref{lemm:vanishing-converge} that the $\HFp$-based Adams spectral sequence for $X$ converges strongly; therefore we may use \Cref{thm:synthetic-Adams} and its corollaries. By \Cref{prop:vanishing-line-Adams} we know that (4) and ($4'$) are equivalent. Using (4) to ground the induction started by \Cref{cor:tau-torsion-bound} we learn (1) and ($1'$) are equivalent. Using \Cref{cor:synth-filt} we may identify the filtration appearing in the definition of a banded vanishing line with the Adams filtration. This allows us to conclude that (2) and (3) are equivalent to ($2'$) and ($3'$), respectively.
\end{proof}

In \Cref{fig:banded} we use \Cref{prop:banded-adams} to illustrate the meaning of a banded vanishing line on $\nu X$. 
As we shall see, \Cref{dfn:banded-vanishing-line} captures the behavior of the modified Adams spectral sequence of a type $1$ spectrum. Moreover, it is formulated in such a way that it is a generic condition, i.e. the full subcategory of synthetic spectra satisfying \Cref{dfn:banded-vanishing-line} for a fixed $m$ and varying $(b \leq d, v, c, r)$ is closed under retracts, bigraded suspensions and cofiber sequences. We prove this genericity in \Cref{lemm:band-retract-susp} and \Cref{prop:band-in-cofiber-seqs}. A key feature of our approach is that we keep explicit track of how the constants $(b \leq d, v, c, r)$ change under retracts, bigraded suspensions and cofiber sequences.

\begin{figure}
  \centering
    \includegraphics[scale=0.6]{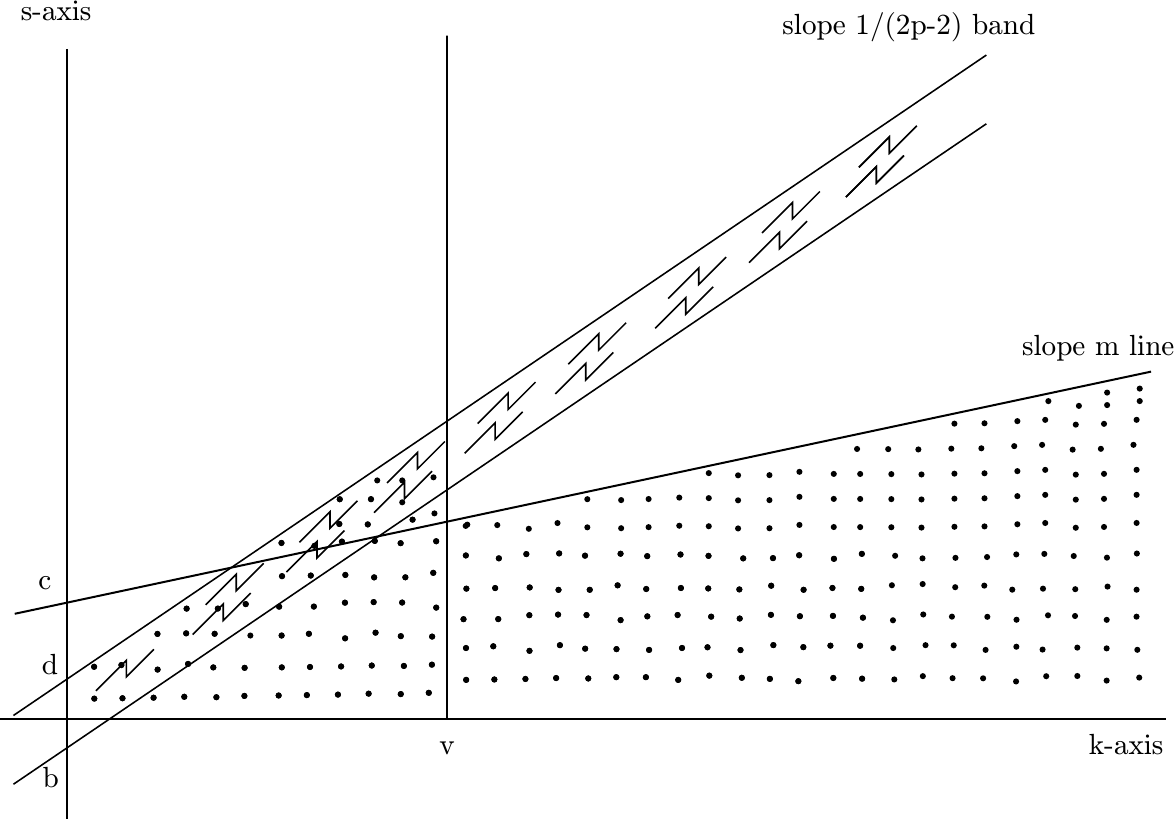}
    \caption{In this figure we display a picture of what the $\mathrm{E}_{r+2}$-page of the $\HFp$-Adams spectral sequence of an $\HFp$-nilpotent complete spectrum that admits an $v_1$-banded vanishing line with parameters $(b \leq d, v, m, c, r)$ might look like. We highlight the following features:\\
      (1) The top region is already empty at the $E_2$-page. \\
      (2) The region indicated with lightning flashes is the band (it is depicted this way since this is how it appears for $\Ss/2$) and contains all classes detected $K(1)$-locally. \\
      (3) The empty region below the band vanishes by the $E_{r+2}$-page. \\
      (4) In the dotted region no conditions are imposed.}
    \label{fig:banded}
\end{figure}

%

\begin{exm} \label{exm:miller-banded}
    The main result of \cite{Miller} implies that the $\HFp$-Adams spectral sequence for the mod $p$ Moore spectrum $C(p)$ admits a $v_1$-banded vanishing line of slope $\frac{1}{p^2-p-1}$ for $p$ odd. In \Cref{sec:Y}, we will show that the methods of \cite{Miller} may also be used to obtain a $v_1$-banded vanishing line of slope $\frac{1}{5}$ in the $\HFt$-Adams spectral sequence of $Y = C(2) \otimes C(\eta)$.
\end{exm}

We begin with the behavior of \Cref{dfn:banded-vanishing-line} under retracts and suspensions.

\begin{lem}[Banded Genericity (part 1)] \label{lemm:band-retract-susp}
  Suppose that a synthetic spectrum $X$ has a $v_1$-banded vanishing line with parameters $(b \leq d, v, m, c, r)$. Then
  \begin{enumerate}
  \item any retract of $X$ has a $v_1$-banded vanishing line with the same parameters as $X$,
  \item $\Sigma^{k,k} X$ has a $v_1$-banded vanishing line with parameters
    $$ \left( b - \frac{1}{2p-2} k \leq d - \frac{1}{2p-2} k, v + k, m, c - mk, r \right), $$
  \item $\Sigma^{0,s} X$ has a $v_1$-banded vanishing line with parameters
    $$ ( b + s \leq d + s, v, m, c + s, r). $$
  \end{enumerate}
\end{lem}

\begin{proof}
  This lemma is a version of \Cref{lemm:suspension-line} for banded vanishing lines.
  As with the earlier lemma it follows from tracking how bigraded homotopy groups change under retracts and bigraded suspensions.
\end{proof}

\begin{prop}[Banded Genericity (part 2)] \label{prop:band-in-cofiber-seqs}
  Let $A \xrightarrow{f} B \xrightarrow{g} C \xrightarrow{\delta} \Sigma A$ be a cofiber sequence of synthetic spectra such that
  \begin{itemize}
  \item $A$ has a $v_1$-banded vanishing line with parameters $(b_A \leq d_A, v_A, m, c_A, r_A)$ and
  \item $C$ has a $v_1$-banded vanishing line with parameters $(b_C \leq d_C, v_C, m, c_C, r_C)$.
  \end{itemize}
  Then, $B$ has a banded vanishing line with parameters $(b_B \leq d_B, v_B, m, c_B, r_B)$, where
  \begin{itemize}
  \item $b_B = \min(b_A, b_C - r_A) \leq \max(d_A, d_C) = d_B$,
  \item $v_B = \max(v_A + 1,v_C, \frac{c_B - b_B}{(2p-2)^{-1}-m})$,
  \item $c_B = \max(c_A + r_C, c_C)$,
  \item $r_B = r_A + \max \left( r_C, \left\lfloor \max(d_A, \min(d_A + r_C, d_C)) - b_C - \frac{1}{2p-2} \right\rfloor \right)$.
  \end{itemize}
\end{prop}

In order to prevent expressions such as $F^{\frac{1}{2p-2}k + b_A}\pi_k(\tau^{-1}A)$ from cluttering the proof of \Cref{prop:band-in-cofiber-seqs}, we introduce the following compact notation (which will not appear outside this section):
\begin{center}
  \begin{tabular}{cc cc cc}
    $\sampi \coloneqq (2p-2)^{-1}$ & $L \coloneqq L_{K(1)}$ &  
    $\bar{A} \coloneqq \tau^{-1}A$ & $\bar{B} \coloneqq \tau^{-1}B$ & $\bar{C} \coloneqq \tau^{-1}C$ \\
  \end{tabular}
\end{center}

Before starting the proof of \Cref{prop:band-in-cofiber-seqs}, we prove two lemmas:

\begin{lem}\label{lemm:lift-filt-bnd}
  Suppose that $A \xrightarrow{f} B \xrightarrow{g} C$ is a cofiber sequence of synthetic spectra,
  where every $\tau$-power torsion element of $\pi_{k,k+s} (C)$ is $\tau^{r}$-torsion.
  Then, the indicated lift exists in the diagram below:
\begin{center}
  \begin{tikzcd}
    F^s\pi_k(A) \ar[r, two heads] \ar[d, hook] &
    \mathrm{Im}(F^sf) \ar[r, hook] \ar[d, hook] &
    \ker(F^sg) \ar[r, hook] \ar[d, hook] \ar[dl, dashed] &
    F^s\pi_k(B) \ar[r] \ar[d, hook] &
    F^s\pi_k(C) \ar[d,hook] \\
    F^{s-r}\pi_k(A) \ar[r, two heads] &
    \mathrm{Im}(F^{s-r}f) \ar[r, hook]  &
    \ker(F^{s-r}g) \ar[r, hook]  &
    F^{s-r}\pi_k(B) \ar[r]  &
    F^{s-r}\pi_k(C)  
  \end{tikzcd}
\end{center}
\end{lem}

\begin{proof}
  Let $D_c$ denote the long exact sequence of bigraded homotopy groups for $A \to B \to C$ considered as an acyclic chain complex such that $\pi_{k,k+c}(B)$ is placed in degree zero. The complex $D_c$ fits into a level-wise short exact sequence $D_c^{\mathrm{tor}} \to D_c \to D_c^{\mathrm{tf}}$ where $D_c^{\mathrm{tor}}$ and $D_c^{\mathrm{tf}}$ are given by the same decorations applied level-wise. This lemma is equivalent to the statement that the map
  $$ H_0(D_s^{\mathrm{tf}}) \xrightarrow{\tau^r} H_0(D_{s-r}^{\mathrm{tf}}) $$
  is zero. Using the cofiber sequence of chain complexes above this map is isomorphic to the map
  $$ H_{-1}(D_s^{\mathrm{tor}}) \xrightarrow{\tau^r} H_{-1}(D_{s-r}^{\mathrm{tor}}). $$
  The latter map is zero because the map
  $$ \pi_{k,k+s}(C)_{\mathrm{tor}} \xrightarrow{\tau^r} \pi_{k,k+s-r}(C)_{\mathrm{tor}} $$
  is zero.
\end{proof}

\begin{lem} \label{lemm:qoppa-exactness}
  Under the hypotheses and notation of \Cref{prop:band-in-cofiber-seqs}, the sequence
  $$ F^{\qoppa -1} \pi_{k+1}(\bar{C}) \to F^\qoppa \pi_{k}(\bar{A}) \to F^\qoppa \pi_{k}(\bar{B}) \to F^\qoppa \pi_{k}(\bar{C}) \to F^{\qoppa + 1}\pi_{k-1}(\bar{A}) $$
  is exact for any $\qoppa$ such that $ mk + c_B \leq \qoppa \leq \sampi k + b_B $.
  Moreover, this sequence is exact at $F^\qoppa \pi_k(\bar{A})$ under the weaker condition that
  $$ mk + c_A \leq \qoppa \leq \sampi (k+1) + b_C + 1. $$
\end{lem}

\begin{proof}
  This sequence is a subsequence of the long exact sequence on homotopy groups for the cofiber sequence $\bar{A} \to \bar{B} \to \bar{C}$, and is therefore automatically a chain complex.
    
  {\bf Exactness at $F^{\qoppa} \pi_k (\bar{A})$.} Consider the diagram
  \begin{center}
    \begin{tikzcd}
      F^{\sampi (k+1) + b_C} \pi_{k+1}(\bar{C}) \ar[dr] \ar[dd, equal] & & & \\
      & F^{\qoppa -1} \pi_{k+1}(\bar{C}) \ar[r] \ar[dl] &
      F^{\qoppa} \pi_k(\bar{A}) \ar[r] \ar[d, hook] &
      F^{\qoppa} \pi_k(\bar{B}) \ar[d] \\
      \pi_{k+1}(L\bar{C}) \ar[rr] & &
      \pi_{k}(L\bar{A}) \ar[r] &
      \pi_{k}(L\bar{B}), 
    \end{tikzcd}
  \end{center}
  where the top left diagonal map exists because
  $$ \qoppa -1 \leq \sampi k + b_B -1 \leq \sampi (k+1) + b_C, $$
  and the middle vertical map is injective because
  $$ mk + c_A \leq mk + c_B \leq \qoppa. $$
  
  {\bf Exactness at $F^{\qoppa} \pi_k (\bar{B})$.} Consider the diagram
  \begin{center}
    \begin{tikzcd}
      F^{\qoppa} \pi_k (\bar{A}) \ar[d, equal] \ar[r, two heads] &
      \mathrm{Im}(F^\qoppa f) \ar[r, hook] \ar[d, hook] &
      \ker(F^\qoppa g) \ar[r, hook] \ar[d, hook] \ar[dl, dashed] &
      F^{\qoppa} \pi_k (\bar{B}) \ar[d, hook] \\
      F^{\qoppa - r_C} \pi_k (\bar{A}) \ar[r, two heads] &
      \mathrm{Im}(F^{\qoppa - r_C} f) \ar[r, hook] &
      \ker(F^{\qoppa - r_C} g) \ar[r, hook]  &
      F^{\qoppa - r_C} \pi_k (\bar{B}),
    \end{tikzcd}
  \end{center}
  where the dashed arrow exists by \Cref{lemm:lift-filt-bnd}, which applies because
  $$ mk + c_C \leq mk + c_B \leq \qoppa, $$
  and the leftmost vertical arrow is an isomorphism because
  $$ mk + c_A \leq mk + c_B - r_C \leq \qoppa - r_C \leq \qoppa \leq \sampi k + b_B \leq \sampi k + b_A. $$

  {\bf Exactness at $F^{\qoppa} \pi_k (\bar{C})$.} Consider the diagram
  \begin{center}
    \begin{tikzcd}
      F^{\qoppa + r_A} \pi_k (\bar{B}) \ar[d, hook] \ar[r, two heads] &
      \mathrm{Im}(F^{\qoppa + r_A}g) \ar[r, hook] \ar[d, hook] &
      \ker(F^{\qoppa + r_A} \delta) \ar[r, hook] \ar[d, equal] \ar[dl, dashed] &
      F^{\qoppa + r_A} \pi_k (\bar{C}) \ar[d, equal] \\
      F^{\qoppa - r_A} \pi_k (\bar{B}) \ar[r, two heads] &
      \mathrm{Im}(F^{\qoppa } g) \ar[r, hook] &
      \ker(F^{\qoppa} \delta) \ar[r, hook]  &
      F^{\qoppa} \pi_k (\bar{C}),
    \end{tikzcd}
  \end{center}
  where the dashed arrow exists by \Cref{lemm:lift-filt-bnd}, which applies because
  $$ m(k-1) + c_A \leq mk + c_B \leq \qoppa + r_A + 1,$$
  and the middle right vertical arrow is an isomorphism because
  \[ mk + c_C \leq mk + c_B \leq \qoppa \leq \qoppa + r_A \leq \sampi k + b_B + r_A \leq \sampi k + b_C. \qedhere \]
\end{proof}

\begin{proof}[Proof of \Cref{prop:band-in-cofiber-seqs}]
  We will prove properties (1)-(4) of \Cref{dfn:banded-vanishing-line} in reverse order.
  Property (4) is obvious from the long exact sequence on bigraded homotopy groups.
  
  Assuming that
  $$ mk + c_B \leq \sampi k + b_B, $$
  which is true whenever $k \geq v_B$,
  we can construct \Cref{fig:banded-chase}.
  The second and third rows of \Cref{fig:banded-chase} are exact by \Cref{lemm:qoppa-exactness}.
  The fifth and sixth rows of \Cref{fig:banded-chase} are also exact.
  The indicated equalities follow easily from the hypotheses. 
    
  \begin{figure}  
    \centering              
    \begin{turn}{90}
      \begin{tikzcd}[row sep = huge]
	F^{\sampi (k+1) + b_C} \pi_{k+1} (\bar{C}) \arrow[d, equal] &
	F^{\sampi k + b_A} \pi_k (\bar{A}) \arrow[d, equal] & &
	F^{\sampi k + b_C} \pi_k (\bar{C}) \arrow[d, equal] & \\
	F^{\sampi k + b_B - 1} \pi_{k+1} (\bar{C}) \arrow[r] \arrow[d, hook] &
	F^{\sampi k + b_B} \pi_k (\bar{A}) \arrow[r] \arrow[d, equal] &
	F^{\sampi k + b_B} \pi_k (\bar{B}) \arrow[r] \arrow[d, hook] &
	F^{\sampi k + b_B} \pi_k (\bar{C}) \arrow[r] \arrow[d, equal] &
	F^{\sampi k + b_B + 1} \pi_{k-1} (\bar{A}) \arrow[d, hook] \\
        F^{mk + c_B - 1} \pi_{k+1} (\bar{C}) \arrow[r] \arrow[dd, hook] &
	F^{mk + c_B} \pi_k (\bar{A}) \arrow[r] \arrow[d, equal] &
	F^{mk + c_B} \pi_k (\bar{B}) \arrow[r] \arrow[dd, hook] &
	F^{mk + c_B} \pi_k (\bar{C}) \arrow[r] \arrow[d, equal] &
	F^{mk + c_B + 1} \pi_{k-1} (\bar{A}) \arrow[d, equal] \\        
        & F^{mk + c_A} \pi_k (\bar{A}) \arrow[d, hook] & &
	F^{mk + c_C} \pi_k (\bar{C}) \arrow[d, hook] &
	F^{m(k-1) + c_A} \pi_{k-1} (\bar{A}) \arrow[d, hook] \\        
        \pi_{k+1} (\bar{C}) \arrow[r] \arrow[d] &
        \pi_k (\bar{A}) \arrow[r] \arrow[d] &
        \pi_k (\bar{B}) \arrow[r] \arrow[d] &
        \pi_k (\bar{C}) \arrow[r] \arrow[d] &
        \pi_{k-1} (\bar{A}) \arrow[d] \\
	\pi_{k+1} (L \bar{C}) \arrow[r] &
        \pi_k (L \bar{A}) \arrow[r] &
        \pi_k (L \bar{B}) \arrow[r] &
        \pi_k (L \bar{C}) \arrow[r] &
	\pi_{k-1} (L \bar{A})
      \end{tikzcd}
    \end{turn}
    \centering
    \caption{}\label{fig:banded-chase}
  \end{figure}

  {\bf Proof of (3).} 
  We wish to show that
  $$ F^{\sampi k + b_B} \pi_{k} (\bar{B}) \to \pi_k (L \bar{B}) $$
  is an isomorphism for $k \geq v_B$.
  The vertical maps from the top row of \Cref{fig:banded-chase} to the bottom row are isomorphisms by hypothesis. The vertical maps from the fourth row to the bottom row are also isomorphisms by hypothesis.
  Thus, we may apply the five lemma to the maps between the second and the bottom rows in order to conclude.

  {\bf Proof of (2).} 
  We wish to show that
  $$ F^{\sampi k + b_B} \pi_{k} (\bar{B}) \to F^{mk + c_B} \pi_k (\bar{B}) $$
  is an isomorphism for $k \geq v_B$.
  This map is automatically injective, so it suffices to apply the four lemma to the maps between the second and third rows of \Cref{fig:banded-chase}.

  {\bf Proof of (1).} 
  Let $w \in \pi_{k,k+s} (B)_{\mathrm{tor}}$ and assume that $s \geq m k + c_B$ and $k \geq v_B$.
  We would like to bound the $\tau$-torsion order of $w$. 

  {\bf Step 1.} We have $ w \in \pi_{k,k+s}(B)_{\mathrm{tor}}$ such that
  $$ mk + c_B \leq s \leq \sampi k + \max(d_A, d_C). $$
  If $s > \sampi k + d_A +r_C$, then $g(\tau^{r_C}w) = 0$ so $\tau^{r_C} w$ lifts to $\pi_{k,k+s-r_C}(A) = 0$ and therefore $\tau^{r_C}w = 0$, hence $\tau^{r_B} w =0$.
  On the other hand, if $s \leq \sampi k + d_A + r_C$, we move on to step 2.

  {\bf Step 2.} We have $ w \in \pi_{k,k+s}(B)_{\mathrm{tor}}$ such that
  $$ mk + c_B \leq s \leq \sampi k + \max(d_A, \min(d_A + r_C, d_C)). $$
  Find the smallest $N$ such that $g(\tau^Nw) = 0$ and an $x \in \pi_{k,k+s-N}(A)$ such that $f(x)= \tau^N w$. We have a bound $N \leq r_C$ coming from the fact that
  $ mk + c_C \leq mk + c_B \leq s. $
  From this we may conclude that $s-N \geq mk + c_B - r_C \geq mk + c_A$.
  
  {\bf Step 3.} We have a $x \in \pi_{k,k+s-N}(A)$ such that $f(x) = \tau^Nw$.
  If $$ \sampi (k+1) + b_C + 1 < s-N, $$ we replace $x$ by $\tau^L x$ where $L$ satisfies
  \[ mk + c_A \leq s-N-L \leq \sampi (k+1) + b_C + 1. \]
  This is possible because \[mk + c_A \leq \sampi k + b_B \leq \sampi (k+1) + b_C,\] which holds since $k \geq v_B$.
  
  {\bf Step 4.} We have a $y \in \pi_{k,k + \qoppa}(A)$ such that $f(y) = \tau^Mw$ where
  $$ mk + c_A \leq \qoppa \leq \sampi (k+1) + b_C + 1. $$
  Consider the diagram
  \begin{center}
    \begin{tikzcd}
      \pi_{k+1, k + \qoppa }(C) \ar[r] \ar[d, two heads] &
      \pi_{k, k + \qoppa }(A) \ar[r, "\tau^{r_A}"] \ar[d, two heads] &
      \pi_{k, k + \qoppa - r_A}(A) \\
      F^{\qoppa -1 } \pi_{k+1}(\bar{C}) \ar[r] &
      F^{\qoppa} \pi_k(\bar{A}) \ar[r] \ar[ur, dashed] &
      F^{\qoppa} \pi_k(\bar{B}),
    \end{tikzcd}
  \end{center}
  where the second row is exact by \Cref{lemm:qoppa-exactness},
  and the dashed arrow exists because any $\tau$-torsion element of $\pi_{k,k+\qoppa}(A)$ has torsion order bounded by $r_A$. The image of $y$ in $F^\qoppa \pi_k(\bar{B})$ is zero by hypothesis, so we can use exactness and surjectivity to produce a lift $z \in \pi_{k+1,k+\qoppa}(C)$ such that $\tau^{r_A}\delta(z) = \tau^{r_A}y$. From this we may conclude that $\tau^{r_A+M}w = 0$. We may now read off that
  \begin{align*}
    r_A + M
    &\leq r_A + \max \left( r_C, \lfloor \sampi k + \max(d_A, \min(d_A + r_C, d_C)) \rfloor - \lfloor \sampi (k+1) + b_C + 1 \rfloor \right) \\
    &\leq r_A + \max \left( r_C, \sampi k + \max(d_A, \min(d_A + r_C, d_C))  - \sampi (k+1) - b_C  \right) \\
    &\leq r_A + \max \left( r_C, \max(d_A, \min(d_A + r_C, d_C)) - b_C - \sampi \right) \\
    &= r_B. \qedhere
    \end{align*}
\end{proof}


%% file: Y-band.tex

In \Cref{exm:miller-banded} we observed that Miller's computation of the $T(1)$-local homotopy of a Moore spectrum at odd primes \cite{Miller} can be summarized by saying that $C(p)$ admits a $v_1$-banded vanishing line of slope $\frac{1}{p^2 - p -1}$. The corresponding calculation at the prime $2$ is Mahowald's computation of the $T(1)$-local homotopy of the spectrum $Y \coloneqq C(2) \otimes C(\eta)$ \cite{MahJEHP}. Unfortunately, Mahowald's proof does not provide a $v_1$-banded vanishing line. In this section we adapt Miller's methods to the prime $2$ in order to obtain a $v_1$-banded vanishing line on $Y$.

\begin{thm}\label{thm:Y-v1-band}
  The $\HFt$-Adams spectral sequence for $Y$ has a $v_1$-banded vanishing line with parameters
  $\left(-\frac{3}{2} \leq 0, 15, \frac{1}{5}, \frac{13}{5}, 1 \right)$.
\end{thm}

To prove \Cref{thm:Y-v1-band}, we apply the Miller square technique of \cite{Miller} \footnote{See \cite[Section 9]{AndrewsMiller} for a corrected and improved exposition of this technique.} to compute the $\HFt$-Adams spectral sequence of $Y$ above a line of slope $\frac{1}{5}$. The Miller square technique relates the differentials in the $\HFt$-Adams spectral sequence to those in the algebraic Novikov spectral sequence. We will use this relation to prove \Cref{thm:Y-v1-band} by producing many differentials in the $\HFt$-Adams spectral sequence for $Y$. Another major input to this section is a computation of Davis and Mahowald \cite{v1ExtDavisMahowald} that determines the $\mathrm{E}_2$-page of this spectral sequence above a line of slope $\frac{1}{5}$.

\begin{rmk}
  In classical language, \Cref{thm:Y-v1-band} is likely well-known to experts (though no proof appears in print) and the authors thank Mark Behrens for a helpful conversation on the subject.

  Although the statement of \Cref{thm:Y-v1-band} involves synthetic spectra,
  its proof is essentially classical.
  In fact, the bulk of this proof is simply a collation of statements from \cite{Miller} and \cite{v1ExtDavisMahowald}.
\end{rmk}

\begin{rmk}
  More recently, work of Gheorghe, Wang and Xu has identified the Miller square as arising via the motivic Adams spectral sequence \cite{GWX}.\footnote{One could equally well phrase things in terms of the $\nu \HFt$-based Adams spectral sequence in $\BP$-synthetic spectra.} In their formulation, the algebraic Novikov spectral sequence is identified with the motivic Adams spectral sequence for the motivic cofiber of $\tau$.
  The link between the two sides of the Miller square is then provided by the motivic Adams spectral sequence for the sphere which has maps out to both the classical Adams spectral sequence (by inverting $\tau$) and the motivic Adams spectral sequence for the cofiber of $\tau$ (by modding out by $\tau$).
  We encourage the reader interested in extending the computations from this section to read \cite{GWX}.
\end{rmk}

\begin{rmk}
  At the end of this section, we will show in \Cref{cor:telescope} that \Cref{thm:Y-v1-band} implies the height $1$ prime $2$ telescope conjecture, recovering the main result of \cite{MahJEHP}.
\end{rmk}

\begin{ntn}
  Throughout this section we will fix a prime $p$ and write $\mathrm{H}_* (X)$ for the mod $p$ homology of a spectrum $X$.
\end{ntn}

%

Let us begin by describing the Miller square technique, which applies to certain spectra $X$, as we recall below. The Miller square consists of the following diagram of spectral sequences:
%

\begin{center}
	\begin{tikzpicture}[baseline= (a).base]
	    \node[scale=.90] (a) at (0,0){
	\begin{tikzcd}
        \Ext_{E_0 \BP_* \BP} ^{s,i,t} (E_0 \BP_*, E_0 \BP_* (X)) \arrow[d, Rightarrow, "\text{Algebraic Novikov}", swap] \arrow[rr,"\cong"] &  & \Ext^{s,t} _{P_*} (\F_p, \Ext^{i,*} _{E_*} (\F_p, \mathrm{H}_*(X))) \arrow[d, Rightarrow, "\text{Cartan-Eilenberg}"] \\[25pt]
        \Ext_{\BP_* \BP} ^{s,t} (\BP_*, \BP_* (X)) \arrow[dr, Rightarrow, "\text{Adams-Novikov}", swap] & & \Ext_{\A_*} ^{s+i,t+i} (\F_p, \mathrm{H}_* (X)) \arrow[dl, Rightarrow, "\HFp-\text{Adams}"] \\[25pt]
		& \pi_{t-s} (X) &  
	\end{tikzcd}
		};
	\end{tikzpicture}
\end{center}

The reader will of course notice that, as we have drawn it, the diagram is not a square. This is because we want to emphasize the fact that the $\mathrm{E}_2$-pages of the algebraic Novikov and Cartan-Eilenberg spectral sequences do not agree in general, but only under the assumption that $X$ is \emph{$(\BP, \HFp)$-good} as defined below.

\begin{dfn}
    We say that a spectrum $X$ is \emph{$(\BP, \HFp)$-good} if the $\HFp$-Adams spectral sequence for $\BP \otimes X$ converges strongly and collapses on the $\mathrm{E}_2$-page\footnote{Note that under this collapse assumption, strong convergence is implied by conditional convergence.}, and the $\HFp$-Adams filtration on $\BP_* (X)$ agrees with the $(p,v_1,v_2, \dots)$-adic filtration.
\end{dfn}

\begin{exm}
    The mod $p$ Moore spectrum $C(p)$ is $(\BP, \HFp)$-good for any prime $p$. The spectrum $Y$ is $(\BP, \HFt)$-good.
\end{exm}

Let us now recall from \cite{Miller} the definitions of the spectral sequences labeled algebraic Novikov and Cartan-Eilenberg in the above diagram.

Before we describe the algebraic Novikov spectral sequence, we require some notation.

\begin{ntn}
    Let $I = (p,v_1,\dots) \subset \BP_* $. Given a $\BP_* \BP$-comodule $M$, we let $E_0 M$ denote the associated graded of $M$ with respect to the $I$-adic topology. We equip $E_0 M$ with the bigrading $(i,t)$, where $i$ is the $I$-adic grading and $t$ is the grading inherited from $M$.
\end{ntn}

\begin{exm}
    In the grading above, we have \[E_0 \BP_* \cong \F_p[q_0, q_1, \dots] \text{ with } \abs{q_i} = (1,2(p^{i}-1))\] and \[E_0 \BP_* \BP \cong E_0 \BP_* [t_0, t_1, \dots] \text{ with } \abs{t_i} = (0,2(p^{i}-1)).\]
\end{exm}

To obtain the algebraic Novikov spectral sequence for a $\BP_* \BP$-comodule $M$, equip the cobar complex $\Omega^{*} (\BP_* \BP, M)$ by the tensor product filtration determined by the $I$-adic filtrations on $\BP_* \BP$ and $M$. This makes $\Omega^{*} (\BP_* \BP, M)$ into a filtered complex, and the algebraic Novikov spectral sequence is the associated spectral sequence.

\begin{fact}[{\cite[Remark 8.4]{Miller}}]
    The algebraic Novikov spectral sequence converges strongly under the assumption that $M$ is of finite type as a $\BP_*$-module.
\end{fact}

On the other hand, the Cartan-Eilenberg spectral sequence in the above diagram is that associated to the extension of Hopf algebras \[P_* \to \A_* \to E_* ,\] where for an odd prime $P_* \cong \F_p [\xi_1, \xi_2, \dots]$ and $E_* \cong \Lambda_{\F_p} [\tau_0, \tau_1 , \dots]$. At the prime $2$, one has $P_* \cong \F_2 [\zeta_1 ^2, \zeta_2 ^2, \dots]$ and $E_* \cong \F_2 [\zeta_1, \zeta_2, \dots] / (\zeta_1 ^2, \zeta_2 ^2, \dots)$.

\begin{cnv}
	Here we follow Milnor \cite{MilnorSteenrod} in calling the polynomial generators of the mod $2$ Steenrod algebra $\zeta_i$ rather than the now more common notation $\xi_i$, which conflicts with the notation for an odd prime.
\end{cnv}

Let us now explain the top horizontal arrow in the above diagram.

\begin{lem}
    If $X$ is $(\BP, \HFp)$-good, then there exists a natural isomorphism \[\Ext^{s,i,t} _{E_0 \BP_* \BP} (E_0 \BP_*, E_0 \BP_* (X)) \cong \Ext^{s,t} _{P_*} (\F_p, \Ext^{i,*} (\F_p, \mathrm{H}_* (X))).\]
\end{lem}

\begin{proof}
    First, one notes that $E_0 \BP_* \BP$ is a split Hopf algebroid in the sense of \cite[Section 7]{Miller}. Indeed, $E_0 \BP_* \BP_*$ splits as $E_0 \BP_* \tilde{\otimes} P_*$ \cite[p. 305]{Miller}, which implies by \cite[Proposition 7.6]{Miller} that \[\Ext^{s,i,t} _{E_0 \BP_* \BP} (E_0 \BP_*, E_0 \BP_* (X)) \cong \Ext^{s,i,t} _{P_*} (\F_p, E_0 \BP_* (X)).\]

    Now, since $\Ext_{E_*} ^{*,*} (\F_p, \mathrm{H}_* (X))$ is the $\mathrm{E}_2$-page of the $\HFp$-Adams spectral sequence converging to $\BP_* (X)$, the desired isomorphism follows from the definition of $(\BP, \HFp)$-good.
\end{proof}
%

The main tool that we use from \cite{Miller} is the following theorem, which relates the $d_2$-differentials in the algebraic Novikov spectral sequence to those in the $\HFp$-Adams spectral sequence, under the assumption that the Cartan-Eilenberg spectral sequence collapses. We first state a piece of notation, and then the theorem.

\begin{ntn}
    We let $F^{\bullet} \Ext^{s,t} _{\A_*} (\F_p, \mathrm{H}_* (X))$ denote the filtration induced by the Cartan-Eilenberg spectral sequence on $\Ext^{s,t} _{\A_*} (\F_p, \mathrm{H}_* (X))$.
\end{ntn}

\begin{thm}[{\cite[Theorem 6.1]{Miller}}]\label{thm:Miller-thm} \todo{Somebody should go and check if I translated what Haynes wrote correctly.}
    Let $X$ denote a $(\BP, \HFp)$-good spectrum, and let $s,t$ be integers such that the Cartan-Eilenberg spectral sequence converging to $\Ext^{s,t}_{\A_*} (\F_p, \mathrm{H}_* (X))$ collapses at the $\mathrm{E}_2$-page in total degrees $(s,t)$ and $(s+2,t+1)$.

    Then the $d_2$-differential $d_2 ^{\HFt-ASS}$ induces a map
    \[d_2 ^{\HFt-\mathrm{ASS}} : F^\bullet \Ext^{s,t}_{\A_*} (\F_p, \mathrm{H}_* (X)) \to F^{\bullet+1} \Ext^{s+2,t+1}_{\A_*} (\F_p, \mathrm{H}_* (X)) \]
    and hence a map
    \[d_2 ^{\HFt-\mathrm{ASS}} : \frac{F^\bullet \Ext^{s,t}_{\A_*} (\F_p, \mathrm{H}_* (X))}{F^{\bullet+1} \Ext^{s,t}_{\A_*} (\F_p, \mathrm{H}_* (X))} \to \frac{F^{\bullet+1} \Ext^{s+2,t+1}_{\A_*} (\F_p, \mathrm{H}_* (X))}{F^{\bullet+2} \Ext^{s+2,t+1}_{\A_*} (\F_p, \mathrm{H}_* (X))}. \]
    Moreover, this associated-graded map may be identified with $-d_2 ^{\hspace{0.05cm}\mathrm{alg-Nov}}$, where $d_2 ^{\hspace{0.05cm}\mathrm{alg-Nov}}$ is the $d_2$-differential in the algebraic Novikov spectral sequence.
\end{thm}

Miller's main application is to the mod $p$ Moore spectrum $X = C(p)$ for an odd prime $p$. In this case, the Cartan-Eilenberg spectral sequence automatically collapses, so \Cref{thm:Miller-thm} applies. Miller is therefore able to compute the $\HFp$-Adams spectral sequence above a line of slope $\frac{1}{p^2-p-1}$ by studying the algebraic Novikov spectral sequence.

The main obstacle to carrying out Miller's program at the prime $2$ is that the Cartan-Eilenberg spectral sequence no longer collapses. What allows us to proceed is a computation of Davis and Mahowald \cite{v1ExtDavisMahowald} that implies that the Cartan-Eilenberg spectral sequence for $Y$ collapses above a line of slope $\frac{1}{5}$. 

The main steps in the proof of \Cref{thm:Y-v1-band} are as follows:

\begin{enumerate}
    \item Using Davis and Mahowald's computation \cite{v1ExtDavisMahowald} of $v_1 ^{-1} \Ext^{s,t} _{\A_*} (\F_2, \mathrm{H}_* (Y))$, deduce that the $v_1$-localized Cartan-Eilenberg spectral sequence collapses for $Y$.
    \item Recall from \cite{Miller} the structure of the $v_1$-localized algebraic Novikov spectral sequence for $Y$.
    \item Show that the $v_1$-local computations above agree with those before $v_1$-localizing above a line of slope $\frac{1}{5}$.
    \item Use \Cref{thm:Miller-thm} to compute the $\HFt$-Adams spectral sequence for $Y$ above a line of slope $\frac{1}{5}$. Conclude that \Cref{thm:Y-v1-band} holds.
\end{enumerate}

We begin by recalling some basic facts about $Y$.

\begin{prop}[{\cite[Theorem 1.2]{v1v2DavisMahowald}}]\label{prop:v1-self-map}
	There is a $v_1$-self map $v_1 : \Sigma^2 Y \to Y$ of $Y$, which is of $\HFt$-Adams filtration one.
\end{prop}

\begin{lem}\label{lemm:h21-lives}
	There is a nonzero element $w_1 \in \pi_5 (Y)$ which is represented in the Adams spectral sequence of $Y$ by the cocycle $h_{2,1} = [\zeta_2 ^2]$.
\end{lem}

\begin{proof}
    This follows immediately from calculating the first five stems of the $\mathrm{E}_2$-page of the Adams spectral sequence for $Y$. See for example the chart on \cite[p. 620]{v1v2DavisMahowald}.
\end{proof}

We now collect the computation of some $v_1$-inverted Ext groups.

\begin{thm}\label{thm:v1-inv-comp}
	There are algebra isomorphisms
    \begin{align} \tag{1} \label{eq:eq1}
        v_1 ^{-1} \Ext^{*,*} _{\mathcal{A}_{*}} (\F_2,\mathrm{H}_* (Y)) \cong \F_2 [v_1^{\pm1}][h_{j,1} \vert j \geq 2],
    \end{align}
    \begin{align} \tag{2} \label{eq:eq2}
        q_1 ^{-1} \Ext^{*,*,*} _{P_{*}} (\F_2, E_0 \BP_* (C(2))) \cong \F_2[q_1 ^{\pm 1}][h_{j,1} \vert j \geq 1],
    \end{align}
    and
    \begin{align} \tag{3} \label{eq:eq3}
        q_1 ^{-1} \Ext^{*,*,*} _{P_{*}} (\F_2, E_0 \BP_* (Y)) \cong \F_2[q_1 ^{\pm 1}][h_{j,1} \vert j \geq 2],
    \end{align}
    where \[\abs{h_{j,1}} = (1,2^{j+1}-2) \text{ in } (\ref{eq:eq1}),\] and \[\abs{h_{j,1}} = (1,0,2^{j+1}-2) \text{ in } (\ref{eq:eq2}) \text{ and } (\ref{eq:eq3}).\]
\end{thm}

\begin{proof}
    We first need to justify that these localized Ext groups admit the structure of algebras. In the case of the second listed group, this follows from the fact that $\BP_* (C(2)) \cong \BP_* / 2$ is a comodule algebra over $\BP_* \BP$. The case of the first group is \cite[Theorem 3.1]{v1ExtDavisMahowald}, and that of the third follows from its proof.

    Now, the first isomorphism is \cite[Theorem 1.3]{v1ExtDavisMahowald}. The second isomorphism follows from \cite[Corollary 3.5]{MillerLocalization}. Finally, the third isomorphism is obtained from the second because the self-map $\eta$ of $C(2)$ induces multiplication by $h_{1,1}$ on localized Ext groups.
\end{proof}

We next recall from \cite{Miller} the computation of the the $v_1$-localized algebraic Novikov spectral sequence for $C(2)$, from which we deduce it for $Y$.

\begin{thm}[{\cite[Equation (9.20)]{Miller}}]\label{thm:localized-alg-Nov-Moore}
    The $d_2$-differentials in the $v_1$-localized algebraic Novikov spectral sequence for $C(2)$ \[q_1 ^{-1} \Ext^{s,i,t} _{P_*} (\F_2, E_0 \BP_* (C(2))) \Rightarrow v_1^{-1} \Ext^{s,t} _{\BP_* \BP} (\BP_*, \BP_*(C(2)))\] are derivations and are determined by \[d_2 (h_{n,1}) = q_1 h_{n-1,1}^2 \text{ for } n \geq 3.\] The spectral sequence collapses at the $\mathrm{E}_3$-term with $\mathrm{E}_3 = \mathrm{E}_\infty$-page \[\F_2 [q_1 ^{\pm 1}][h_{1,1}, h_{2,1}]/(h_{2,1}^2).\]
\end{thm}

\begin{cor}\label{prop:localized-alg-Nov}
    The $d_2$-differentials in the $v_1$-localized algebraic Novikov spectral sequence for $Y$ \[q_1 ^{-1} \Ext^{s,i,t} _{P_*} (\F_2, E_0 \BP_* (Y)) \Rightarrow v_1^{-1} \Ext^{s,t} _{\BP_* \BP} (\BP_*, \BP_*(Y))\] are derivations and are determined by \[d_2 (h_{n,1}) = q_1 h_{n-1,1}^2 \text{ for } n \geq 3.\] The spectral sequence collapses at the $\mathrm{E}_3$-term with $\mathrm{E}_3 = \mathrm{E}_\infty$-page \[\F_2 [q_1 ^{\pm 1}][h_{2,1}]/(h_{2,1}^2).\]
\end{cor}

This gives rise to a convenient computation of the $K(1)$-local homotopy of $Y$.

\begin{cor}\label{cor:K1-homotopy-Y}
    The $K(1)$-local homotopy of $Y$ is \[ \pi_* (L_{K(1)} Y) \cong \F_2 [v_1 ^{\pm 1}][w_1] / (w_1 ^2) \] as a $\ZZ[v_1]$-module.
\end{cor}

\begin{proof}
    The $v_1$-localized Adams-Novikov spectral sequence for $Y$ converges to the homotopy of $L_{K(1)} Y$ by the Localization Theorem \cite[Theorem 7.5.2]{RavenelOrange}. By \Cref{prop:localized-alg-Nov}, the $\mathrm{E}_2$-page is concentrated in filtrations $0$ and $1$, so the spectral sequence collapses to the desired isomorphism.
\end{proof}

Our next goal is to show that the $v_1$-localized computations above are in fact valid above a line of slope $\frac{1}{5}$.

\begin{lem}\label{lemm:vanishing-results}
    We have \[\Ext^{s,i,t} _{P_*} (\F_2, E_0 \BP_* (Y) / q_1) = 0\] when $s+i > \frac{1}{5} (t-s) + \frac{4}{5}$ and \[\Ext^{s,t} _{\A_*} (\F_2, \mathrm{H}_* (\cof(Y \xrightarrow{v_1} Y))) = 0\] for $s > \frac{1}{5} (t-s) + \frac{4}{5}$. As a consequence, the maps \[\Ext^{s,i,t} _{P_*} (\F_2, E_0 \BP_* (Y)) \to q_1 ^{-1} \Ext^{s,i,t} _{P_*} (\F_2, E_0 \BP_* (Y))\] and \[\Ext^{s,t} _{\A_*} (\F_2, \mathrm{H}_* (Y)) \to v_1 ^{-1} \Ext^{s,t} _{\A_*} (\F_2, \mathrm{H}_* (Y))\] are isomorphisms for $s+i > \frac{1}{5} (t-s) + \frac{7}{5}$ and $s > \frac{1}{5} (t-s) + \frac{7}{5}$, respectively. Moreover, they are surjections for $s + i > \frac{1}{5}(t-s) + \frac{1}{5}$ and $s > \frac{1}{5}(t-s) + \frac{1}{5}$, respectively.
\end{lem}

\begin{proof}
    We begin with the vanishing statement for $\Ext^{s,i,t} _{P_*} (\F_2, E_0 \BP_* (Y) / q_1)$. Let $M$ denote the sub-comodule of $P_*$ spanned by $1$ and $\zeta_1 ^2$ and recall that there is a degree-doubling isomorphism $\A_* \cong P_*$ which sends $\zeta_i$ to $\zeta_i ^2$. Under this isomorphism, $M$ corresponds to the $\A_*$-comodule $\mathrm{H}_* (C(2))$. By \cite[Theorem 2.1]{AdamsPer}, $\Ext^{s,t} _{\A_*} (\F_2, \mathrm{H}_* (C(2)))$ vanishes for $s > \frac{1}{2} (t-s) + 1.$ It follows that $\Ext^{s,t} _{P_*} (\F_2, M)$ vanishes for $s > \frac{1}{2} (\frac{1}{2}t-s) + 1$, i.e. for $s > \frac{1}{5} (t-s) +\frac{4}{5}$.

	We now note that $E_0 \BP_* (Y) / q_1 \cong M \otimes_{\F_2} \F_2 [q_2, q_3, \dots].$ There are therefore a series of Bockstein spectral sequences starting from $\Ext^{s,i,t} _{P_*} (\F_2, M)$ and converging to $\Ext^{s,i,t} _{P_*} (\F_2, E_0 \BP_* (Y) / q_1)$. Since each of the $q_i$ for $i \geq 2$ lies below the plane of interest, this implies the result.

    The vanishing result for $\Ext^{s,t} _{\A_*} (\F_2, \mathrm{H}_* (\cof(Y \xrightarrow{v_1} Y)))$ follows from the above vanishing result and the Cartan-Eilenberg spectral sequence.

	The translation of these vanishing results into the desired isomorphisms and surjections follows from the long exact sequences
	\begin{center}
		    \begin{tikzcd}[column sep=tiny]
		\dots \arrow[r] & \Ext^{s-1,i+1,t+2} _{P_*} (\F_2, E_0 \BP_* (Y) / q_1) \arrow[rr] & \arrow[d, phantom, ""{coordinate, name=Z}] & \Ext^{s,i,t} _{P_*} (\F_2, E_0 \BP_* (Y)) \arrow[dll, "\cdot q_1"'{pos=1}, rounded corners, to path={ -- ([xshift=2ex]\tikztostart.east) |- (Z) [near end]\tikztonodes -| ([xshift=-2ex]\tikztotarget.west) --(\tikztotarget)}] & \\
		& \Ext^{s,i+1,t+2} _{P_*} (\F_2, E_0 \BP_* (Y)) \arrow[rr] & \vphantom{a} & \Ext^{s,i+1,t+2} _{P_*} (\F_2, E_0 \BP_* (Y)/q_1) \arrow[r] & \dots
	\end{tikzcd}
	\end{center}
	and
	\begin{center}
	\begin{tikzcd}[column sep = tiny]
        \dots \arrow[r] & \Ext^{s, t+3} _{\A_*} (\F_2, \mathrm{H}_* (\cof(Y \xrightarrow{v_1} Y))) \arrow[rr] & \arrow[d, phantom, ""{coordinate, name=Z}] & \Ext^{s, t} _{\A_*} (\F_2, \mathrm{H}_* (Y)) \arrow[dll, "\cdot v_1"'{pos=1}, rounded corners, to path={ -- ([xshift=2ex]\tikztostart.east) |- (Z) [near end]\tikztonodes -| ([xshift=-2ex]\tikztotarget.west) --(\tikztotarget)}] & \\
        & \Ext^{s+1,t+3} _{\A_*} (\F_2, \mathrm{H}_* (Y)) \arrow[rr] & \vphantom{a} & \Ext^{s+1,t+3} _{\A_*} (\F_2, \mathrm{H}_* (\cof(Y \xrightarrow{v_1} Y))) \arrow[r] & \dots.
	\end{tikzcd} \qedhere
	\end{center}
\end{proof}

\begin{cor}\label{prop:CESS-collapse}
    For $s + i > \frac{1}{5} (t-s) + \frac{7}{5}$, the Cartan-Eilenberg spectral sequence \[\Ext^{s,i,t} _{P_*} (\F_2, E_0 \BP_* (Y)) \Rightarrow \Ext^{s+i,t+i} _{\A_*} (\F_2, \mathrm{H}_* (Y))\] collapses at the $\mathrm{E}_2$-page.
\end{cor}

\begin{proof}
    It suffices to show that the $\mathrm{E}_2$-page and the target are of the same finite dimension as bigraded $\F_2$-vector spaces in this range. This follows from \Cref{thm:v1-inv-comp} and \Cref{lemm:vanishing-results}.
\end{proof}

\begin{cor}\label{cor:alg-nov-range}
    For $s + i > \frac{1}{5} (t-s) + \frac{18}{5}$, the algebraic Novikov spectral sequence \[\Ext^{s,i,t} _{P_*} (\F_2, E_0 \BP_* (Y)) \Rightarrow \Ext^{s,t} _{\BP_* \BP} (\BP_*, \BP_*(Y))\] agrees with the $v_1$-localized algebraic Novikov spectral sequence.
\end{cor}

\begin{proof}
    There is a map from the algebraic Novikov spectral sequence to its $v_1$-localized version, which by \Cref{lemm:vanishing-results} is an equivalence on $\mathrm{E}_2$-pages for $s+i > \frac{1}{5} (t-s) + \frac{7}{5}$. We may therefore lift all $d_2$-differentials that lie entirely in this range, which shows that the map from the $\mathrm{E}_3$-page of the algebraic Novikov spectral sequence to the $v_1$-localized algebraic Novikov spectral sequence is an equivalence for $s+i > \frac{1}{5}(t-s)+\frac{18}{5}$, since all entering $d_2$-differentials in this range originate in the range $s+i > \frac{1}{5} (t-s) + \frac{7}{5}$.

    The classes left on the $\mathrm{E}_3$-page in the region $s+i > \frac{1}{5}(t-s) + \frac{18}{5}$ cannot be the source of higher differentials by sparsity, and they cannot be the targets of higher differentials because they are detected in the $v_1$-localized Ext groups. It follows that $\mathrm{E}_3 = \mathrm{E}_\infty$ in the region $s+i > \frac{1}{5} (t-s) + \frac{18}{5}$, as desired.
\end{proof}

Finally, we are able to combine the above results with \Cref{thm:Miller-thm} to compute the $\HFt$-Adams spectral sequence of $Y$ above a line of slope $\frac{1}{5}$, at least up to an associated graded. From this we will deduce \Cref{thm:Y-v1-band}.

\begin{prop}\label{prop:Y-main-prop}
    For $s > \frac{1}{5}(t-s) + \frac{12}{5}$, the $\HFt$-Adams spectral sequence \[\Ext^{s,t} _{\A_*} (\F_2, \mathrm{H}_* (Y)) \Rightarrow \pi_{t-s} (Y)\] collapses at the $E_3$-page. Moreover, the map \[ F^{\frac{1}{5}k + \frac{13}{5}} \pi_k (Y) \to \pi_k (L_{K(1)} Y) \] is an isomorphism for $k \geq 15$.
\end{prop}

\begin{proof}
    We begin by noting that in the range $s > \frac{1}{5} (t-s) + \frac{29}{5}$, all entering $d_2$-differentials originate in the range $s > \frac{1}{5} (t-s) + \frac{18}{5}$. It therefore follows from \Cref{thm:Miller-thm}, \Cref{prop:localized-alg-Nov}, \Cref{prop:CESS-collapse} and \Cref{cor:alg-nov-range} that at most the elements $v_1 ^i$ and $v_1 ^i h_{2,0}$ survive to the $\mathrm{E}_3$-page of the spectral sequence in this range. These elements do in fact survive to represent nonzero elements of the $\mathrm{E}_\infty$-page by \Cref{prop:v1-self-map}, \Cref{lemm:h21-lives}, and \Cref{cor:K1-homotopy-Y}. It follows that, for $s > \frac{1}{5} (t-s) + \frac{29}{5}$, the spectral sequence collapses at the $\mathrm{E}_3$-page, and the $v_1 ^i$ and $v_1^i h_{2,0}$ are all of the nonzero classes on the $\mathrm{E}_3$-page.

    We may in fact extend the above description to the range $s > \frac{1}{5} (t-s) + \frac{12}{5}$ as follows. Since $v_1$ lifts to a self-map of $Y$ by \Cref{prop:v1-self-map}, multiplying by $v_1$ commutes with differentials in the $\HFt$-Adams spectral sequence. Now, it follows from \Cref{lemm:vanishing-results} that multiplication by $v_1$ induces an isomorphism on $\im(d_2)$ for any $d_2$ with target in the range $s > \frac{1}{5}(t-s) + \frac{12}{5}$, hence source in the range $s > \frac{1}{5} (t-s) + \frac{1}{5}$. This is because the source lies in the $v_1$-surjectivity region and the target lies in the $v_1$-periodic region. It follows that the description of the spectral sequence appearing in the previous paragraph applies in fact to the range $s > \frac{1}{5} (t-s) + \frac{12}{5}$.

    We conclude that the only classes in $\pi_k (Y)$ detected in Adams filtration at least $\frac{1}{5} k + \frac{13}{5}$ are $v_1 ^i$ and $v_1 ^i w_1$. By \Cref{cor:K1-homotopy-Y}, these classes map isomorphically to the homotopy of $L_{K(1)} Y$. Thus to check that \[F^{\frac{1}{5} k + \frac{13}{5}} \pi_k (Y) \to \pi_k (L_{K(1)} Y)\] is an isomorphism, it suffices to check that the classes $v_1 ^i$ and $v_1 ^i h_{2,1}$ are in the range $s \geq \frac{1}{5}(t-s) + \frac{13}{5}$. A short calculation shows that this happens when $i \geq 5$, hence when the total degree is at least $15$.
\end{proof}

\begin{proof}[Proof of \Cref{thm:Y-v1-band}]
    By \Cref{prop:banded-adams}, we see that there are two things left to check beyond \Cref{prop:Y-main-prop}. The first is that $\Ext_{\A_{*}} ^{s,t} (\F_2, \mathrm{H}_* (Y)) = 0$ for $s > \frac{1}{2} (t-s)$, which follows from the computation \[\Ext_{\A(1)_*} ^{s,t} (\F_2, \mathrm{H}_*(Y)) \cong \Ext_{\F_2 [\overline{\zeta}_2]/(\overline{\zeta}_2 ^2)} (\F_2, \F_2) \cong \F_2 [v_1]\] and \cite[Proposition 3.2]{Miller-Wilkerson}. The second is that the classes $v_1 ^i$ and $v_1^i h_{2,1}$ lie in the region $s \geq \frac{1}{2} (t-s) - \frac{3}{2}$, which is easily verified.
\end{proof}

Finally, we note down a proof of the telescope conjecture at chromatic height $1$ and the prime $2$, based on \Cref{thm:Y-v1-band}. It is similar to Miller's proof at an odd prime \cite{Miller} and different from the $2$-primary proof of Mahowald \cite{MahJEHP}, which uses $bo$-resolutions. We begin with the following proposition.

\begin{prop}\label{prop:v1-band-tele}
    Let $X$ be a type $1$ finite spectrum\footnote{A finite spectrum $X$ is said to be type $1$ if $\mathrm{H}_* (X;\mathbb{Q}) = 0$ and $K(1)_* (X) \neq 0$.} whose $\HFp$-Adams spectral sequence admits a $v_1$-banded vanishing line with parameters $(b \leq d, w, m, c, r)$, and suppose $v : \Sigma^{n(2p-2)} X \to X$ is a $v_1$-self map of $\HFp$-Adams filtration $n$. Then the map \[v^{-1} \pi_{*} X \to \pi_* (L_{K(1)} X) \] is an isomorphism.
\end{prop}

\begin{proof}
    Inverting $v$ in the $\HFp$-Adams spectral sequence gives rise to the $v$-periodic $\HFp$-Adams spectral sequence, which converges to $v^{-1} \pi_{*} (X)$ by \cite[Theorem 2.13]{TripleLoop}. This theorem applies by the assumption on the Adams filtration of $v$, as well as the fact that $\nu_{\HFp} (X)$ has a finite-page vanishing line of slope $\frac{1}{2p-2}$ by definition of $v_1$-banded vanishing line.




    By the assumption on the $\HFp$-Adams filtration of $v$, $\displaystyle\bigoplus_{k} F^{m k + c} \pi_k (X)$ is a $\Z[v]$-submodule of $\pi_k (X)$, so that we have a factorization \[F^{mk+c} \pi_k (X) \to v^{-1} F^{mk +c} \pi_k (X) \to v^{-1} \pi_k (X) \to \pi_k (L_{K(1)} X).\]

    Since both $v^{-1} \pi_k (X)$ and $\pi_k (L_{K(1)} X)$ are $v$-periodic, it suffices to show that \[v^{-1} \pi_k (X) \to \pi_k (L_{K(1)} X)\] is an equivalence for $k \gg 0$. By the assumption that the $\HFp$-Adams spectral sequence of $X$ admits a $v_1$-banded vanishing line, the map \[F^{mk +c} \pi_k (X) \to \pi_k (L_{K(1)} X)\] is an equivalence for $k \geq w$.
    This implies that \[F^{mk +c} \pi_k (X) \to v^{-1} F^{mk +c} \pi_k (X)\] is an equivalence for $k \geq w$. Therefore it suffices to show that \[v^{-1} F^{mk+c} \pi_k (X) \to v^{-1} \pi_k (X)\] is an equivalence, which follows from the fact that $v$ acts nilpotently on \[\pi_k(X) / (F^{mk+c} \pi_k (X)),\] since $m < \frac{1}{2p-2}$.
\end{proof}


\begin{cor}[Telescope Conjecture at height $1$ and the prime $2$] \label{cor:telescope}
	Suppose that the prime is $2$. Then the Bousfield classes of $K(1)$ and $v^{-1} X$ are equal for any type $1$ spectrum $X$ with $v_1$-self map $v : \Sigma^{n(2p-2)} X \to X$.
\end{cor}

\begin{proof}
    Since $v_1 : \Sigma^2 Y \to Y$ has $\HFt$-Adams filtration one, \Cref{thm:Y-v1-band} and \Cref{prop:v1-band-tele} imply that $v ^{-1} Y \to L_{K(1)} Y$ is an equivalence, so this follows as in \cite[Section 4]{BousfieldLocalization} and the proof of \cite[Theorem 10.12]{RavenelConjs}.
\end{proof}

%% file: Mod8.tex
Our main goal in this section is to prove \Cref{thm:mod8-main-thm}, which was a key input to our proof of \Cref{thm:mainboundary} in \Cref{sec:FinishingStolz}.

\begin{thm}\label{thm:mod8-main-thm}
    Let $F^s \pi_k (C(8)) \subseteq \pi_k (C(8))$ denote the elements of $\HFt$-Adams filtration at least $s$. Then for $k \geq 126$, the image of the map \[F^{\frac{1}{5} k + 15} \pi_k (C(8)) \to \pi_{k-1} (\Ss)\] is contained in the subgroup of $\pi_{k-1} (\Ss)$ generated by the image of $J$ and the $\mu$-family.

\end{thm}

We will prove \Cref{thm:mod8-main-thm} by combining the banded genericity techonology of \Cref{sec:BandVan} with the main result of \Cref{sec:Y}. Before we explain further, let us fix some notation.

\begin{cnv}
    In this section, synthetic spectra will always be taken with respect to $\HFt$, and we will denote $\Sigma^{*,*}\nu_{\HFt}( \Ss_2^\wedge )$ by $\Ss_2^{*,*}$. Similarly, we will let $\Ss_2$ denote the $2$-complete sphere.
\end{cnv}


\begin{ntn}
By the calculations of \Cref{prop:syn-toda-range}, we see that that there are classes $\wt{2} \in \pi_{0,1} \Ss_2^{0,0}$, $\wt{\eta} \in \pi_{1,2} \Ss_2^{0,0}$ and $\wt{\nu} \in \pi_{3,4} \Ss_2^{0,0}$ which satisfy relations $\tau \wt{2} = \nu (2) = 2$, $\tau \wt{\eta} = \nu (\eta)$ and $\tau \wt{\nu} = \nu (\nu)$. Moreover, we let $\wt{2^n} = \wt{2}^n$.
\end{ntn}

\begin{lem} \label{fact:2-complete-fine}
    The natural map \[ [\Ss_2^{a,b}, \Ss_2 ^{0,0}] \to \pi_{a,b} (\Ss_2 ^{0,0})\] is an isomorphism for all $a,b$.
  Furthermore, for any $n \geq 0$ there is an equivalence
    \[\nu C(2^n) \simeq \mathrm{cof}(\Ss_2^{0,1} \xrightarrow{\tau^{n-1} \wt{2}^n} \Ss_2^{0,0}). \]
\end{lem}

\begin{proof}
  The first claim follows from \cite[Proposition 5.6]{Pstragowski}, which implies that $\Ss^{0,0} _2$ is the $\nu \HFt$-localization of $\Ss^{0,0}$. 

  To prove the second claim, we note that the cofiber sequence $ \Ss_2^0 \to C(2^n) \to \Ss_2^1 $
    is short exact on $\HFt$-homology, so by \Cref{lemm:syn-cof} it induces a cofiber sequence $\Ss_2 ^{0,0} \to \nu C(2^n) \to \Ss_2 ^{1,1}$.
    Thus $\nu C(2^n)$ is the cofiber of a map $\Ss_2 ^{0,1} \to \Ss_2 ^{0,0}$ whose image under the functor $\tau^{-1}$ is $2^n$.
    The result therefore follows from the fact that $\pi_{0,*} (\Ss_2)$ is $\tau$-torsion free.
\end{proof}

\begin{ntn}
  For convenience, we will use the following notation:
  $$ C(\tau^a \wt{2}^b) \coloneqq \mathrm{Cof}(\Ss_2^{0,b-a} \xrightarrow{\tau^a \wt{2}^b} \Ss_2^{0,0}). $$
\end{ntn}

\begin{rmk}\label{rmk:modified-mod8}
    The synthetic spectrum $C(\wt{2}^n)$ encodes the modified $\HFt$-Adams spectral sequence for $C(2^n)$. See \cite[Section 3]{BHHM} for the notion of a modified Adams spectral sequence. \todo{Is there a better reference to cite here?}
\end{rmk}

Our next goal will be to establish a $v_1$-banded vanishing line of slope $\frac{1}{5}$ for $C(\wt{8})$ with explicit parameters. We will do this via a thick subcategory argument.


\begin{lem}\label{lemm:eta-splitting}
    There is a splitting of synthetic spectra
    \[C(\wt{2}) \otimes C(\wt{\eta}^3) \simeq C(\wt{2}) \oplus \Sigma^{4,6} C(\wt{2}).\]
\end{lem}

\begin{proof}
  After inverting $\tau$ this splitting becomes the classical fact that
  $ C(2) \otimes C(\eta^3) \simeq C(2) \oplus \Sigma^{4} C(2) $
  and the proof we give simply lifts this argument to the synthetic setting.
  
  The splitting follows two statements:
  that $\wt{\eta}^3 = \wt{4} \wt{\nu}$ as self-maps of $C(\wt{2})$ and
  that $\wt{4}$ is null as a self-map of $C(\wt{2})$.
  The first fact follows from \Cref{prop:syn-toda-range}, which shows that the relation $\wt{\eta}^3 = \wt{4} \wt{\nu}$ holds in the homotopy of $\Ss^{0,0} _2$.
  To prove that $\wt{4}$ is null as a self-map of $C(\wt{2})$, we examine the following commutative diagram
  
  \begin{center}
    \begin{tikzcd}
      \Ss^{0,0}_2 \arrow[r] \arrow[d, "\wt{2}"] & C(\wt{2}) \arrow[r] \arrow[d, "\wt{2}"] \arrow[dl, dashrightarrow] & \Ss^{1,1}_2 \arrow[d, "\wt{2}"] \\
      \Ss^{0,0}_2 \arrow[r] \arrow[d, "\wt{2}"] & C(\wt{2}) \arrow[r] \arrow[d, "\wt{2}"] & \Ss^{1,1}_2 \arrow[dl, dashrightarrow] \arrow[d, "\wt{2}"] \\
      \Ss^{0,0}_2 \arrow[r] & C(\wt{2}) \arrow[r] & \Ss^{1,1}_2 ,
    \end{tikzcd}
  \end{center}
  where the rows are cofiber sequences and the dashed arrows exist because of the canonical nullhomotopies of \[C(\wt{2}) \to \Ss^{1,1}_2 \xrightarrow{\wt{2}} \Ss^{1,1}_2 \text{ and } \Ss^{0,0}_2 \xrightarrow{\wt{2}} \Ss^{0,0}_2 \to C(\wt{2}).\]
    We wish to show that the composite of the middle two vertical arrows is null. Using the dashed arrows, we may factor this through the composition of the middle two horizontal arrows, which is null because they form a cofiber sequence.
\end{proof}

\begin{prop}\label{prop:banded-numbers}
  The synthetic spectra $X$ in the table below admit $v_1$-banded vanishing lines of slope $\frac{1}{5}$ and remaining parameters as follows:
  \renewcommand{\arraystretch}{1.6}
  \begin{center}
    \begin{tabular}{|c||c|c|c|c|c|}
      \hline
      \textup{Synthetic Spectrum X} & $b$ & $d$ & $v$ & $c$ & $r$ \\ \hline \hline
      $C(\wt{2}) \otimes C(\wt{\eta})$ &
      $-1.5$ & $0$ & $15$ & $2.6$ & $1$ \\ \hline
      $C(\wt{2}) \otimes C(\wt{\eta}^2)$ &
      $-2.5$ & $0.5$ & $23$ & $4.4$ & $2$ \\ \hline
      $C(\wt{2}) \otimes C(\wt{\eta}^3)$ &
      $-3.5$ & $1$ & $32 + \frac{1}{3}$ & $6.2$ & $4$ \\ \hline
      $C(\wt{2})$ &
      $-3.5$ & $1$ & $28 + \frac{1}{3}$ & $5$ & $4$ \\ \hline
      $C(\wt{4})$ &
      $-7.5$ & $2$ & $58 + \frac{1}{3}$ & $10$ & $9$ \\ \hline
      $C(\wt{8)}$ &
      $-12.5$ & $3$ & $91 + \frac{2}{3}$ & $15$ & $15$ \\ \hline
    \end{tabular}
  \end{center}
  \renewcommand{\arraystretch}{1.0}
\end{prop}

\begin{proof}
  Inductively apply \Cref{lemm:band-retract-susp} and \Cref{prop:band-in-cofiber-seqs} to the Bockstein cofiber sequences
  \begin{align*}
      \Sigma^{1,2} C(\wt{2}) \otimes C(\wt{\eta}) \to C(\wt{2}) &\otimes C(\wt{\eta}^2) \to C(\wt{2}) \otimes C(\wt{\eta}) \\
      \Sigma^{1,2} C(\wt{2}) \otimes C(\wt{\eta}^2) \to C(\wt{2}) &\otimes C(\wt{\eta}^3) \to C(\wt{2}) \otimes C(\wt{\eta}) \\
      \Sigma^{0,1} C(\wt{2}) \to &C(\wt{4}) \to C(\wt{2}) \\
      \Sigma^{0,1} C(\wt{4}) \to &C(\wt{8}) \to C(\wt{2})
  \end{align*}
    using \Cref{thm:Y-v1-band} as a base case and \Cref{lemm:eta-splitting} to go from the second to the third sequence.
\end{proof}

\begin{rmk}
    The numbers in \Cref{prop:banded-numbers} can likely be improved by more carefully accounting for the behavior of the classes in the band under the cofiber sequences used in the proof. In particular, we believe that one could improve the torsion order bound in the $v_1$-banded vanishing line for $C(\wt{2})$ to $3$. This would imply by \Cref{prop:banded-adams} that the $v_1$-localized Adams spectral sequence for $C(2)$ collapses at the $\mathrm{E}_5$-page. This result was announced by Mahowald \cite[Theorem 5]{MahBull}, but to the best of our knowledge a proof has never appeared in the literature.
\end{rmk}

\begin{prop} \label{prop:I-want-to-be-done}
  If $C(\wt{2}^n)$ admits a banded vanishing line with parameters $(m, c, r, b, d, v)$, then
  the image of the map
  $$ F^{mk+c}\pi_k C(2^n) \to \pi_{k-1}(\Ss) $$
  is contained in the subgroup of $\pi_{k-1}(\Ss)$ generated by the image of $J$ and the $\mu$-family as long as $ k \geq v$ and
  $$ \frac{1}{2}k + b - n + 1 \geq \frac{3}{10}(k-1) + 4 + v_2(k+1) + v_2(k). $$
  Recall that $v_2 (k)$ denotes the $2$-adic valuation of $k$.
\end{prop}

\begin{proof}
  First, we note that the conclusion holds trivially for $k \leq 1$.
  Next, using that $\pi_{k-1}(\Ss) \cong \pi_{k-1}(\Ss_2)$ for $k > 1$ and that the $\HFt$-Adams filtrations on each group agree, we may replace $\Ss$ in the theorem statement by $\Ss_2$. 

  Consider the diagram below, where each row is a cofiber sequence and the middle and right vertical maps are projection onto the top cell:
  \begin{center}
    \begin{tikzcd}
      \Sigma^{0,n}C\tau^{n-1} \ar[r] \ar[d, equal] &
      C(\tau^{n-1} \wt{2}^n) \ar[r] \ar[d] &
      C(\wt{2}^n) \ar[d] \\
      \Sigma^{0,n}C\tau^n \ar[r] &
      \Ss_2^{1, 1} \ar[r, "\tau^n"] &
      \Ss_2^{1, n} 
    \end{tikzcd}
  \end{center}
  By \Cref{fact:2-complete-fine}, there is an equivalence $ \nu C(2^n) \simeq C(\tau^{n-1} \wt{2}^n) $.
  Note that under $\tau^{-1}$ the map $\nu C(2^n) \to \Ss^{1,n}$ becomes projection onto the top cell.
  This implies that projection to the top cell induces maps
  $$ F^{s} \pi_k( C(2^n)) \to F^s \pi_k (\tau^{-1} C(\wt{2}^n)) \to F^{s-n+1} \pi_{k-1} \Ss, $$
  where $ F^s \pi_k(\tau^{-1} C(\wt{2}^n)) $ is as in \Cref{dfn:synth-filt}.\footnote{ \Cref{rmk:modified-mod8} allows us to identify this filtration with the \emph{modified} $\HFt$-Adams filtration.}
  We finish by using the hypothesis that $C(\wt{2}^n)$ has a banded vanishing line.
  It follows that, for $k \geq v$, the maps induced by projection to the top cell factor as
  \begin{align*}
    F^{mk + c} \pi_k (C(2^n))
    &\to F^{mk + c} \pi_k (\tau^{-1} C(\wt{2}^n) )
    = F^{\frac{1}{2}k + b} \pi_k (\tau^{-1} C(\wt{2}^n) )\\
    &\to F^{\frac{1}{2}k + b - n + 1} \pi_{k-1} \Ss_2.
  \end{align*}
  
  It therefore suffices to find a $k \geq v$ large enough so that every element of $\pi_{k-1} \Ss$ which has $\HFt$-Adams filtration at least $\frac{1}{2}k + b - n + 1$ is in the subgroup generated by the image of $J$ and the $\mu$-family.
  
  \Cref{thm:gamma-upper}(1) states that every element in $\pi_{k-1} \Ss$ which has $\HFt$-Adams filtration at least $\frac{3}{10}(k-1)+4+v_2 (k+1) + v_2 (k)$ is in the subgroup generated by the image of $J$ and the $\mu$-family. The result follows.
\end{proof}

\begin{proof}[Proof of \Cref{thm:mod8-main-thm}]
  Using Propositions \ref{prop:banded-numbers} and \ref{prop:I-want-to-be-done}, it will suffice to show that the following inequality holds for all $k \geq 126$:
  $$ \frac{1}{2}k - 14.5 \geq \frac{3}{10} (k-1) + 4 + v_2 (k+1) + v_2 (k). $$
  Rearranging, clearing denominators and applying the bound
  $$ \log_2(k+1) \geq v_2(k+1)+v_2(k), $$
  we find that it suffices to show that
  $$ k \geq 91 + 5\log_2(k+1). $$
  Taking derivatives, we find that the left hand side increases faster than the right hand side as soon as $k \geq 9$.
  Thus, to show the inequality holds for $k \geq 126$ it suffices to note that
  \[ 126 \geq 91 + 5 \log_2 (127) \approx 125.94. \qedhere\]


\end{proof}

%% file: SynRevHom.tex
In this appendix, we provide the technical details of the proof of \Cref{thm:synthetic-Adams}, as well as a computation of the $\HFt$-synthetic bigraded homotopy groups in the Toda range. The computation of synthetic homotopy groups highlights many of the subtleties within the statement of \Cref{thm:synthetic-Adams}. We have tried to make this appendix as self-contained as possible. 
Understanding the techniques introduced in this appendix is not necessary in order to read the remainder of the paper. For convenience, we recall the statement of \Cref{thm:synthetic-Adams}.

\begin{thm}[\Cref{thm:synthetic-Adams}]
  Let $X$ denote an $E$-nilpotent complete spectrum with strongly convergent $E$-based Adams spectral sequence.
  Then we have the following description of the bigraded homotopy groups of $\nu X$.

  Let $x$ denote a class in topological degree $k$ and filtration $s$ of the $\mathrm{E}_2$-page of the $E$-based Adams spectral sequence for $X$. The following are equivalent:
  \begin{itemize}
  \item[(1a)] Each of the differentials $d_2$,...,$d_r$ vanish on $x$.
  \item[(1b)] $x$, viewed as an element of $\pi_{k,k+s}(C\tau \otimes \nu X)$, lifts to $\pi_{k,k+s}(C\tau^r \otimes \nu X)$.
  \item[(1c)] $x$ admits a lift to $\pi_{k,k+s}(C\tau^r \otimes \nu X)$ whose image under the $\tau$-Bockstein
  $$ C\tau^r \otimes \nu X \to \Sigma^{1,-r} C\tau \otimes \nu X $$
  is equal to $-d_{r+1}(x)$.
  \end{itemize}

  If we moreover assume that $x$ is a permanent cycle, then there exists a (not necessarily unique) lift of $x$ along the map
  $\pi_{k,k+s}(\nu X) \to \pi_{k,k+s}(C\tau \otimes \nu X)$.  For any such lift, $\wt{x}$, the following statements are true:
  \begin{itemize}
  \item[(2a)] If $x$ survives to the $\mathrm{E}_{r+1}$-page, then $\tau^{r-1} \wt{x} \neq 0$.
  \item[(2b)] If $x$ survives to the $\mathrm{E}_\infty$-page, then the image of $\wt{x}$ in $\pi_{k} (X)$ is of $E$-Adams filtration $s$ and detected by $x$ in the $E$-based Adams spectral sequence.
  \end{itemize}
  Furthermore, there always exists a choice of lift $\wt{x}$ satisfying additional properties:
  \begin{itemize}
  \item[(3a)] If $x$ is the target of a $d_{r+1}$-differential, then we may choose $\wt{x}$ so that $\tau^r \wt{x} = 0$.
  \item[(3b)] If $x$ survives to the $\mathrm{E}_\infty$-page, and $\alpha \in \pi_k X$ is detected by $x$, then we may choose $\wt{x}$ so that
    $\tau^{-1} \wt{x} = \alpha$. In this case we will often write $\wt{\alpha}$ for $\wt{x}$.
  \end{itemize}
  Finally, the following generation statement holds:
  \begin{itemize}
  \item[(4)] Fix any collection of $\wt{x}$ (not necessarily chosen according to (3)) such that the $x$ span the permanent cycles in topological degree $k$.  Then the $\tau$-adic completion of the $\Z[\tau]$-submodule of $\pi_{k,*}(\nu X)$ generated by those $\wt{x}$ is equal to $\pi_{k,*} (\nu X)$.
  \end{itemize}
\end{thm}

\begin{rmk}
  As a foreward to the proof, we provide some commentary on the provenance of \Cref{thm:synthetic-Adams}.
  Through the equivalence between $(\Syn_{\BP}^{\mathrm{ev}})_p$ and $\mathcal{SH}(\mathbb{C})_p^{\mathrm{cell}}$ (see \cite[Theorem 1.4]{Pstragowski}) this theorem specializes to provide a translation between the $p$-complete, bigraded motivic stable stems over $\mathbb{C}$ and the Adams--Novikov spectral sequence.
  In fact, the proof we give was directly inspired by the literature on motivic stable stems and their connection to the Adams--Novikov spectral sequence.

  The core argument originates in \cite[Lemma 15]{HKO} where is it observed that at the prime $2$ the differentials in the motivic Adams--Novikov spectral sequence can be formally deduced from those in the classical Adams--Novikov spectral sequence.\footnote{See \cite[Proposition 5.6]{Stahn} for a version at odd primes.} The corresponding result in our setting is \Cref{thm:synth-ASS} where we identify the $\nu E$-based Adams spectral sequence for $\nu X$ in terms of the $E$-based Adams spectral sequence for $X$.

  Building on knowledge from extensive calculations of motivic stable stems, Isaksen then recognized that not only do these spectral sequences determine one another, but the motivic stable stems and Adams--Novikov spectral sequence (with its hidden extensions) contain essentially equivalent information (see the introduction to \cite[Chapter 6]{StableStems}). The remainder of the proof of \Cref{thm:synthetic-Adams} involves formalizing these ideas and dealing with questions of convergence.
  \todo{I'm not too happy with this remark, examine again with fresh eyes.}
\end{rmk}


%% file: SynRevAdams.tex
This subsection is generally organized in order of increasing strength of hypotheses and some results are proved in greater generality than stated in \Cref{thm:synthetic-Adams}. Before we begin we will need to recall more material from \cite{Pstragowski}.

\begin{rec}
  In \cite[Section 4.2, Lemma 4.29 and Proposition 4.35]{Pstragowski},
  Pstragowski introduces a $t$-structure on $\Syn_E$ which satisfies the following properties:
  \begin{enumerate}
  \item The heart, $\mathrm{Syn}_E ^{\heartsuit}$, is equivalant to the abelian category of $E_* E$-comodules.
  \item Given a spectrum $X$, the $0^{\mathrm{th}}$ homotopy object of $\nu X$ with respect to this $t$-structure is naturally equivalent to $E_*X$ as an $E_*E$-comodule and naturally equivalent to $C\tau \otimes \nu(X)$ as an object of $\Syn_E$.
  \item If we let $Y(-)$ denote the right adjoint to inverting $\tau$,
    then we have a natural equivalence between the connective cover of $Y(X)$ and $\nu X$.  
  \item This $t$-strucutre is right complete and compatible with filtered colimits.
  \end{enumerate}
\end{rec}

\begin{ntn}
  In view of (2) we will refer to the $t$-structure introduced above as the \emph{homological $t$-structure} on $\Syn_E$.
  For notational brevity we will use subscripts to denote truncation with respect to this $t$-structure, so that $A_{\geq n}$ refers to the $n$-connective cover of a synthetic spectrum $A$ in the homological $t$-structure.  
\end{ntn}

\begin{cnv}
For the remainder of this subsection $X$ will denote a spectrum.
\end{cnv}
Our analysis of the relation between the bigraded homotopy groups of $\nu X$ and the $E$-based Adams spectral sequence for $X$ will hinge on an understanding of the $\nu E$-based Adams spectral sequence for $\nu X$. We begin by considering the canonical $E$-Adams tower for $X$, constructed below:

\begin{center}
  \begin{tikzcd}
    \cdots \ar[r] &
    X_2 \ar[r, "f_2"] \ar[d, "i_2"] &
    X_1 \ar[r, "f_1"] \ar[d, "i_1"] &
    X_0 \ar[r, equal] \ar[d, "i_0"] &
    X. \\
    & E \otimes X_2 &
    E \otimes X_1 &
    E \otimes X_0 &
  \end{tikzcd}
\end{center}
Each of the $f_i$ is zero on $E$-homology. Therefore, using \Cref{rmk:strong-monoidal} and \Cref{lemm:syn-cof}, we may identify the canonical $\nu E$-Adams tower of $\nu X$ as
\begin{center}
  \begin{tikzcd}
    &
    \Sigma^{0,2} \nu X_3 \ar[d, "\nu(f_3)"] \ar[dl, "\tau"'] &
    \Sigma^{0,1} \nu X_2 \ar[d, "\nu(f_2)"] \ar[dl, "\tau"'] &
    \nu X_1 \ar[d, "\nu(f_1)"] \ar[dl, "\tau"'] & \\
    \cdots \ar[r] &
    \Sigma^{0,2} \nu X_2 \ar[r, "\wt{f_2}"] \ar[d, "\nu(i_2)"] &
    \Sigma^{0,1} \nu X_1 \ar[r, "\wt{f_1}"] \ar[d, "\nu(i_1)"] &
    \nu X_0 \ar[r, equal] \ar[d, "\nu(i_0)"] &
    \nu X. \\
    &
    \nu E \otimes \Sigma^{0,2} \nu X_2 &
    \nu E \otimes \Sigma^{0,1} \nu X_1 &
    \nu E \otimes \nu X_0 &
  \end{tikzcd}
\end{center}


\begin{ntn}
The above tower gives rise to a spectral sequence
 $$ \mathrm{E}^{s,k,w}_1 \coloneqq \pi_{k, k+w} \left( \nu E \otimes \Sigma^{0,s} \nu X_s \right) \Longrightarrow \pi_{k, k+w}(\nu X), $$
 with differentials of tridegree $(r, -1, 1)$. Note that multiplication by $\tau$ lowers the $w$ grading by $1$ but preserves the $s$ and $k$ gradings.
  We use the notation $\mathrm{E}^{s,k,w}_r$ for page $r$ of this spectral sequence.

  Analogously, we use $\mathrm{E}^{s,k}_r$ to refer to the groups in the $E$-Adams spectral sequence for $X$
 $$ \mathrm{E}^{s,k}_1 \coloneqq \pi_k \left( E \otimes X_s \right) \Longrightarrow \pi_k(X), $$
with differentials of bidegree $(r, -1)$. \footnote{Our grading choices do not agree with the usual conventions for Adams spectral sequences. However, we prefer them because each of the indices has a clear interpretation: $k$ is the topological degree, $w$ is the weight and $s$ is the filtration.}
\end{ntn}
 Note that inverting $\tau$ determines a map of spectral sequences
 $$ \mathrm{E}_r^{s,k,w} \to \mathrm{E}_r^{s,k}. $$

 \begin{ntn}
 Let $\mathrm{B}_r^{s,k}$ denote the subgroup of $\mathrm{E}_2^{s,k}$ generated by the images of the differentials $d_2$ through $d_r$. Let $\mathrm{Z}_r^{s,k}$ denote the larger subgroup of $\mathrm{E}_2^{s,k}$ given by those classes on which $d_2$ through $d_r$ vanish. Then, $\mathrm{E}_{r+1}^{s,k} \cong \mathrm{Z}_r^{s,k}/\mathrm{B}_r^{s,k}$.
 \end{ntn}

\begin{thm} \label{thm:synth-ASS}
    The $\nu E$-based Adams spectral sequence for $\nu X$ is determined by the $E$-based Adams spectral sequence for $X$ in the following way:
    \begin{enumerate}
    \item $ \mathrm{E}^{s,k,w}_1 \cong \mathrm{E}^{s,k}_1 \otimes \ZZ[\tau]$,
      where $\mathrm{E}^{s,k}_1$ is considered to be in tridegree $(s,k,s)$.
    \item $ \mathrm{E}^{s,k,w}_2 \cong \mathrm{E}^{s,k}_2 \otimes \ZZ[\tau]$,
      where $\mathrm{E}^{s,k}_2$ is considered to be in tridegree $(s,k,s)$.
    \item Given a differential $d_{r, \mathrm{top}} (x) = y$, there is a differential $d_r (x) = \tau^{r-1} y$. Moreover, all differentials arise in this way.
    \end{enumerate}
\end{thm}

\begin{proof}
  The proof is very similar to that of \cite[Lemma 15]{HKO}.
  Statement (1) follows from \cite[Proposition 4.21]{Pstragowski}.
  Statement (2) follows from statement (3).
  We now prove statement (3) by induction. Suppose that we have proved the statement through the $\mathrm{E}_r$-page.  To prove it for the $\mathrm{E}_{r+1}$-page, we calculate the differential
  $$ d_r : \mathrm{E}^{s,k,w}_r \to \mathrm{E}_r^{s+r,k-1,w+1}. $$
  Note that, by the inductive hypothesis, the $\mathrm{E}_r$-page in every tridegree which can be the target of a $d_r$ differential consists of $\tau$-torsion free elements.  On the other hand, upon inverting $\tau$ we must obtain the differential
$$\tau^{-1} d_r = d_{r,\mathrm{top}}: \mathrm{E}^{s,k}_{r} \to \mathrm{E}^{s+r,k-1}_{r},$$
which determines the $d_r$ differential by the above.
\end{proof}

As a corollary of this description of the $\nu E$-Adams spectral sequence, we obtain the following more explicit statement.

\begin{cor} \label{cor:synth-ASS}
  For $2 \leq r \leq \infty$, there are natural isomorphisms:
  \begin{enumerate}
  \item $\mathrm{E}_r^{s,k,w} \cong 0$ for $w > s$.
  \item $\mathrm{E}_r^{s,k,w} \cong \mathrm{Z}_{r-1}^{s,k}$ for $w = s$.
  \item $\mathrm{E}_r^{s,k,w} \cong \mathrm{Z}_{r-1}^{s,k}/\mathrm{B}_{s-w+1}^{s,k}$ for $s-r+1 < w < s$.
  \item $\mathrm{E}_r^{s,k,w} \cong \mathrm{E}_r^{s,k}$ for $w \leq s-r+1$ and $w \leq 0$.
  \end{enumerate}
    In particular, the map $\mathrm{E}_r ^{s,k,w} \to \mathrm{E}_r ^{s,k,w-1}$ induced by multiplication by $\tau$ is surjective for $w \leq s$.
\end{cor}

Our next order of business will be to determine the $\nu E$-based Adams spectral sequence for $C \tau^p \otimes \nu X$.

\begin{ntn}
    We use the notation ${}^p \mathrm{E}_r ^{s,k,w}$ to denote the groups on page $r$ of the $\nu E$-based Adams spectral sequence
    \[ {}^p\mathrm{E}_1^{s,k,w} \coloneqq \pi_{k,k+w}( \nu E \otimes C\tau^p \otimes \Sigma^{0,s}\nu X_s ) \Longrightarrow \pi_{k,k+w} (\nu X),\]
  and similarly for the later pages.
\end{ntn}

%
\begin{cor} \label{cor:synth-ctau-ASS}
  For $p \geq 1$ and $2 \leq r \leq \infty$, there are natural isomorphisms:
  \begin{enumerate}
  \item ${}^p\mathrm{E}_r^{s,k,w} \cong 0$ for $w > s$.
  \item ${}^p\mathrm{E}_\infty^{s,k,w} \cong \mathrm{Z}_{p-s+w}^{s,k}/\mathrm{B}_{s-w+1}^{s,k}$ for $s \geq w > s-p$.
  \item ${}^p\mathrm{E}_r^{s,k,w} \cong 0$ for $s-p \geq w$.
  \end{enumerate}
\end{cor}

\begin{proof}
    This follows from considering the map of $\nu E$-based Adams spectral sequences induced by the map $\nu X \to C\tau^p \otimes \nu X$. \todo{Maybe elaborate on this.}
\end{proof}

In order to use the theorem and corollaries we have just proved we will need to make a digression and discuss completeness and convergence.

\begin{dfn}
  We say that a synthetic spectrum $A$ is \emph{$\tau$-complete} if the $\tau$-Bockstein tower of $A$ is convergent: that is, if the canonical map
  \[A \to \varprojlim_n C\tau^n \otimes A\] is an equivalence.
\end{dfn}

\begin{prop} \label{lemm:Elocal-taucomplete} \label{lemm:easy-e-compton}
  The following are equivalent:
  \begin{enumerate}
  \item $X$ is $E$-nilpotent complete.
  \item $\nu X$ is $\nu E$-nilpotent complete.
  \item $\nu X$ is $\tau$-complete.
  \end{enumerate}
\end{prop}

The proof of \Cref{lemm:Elocal-taucomplete} will rely on the following two lemmas.

\begin{lem} \label{lemm:comp1}
  The synthetic spectrum $C\tau^p \otimes \nu X$ is $\nu E$-nilpotent complete.
\end{lem}

\begin{proof}
    By induction on $p$ via the the Bockstein sequences
    \[\Sigma^{0,-1} C\tau^{p-1} \otimes \nu X \to C \tau^p \otimes \nu X \to C \tau \otimes \nu X,\]
  we see that it suffices to prove the lemma for $p=1$.
  Tensoring the canonical Adams resolution for $\nu X$ with $C\tau$, we obtain an Adams resolution of $C\tau \otimes \nu X$.
  Using \cite[Lemma 4.29]{Pstragowski} repeatedly, we learn that $C\tau \otimes \Sigma^{0,s}\nu X$ is $(-s)$-coconnective in the homological $t$-structure. Thus, the inverse limit of this Adams resolution for $\nu X$ is trivial because this $t$-structure is right complete.
\end{proof}

\begin{lem} \label{lemm:comp2}
  The synthetic spectrum $\nu E \otimes \nu X$ is $\tau$-complete.
\end{lem}

\begin{proof}
  In order to show that $\nu E \otimes \nu X$ is $\tau$-complete we will show that the inverse limit under iterated multiplication by $\tau$ is trivial.
  From \Cref{rmk:dualizable-generators} it suffices to check triviality on maps in from suspensions of finite projectives.
  Pick a finite $E_*$-projective $P$ and an integer $k$. Using \Cref{rmk:strong-monoidal} and the dualizability statement from \Cref{rmk:dualizable-generators}, we obtain an equivalence
  $$ \Hom(\Sigma^k \nu P, \varprojlim_\tau \Sigma^{0,-s}\nu E \otimes \nu X) \simeq \varprojlim_\tau \Hom(\Ss^{k,s}, \nu (E \otimes DP \otimes X)). $$
  Note that the spaces in the inverse limit on the right hand side are each $(s-k)$-connective by \cite[Proposition 4.21]{Pstragowski}. Therefore as $s \to \infty$ the right hand side becomes infinitely connective and thereby trivial.
\end{proof}

\begin{proof}[Proof of \Cref{lemm:easy-e-compton}]
  First we show that $E$-nilpotent completeness is equivalent to $\nu E$-nilpotent completeness.
  Using \cite[Lemma 4.29 and Proposition 4.35]{Pstragowski}, we may rewrite the canonical $\nu E$-Adams tower for $\nu X$ as
  $$ \cdots \to Y(X_2)_{\geq -2} \xrightarrow{\wt{f_2}} Y(X_1)_{\geq -1} \xrightarrow{\wt{f_1}} Y(X_0)_{\geq 0},$$
	where
	$\cdots \to X_2 \to X_1 \to X_0$
	is the canonical Adams tower for $X$.
  Inverting $\tau$ on the $\nu E$-Adams tower recovers the image of the $E$-Adams tower under $Y(-)$, and there is a fiber sequence
  $$ \left( \varprojlim_s Y(X_s)_{\geq -s} \right) \to \left( \varprojlim_s Y(X_s) \right) \to \left( \varprojlim_s Y(X_s)_{\leq -s} \right). $$
  The right hand term vanishes because the homological $t$-structure is right complete \cite[Proposition 4.16]{Pstragowski}. Furthermore, since $Y$ is a right adjoint, we may pull the inverse limit inside the functor $Y$ in the middle term. Thus, we obtain an equivalence
  $$ \left( \varprojlim_s Y(X_s)_{\geq -s} \right) \simeq Y \left( \varprojlim_s X_s \right).$$
  From the fact that $Y$ is fully faithful we conclude that the left hand side vanishes if and only if the inverse limit of the $E$-Adams tower for $X$ vanishes.
  
  Next we show that $\nu E$-nilpotent completeness is equivalent to $\tau$-completeness. Consider the following diagram:
$$  \begin{tikzcd}[column sep=small]
    \varprojlim_s \varprojlim_\tau \Sigma^{0,s-p} \nu X_s \ar[r] & \cdots \ar[r] &
    \varprojlim_\tau \Sigma^{0,1-p} \nu X_1 \ar[r, "\wt{f_1}"] \ar[d, "\nu(i_1)"] &
    \varprojlim_\tau \Sigma^{0,-p}\nu X \ar[d, "\nu(i_0)"] \\
    & &
    \varprojlim_\tau \Sigma^{0,1-p} (\nu E \otimes \nu X_1) &
    \varprojlim_\tau \Sigma^{0,-p} (\nu E \otimes \nu X). &
			\end{tikzcd} $$
Here, the limits over $\tau$ refer to limits, as $p$ varies, under multiplication by $\tau$ maps.
Using \Cref{lemm:comp2}, each object on the second row vanishes. We obtain an equivalence
$$ \varprojlim_s \varprojlim_\tau \Sigma^{0,s-p}\nu X_s \simeq \varprojlim_\tau \Sigma^{0,-p}X. $$
Dually, using \Cref{lemm:comp1}, we learn that
$$ \varprojlim_\tau \varprojlim_s \Sigma^{0,s-p}\nu X_s \simeq \varprojlim_s \Sigma^{0,s}\nu X_s. $$
Together these equalities finish the proof.
\end{proof}




We are now ready to prove the first part of \Cref{thm:synthetic-Adams}.

\begin{proof}[Proof of \Cref{thm:synthetic-Adams}(1)]
  Since $X$ is $E$-nilpotent complete, it follows from \Cref{lemm:easy-e-compton} that $C\tau^r \otimes \nu X$ is $\nu E$-nilpotent complete and $\tau$-complete. From \Cref{cor:synth-ctau-ASS} we can read off that the $\nu E$-based Adams spectral sequence for $C\tau^r \otimes \nu X$ converges strongly. Further, we can directly read off that (1a) and (1b) are equivalent. Clearly (1c) implies (1b). We will now prove (1c) assuming (1b).  If $d_{r+1}(x) = 0$, then we may finish by (1a), so we assume otherwise.

  We will prove (1c) by working directly with the cofiber sequence of Adams towers associated to the relevant Bockstein sequence. Before we begin, we fix some notation,
  $$ {}^r\mathrm{D}^{s,k,w} \coloneqq \pi_{k,k+w}( C\tau^r \otimes \Sigma^{0,s}\nu X_s ). $$
	Now, consider the following diagram of exact sequences, obtained by applying homotopy groups to the smash product of the two extended cofiber sequences
\[\Sigma^{-1,0} C\tau \to C\tau^{r+1} \to C\tau^{r} \to C\tau, \text{ and}\]
\[\Sigma^{0,s+1} \nu X_{s+1} \to \Sigma^{0,s} \nu X_{s} \to \nu E \otimes \Sigma^{0,s} \nu X_{s} \to \Sigma^{1,s+1} \nu X_{s+1}.\]
  \begin{center}
    \begin{tikzcd}
      {}^1\mathrm{D}^{s+1,k,s+r} \ar[r] \ar[d, "f"] &
      {}^{r+1}\mathrm{D}^{s+1,k,s} \ar[r] \ar[d, "f"] &
      {}^r\mathrm{D}^{s+1,k,s} \ar[r] \ar[d, "f"] &
      {}^1\mathrm{D}^{s+1,k-1,s+r+1} \ar[d, "f"] \\
      {}^1\mathrm{D}^{s,k,s+r} \ar[r] \ar[d, "i"] &
      {}^{r+1}\mathrm{D}^{s,k,s} \ar[r] \ar[d, "i"] &
      {}^r\mathrm{D}^{s,k,s} \ar[r] \ar[d, "i"] &
      {}^1\mathrm{D}^{s,k-1,s+r+1} \ar[d, "i"] \\
      0 \ar[r] \ar[d, "\partial"] &
      {}^{r+1}\mathrm{E}_1^{s,k,s} \ar[r] \ar[d, "\partial_{r+1}"] &
      {}^r\mathrm{E}_1^{s,k,s} \ar[r] \ar[d, "\partial_r"] &
      0 \ar[d, "\partial"] \\
      {}^1\mathrm{D}^{s+1,k-1,s+r+1} \ar[r] &
      {}^{r+1}\mathrm{D}^{s+1,k-1,s+1} \ar[r] &
      {}^r\mathrm{D}^{s+1,k-1,s+1} \ar[r] &
      {}^1\mathrm{D}^{s+1,k-2,s+r+2} 
    \end{tikzcd}
  \end{center}
	
  In this diagram we may pick a representative of $x$ in ${}^r\mathrm{E}_1^{s,k,s} = {}^{r+1}\mathrm{E}_1^{s,k,s}$ (which we will also denote $x$).
  Let $y = d_r(x)$ denote the target of the relevant differential in the $E$-based Adams spectral sequence and consider $y$ as an element of ${}^1\mathrm{D}^{s+1,k-1,s+r+1}$.
  We claim that there exists a $y' \in {}^{r+1}\mathrm{D}^{s+1,k-1,s+r+1}$ such that $\partial_{r+1} (x) = \tau^r y'$ and $y$ maps to $\tau^r y'$.
  Indeed, this follows from \Cref{thm:synth-ASS} and the fact that $\tau^r$ as an endomorphism of $C\tau^{r+1}$ factors through $C\tau$.

  Then, by standard manipulations of exact sequences arising from smash products of cofiber sequences (as in, e.g., \cite[Lemma 9.3.2]{AndrewsMiller}), there is some $\wt{x}$ in ${}^r\mathrm{D}^{s,k,s}$ which maps to both $x$ in ${}^r\mathrm{E}_1^{s,k,s}$ and $-f(y)$ in ${}^1\mathrm{D}^{s,k-1,s+r+1}$. The image of $\wt{x}$ along the map
  $$ {}^r\mathrm{D}^{s,k,s} \to \pi_{k,k+s}(C\tau^r \otimes \nu X) $$
  is the desired class.
\end{proof}

\begin{prop} \label{lemm:strong-etau-complete} \label{cor:strong-conv}
  Let $X$ denote an $E$-nilpotent complete spectrum.
  Then the following are equivalent:
  \begin{enumerate}
  \item The $E$-based Adams spectral sequence for $X$ converges strongly.
  \item The $\nu E$-based Adams spectral sequence for $\nu X$ converges strongly.
  \item The $\tau$-Bockstein spectral sequence for $\nu X$ converges strongly.
  \end{enumerate}
\end{prop}

In order to prove \Cref{lemm:strong-etau-complete} we recall the following theorem of Boardman, which provides a useful characterization of strong convergence.

\begin{thm}[{\cite[Theorem 7.3]{Boardman}}] \label{thm:strong-conv}
    Given an $E$-nilpotent complete spectrum $X$, the following two conditions are equivalent:
    \begin{itemize}
        \item The $E$-based Adams spectral sequence of $X$ converges strongly.
        \item $\varprojlim_{r}^{1} \mathrm{E}^{s,t}_r (X) = 0$ for pair of integers $s$ and $t$.
    \end{itemize}
		Analogous $\mathrm{lim}^1$ conditions determine strong convergence of $\nu E$-based Adams spectral sequences and $\tau$-Bockstein spectral sequences.
\end{thm}

Note that the second bullet point in the above theorem makes sense because $\mathrm{E}^{s,t} _{r+1} \subseteq \mathrm{E}^{s,t} _r$ as soon as $r>s$.

\begin{proof}[Proof of \Cref{lemm:strong-etau-complete}]
  Since $X$ is $E$-nilpotent complete, it follows from \Cref{lemm:easy-e-compton} that $\nu X$ is $\nu E$-nilpotent complete and $\tau$-complete.

  Using \Cref{thm:strong-conv}, to prove that (1) is equivalent to (2) it suffices to show that $\varprojlim_r \mathrm{E}^{s,k}_{r} = 0$ if and only if $\varprojlim_r \mathrm{E}^{s,k,w}_{r} = 0$. 
  In fact, these groups are isomorphic:
    \begin{align*}
      {\varprojlim_{r}}^1 \mathrm{E}^{s,k} _{r}
      \cong {\varprojlim_{r}}^1 \mathrm{Z}^{s,k}_{r}/\mathrm{B}_{s}^{s,k} 
      \cong {\varprojlim_{r}}^1 \mathrm{Z}^{s,k}_{r} 
      \cong {\varprojlim_{r}}^1 \mathrm{Z}^{s,k}_{r}/\mathrm{B}_{s-w+1}^{s,k} 
      \cong {\varprojlim_{r}}^1 \mathrm{E}^{s,k,w}_{r}.
    \end{align*}

We next prove the equivalence of the second and third conditions.  Let $\beta_r^{s,k,w}$ denote the groups in the $\tau$-Bockstein spectral sequence indexed so that the spectral sequence takes the form
    $$\beta_1^{s,k,w} \cong \pi_{k,k+w}(\Sigma^{0,-s} C\tau \otimes \nu X) \Longrightarrow \pi_{k,k+w} (\nu X).$$
Combining \Cref{thm:synthetic-Adams}(1) and \Cref{cor:synth-ASS}, we learn that
	$$\beta_r^{s,k,w} \cong \mathrm{E}_{r+1}^{w+s,k,w}.$$
Boardman's theorem applies since both of these spectral sequences are conditionally convergent.  Since the spectral sequences are furthermore isomorphic, up to reindexing, one converges strongly if and only if the other does.
\end{proof}

\begin{ntn}
    We let $\mathrm{F}^s \pi_{k,k+w} (\nu X) \subseteq \pi_{k,k+w} (\nu X)$ denote the $\nu E$-Adams filtration, and we let $\mathrm{F}^{s}_{\tau} \pi_{k,k+w}(\nu X)$ denote the $\tau$-Bockstein filtration.
\end{ntn}

\begin{cor} \label{cor:tau-surj}
  Suppose $X$ is $E$-nilpotent complete and that its $E$-based Adams spectral sequence converges strongly. Then
    $$\mathrm{F}^{s} \pi_{k,k+w}(\nu X) = \mathrm{F}^{s-w}_{\tau} \pi_{k,k+w}(\nu X),$$
    where for $k < 0$ we set
    \[\mathrm{F}^{k} _{\tau} \pi_{k,k+w} (\nu X) = \pi_{k,k+w} (\nu X).\]
    In particular, for $w \leq s$, the map
  $$ \mathrm{F}^s\pi_{k,k+w}(\nu X) \xrightarrow{\cdot \tau} \mathrm{F}^s\pi_{k,k+w-1}(\nu X) $$
  is surjective.

\end{cor}

\begin{proof}
  By \Cref{cor:strong-conv}, the $\nu E$-based Adams spectral sequence for $\nu X$ converges strongly.

    Now, the inclusion
    $$\mathrm{F}^{s-w} _{\tau} \pi_{k,k+w}(\nu X) \subseteq \mathrm{F}^{s} \pi_{k,k+w}(\nu X)$$
    follows from a downward induction on $w$, starting from \Cref{cor:synth-ASS}(1), which implies the desired result for $w \geq s$.
    On the other hand, to see that 
    $$\mathrm{F}^{s} \pi_{k,k+w}(\nu X) \subseteq \mathrm{F}^{s-w}_{\tau} \pi_{k,k+w}(\nu X)$$
    for all $s$, it suffices by strong convergence to show that, whenever $w \leq s$, multiplication by $\tau$ is surjective as a map $\mathrm{E}_{\infty} ^{s,k,w-1} \to \mathrm{E}_{\infty} ^{s,k,w}$. This is a consequence of \Cref{cor:synth-ASS}.
%
\end{proof}

\begin{proof}[Proof of \Cref{thm:synthetic-Adams}(2)-(4)]
    We begin by noting that \Cref{cor:strong-conv} implies that the $\nu E$-based Adams spectral sequence for $\nu X$ converges strongly.
    Recall that this means that:

    \begin{enumerate}
    \item $\mathrm{F}^s \pi_{k,k+w}( \nu X )/ \mathrm{F}^{s+1} \pi_{k,k+w}( \nu X ) \cong \mathrm{E}^{s,k,w} _\infty$.
    \item The filtration $\mathrm{F}^{\bullet} \pi_{k,k+w}( \nu X )$ is complete and Hausdorff.
    \end{enumerate}

    Now, we examine the reduction map
    $$ \nu X \to C\tau \otimes \nu X $$
    through the $\nu E$-based Adams spectral sequence.
    As discussed in \Cref{cor:synth-ctau-ASS}, the $\nu E$-based Adams spectral sequence for $C\tau \otimes \nu X$ has $\mathrm{E}_2$-term given by
    $$ {}^1\mathrm{E}_2^{s,k,w} \cong \begin{cases} \mathrm{E}^{s,k}_{2}, & \text{if } s=w \\ 0, & \text{otherwise.} \end{cases} $$
    It follows that the spectral sequence collapses at the $\mathrm{E}_2$-page, that there is no space for extension problems and that the map
    $$ \mathrm{Z}_\infty^{s,k} \cong \mathrm{E}_\infty^{s,k,s} \to {}^1\mathrm{E}_\infty^{s,k,s} \cong \mathrm{E}_2^{s,k} $$
    is just the usual inclusion. This produces a factorization
    $$\pi_{k,k+s}(\nu X)=F^{s} \pi_{k,k+s}(\nu X) \twoheadrightarrow \mathrm{E}_\infty^{s,k,s} \cong \mathrm{Z}_\infty^{s,k} \subseteq \mathrm{E}_2^{s,k} \cong \pi_{k,k+s}(C\tau \otimes \nu X). $$
    The surjectivity of the first map implies that we can always pick an $\wt{x}$.

    (2a) On the associated graded, multiplication by $\tau^{r-1}$ can be identified with
    $$ \mathrm{E}_\infty^{s,k,s} \cong \mathrm{Z}_\infty^{s,k}(X) \to \mathrm{Z}_\infty^{s,k}(X) / \mathrm{B}_{r}^{s,k}(X). $$
    Therefore, as long as $x$ survives to the $\mathrm{E}_{r+1}$-page, any lift $\wt{x}$ will have $\tau^{r-1}\wt{x} \neq 0$.
    
    (2b) It suffices to note that the $\nu E$-based Adams spectral sequence for $\nu X$ is sent to the $E$-based Adams spectral sequence for $X$ under $\tau^{-1}$ and that the induced map
    $$ \mathrm{E}_\infty^{s,k,s} \cong \mathrm{Z}_\infty^{s,k}(X) \to \mathrm{E}_\infty^{s,k} $$
    is just the usual projection.

    (3a) For this we now suppose that $x \in \mathrm{B}_{r+1}^{s,k}$ and consider the following diagram:
    \begin{center}
      \begin{tikzcd}
        0 \ar[r] &
        (\mathrm{F}^s\pi_{k,k+s}(\nu X))[\tau^r] \ar[r] \ar[d] &
        \mathrm{F}^s\pi_{k,k+s}(\nu X) \ar[r, "\cdot \tau^r"] \ar[d, two heads] &
        \mathrm{F}^s\pi_{k,k+s-r}(\nu X) \ar[r] \ar[d, two heads] &
        0 \\
        0 \ar[r] &
        \mathrm{E}_\infty^{s,k,s}[\tau^r] \ar[r] \ar[d, equal] &
        \mathrm{E}_\infty^{s,k,s} \ar[r] \ar[d, equal] &
        \mathrm{E}_\infty^{s,k,s-r} \ar[r] \ar[d, equal] &
        0 \\
        0 \ar[r] &
        \mathrm{B}_{r+1}^{s,k} \ar[r] &
        \mathrm{Z}_\infty^{s,k} \ar[r] &
        \mathrm{Z}_\infty^{s,k}/\mathrm{B}_{r+1}^{s,k} \ar[r] &
        0,
      \end{tikzcd}
    \end{center}
    where the rows are exact by Corollaries \ref{cor:synth-ASS} and \ref{cor:tau-surj}.
    It will suffice to show that left most vertical map is surjective.
    This follows from the snake lemma together with the fact that the map
    $$ \mathrm{F}^{s+1}\pi_{k,k+s}(\nu X) \xrightarrow{\cdot \tau^r} \mathrm{F}^{s+1}\pi_{k,k+s-r}(\nu X) $$
    is surjective, which follows from Corollary \ref{cor:tau-surj}.
    
    (3b) We now suppose that we are given $\alpha \in \pi_k (X)$ detected by $x$. In particular, $x$ is not the target of a differential in the $E$-based Adams spectral sequence for $X$. Then we may modify $\wt{x}$ by elements of higher $\nu E$-filtration without affecting conditions (1a) through (2b). Let $\beta = \tau^{-1} \wt{x}$. Then $\alpha - \beta \in \mathrm{F}^{s+1} \pi_k (X)$. It follows from \Cref{lemm:adams-fil1} that there exists some $e_1 \in \pi_{k,k+s+1} (\nu X)$ such that $\tau^{-1} e_1 = \alpha-\beta$. It follows from \Cref{cor:synth-ASS} that $e_1$ must be in $\nu E$-Adams filtration at least $s+1$. Replacing $\wt{x}$ with $\wt{x} + e_1$, we obtain $\tau^{-1} \wt{x} = \alpha$, as desired.

    (4) Finally, we verify the generation statement.
    Let $A$ denote the $\Z[\tau]$-submodule of $\pi_{k,*}(\nu X)$ generated by the $\wt{x}$, and let $B$ denote the $\tau$-adic completion of $A$. Our first claim is that $B$ remains a natural submodule of $\pi_{k,*}(\nu X)$, which follows from the fact that the $\tau$-adic filtration on $\pi_{k,*}(\nu X)$ is complete and Hausdorff by strong convergence.
		Now, since the inclusion $B \to \pi_{k,*}(\nu X)$ is one between $\tau$-complete objects, we need only note that the map
            $$ B/\tau \to \pi_{k,k+*}(\nu X)/\tau \cong \mathrm{F}^*\pi_{k,k+*}(\nu X)/\mathrm{F}^{*+1}\pi_{k,k+*}(\nu X) \cong \mathrm{E}_\infty^{*,k,*} $$
    is a surjection. The middle isomorphism above follows from \Cref{cor:tau-surj}.
		%
\end{proof}


%% file: SyntheticTodaRange.tex
In order to illustrate the complexities present in synthetic homotopy groups we will compute the bigraded groups $\pi_{k,*}(\nu_{\HFt} \Ss^{\wedge} _2)$ in the Toda range ($k \leq 19$).
We will see that these groups reflect the entire structure of the $\HFt$-Adams spectral sequence for $\Ss^{\wedge} _2$, including hidden extensions.
For brevity, throughout this section $\pi_{a,b}$ will refer to $\pi_{a,b}(\nu_{\HFt} \Ss^\wedge_2)$.

The $\HFt$-Adams spectral sequence for $\Ss^\wedge_2$ converges strongly because $\Ss^\wedge_2$ is $\HFt$-nilpotent complete and each of the groups on its $\mathrm{E}_2$-term are finite.
There are no differentials in the $\HFt$-Adams spectral sequence for $\Ss^{\wedge} _2$ in topological degree less than or equal to $13$.
For topological degrees $14$ through $19$ we reproduce the spectral sequence below in \Cref{fig:syn-toda-range-sseq}.

\begin{prop} \label{prop:syn-toda-range}
  For $k \leq 19$, 
  $\pi_{k,*}$ is presented as a $\tau$-adically complete algebra by generators
  
  \begin{tabular*}{\textwidth}{c @{\extracolsep{\fill}} c c c c }
    $ \tau \in \pi_{0,-1} $ &
    $ \wt{\sigma} \in \pi_{7,8} $ &
    $ \wt{\kappa} \in \pi_{14,18} $ &
    $ \wt{P^2h_1} \in \pi_{17,26} $ \\
    $ \wt{2} \in \pi_{0,1} $ & 
    $ \wt{\epsilon} \in \pi_{8,11} $ &
    $ \wt{\rho} \in \pi_{15,19} $ &
    $ \wt{\nu^*} \in \pi_{18,20} $ \\
    $ \wt{\eta} \in \pi_{1,2} $ &
    $ \wt{Ph_1} \in \pi_{9,14} $ &
    $ \wt{\eta^*} \in \pi_{16,18} $ & 
    $ \wt{c_1} \in \pi_{19,22} $ \\
    $ \wt{\nu} \in \pi_{3,4} $ & 
    $ \wt{Ph_2} \in \pi_{11,16} $ & 
    $ \wt{Pc_0} \in \pi_{16,24} $ &
    $ \wt{P^2h_2} \in \pi_{19,28} $
  \end{tabular*}
  \newline

  subject to relations
  \begin{align*}
    0 &=
    \wt{2} \wt{\eta} =
    \wt{\eta} \wt{\nu} =
    \wt{2} \wt{\nu}^2 =
    \wt{2}^4 \wt{\sigma} =
    \wt{\nu} \wt{\sigma} =
    \wt{\eta} \wt{\sigma}^2 =
    \wt{2} \wt{\epsilon} =
    \wt{\eta}^2 \wt{\epsilon} =
    \wt{\nu} \wt{\epsilon} =
    \wt{\sigma} \wt{\epsilon} & \\
    &=
    \wt{2} \wt{Ph_1} =
    \wt{\nu} \wt{Ph_1} =
    \wt{\eta} \wt{Ph_2} =
    \wt{\sigma} \wt{Ph_2} =
    \wt{\epsilon} \wt{Ph_2} = 
    \wt{2}^3 \wt{\kappa} = 
    \wt{2}^5 \wt{\rho} = 
    \wt{\nu} \wt{\rho} & \mathrm{(0)}\\
    &=
    \wt{2} \wt{Pc_0} = 
    \wt{\eta}^2 \wt{Pc_0} = 
    \wt{\nu} \wt{Pc_0} = 
    \wt{2} \wt{\eta^*} = 
    \wt{\nu} \wt{\eta^*} = 
    \wt{2} \wt{P^2h_1} = 
    \wt{\eta} \wt{\nu^*} = 
    \wt{2} \wt{c_1} &
  \end{align*}

  \begin{tabular*}{\textwidth}{c @{\extracolsep{\fill}} l c l c l c }
    $\mathrm{(1)}$ & $ \wt{\eta}^3 = \wt{2}^2 \wt{\nu} $ &
    $\mathrm{(4)}$ & $ \wt{\eta}^2 \wt{Ph_1} = \wt{2}^2 \wt{Ph_2} $ &
    $\mathrm{(7)}$ & $ \wt{\eta}^2 \wt{\eta^*} = \wt{2}^2 \wt{\nu^*} $ \\
    $\mathrm{(2)}$ & $ \wt{\eta} \wt{\rho}  = \tau^2 \wt{Pc_0} $ &
    $\mathrm{(5)}$ & $ \wt{\epsilon} \wt{Ph_1} = \wt{\eta} \wt{Pc_0} $ &
    $\mathrm{(8)}$ & $ \wt{\eta}^2 \wt{P^2h_1} = \wt{2}^2 \wt{P^2h_2} $ \\
    $\mathrm{(3)}$ & $ \wt{\nu}^3 = \wt{\eta}^2 \wt{\sigma} + \tau \wt{\eta} \wt{\epsilon} $ &
    $\mathrm{(6)}$ & $ \wt{Ph_1}^2 = \wt{\eta} \wt{P^2h_1} $ &
    $\mathrm{(9)}$ & $ \tau \wt{2} = 2 $ \\
    & & \\
    $\mathrm{(10)}$ & $0 = 2 \wt{\sigma}^2$ &
    $\mathrm{(12)}$ & $0 = 2 \wt{\nu} \wt{\kappa}$ &
    $\mathrm{(14)}$ & $2 \wt{\kappa} = \wt{2}^2 \wt{\sigma}^2$ \\
    $\mathrm{(11)}$ & $0 = \tau \wt{\eta}^2 \wt{\kappa}$ &
    $\mathrm{(13)}$ & $\wt{\nu} \wt{Ph_2} = \wt{2}^2 \wt{\kappa} $ & 
    $\mathrm{(15)}$ & $ \wt{\epsilon}^2 = \wt{\eta}^2 \wt{\kappa} = \wt{\sigma} \wt{Ph_1} + \tau \wt{Pc_0}. $ \\
  \end{tabular*}  
\end{prop}

\begin{figure}[t]
  \centering
  Synthetic and usual Adams spectral sequences for the $2$-complete sphere \\
   \includegraphics[trim={3.9cm 15.2cm 3.9cm 4cm}, clip, scale=1]{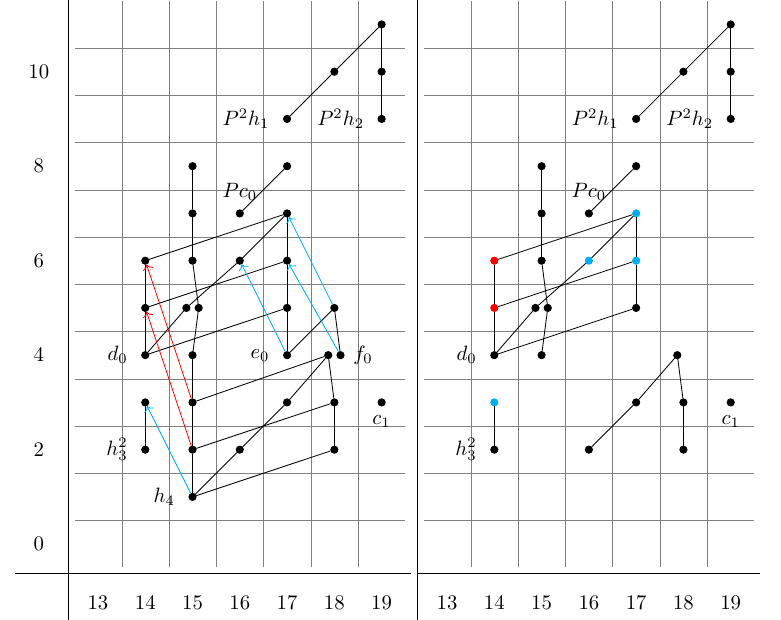}
   \caption{Left: Adams spectral sequence for the sphere, with differentials color-coded by length. Right: $\mathrm{E}_\infty$-page of the synthetic Adams spectral sequence for $\nu_{\HFt} \Ss_2^\wedge$. Black dots indicate a copy of $\F_2[\tau]$, red dots indicate a copy of $\F_2[\tau]/\tau^2$ and blue dots indicate a copy of $\F_2[\tau]/\tau$.}
   \label{fig:syn-toda-range-sseq}
\end{figure}

Before proving \Cref{prop:syn-toda-range} we discuss some of the subtleties which appear in this range.
The results of this proposition are also summarized in \Cref{fig:syn-toda-range}.

\begin{rmk} \label{rmk:subtle-mult}
  The first hidden extension in the Adams spectral sequence occurs in stem 9 where on the $\mathrm{E}_2$-page $h_2^3 = h_1^2h_3$ but in homotopy $\nu^3 = \eta^2\sigma + \eta\epsilon$. Synthetically the presence of this hidden term is reflected by the appearance of a $\tau$ in relation (3) where
  \[ \wt{\nu}^3 = \wt{\eta}^2 \wt{\sigma} + \tau \wt{\eta} \wt{\epsilon}. \]

  Similarly in stem 16 the hidden extension from $h_0^3h_4$ to $Pc_0$ is reflected by relation (2) where $ \wt{\eta} \wt{\rho}  = \tau^2 \wt{Pc_0} $. Note that in this case the multiplication jumps by 2 Adams filtrations and therefore 2 $\tau$'s appear in the product. This product relation is depicted by the green line originating from $\wt{\rho}$ in \Cref{fig:syn-toda-range}

  Another subtlety is that products that are classically zero need not be zero synthetically (though they will be $\tau$-power torsion). In this range the key example of this is relation (14) where $\wt{2}^2 \wt{\sigma}^2 = \tau \wt{2} \wt{\kappa}$. In this relation we see a product which is $\tau$-torsion hidden extend to a $\tau^2$-torsion class and it is depicted by the bent green line in \Cref{fig:syn-toda-range}. A combination of these features appears in relation (15). 
\end{rmk}

\begin{rmk}
In \Cref{prop:syn-toda-range} the generators are chosen using \Cref{thm:synthetic-Adams}(3).
It is important to note that there are ambiguities in this notation.
For some classes $\wt{x}$, $x$ refers to an element of the homotopy $\Ss^{\wedge}_2$.  In these cases $\wt{x}$ is determined up to $\tau$-power torsion classes of higher $\nu \HFt$-Adams filtration.
For other classes $\wt{x}$, $x$ refers to a permanent cycle on the $\mathrm{E}_2$-page of the Adams spectral sequence for $\Ss^{\wedge}_2$.  These classes are only determined up to elements of higher $\nu \HFt$-Adams filtration.  However, in the case that $x$ is the target of a $d_{m+1}$-differential, we more precisely define $\wt{x}$ up to elements of higher $\nu \HFt$-filtration which are $\tau^m$-torsion.

In particular, note that the classes $\wt{2}$, $\wt{\eta}$ and $\wt{\nu}$ are unambiguously determined. On the other hand, one could, for example, replace $\wt{\kappa}$ with $3\wt{\kappa}$ or $\wt{c_1}$ with $\wt{c_1} + a\tau^6 \wt{P^2h_2}$. Nevertheless, we claim that the proposition is valid for any collection of generators provided by \Cref{thm:synthetic-Adams}(3) \emph{as long as we choose a $\wt{c_1}$ which is 2-torsion.}
\end{rmk}

\begin{rmk}
It is also important to note that multiplication may not interact nicely with the tilde notation: $\wt{x} \wt{y}$ might not be a valid choice of representative for $\wt{xy}$ since $\wt{x}\wt{y}$ may not satisfy the $\tau$-torsion requirement that \Cref{thm:synthetic-Adams}(3) places on $\wt{xy}$. However, it is true that $\wt{x} \wt{y} - \wt{xy}$ is divisible by $\tau$, and often this can be used to show that $\wt{x}\wt{y}$ does in fact satisfy the $\tau$-torsion requirement.

Furthermore, when solving extension problems one needs to be careful about exactly which bigraded homotopy elements one chooses. For example, both $\wt{\sigma}^2$ and $\wt{\sigma}^2 + \tau^2 \wt{\kappa}$ are valid choices of $\wt{h_3^2}$, but $\wt{\eta} \wt{\sigma}^2 = 0$ whereas $\wt{\eta} (\wt{\sigma}^2 + \tau^2 \wt{\kappa}) = \tau^2 \wt{\eta} \wt{\kappa} \neq 0$.
\end{rmk}


\begin{figure}
  \centering
  $\pi_{k,k+s} \nu_{\HFt} (\Ss^{\wedge} _2)$ for $13 \leq k \leq 19$
\includegraphics[trim={3.8cm 9.6cm 3.8cm 5cm}, clip, scale=1]{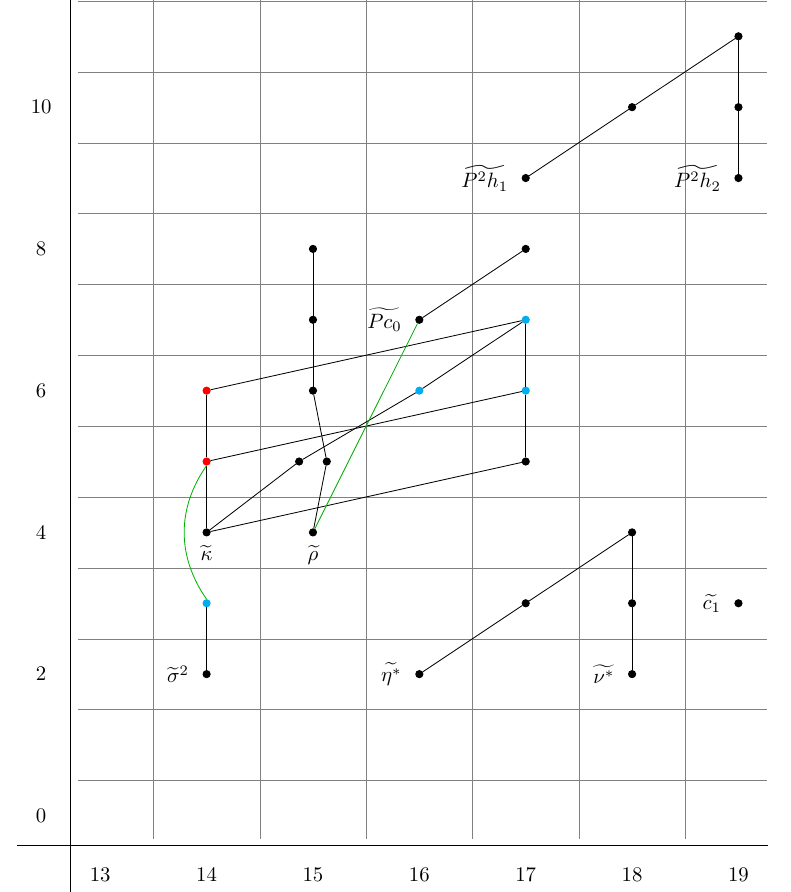}
  \caption{A picture of $\pi_{k,*}$ for $13 \leq k \leq 19$. We index the picture so that bidegree $(k,k+s)$ corresponds to position $(k,s)$.
    Black dots indicate non-$\tau$-torsion classes,
    red dots indicate $\tau^2$-torsion and
    blue dots indicate $\tau$-torsion.
    We suppress all $\tau$-multiples in this figure.
    Black lines correspond to $\wt{2}$, $\wt{\eta}$ and $\wt{\nu}$ multiplications which are detected at the level of $C\tau$.
    Green lines are used for more complicated $\wt{2}$ and $\wt{\eta}$ multiplications.
    In this range the green line indicate a multiplication which hits a power of $\tau$ times the indicated dot (see \Cref{rmk:subtle-mult} for further discussion).
    }
  \label{fig:syn-toda-range}
\end{figure}

\begin{proof}
  Using \Cref{thm:synthetic-Adams} we may produce the generators listed above.
  This theorem also lets us conclude that the $\tau$-adic completion of the algebra they generate is equal to $\pi_{k,*}$ for $k \leq 19$.

  Before we continue we use \Cref{cor:tau-torsion-bound} to find which bigraded groups have $\tau$-power torsion elements. The only bigraded groups with $k \leq 19$ for which $\pi_{k,k+s}^{\mathrm{tor}}$ is nonzero are
  $$ \pi_{14,17},\ \pi_{14,18},\ \pi_{14,19},\ \pi_{14,20},\ \pi_{16,22},\ \pi_{17,23},\ \pi_{17,24}. $$
  This means that $\tau^{-1} : \pi_{k,k+s} \to \pi_k$ is an inclusion in all other bidegrees.
  Moreover, since the functor $\tau^{-1}$ is symmetric monoidal, it follows that these inclusions respect the multiplicative structure on both sides. Thus, we may deduce that (0)--(9) follow from the associated relations in usual homotopy groups.

  To prove the relation (10), note that the element $\wt{\sigma}$ lives in an odd topological degree. Therefore, we learn that $2 \wt{\sigma}^2 = 0$ by considering the $\mathbb{E}_\infty$-ring structure on $\nu_{\HFt}(\Ss_2^\wedge)$ (see \cite[Remark 4.10]{Pstragowski}). 

  Relations (11) and (12) follow from the fact that both $\eta^2\kappa$ and $2\nu\kappa$ are zero in the usual homotopy groups of $\Ss_2^{\wedge}$.  Therefore both $\wt{\eta}^2\wt{\kappa}$ and $\wt{2}\wt{\nu}\wt{\kappa}$ are $\tau$-power torsion.  Since they live in bidegrees containing only simple $\tau$-torsion, it follows that $\tau$ times them is zero.  Note that $\tau \wt{2}\wt{\nu}\wt{\kappa}=2 \wt{\nu}\wt{\kappa}$.

  To prove (13) and (15) we consider the ring map
  $$ \nu_{\HFt}(\Ss^\wedge_2) \to C\tau \otimes \nu_{\HFt}(\Ss^\wedge_2). $$
  Because there are no $\tau$-power torsion elements which are also divisible by $\tau$ in $\pi_{14,20}$ or $\pi_{16,22}$, this map induces isomorphisms
  $$ \pi_{14,20}^{\mathrm{tor}} \cong \Ext^{6,20}_{\A_*}(\F_2, \F_2)\ \ \text{ and }\ \ \pi_{16,22}^{\mathrm{tor}} \cong \Ext^{6,22}_{\A_*}(\F_2, \F_2). $$
  Thus, once we know that each term is zero in the usual homotopy groups we can read (13) and (15) off from the corresponding relation in the $\mathrm{E}_2$ page.

  In the Toda range (14) is the most difficult relation.
  To obtain it, we will make use of the long exact sequence
  \begin{center}
    \begin{tikzcd}
      \cdots \ar[r] &
      \pi_{k+1,k+s-1} \ar[r] \arrow[draw=none]{d}[name=Z, shape=coordinate]{} &
      \Ext^{s-2,k+s-1}_{\A_*}(\F_2, \F_2) \arrow[rounded corners,to path={ -- ([xshift=2ex]\tikztostart.east)|- (Z) [near end]\tikztonodes-| ([xshift=-2ex]\tikztotarget.west)-- (\tikztotarget)}]{dl} & \\
        &
      \pi_{k,k+s} \ar[r, "\cdot \tau"] &
      \pi_{k,k+s-1} \ar[r] &
      \cdots.
    \end{tikzcd}
  \end{center}
  From (10) and the torsion bound on $\pi_{14,18}$ we know that $\wt{2} \wt{\sigma}^2$ and $2 \wt{\kappa}$ are both simple $\tau$-torsion. Thus, they lift to non-zero classes in $\Ext^{1,16} _{\A_*} (\F_2, \F_2)$ and $\Ext^{2,17} _{\A_*} (\F_2, \F_2)$, respectively. These classes must be $h_4$ and $h_0 h_4$, and hence are related by multiplication by $h_0$.
	This implies that their images are related by multiplication by $\wt{2}$, as desired.

    Finally, using parts (2a), (3a), (3b) and (4) of \Cref{thm:synthetic-Adams}, one may compute the length of $\pi_{k,k+s}$ as a $\Z_2$-module for each $k \leq 19$ and all $s$. From this we may conclude that there are no further relations for size reasons.
\end{proof}


%% file: Appendix-s0.tex
In this appendix we study vanishing curves in Adams spectral sequences via an explicit analysis of Adams towers and their Postnikov truncations. These techniques were developed in order to answer Question 3.33 from \cite{Akhil}, which asks about the linear term in the vanishing curve of the $\BP\langle n \rangle$-Adams spectral sequence for the sphere. At the prime $3$ our results provide the upper bound on the left-hand-side of \cref{eqn:BIG} necessary in the proof of \Cref{thm:intro-main}. As a corollary we obtain new bounds on the $p$-torsion order of the stable homotopy groups of spheres.

Before proceeding further we should highlight several differences between the perspective on vanishing curves taken in \Cref{sec:AppendixVL} and this appendix. In the main body of the paper vanishing curves are interpreted in terms of the bigraded homotopy groups of a synthetic spectrum and are often implicitly linear and finite page. The emphasis is mostly on genericity results. \Cref{sec:AppendixVL} inherits the technical assumption that we must work only with ring spectra which are of Adams type from \cite{Pstragowski}. In this appendix we will not consider finite-page vanishing lines, instead confining ourselves to the vanishing curve present at the $E_\infty$ page. Our emphasis is on exploiting naturality in the choice of ring spectrum. This appendix works with the approach to descent developed by Akhil Mathew in \cite{Akhil} and thereby inherits the technical assumption that all ring spectra admit an $\mathbb{E}_1$ multiplication.

In Section \ref{app:prelim} we recall the definition of the vanishing curve, review previous results and state our main theorem, which is a collection of novel bounds on various vanishing curves.
In Section \ref{app:ABvl} we give a comparison theorem for vanishing curves over different rings. This comparison theorem is the key technical advance in this appendix. 
In Section \ref{app:blocks} we finish the proof of the comparison theorem.
In Section \ref{app:reductions} we use the comparison theorem and theorems of Davis-Mahowald and Gonz\'alez \cite{DM3} \cite{GonzalezRegular} to prove the main theorem.

\begin{cnv} \label{cnv:desc-ring}
  Throughout this appendix we will adopt the following conventions:
  \begin{enumerate}
  \item All spectra will be $p$-local for a fixed prime $p$.
  \item Rings and ring morphisms will refer to objects and morphisms of $\Alg(Sp)$. \footnote{The results of this appendix remain true if we replace $Alg(Sp)$ with the full subcategory of $Alg^{\mathbb{E}_0}(\Sp)$ on those objects that admit an $\mathbb{A}_2$ structure. We opt to work in less generality for convenience and so that we can avoid reproving many statements from \cite{Akhil}.}
  \item A ring $R$ will also be assumed to satisfy the following hypotheses:
    \begin{itemize}
    \item $R$ is $p$-local and connective,
    \item $\pi_0(R) \cong \Z_{(p)}$,
    \item $\pi_i(R)$ is a finitely generated $\Z_{(p)}$ module for all $i$.
    \end{itemize}
    Moreover, $A$ and $B$ will also denote rings satisfying the same hypotheses. \footnote{This convention ensures that the Adams spectral sequence based on $R$ converges for every connective $p$-local spectrum.}
  \item In order to make concise statements about the asymptotics of various functions we will make use both big $O$ and little $o$ notation. 
  \end{enumerate}
\end{cnv}

\begin{ntn}
  In this appendix we will adopt the following notation in order to simplify expressions,
  \begin{enumerate}
  \item $q = 2p-2$,
  \item $v_p(k)$ will denote the $p$-adic valuation of an integer $k \in \Z$,
  \item if $p \neq 2$, 
    $$ \ell(k) = \begin{cases} v_p(k+2) & k+2 \equiv 0 \pmod q \\ 0 & k+2 \not\equiv 0 \pmod q \end{cases}, $$
    if $p = 2$, 
    $$ \ell(k) = \begin{cases} v_2(k+1) & k \text{ odd} \\ v_2(k+2) & k \text{ even} \end{cases}. $$
  \end{enumerate}
  We will sometimes use that $\ell(k) \in O(\log(k))$.
\end{ntn}


%% file: Appendix-s1.tex
We begin by defining two functions attached to a ring $R$ which we will refer to as the $R$-Adams spectral sequence vanishing curves. Although the function $g_R$ defined below has a more direct interpretation as a vanishing curve it will turn out that $f_R$ has more tractable properties. For example, $f_R$ is sub-additive while $g_R$ has no such property.

\begin{dfn}
  \label{dfn:exp-of-nilp}
  Given a ring spectrum $R$ as above,
  \begin{itemize}
  \item Let $g_R(k)$ denote the minimal $m$ such that every
    $\alpha \in \pi_k(\Ss)$ whose $R$-Adams filtration is strictly greater than $m$ is zero. \footnote{Our function $g_{\BP}$ is equal to the function $g$ defined by Hopkins in \cite{''''}.}
  \item Let $f_R(k)$ denote the minimal $m$ such that
    for every connective $p$-local spectrum $X$, $i < k$, and $\alpha \in \pi_i(X)$,
    if $\alpha$ has $R$-Adams filtration at least $m$, then $\alpha = 0$. 
  \item Let $\Gamma (k)$ denote the minimal $m$ such that every $\alpha \in \pi_k(\Ss^0)$ whose $\HFp$-Adams filtration is strictly greater than $m$ is detected in the $K(1)$-local sphere ($\Gamma$ does not depend on a choice of $R$).
    \footnote{\Cref{dfn:h} is equivalent to the definition given here by our knowledge of the homotopy of the $K(1)$-local sphere.}
  \end{itemize}
\end{dfn}

\begin{rmk}\label{rmk:g-f-comp}
  \label{rmk:g->f}
  The $X = \Ss^0$ case in the definition of $f_R(k)$ implies that $$g_R(k) \leq f_R(k+1) - 1. $$
\end{rmk}

Several classic results in stable homotopy theory can be reformulated as bounds on the functions $f_R, g_R$ and $\Gamma$ for various rings $R$.
In \cite{MathewExponent} and \cite{Akhil}, work of Adams \cite{AdamsPer} and Luilevicius \cite{liul} is reformulated into the pair of inequalities
$$ f_{\Z_p}(k) \leq \frac{1}{q}k + O(1) \ \ \ \text{ and }\ \ \  \Gamma(k) \leq \frac{1}{q}k + O(1). $$
Later, in \cite{DM3}, Davis and Mahowald showed that at the prime $2$,
$$ g_{\text{bo}}(k) \leq \frac{1}{5}k + O(\log(k)) \ \ \ \text{ and }\ \ \  \Gamma(k) \leq \frac{3}{10}k + O(\log(k)) $$
In \cite{Gonz00}, Gonzalez proved the analogous results at odd primes, 
$$ g_{\BP\langle 1 \rangle}(k) \leq \frac{1}{p^2 - p -1}k + O(\log(k))\  \text{ and }\  \Gamma(k) \leq \frac{(2p-1)}{(2p-2)(p^2 - p -1)}k + O(\log(k)) $$
Finally, another formulation of the Nilpotence theorem \cite{DHS} worked out by Hopkins and Smith is that\footnote{See \cite[Theorem 3.30]{Akhil} for a published account of this argument.}
$$ f_{\BP}(k) = o(k). $$
One of the purposes of \Cref{sec:ANvl} was to provide the first effective bound on $f_{\BP}(k)$ which is not already present at the $\mathrm{E}_2$-page.

In the situation where $R$ is both an $\mathbb{E}_1$-ring and of Adams type we have the following lemma which relates $f_R$ and $g_R$ to weak and strong vanishing lines in synthetic spectra. 

\begin{lem}
  \label{lemm:strong-vl->f}
  Suppose $R$ is both an $\mathbb{E}_1$ ring and of Adams type,
  \begin{itemize}
      \item If $\nu_{R} (\mathbb{S}^{0})$ has a finite-page vanishing line of slope $m$ and intercept $c$,
    then
    $$ g_R(k) \leq mk + c. $$
\item If $\nu_{R} (\mathbb{S}^{0})$ has a strong finite-page vanishing line of slope $m$ and intercept $c$,
    then
    $$ f_R(k) \leq m(k-1) + c + 1.  $$
  \end{itemize}
\end{lem}

\begin{proof}
  By Lemma \ref{lemm:adams-fil1} each nonzero class $\alpha \in \pi_j(X)$ whose $R$-Adams filtration is $\geq n$ yields a non-$\tau$-torsion class $\tilde{\alpha} \in \pi_{j,j-n}(\nu X)$.
  \todo{It may be appropriate to say something about tau-completeness here. AS: It doesn't seem to be required, if I read the lemma right.}
\end{proof}

Applying \Cref{lemm:strong-vl->f} to \Cref{thm:AN-vl} we obtain the following corollary.

\begin{cor}
  \label{cor:ANvl}
  For each odd prime,
  $$ f_{\BP}(k) \leq \frac{1}{p^3 - p -1}k + 2p^2 - 4p + 10 - \frac{2p^2+2p-9}{p^3 - p -1}. $$
\end{cor}
The main theorem of this appendix is the following.

\begin{thm}\label{thm:app-main}\ 
  \begin{enumerate}
  \item For each prime and $n \in \Z_{\geq 0}$,
    $$ f_{\BP\langle n \rangle}(k) \leq \frac{1}{|v_{n+1}|}k + \left(1 + \frac{1}{|v_{n+1}|} \right) f_{\BP}(k) - \frac{1}{|v_{n+1}|}. $$
  \item For each prime,
    $$ \Gamma(k) \leq \frac{(q+1)}{q|v_2|}k + \frac{(q+1)(|v_2| + 1)}{q|v_2|}f_{\BP}(k) + \ell(k). $$
  \item For each odd prime,
    $$ f_{\BP\langle 1\rangle}(k) \leq \frac{p+2}{2(p^3 - p -1)}k + 2p^2 - 4p + 11. $$
  \item For $p=3$,
    $$ \Gamma(k) \leq \frac{25}{184}k + 19 + \frac{1133}{1472} + \ell(k), $$
    and for $p \geq 5,$
    $$ \Gamma(k) \leq \frac{(2p-1)(p+2)}{4(p-1)(p^3 - p -1)}k + 2p^2 - 3p + 11 + \ell(k). $$
  \end{enumerate}
\end{thm}

The proof of this theorem will occupy the remainder of this appendix.
Once we have proved \ref{thm:app-main}(1) the rest of the theorem follows by relatively standard arguments. Note that \ref{thm:app-main}(1), when combined with the Nilpotence theorem, implies the following corollary which appeared as question 3.33 in \cite{Akhil}.

\begin{cor}
     \[f_{\BP \langle n \rangle} (k) \leq \frac{1}{\abs{v_{n+1}}}k + o(k).\]
\end{cor}

\begin{rmk}
  Similarly, using the Nilpotence theorem, the bound on $\Gamma$ given in \Cref{thm:app-main}(2) at the prime 2 simplifies to
  $$ \Gamma(k) \leq \frac{1}{4}k + o(k). $$
  Although this is asymptotically better than the result of Davis-Mahowald quoted above, because we don't have explicit control over the error term it is unsuitable for use in \cref{sec:Finale}.
  In fact, as observed by Stolz \cite[p. XX]{StolzBook}, any further improvement of the slope of a linear bound on $\Gamma(k)$ would imply \Cref{thm:intro-main} at the prime $2$ for $k \gg 0$. This would, at least for $k \gg 0$, bypass the need for \Cref{thm:AdamsBound}. 
\end{rmk}

\begin{cnj} \label{cnj:vanishing-line}
  $$ \Gamma(k) \leq \frac{1}{|v_2|}k + O(1). $$
\end{cnj}

The application of \ref{thm:app-main}(2) to bounding torsion exponents in the stable homotopy groups of spheres was explained in \Cref{ssec:exp}. Ultimately, torsion exponent bounds arise as a corollaries to bounds on $\Gamma(k)$. A more numerically precise result is obtained at odd primes by using \ref{thm:app-main}(4). The mysterious ``sublinear error term'' present in \Cref{thm:torsion-bound} is a residue of the non-effective nature of the Nilpotence theorem.





%% file: Appendix-s4.tex
The novel part of the proof of \Cref{thm:app-main} is the following comparison theorem which allows us to relate vanishing lines for different rings.

\begin{thm}[Comparison Theorem]
  \label{thm:AB-vl}\
  
  Suppose we have a ring map $i : A \to B$, then
  \begin{enumerate}
  \item $g_A(k) \leq g_B(k)$,
  \item $f_A(k) \leq f_B(k)$,
  \item if $i$ becomes an equivalence after applying $\tau_{<m}$, then
    $$ f_B(k) \leq f_A(k) + \left\lfloor \frac{k + f_A(k) - 1}{m} \right\rfloor \leq \frac{1}{m}k + \left( 1 + \frac{1}{m} \right) f_A(k) - \frac{1}{m}. $$
  \end{enumerate}
\end{thm}

\begin{rmk}
  In \cite{BBBCX2}, Conjecture 9.4.2 asks whether there is a finite-page vanishing line of slope $\frac{1}{13}$ in the $\mathrm{tmf}$-Adams spectral sequence for a particular spectrum. We can provide the following evidence in favor of this conjecture:
  The map
  $$ \mathrm{tmf} \to \mathrm{tmf}_1(3) = \BP\langle 2 \rangle $$
  allows us to apply \Cref{thm:AB-vl}(2), \Cref{thm:app-main}(1) and the Nilpotence theorem in order to conclude that
  $$ f_{\mathrm{tmf}}(k) \leq f_{\BP\langle 2 \rangle}(k) \leq \frac{1}{14}k + o(k) $$ 
  Note that the bound on $f_{\mathrm{tmf}}$ is not guaranteed to appear at any finite page. \todo{Perhaps note the consequence for the $v_2 ^2$-elements in the telescope?}
\end{rmk}

The first two statements of \ref{thm:AB-vl} follow easily from the fact that $i:A \to B$ induces a map of canonical Adams resolutions. The proof of the third statement will occupy most of Sections \ref{app:ABvl} and \ref{app:blocks}. In this proof we will rely on an alternative interpretation of $f_R$ from \cite{Akhil}. In order to recall this interpretation we begin by defining a natural filtration on the thick $\otimes$-ideal generated by $R$.

\begin{dfn}
  Given a set of spectra, $S$, the thick $\otimes$-ideal generated by $S$ consists of the smallest collection of spectra, $\text{Thick}^{\otimes}(S)$, closed under finite (co)limits and retracts, such that $X \otimes s \in \text{Thick}^{\otimes}(S)$ for all $s \in S$. We equip $\text{Thick}^{\otimes}(S)$ with the following filtration:
  \begin{itemize}
  \item $\text{Thick}^{\otimes}(S)_0 = \{0\}$,
  \item $\text{Thick}^{\otimes}(S)_1$ consists of retracts of spectra of the form $X \otimes s$ where $s \in S$,
  \item $\text{Thick}^{\otimes}(S)_n$ consists of retracts of extensions of objects of $\text{Thick}^{\otimes}(S)_{n-1}$ by objects of $\text{Thick}^{\otimes}(S)_1$.
  \end{itemize}
  We will only make use of this definition in the case where $S = \{R\}$.
\end{dfn}

\begin{rmk}
  For any $R$-module $M$ the unit map $M \to R \otimes M$ and the action map $R \otimes M \to M$ exhibit $M$ as a retract of $R \otimes M$, therefore $M \in \text{Thick}^\otimes(R)_1$.
\end{rmk}

The function $f_R$ can then be interpreted in terms of this filtration. 

\newcounter{fint}
\begin{prop}[{\cite[Definitions 2.28 and 3.26, and Proposition 3.28]{Akhil}}]
  \label{prop:f-interpret-I}\ 
  
  Let $I := \mathrm{fib}(\Ss \to R)$, then the following are equivalent:
  \begin{enumerate}
  \item $f_R(k) \leq n$,
  \item $\tau_{<k}\Ss^0 \in \textup{Thick}^{\otimes}(R)_n$,
  \item the map $I^{\otimes n} \to \Ss^0$ becomes null after tensoring with $\tau_{<k}\Ss^0$.
    \setcounter{fint}{\value{enumi}}    
  \end{enumerate}
\end{prop}

Sadly, none of these conditions are particularly convenient for the proof we have in mind. In order to remedy this we introduce two further equivalent conditions:

\begin{prop}
  \label{prop:f-interpret-II}
  The list of equivalent conditions from \Cref{prop:f-interpret-I} can be extended to include:
  \begin{enumerate}
    \setcounter{enumi}{\value{fint}}
  \item $\tau_{<k}\Ss^0$ is a retract of an object which has a length $n$ resolution by connective $R$-modules.
  \item $\tau_{<k}\Ss^0$ is a retract of an object which has a length $n$ resolution by connective induced $R$-modules.
  \end{enumerate}
\end{prop}

Before proving \Cref{prop:f-interpret-II} we set up some notation and conventions for manipulating finite resolutions of spectra.

\begin{dfn}
  A length $N$ resolution of a spectrum $X_0$ will consist of a diagram
  \begin{center}
    \begin{tikzcd}
        F_{N-1} \ar[r] & X_{N-2} \ar[r] \ar[d] & \cdots \ar[r] & X_1 \ar[r] \ar[d] & X_0 \ar[d] \\
        & F_{N-2} & & F_1 & F_0
    \end{tikzcd}
  \end{center}
  such that each $X_{j+1} \to X_j \to F_j$ is a cofiber sequence with the convention that $X_{N-1} = F_{N-1}$.
  In this situation we will say that we have a resolution of $X_0$ by $F_{N-1},\dots,F_0$.
\end{dfn}

\begin{ntn}
  We will adopt the following compact notation
  $$ [F_N, \dots, F_1, F_0; X] $$
  to express a resolution of $X$ by $F_{N},\dots,F_0$.
  It is important to note that this notation suppresses much of the data of a resolution.
\end{ntn}

\begin{wrn}
  Sometimes we will write $[\dots, F_1, F_0;X]$ for a resolution. Although this suggests an infinite-length resolution, in this appendix all resolutions will be finite length and this will simply be used to avoid specifying the length of a resolution. 
\end{wrn}

\begin{rmk}
  In the length two case the notation $[A, B; X]$ simply refers to a cofiber sequence
  $$ A \to X \to B. $$
\end{rmk}

\begin{proof}[Proof of \Cref{prop:f-interpret-II}]\
  
  $(5) \Rightarrow (4)$, clear.
  
  $(4) \Rightarrow (2)$,
  As remarked above every $R$-module is in $\text{Thick}^{\otimes}(R)_1$, therefore $\tau_{<k}\Ss^0$ is a retract of an $n$-fold extension of elements of $\text{Thick}^{\otimes}(R)_1$.
  
  $(3) \Rightarrow (5)$,
  Consider the following length $n+1$ resolution of $\Ss$,
  \begin{center}
    \begin{tikzcd}
      I^{\otimes n} \ar[r] &
       I^{\otimes (n-1)} \ar[r] \ar[d] &
      \cdots \ar[r] &
       I \ar[r] \ar[d] &
      \Ss \ar[d] \\
      &  R \otimes I^{\otimes (n-1)} & &
      R \otimes I &
      R
    \end{tikzcd}
  \end{center}
  From it we can produce a length $n$ resolution of $\text{cof}(I^{\otimes n} \to \Ss)$,
  \begin{center}
    \begin{tikzcd}
      R \otimes I^{\otimes (n-1)} \ar[r] &
      I^{\otimes (n-2)}/I^{\otimes n} \ar[r] \ar[d] &
      \cdots \ar[r] &
       I / I^{\otimes n} \ar[r] \ar[d] &
      \Ss / I^{\otimes n} \ar[d] \\
      &  R \otimes I^{\otimes (n-2)} & &
      R \otimes I &
      R
    \end{tikzcd}
  \end{center}
  Upon tensoring with $\tau_{<k}\Ss^0$ we obtain a length $n$ resolution of $(\tau_{<k}\Ss^0) \otimes (\Ss/I^{\otimes n})$ by connective, induced $R$-modules. Finally, by hypothesis,
  $$ (\tau_{<k}\Ss^0) \otimes (\Ss/I^{\otimes n}) \simeq (\tau_{<k}\Ss^0) \oplus (\tau_{<k}\Ss^0 \otimes \Sigma I^{\otimes n}). $$
\end{proof}

Using condition (4) we reduce the proof of the Theorem \ref{thm:AB-vl} to the following problem:
take a resolution of $X$ by $A$-modules and produce from it the shortest possible resolution of $X$ by $B$-modules. In order to provide a simple illustration of the methods we will use in the general case we first work the following example in detail.

\begin{qst*} Suppose that a spectrum $X$ sits in a cofiber sequence $C \to X \to D$ where $C,D$ are $\textup{BP}$-modules and $C,D,X \in \Sp_{[0,10]}$. What is the shortest resolution of $X$ by $\textup{BP}\langle 1 \rangle$-modules?
\end{qst*}

{\bf Strategy 1:}
We know that the map $\textup{BP} \to \textup{BP}\langle 1 \rangle$ is an equivalence after we apply $\tau_{<6}$, therefore any $\textup{BP}$ module in $\Sp_{[k,k+5]}$ is automatically a $\textup{BP}\langle 1 \rangle$ module.\footnote{A proof of this will appear in much greater generality in the next section} Knowing this trick we can break each of $C$ and $D$ into two $\textup{BP}\langle 1 \rangle$-modules and produce a new resolution of $X$ which uses 4 $\textup{BP}\langle 1 \rangle$-modules:
\begin{center}
  \begin{tikzcd}
    \tau_{[6,10]}C \ar[r] & C \ar[d] \ar[r] & F \ar[d] \ar[r] & X \ar[d] \\
    & \tau_{[0,5]}C & \tau_{[6,10]}D & \tau_{[0,5]}D
  \end{tikzcd}
\end{center}

This is a start, but it turns out we can do better.

{\bf Strategy 2:}
For our second approach we will start with a slightly modified version of the first resolution we produced:
\begin{center}
  \begin{tikzcd}
    \tau_{[5,10]}C \ar[r] & C \ar[d] \ar[r] & F \ar[d] \ar[r] & X \ar[d] \\
    & \tau_{[0,4]}C & \tau_{[6,10]}D & \tau_{[0,5]}D
  \end{tikzcd}
\end{center}
Now, we can expand this resolution into the diagram below where each square is cartesian.
\begin{center}
  \begin{tikzcd}
    \tau_{[5,10]}C \ar[r] \ar[d] &
    C \ar[r] \ar[d] &
    F \ar[r] \ar[d] &
    X \ar[d] \\
    0 \ar[r] &
    \tau_{[0,4]}C \ar[r] \ar[d] &
    G \ar[d] \ar[r] &
    H \ar[d] \\
    & 0 \ar[r] &
    \tau_{[6,10]}D \ar[r] \ar[d] &
    D \ar[d] \\
    & & 0 \ar[r] &
    \tau_{[0,5]}D 
  \end{tikzcd}
\end{center}

Notably the cofiber sequence which $G$ sits in is ``backwards''.
In fact, if we expand it a little bit
$$ \Sigma^{-1}\tau_{[6,10]}D \to \tau_{[0,4]}C \to G \to \tau_{[6,10]}D $$
we see that the attaching map must be zero for connectivity reasons and therefore $G \simeq \tau_{[0,4]}C \oplus \tau_{[6,10]}D$.
As a direct sum of $\textup{BP}\langle 1 \rangle$-modules this is in fact a $\textup{BP}\langle 1 \rangle $-module as well. Thus, we can produce the following length $3$ resolution of $X$
\begin{center}
  \begin{tikzcd}
    \tau_{[5,10]}C \ar[r] & F \ar[d] \ar[r] & X \ar[d] \\
    & \tau_{[0,4]}C \oplus \tau_{[6,10]}D & \tau_{[0,5]}D
  \end{tikzcd}
\end{center}


Both strategies had four main steps:
\begin{enumerate}
\item start with a resolution by bounded $A$-modules,
\item construct a new resolution from the given one,
\item count the length of the new resolution.
\item show that the new resolution is in fact a resolution by $B$-modules,
\end{enumerate}
In the proof of Theorem \ref{thm:AB-vl} each of these elements will be replaced by a lemma (whose proofs we defer to the next section).

\begin{lem}
  \label{lemm:resolution-prep}
  Given a resolution $[F_N,\dots,F_0; X]$ such that each $F_j$ is an $A$-module and all the $F_j$ are connective, there exists another resolution $[F_N',\dots,F_0'; \tau_{\leq M-1}X]$ such that each $F_j'$ is an $A$-module in $\Sp_{[0,M]}$.
\end{lem}

\begin{lem}
  \label{lemm:new-resolution}
  Given a resolution $[F_N, \dots, F_0; X]$ where each $F_j$ lives in $\Sp_{[0,K]}$
  we can construct another resolution
  $$ \left[ \dots, \left( \bigoplus_{0 \leq i \leq j} \tau_{[(j-i)m-i,(j-i+1)m-i)}F_i \right) , \dots, \left( \tau_{[-1,m-1)}F_1  \oplus  \tau_{[m,2m)}F_0 \right) , \left( \tau_{[0,m)}F_0 \right) ; X \right] $$
\end{lem}

\begin{rmk}
  Since \Cref{lemm:new-resolution} is more complicated than the others we pause here to note that in the case where $N=0$ this lemma just produces the $m$-speed Postnikov tower for $X$.
  In general, the lemma takes Postnikov-type towers for each $F_j$ and shuffles them together.
  Ultimately this is no more than a careful elaboration on the manipulations used in strategy 2 above.
\end{rmk}

\begin{lem}
  \label{lemm:length-count}
  The resolution produced by the Lemma \ref{lemm:new-resolution} has length,
  $$ 1 + N + \left\lfloor \frac{K + N}{m} \right\rfloor. $$
\end{lem}

\begin{lem}
  \label{lemm:res-by-B-mod}
  Given a ring map $A \to B$ which becomes an equivalence after applying $\tau_{<m}$ any $A$-module in $\Sp_{[a,a+m)}$ can be given the structure of a $B$-module.
\end{lem}

\begin{proof}[Proof of Theorem \ref{thm:AB-vl}]
  Let $N+1 = f_A(k)$, then by Proposition \ref{prop:f-interpret-II}(4) there exists a $Y$ such that
  \begin{enumerate}
  \item $\tau_{<k}\Ss^0$ is a retract of $Y$ and
  \item $Y$ has a resolution $[F_{N},\dots,F_0; Y]$ by connective $A$-modules.
  \end{enumerate}
  Next we apply Lemma \ref{lemm:resolution-prep} to obtain a resolution $[G_{N},\dots,G_0; \tau_{<k}Y]$ where each $G_j$ is an $A$-module in $\Sp_{[0,k]}$.

  At this point we apply Lemma \ref{lemm:new-resolution} to
  $ [G_{N},\dots,G_0; \tau_{<k}Y] $ with $K = k$, $m = m$ to obtain a new resolution
  $ [\dots,H_1,H_0; \tau_{<k}Y] $.
  Each of the $H_j$ is a direct sum of finitely many terms of the form $\tau_{[a,a+m)}G_i$.
  By Lemma \ref{lemm:res-by-B-mod} each of these terms is then a $B$-module,
  thus $H_j$ is a $B$-module as well.
  Finally, we note that $\tau_{<k}\Ss^0$ is a retract of $\tau_{<k}Y$, therefore by Proposition \ref{prop:f-interpret-II}(4) $f_B(k)$ is bounded by the length of the resolution we have produced and Lemma \ref{lemm:length-count} lets us conclude that
  $$ f_B(k) \leq 1 + N + \left\lfloor \frac{k + N}{m} \right\rfloor. $$
\end{proof}


%% file: Appendix-s5-E1.tex
In this subsection we prove the four lemmas used in the proof of \Cref{thm:AB-vl}.
Lemmas \ref{lemm:resolution-prep} and \ref{lemm:res-by-B-mod} follow from standard manipulations of Postnikov towers for $R$-modules. \Cref{lemm:new-resolution} requires the iterated application of several simple maneuvers that modify finite resolutions. After laying out the necessary constructions the proof is straightforward. 

\begin{lem}[{\cite[Proposition 7.1.1.13]{HA}}]
  \label{lemm:truncate-modules}\
  Let $U : \mathrm{LMod}_R \to \Sp$ denote the functor which sends a left $R$-module to its underlying spectrum. Let $\mathrm{LMod}_R^{\geq 0}$ (resp. $\mathrm{LMod}_R^{\leq 0}$) denote the full subcategory of $\mathrm{LMod}_R$ on those left $R$-modules whose underlying spectrum is connective (coconnective). Then, $\mathrm{LMod}_R^{\geq 0}$ and $\mathrm{LMod}_R^{\leq 0}$ determine an accessible $t$-structure on $\mathrm{LMod}_R$ such that
  \begin{enumerate}
  \item $U(\tau_{\geq 0}M) \simeq \tau_{\geq 0}U(M)$,
  \item $U(\tau_{\leq 0}M) \simeq \tau_{\leq 0}U(M)$ and
  \item the natural functor $\pi_0U : \mathrm{LMod}_R^\heartsuit \to \Sp_{(p)}^\heartsuit$ is an equivalence. \footnote{Recall that by convention $\pi_0R \cong \Z_{(p)}$.}
  \end{enumerate}
\end{lem}


\begin{proof}[Proof of Lemma \ref{lemm:res-by-B-mod}]
  It will suffice to prove the lemma in the case where $a=0$.
  We would like to show that the left $B$-module $$ \tau_{<m} ( B \otimes_A M) $$
  is equivalent to $M$. Consider the following diagram of spectra
  \begin{center}
    \begin{tikzcd}
      (\tau_{\geq m}B) \otimes_A M \ar[r] &
      B \otimes_A M \ar[r] &
      (\tau_{<m}B) \otimes_A M \ar[d, equal] \\
      (\tau_{\geq m}A) \otimes_A M \ar[r] &
      A \otimes_A M \ar[r] &
      (\tau_{<m}A) \otimes_A M
    \end{tikzcd}
  \end{center}
  where both rows are cofiber sequences.
  In order to produce a chain of equivalences
  $$ \tau_{<m}(B \otimes_A M) \simeq \tau_{<m}(\tau_{<m}B \otimes_A M) \simeq \tau_{<m}(\tau_{<m}A \otimes_A M) \simeq \tau_{<m}(A \otimes_A M) $$
  it will suffice to show that $(\tau_{\geq m}B) \otimes_A M$ and $(\tau_{\geq m}A) \otimes_A M$ are $m$-connective. This follows from the fact that a relative tensor product of connective modules over a connective ring is connective.
\end{proof}

\begin{proof}[Proof of Lemma \ref{lemm:resolution-prep}]
  We are given a resolution:
  \begin{center}
    \begin{tikzcd}
      F_N \ar[r] & X_{N-1} \ar[r] \ar[d] & \cdots \ar[r] & X_1 \ar[r] \ar[d] & X \ar[d] \\
    & F_{N-1} & & F_1 & F_0
    \end{tikzcd}
  \end{center}
  From this we construct the resolution:
  \begin{center}
    \begin{tikzcd}
      \tau_{\leq M-1}F_N \ar[r] &
      \tau_{\leq M-1}X_{N-1} \ar[r] \ar[d] &
      \cdots \ar[r] &
      \tau_{\leq M-1}X_1 \ar[r] \ar[d] &
      \tau_{\leq M-1}X \ar[d] \\
      & F_{N-1}' & &
      F_1' &
      F_0'
    \end{tikzcd}
  \end{center}

  In order to finish the proof we just need to analyze $F_j'$.
  For $j \neq N$, we can construct the following diagram of spectra where $Y$ and $Z$ are chosen so that each row is a cofiber sequence.
  \begin{center}
    \begin{tikzcd}
      \tau_{\leq M-1}X_{j+1} \ar[r] \ar[d, equal] &
      Y \ar[r] \ar[d] &
      \tau_{\leq M} F_j \ar[d] \ar[r] &
      \tau_{\leq M} \Sigma X_{j+1} \ar[d, equal] \ar[r] &
      \Sigma Y \ar[d] \\
      \tau_{\leq M-1}X_{j+1} \ar[r] \ar[d] &
      \tau_{\leq M-1}X_j \ar[r] \ar[d, equal] &
      F_j' \ar[d] \ar[r] &
      \tau_{\leq M} \Sigma X_{j+1} \ar[d] \ar[r] &
      \tau_{\leq M} \Sigma X_j \ar[d, equal] \\
      Z \ar[r] &
      \tau_{\leq M-1} X_j \ar[r] &
      \tau_{\leq M-1} F_j \ar[r] &
      \Sigma Z \ar[r] &
      \tau_{\leq M} \Sigma X_j
    \end{tikzcd}
  \end{center}
  On long exact sequences of homotopy groups this diagram becomes:
  \begin{center}
    \begin{tikzcd}
      \cdots \ar[r] &
      0 \ar[r] \ar[d] &
      A \ar[r] \ar[d] &
      \pi_M(F_j) \ar[d] \ar[r] &
      \pi_{M-1}(X_{j+1}) \ar[d] \ar[r] &
      \pi_{M-1}(X_j)  \ar[d] \ar[r] &
      \cdots \\
      \cdots \ar[r] &
      0 \ar[r] \ar[d] &
      0 \ar[r] \ar[d] &
      \pi_M(F_j') \ar[d] \ar[r] &
      \pi_{M-1}(X_{j+1}) \ar[d] \ar[r] &
      \pi_{M-1}(X_j) \ar[d] \ar[r] &
      \cdots \\
      \cdots \ar[r] &
      0 \ar[r] &
      0 \ar[r] &
      0 \ar[r] &
      B \ar[r] &
      \pi_{M-1}(X_j) \ar[r] &
      \cdots
    \end{tikzcd}
  \end{center}
  where $A, \pi_{M}(F_j')$ and $B$ are the kernels of the next map in the sequence.
  From this we can read off that the sequence 
  $$ \tau_{\leq M}F_j \to F_j' \to \tau_{\leq M -1}F_j $$
  becomes an equivalence after applying $\tau_{\leq M-1}$ and induces a surjection on $\pi_M$.
  In particular, this tells us that $F_j' \in \Sp^{[0,M]}$.
  What remains is to show that $F_j'$ is an $A$-module.

  In order to do this we recall \cite[Proposition 1.3.15]{BBDG}:
  Let $\mathcal{C}$ be a triangulated category equipped with a left and right complete $t$-structure.
  Suppose we are given an object $P \in \mathcal{C}$
  and a quotient map $\pi_MP \to Q$ in $\mathcal{C}^\heartsuit$.
  Then, there is a unique object $P'$ equipped with a factorization
  $$ \tau_{\leq M}P \to P' \to \tau_{\leq M-1}P $$
  such that $P' \in \mathcal{C}^{\leq M}$ and after applying $\pi_M$ this sequence becomes
  $$ \pi_M(P) \to Q \to 0. $$
  
  Using the existence part of this proposition with $\mathcal{C} = \mathrm{LMod}_R$
  we construct an $A$-module $P$.
  Using the properties of $U$ from \Cref{lemm:truncate-modules} we may apply the uniqueness assertion in the proposition to conclude that $UP \simeq F_j'$.
  
  In the $j=N$ case we have $F_N' := \tau_{\leq M-1}F_j$.
  This objects clearly lives in $\Sp_{[0,M]}$ and is an $A$-module by \Cref{lemm:truncate-modules}.
\end{proof}

Before proceeding with the proof of Lemma \ref{lemm:new-resolution} we introduce several basic constructions which we will need in order to efficiently manipulate resolutions.
The first pair of constructions (which are inverse to each other) codify the process of inserting a resolution into another resolution and extracting a piece of a resolution.

\begin{cnstr}[Compression]  
  Given a resolution
  $$ [\dots, F_j, \dots, F_{j-a}, \dots, F_0; X] $$
  we construct resolutions
  $$ [\dots, F_{j+1}, G, F_{j-a-1}, \dots, F_0; X]\ \ \ \text{ and }\ \ \ [F_j, \dots, F_{j-a}; G]. $$
\end{cnstr}

\begin{proof}
  The desired resolution is given by
  \begin{center}
    \begin{tikzcd}
      \cdots \ar[r] & X_{j+2} \ar[r] \ar[d] & X_{j+1} \ar[r] \ar[d] & X_{j-a} \ar[r] \ar[d] & X_{j-a-1} \ar[r] \ar[d] & \cdots \ar[r] & X \ar[d] \\
      & F_{j+2} & F_{j+1} & G & F_{j-a-1} & & F_0
    \end{tikzcd}
  \end{center}

  The resolution of $G$ is given by 
  \begin{center}
    \begin{tikzcd}
      F_j \ar[r] & X_{j-1}/X_{j+1} \ar[r] \ar[d] & X_{j-2}/X_{j+1} \ar[r] \ar[d] & \cdots \ar[r] & X_{j-a}/X_{j+1} =: G \ar[d] \\
      & F_{j-1} & F_{j-2} & & F_{j-a}
    \end{tikzcd}
  \end{center}
\end{proof}

\begin{cnstr}[Insertion]
  Given a resolution $ [F_n, \dots , F_0; X] $
  and another resolution $ [G_m, \dots, G_0; F_j] $
  we construct a third resolution
  $$ [F_n, \dots, F_{j+1}, G_m, \dots, G_0, F_{j-1}, \dots, F_0; X] $$
\end{cnstr}

\begin{proof}
  We will make our construction by induction on $m$. The $m=0$ case is trivial.

  In the $m=1$ case let $H$ denote the fiber of the composite
  $ X_j \to F_j \to G_0 $.
  Then, we obtain a natural maps $ X_{j+1} \to H \to X_j $
  and the desired resolution is given by
  \begin{center}
    \begin{tikzcd}
      \cdots \ar[r] & X_{j+1} \ar[r] \ar[d] & H \ar[r] \ar[d] & X_j \ar[r] \ar[d] & X_{j-1} \ar[d] \ar[r] & \cdots \ar[r] & X \ar[d] \\
      & F_{j+1} & G_1 & G_0 & F_{j-1} & & F_0 
    \end{tikzcd}
  \end{center}
  
  For the induction step we compress $[G_m, \dots, G_0; F_j]$ into a resolution $[G_m, \dots, G_2, K; F_j]$ insert this new resolution into the given one, then apply insertion again, this time with  $[G_1,G_0; K]$ instead, which finishes the construction.
\end{proof}

When appropriate connectivity hypotheses are satisfied we can use compression and insertion to swap the order of the terms in a resolution.

\begin{lem}[Splitting Lemma]
  When the compression construction is applied with $a=1$, if $F_{j-1}$ is $k$-connective and $F_j$ is $(k-2)$-coconnective, then $G \simeq F_j \oplus F_{j-1}$. 
\end{lem}

\begin{proof}
  The attaching map in the cofiber sequence building $G$ is 0 for connectivity reasons.
\end{proof}

\begin{cnstr}[Swapping]
  Given a resolution
  $$ [F_n, \dots F_j, A \oplus B, F_{j-1}, \dots, F_0; X] $$
  we can construct another resolution
  $$ [F_n, \dots, F_j, A, B, F_{j-1}, \dots, F_0; X] $$
  by inserting the resolution $[A,B; A \oplus B]$ into the given resolution.
\end{cnstr}

Finally, we have a second pair of inverse constructions where we slice off the leading object of a resolution or add a new leading term.

\begin{cnstr}[Slicing]
  Given a resolution $ [F_n, \dots, F_0; X] $
  we will construct another resolution
  $ [F_n, \dots, F_1; X_1] $
  where $X_1$ sits in a cofiber sequence $X_1 \to X \to F_0$.
\end{cnstr}

\begin{proof}
  The desired resolution is given by
  \begin{center}
    \begin{tikzcd}
      \cdots \ar[r] & X_2 \ar[r] \ar[d] & X_1 \ar[d] \\
      & F_2 & F_1
    \end{tikzcd}
  \end{center}
  together with the cofiber sequence
  $X_1 \to X \to F_0$.
\end{proof}

\begin{cnstr}[Appending]
  Given a resolution $ [F_n, \dots, F_0; X] $
  and a cofiber sequence $X \to Y \to A$ 
  we can construct another resolution
  $ [F_n, \dots, F_0, A; Y] $
\end{cnstr}

\begin{proof}
  The desired resolution is given by
  \begin{center}
    \begin{tikzcd}
      \cdots \ar[r] & X_2 \ar[r] \ar[d] & X_1 \ar[d] \ar[r] & X \ar[r] \ar[d] & Y \ar[d] \\
      & F_2 & F_1 & F_0 & A
    \end{tikzcd}
  \end{center}
\end{proof}

\begin{proof}[Proof of \Cref{lemm:new-resolution}]
  We will prove the proposition by induction on $N$.
  For the base case we replace the resolution $[X;X]$ with
  $$[\dots, \tau_{[2m,3m)}X, \tau_{[m,2m)}X, \tau_{[0,m)}X; X]$$
  which is a variant on the Postnikov resolution.
  For the induction step we start by slicing and suspending the given resolution to obtain a new resolution
  $$ [\Sigma F_N, \dots, \Sigma F_1; \Sigma X_1] $$
  Next we apply the $(N-1)$ case of this proposition to this resolution
  $$ \left[ \dots, \left( \bigoplus_{0 \leq i \leq j} \tau_{[(j-i)m-i,(j-i+1)m-i)}(\Sigma F_{i+1}) \right) , \dots , \left( \tau_{[0,m)} (\Sigma F_1) \right) ; \Sigma X_1 \right]. $$
  After desuspending this resolution and appending $X$ to the front we obtain a resolution,
  $$ \left[ \dots, \left( \bigoplus_{1 \leq i \leq j} \tau_{[(j-i)m-i,(j-i+1)m-i)} F_i \right) , \dots , \left( \tau_{[-1,m-1)} F_1 \right), F_0 ; X \right]. $$
  In order to simplify notation we will define
  $$ G_j := \left( \bigoplus_{1 \leq i \leq j} \tau_{[(j-i)m-i,(j-i+1)m-i)} F_i \right) $$
    
  Now we make two observations that will let us finish the proof:\\
  \noindent (1) We're trying to produce a resolution whose terms are $G_j \oplus \tau_{[jm,(j+1)m)}F_0$,\\
  \noindent (2) $G_j$ is $(jm-2)$-coconnective.

  Next, we insert the the same variant of the Postnikov resolution considered in the base case (this time applied to $F_0$). This produces a resolution,
  $$ [\dots, G_2, G_1, \dots, \tau_{[2m,3m)}F_0, \tau_{[m,2m)}F_0, \tau_{[0,m)}F_0; X]$$  
  We now apply the Splitting Lemma and the Swap Construction repeatedly in order to move the terms $\tau_{[am,(a+1)m)}F_0$ to the left until we saturate the coconnectivity from (2). This yields a resolution
  $$ [\dots, \tau_{[2m,3m)}F_0 \oplus G_2, \tau_{[m,2m)}F_0 \oplus G_1, \tau_{[0,m)}F_0; X], $$
  which completes the proof.
\end{proof}

\begin{proof}[Proof of Lemma \ref{lemm:length-count}]
  The last term in the resolution from Lemma \ref{lemm:new-resolution} which contains a truncation of $F_i$ as a summand has $j$ such that
  $$ K \in [(j-i)m-i, (j-i+1)m - i) $$
  Thus, there are $j+1$ terms in the resolution where $j$ is the integer such that
  $$ K \in [(j-N)m-N, (j-N+1)m - N) $$
 From this we may conclude that
 $$ j+1 = 1 + N + \left\lfloor \frac{K+N}{m} \right\rfloor $$
\end{proof}


%% file: Appendix-s2.tex
For clarity of exposition we prove the various parts of \Cref{thm:app-main} as separate lemmas. Before proceeding we summarize each of these parts.
\begin{enumerate}
\item \Cref{thm:app-main}(1) is a direct corollary of \Cref{thm:AB-vl}.
\item \Cref{thm:app-main}(2) is a corollary of \Cref{thm:app-main}(1) and a bound on $\Gamma(k)$ in terms of $f_{\BP\langle 1\rangle}(k)$ due to Davis and Mahowald \cite{DM3} at $p=2$ and Gonz\'alez \cite{GonzalezRegular} at odd primes.
\item \Cref{thm:app-main}(3) is a corollary of the $n=1$ case of \Cref{thm:app-main}(1) combined with \Cref{cor:ANvl}.
\item \Cref{thm:app-main}(4) is proved in the same way as \Cref{thm:app-main}(2) except using \Cref{thm:app-main}(3) instead of \Cref{thm:app-main}(1). 
\end{enumerate}

\begin{cor}[\Cref{thm:app-main}(1)]
  \label{cor:bpn-vl}
  $$ f_{\textup{BP}\langle n \rangle}(k) \leq \frac{1}{|v_{n+1}|}k + \left(1 + \frac{1}{|v_{n+1}|}\right)f_{\textup{BP}}(k) - \frac{1}{|v_{n+1}|} = \frac{k}{|v_{n+1}|} + o(k) $$
\end{cor}

\begin{proof}
    Apply \Cref{thm:AB-vl} to the map of $\mathbb{E}_1$-algebras $\BP \to \BP\langle n \rangle$, which exists by \cite[Corollary 3.2]{Angeltveit}. While the statement of \cite[Corollary 3.2]{Angeltveit} asks for $R$ to be an $\mathbb{E}_\infty$-ring spectrum, its proof only requires that $R$ be $\mathbb{E}_2$. The spectrum $\BP$ admits the structure of an $\mathbb{E}_4$-ring (and therefore $\mathbb{E}_2$-ring) by \cite{BasterraMandell}.
\end{proof}

This corollary used only very coarse information about $\BP\langle n \rangle$.
In fact, the same conclusions hold with $\BP\langle n \rangle$ replaced by $\tau_{<|v_{n+1}|}\BP$.
We believe that the actual vanishing curve for $\BP\langle n \rangle$ has only a constant ``error term''. As such, we make the following conjecture:

\begin{cnj}
  $$ f_{\textup{BP}\langle n \rangle}(k) = \frac{k}{|v_{n+1}|} + O(1). $$
\end{cnj}

\begin{rmk}
    The $n=0$ case of this conjecture is essentially due to Adams and Luilevicius and appeared in the discussion following \Cref{rmk:g-f-comp}. For $n > 0$ this conjecture is open.
\end{rmk}

In order to prove \Cref{thm:app-main}(2) we need a technique which allows us to bound $\Gamma(k)$. This is provided by the following pair of theorems proved by Davis-Mahowald at $p=2$ and Gonz\'alez at $p\neq 2$.
  
\begin{thm}[{\cite[Theorem 5.1]{DM3}}]
  \label{thm:DM-bo}
  Let
  $$\mathbb{S}^0 = S_0 \xleftarrow{f_1} S_1 \xleftarrow{f_2} S_2 \xleftarrow{f_3} \cdots $$
  denote the canonical $bo$-Adams resolution of $\Ss^0$ and suppose we are given $\alpha_s \in \pi_n(S_s)$ such that $\alpha_s$ maps to zero under the composite
  $$\pi_n(S_s) \to \pi_n(S_0) \to \pi_n(L_{K(1)}\Ss) $$
  and $AF(\alpha_s) \geq \epsilon(n,s)$.\footnote{If $\alpha \in \pi_*(X)$, then $AF(\alpha)$ denotes the $\HFp$-Adams filtration of the class $\alpha$.} Then, there exists an $\alpha_{s+1}$ such that $f_s(f_{s+1}(\alpha_{s+1})) = f_s(\alpha_s)$ and $AF(\alpha_{s+1}) \geq AF(\alpha_s) - \delta(n,s)$ where the values of $\epsilon(n,s)$ and $\delta(n,s)$ are given in the table below:
  \begin{center}
    \begin{tabular}{|c||c|c|}\hline
      \text{s}  & $\epsilon(n,s)$ & $\delta(n,s)$ \\\hline\hline
      $0$ & $1$ & $1$ \\\hline
      $1$ & $\max(1, v_2(n+1)-1)$ & $\max(1, v_2(n+1)-1)$ \\\hline
      $2$ & $v_2(n+2)+1$ & $v_2(n+2)+1$ \\\hline
      $\geq 3$ & $2$ & $\begin{cases} 2 & n+s \equiv 0,1,2,4 \pmod 8 \\ 1 & n+s \equiv 3,5,6,7 \pmod 8 \end{cases}$ \\\hline
    \end{tabular}
  \end{center}
  Note that $\alpha_s$ maps to zero in $L_{K(1)}\Ss$ automatically if $s \geq 2$.
\end{thm}

\begin{thm}[{\cite[Theorem 7.5]{GonzalezRegular}}]
  \label{thm:gon-bp1}
  Let
  $$\mathbb{S}^0 = S_0 \xleftarrow{f_1} S_1 \xleftarrow{f_2} S_2 \xleftarrow{f_3} \cdots $$
  denote the canonical $\textup{BP}\langle 1\rangle$-Adams resolution of $\mathbb{S}^0$ and suppose we are given $\alpha_s \in \pi_n(S_s)$ such that $\alpha_s$ maps to zero under the composite
  $$ \pi_n(S_s) \to \pi_n(S_0) \to \pi_n(L_{K(1)}\Ss) $$
  and $AF(\alpha_s) \geq \epsilon(n,s)$. Then, there exists an $\alpha_{s+1}$ such that $f_s(f_{s+1}(\alpha_{s+1})) = f_s(\alpha_s)$ and $AF(\alpha_{s+1}) \geq AF(\alpha_s) - \epsilon(n,s)$ where the values of $\epsilon(n,s)$ are given in the table below:
  \begin{center}
    \begin{tabular}{|c|c|}\hline
      \text{s} & $\epsilon(n,s)$ \\\hline\hline
      $0,1$ & $1$ \\\hline
      $2$ & $1 + \ell(n)$ \\\hline
      $\geq 3$ & $\begin{cases} 2 & n+s \equiv 0 \pmod{q} \\ 1 & n+s \not\equiv 0 \pmod{q}\end{cases}$ \\\hline
    \end{tabular}
  \end{center}
  Note that $\alpha_s$ maps to zero in $L_{K(1)}\Ss$ automatically if $s \geq 2$.
\end{thm}

\begin{cor}
  \label{cor:bp1->h}
  \begin{align*}
    & \Gamma(k) \leq \frac{3}{2} g_{bo}(k) + \frac{3}{2} + \ell(k) & \text{ at }\ p = 2 \\
    \text{and}\ \ \ \ \  & \Gamma(k) \leq \frac{q+1}{q} g_{\textup{BP}\langle 1 \rangle}(k) + 1 - \frac{2}{q} + \ell(k) & \text{ at }\ p \neq 2.
  \end{align*}  
\end{cor}

\begin{proof}
  Suppose $p=2$, then we can read off from Theorem \ref{thm:DM-bo} that
  if $\alpha \in \pi_k(\Ss)$ is a class which maps to zero in $L_{K(1)}\Ss$ and
  \begin{align*}
    AF(\alpha) \geq &1 + \max(1, v_2(k+1)-1) + \big( 1 + v_2(k+2) \big) \\
    &+ (N-3) + \left| \left\{ (k+s) \equiv 0,1,2,4 \pmod 8\ |\ 3 \leq s < N \right\} \right| + 1
  \end{align*}
  then $\alpha$ has $bo$-Adams filtration at least $N$.
  Once $N > g_{bo}(k)$ we automatically have $\alpha = 0$.
  Stated another way, we have
  \begin{align*}
    \Gamma(k)+1 &\leq 1 + \max(1, v_2(n+1)-1) + \big( 1 + v_2(n+2) \big) \\
    &\ \ \ + (g_{bo}(k) - 2) + \left| \left\{ (n+s) \equiv 0,1,2,4 \pmod 8\ |\ 3 \leq s < (g_{bo}(k)+1) \right\} \right| + 1 \\
    &\leq 3 + \begin{cases} v_2(n+1) & n \text{ odd} \\ v_2(n+2) & n \text{ even} \end{cases} + (g_{bo}(k) - 2) + \frac{1}{2}(g_{bo}(k) - 2) + \frac{5}{2} \\
    &\leq \frac{3}{2}g_{bo}(k) + \frac{5}{2} + \ell(k).
\end{align*}

  Suppose $p \neq 2$, then we can read off from Theorem \ref{thm:gon-bp1} that
  if $\alpha \in \pi_k(\Ss)$ is a class which maps to zero in $L_{K(1)}\Ss$ and
  $$ AF(\alpha) \geq 3 + \ell(k) + (N-3) + \left| \left\{ (k+s) \equiv 0 \pmod q\ |\ 3 \leq s < N \right\} \right| $$
  then $\alpha$ has $BP\langle 1 \rangle$-Adams filtration at least $N$.
  Once $N > g_{BP\langle 1 \rangle}(k)$ we automatically have $\alpha = 0$.
  Stated another way, we have
  \begin{align*}
    \Gamma(k)+1 &\leq 3 + \ell(k) + (g_{BP\langle 1 \rangle}(k)-2) \\
    &\ \ \ + \left| \left\{ (k+s) \equiv 0 \pmod q\ |\ 3 \leq s < (g_{BP\langle 1 \rangle}(k)+1) \right\} \right| \\
    &\leq 1 + \ell(k) + g_{BP\langle 1 \rangle}(k) + \frac{1}{q}(g_{BP\langle 1 \rangle}(k) - 2) + 1 \\
    &\leq 2 - \frac{2}{q} + \ell(k) + \frac{q+1}{q} g_{BP\langle 1 \rangle}(k).
  \end{align*}
\end{proof}

\begin{cor}[\Cref{thm:app-main}(2)]
\begin{align*}
  & \Gamma(k) \leq \frac{1}{4}k + \frac{7}{4}f_{\textup{BP}}(k+1) + \ell(k) & \text{ at }\ p = 2 \\
  \text{and}\ \ \ \ \  & \Gamma(k) \leq \frac{q+1}{q|v_2|}k + \frac{(q+1)(|v_2|+1)}{q|v_2|}f_{\textup{BP}}(k+1) - \frac{3}{q} + \ell(k) & \text{ at }\ p \neq 2. 
\end{align*}
\end{cor}

\begin{proof}
  At $p = 2$, using \Cref{cor:bp1->h}, \Cref{rmk:g->f}, \Cref{thm:AB-vl} and \Cref{thm:app-main}(1) we obtain:
  \begin{align*}
    \Gamma(k)
    &\leq \frac{3}{2} g_{bo}(k) + \frac{3}{2} + \ell(k) 
    \leq \frac{3}{2} (f_{bo}(k+1) - 1) + \frac{3}{2} + \ell(k) \\
    &\leq \frac{3}{2} f_{\BP \langle 1 \rangle}(k+1) + \ell(k) 
    \leq \frac{3}{2} \left( \frac{1}{6}(k+1) + \frac{7}{6} f_{\BP}(k+1) - \frac{1}{6} \right) + \ell(k) \\
    &\leq \frac{1}{4}k  + \frac{7}{4} f_{\BP}(k+1) + \ell(k).
  \end{align*}
  
  At $p \neq 2$,
  using \Cref{cor:bp1->h}, \Cref{rmk:g->f} and \Cref{thm:app-main}(1) we obtain:
  \begin{align*}
    \Gamma(k)
    &\leq \frac{q+1}{q} g_{\textup{BP}\langle 1 \rangle}(k) + 1 - \frac{2}{q} + \ell(k) \\
    &\leq \frac{q+1}{q} \left( f_{\textup{BP}\langle 1 \rangle}(k+1) - 1 \right) + 1 - \frac{2}{q} + \ell(k) \\
    &\leq \frac{q+1}{q} \left( \frac{1}{|v_{2}|}(k+1) + \frac{|v_2| + 1}{|v_{2}|} f_{\BP}(k+1) - \frac{1}{|v_{2}|} \right) - \frac{3}{q} + \ell(k) \\
    &\leq \frac{q+1}{q|v_{2}|}k + \frac{(q+1)(|v_2| + 1)}{q|v_{2}|} f_{\BP}(k+1) - \frac{3}{q} + \ell(k). 
  \end{align*}
\end{proof}

\begin{cor}[\Cref{thm:app-main}(3)]
  For each odd prime,
  $$ f_{\textup{BP}\langle 1 \rangle}(k) \leq \frac{p+2}{2(p^3 - p - 1)}k + 2p^2 - 4p + 11. $$
\end{cor}

\begin{proof}
  We specialize \Cref{cor:bpn-vl} to the $n=1$ case and plug in the bound on $f_\BP$ obtained in \Cref{cor:ANvl}.
  \begin{align}\label{eqn:bp1}
    \nonumber f_{\textup{BP}\langle 1 \rangle}(k)
    &\leq \frac{1}{|v_2|}k + \frac{1 + |v_2|}{|v_2|} f_\BP(k) - \frac{1}{|v_2|} \\
    \nonumber &\leq \frac{1}{|v_2|}k + \frac{1 + |v_2|}{|v_2|} \left( \frac{1}{p^3 - p -1}k + 2p^2 - 4p + 10 - \frac{2p^2+2p-9}{p^3 - p -1} \right) - \frac{1}{|v_2|} \\
      &\leq \frac{p+2}{2(p^3 - p -1)}k + 2p^2 - 4p + 11 - \frac{2p - 6}{p^2 -1} - \frac{(2p^2+2p-9)(2p^2 -1)}{(2p^2 -2)(p^3 - p -1)} \\
    &< \frac{p+2}{2(p^3 - p -1)}k + 2p^2 - 4p + 11
  \end{align}
\end{proof}

\begin{cor}[\Cref{thm:app-main}(4)]
  For $p=3$,
  $$ \Gamma(k) \leq \frac{25}{184}k + 19 + \frac{1133}{1472} + \ell(k), $$
  and for $p \geq 5$,
  $$ \Gamma(k) \leq \frac{(2p-1)(p+2)}{4(p-1)(p^3 - p -1)}k + 2p^2 - 3p + 11 + \ell(k). $$
\end{cor}

\begin{proof}
  In the proof of \Cref{thm:app-main}(2) we obtained
  $$ \Gamma(k) \leq \frac{q+1}{q} (f_{\BP\langle 1 \rangle}(k+1) - 1) + 1 - \frac{2}{q} + \ell(k) $$
  for each odd prime. Using the intermediate bound on $f_{\BP\langle 1 \rangle}(k)$ from \Cref{eqn:bp1} we obtain a bound on $\Gamma(k)$ which simplifies to
\begin{align*}
  \Gamma(k)
    &\leq \frac{(2p-1)(p+2)}{4(p-1)(p^3 - p -1)}k + 2p^2 - 3p + 10 \\ &+ \frac{-4p^5 + 26p^4 + 19p^3 - 52p^2 - 27p + 35}{(2p-2)(2p^2 - 2)(p^3 - p -1)} + \ell(k).
\end{align*}
For all $p \geq 5$ the second to last term is less than $1$.
\end{proof}
